\documentclass[12pt]{amsart}
\usepackage{amsmath,amssymb,amsfonts,comment,verbatim,amsthm}
\usepackage{enumerate}

\newtheorem{proposition}{Proposition}
\newtheorem{theorem}{Theorem}
\newtheorem{lemma}{Lemma}

\newtheorem{corollary}{Corollary}

\theoremstyle{remark}
\newtheorem{remark}{Remark}
\numberwithin{proposition}{section}
\numberwithin{lemma}{section}
\numberwithin{equation}{section}
\begin{document}

\title{Invariance of the Gibbs measure for the periodic quartic gKdV}
\author{Geordie Richards}
\address{Institute for Mathematics and its Applications, Minneapolis, MN, USA}
\email{geordie@ima.umn.edu}
\thanks{This is part of the author's Ph.D. thesis recently completed at the University of Toronto advised by J. Colliander and T. Oh.  The author also thanks J. Quastel and N. Tzvetkov for suggestions related to this work.}
\maketitle

\begin{abstract}
We prove invariance of the Gibbs measure for the (gauge transformed) periodic quartic gKdV.  The Gibbs measure is supported on $H^{s}(\mathbb{T})$ for $s<\frac{1}{2}$, and the quartic gKdV is analytically ill-posed in this range.  In order to consider the flow in the support of the Gibbs measure, we combine a probabilistic argument and the second iteration and construct local-in-time solutions to the (gauge transformed) quartic gKdV almost surely in the support of the Gibbs measure.  Then, we use Bourgain's idea to extend these local solutions to global solutions, and prove the invariance of the Gibbs measure under the flow.  Finally, Inverting the gauge, we construct almost sure global solutions to the (ungauged) quartic gKdV below $H^{\frac{1}{2}}(\mathbb{T})$.
\end{abstract}

\tableofcontents

\section{Introduction}

In this paper, we consider the periodic quartic generalized Korteweg-de Vries (gKdV) equation
\begin{align}
\left\{
\begin{array}{ll} \partial_{t}u + \partial_{x}^{3}u = \frac{1}{4}\partial_{x}(u^{4})  , \ \
 x\in \mathbb{T}, t\in\mathbb{R},
\\
u(x,0) = u_{0}(x),
\end{array} \right.
\label{Eqn:gKdV}
\end{align}
where $u$ is a real-valued function on $\mathbb{T}\times \mathbb{R}$ with $\mathbb{T}=[0,2\pi)$ and the mean of $u_0$ is zero.  From conservation of the mean, it follows that the solution $u(t)$ of \eqref{Eqn:gKdV} (if it exists) has spatial mean zero for all $t\in\mathbb{R}$.  Throughout this paper, we assume that the spatial mean $\hat{u}(0,t)$ is always zero for all $t\in \mathbb{R}$.

The system \eqref{Eqn:gKdV} is a special case of the gKdV equation
\begin{align}
\left\{
\begin{array}{ll} \partial_{t}u + \partial_{x}^{3}u = \frac{1}{p}\partial_{x}(u^{p})  , \ \
x\in \mathbb{T}, t\in\mathbb{R},  p\geq2 \  \text{integer},
\\
u(x,0) = u_{0}(x).
\end{array} \right.
\label{Eqn:gKdV-p}
\end{align}
The KdV (\eqref{Eqn:gKdV-p} with $p=2$) is a canonical model for dispersive waves in physics.  This equation has a rich history and the related literature is extensive.  The modified KdV (mKdV, $p=3$) has also appeared in physics, and it is closely related to KdV through the Miura transform.  Higher power gKdV equations ($p\geq 4$) have been studied mainly by mathematicians; there is interest in exploring the balance of a stronger nonlinearity with dispersion.

\smallskip

The system \eqref{Eqn:gKdV-p} has a conserved (if it is finite) Hamiltonian given by
\begin{align*}
H(u):= \frac{1}{2}\int_{\mathbb{T}} u_{x}^{2}dx +
\frac{1}{p(p+1)}\int_{\mathbb{T}} u^{p+1}dx.
\end{align*}
Then \eqref{Eqn:gKdV-p} can be reformulated as
\begin{align}
\partial_{t} u =\partial_{x}\frac{\partial H}{\partial u},
\label{Eqn:HAM}
\end{align}
where $\frac{\partial H}{\partial u}$ is the Fr\'{e}chet derivative with respect to the $L^2(\mathbb{T})$-inner product\footnote{This is at least formally correct, for the rigorous definition of gKdV as a Hamiltonian system with Poisson structure $J=\frac{\partial}{\partial x}$ on the Sobolev space  $H^{-\frac{1}{2}}_{0}(\mathbb{T})$ (mean-zero functions) equipped with a compatible symplectic form, see \cite{KUK}.}.  This Hamiltonian structure leads to a natural question: is the Gibbs measure $``d\mu = e^{-H(u)}du"$ invariant under the flow of \eqref{Eqn:gKdV-p}?


The Gibbs measure $\mu$ for \eqref{Eqn:gKdV-p}, first constructed in Lebowitz-Rose-Speer \cite{LRS}, is supported on $H^{\frac{1}{2}-}(\mathbb{T}) = \bigcap_{s<\frac{1}{2}}H^{s}(\mathbb{T})$ (for $p\leq 5$ only, with appropriate restrictions).  To ask the question of its invariance under the flow, one needs to prove that the evolution of \eqref{Eqn:gKdV-p} is well-defined (globally-in-time) for initial data in the support of $\mu$.  For this step, it suffices to prove global well-posedness of \eqref{Eqn:gKdV-p} in $H^{s}(\mathbb{T})$ for some $s<\frac{1}{2}$.

\smallskip

Let us recall some well-posedness results for \eqref{Eqn:gKdV} and \eqref{Eqn:gKdV-p}.  In \cite{B1}, Bourgain introduced a weighted space-time Sobolev space $X^{s,b}$ whose norm is given by
\begin{align}
\|u\|_{X^{s,b}(\mathbb{T}\times\mathbb{R})}=\|\langle n\rangle^{s}\langle \tau-n^{3}\rangle^{b}\hat{u}(n,\tau)\|_{L^{2}_{n,\tau}(\mathbb{Z}\times\mathbb{R})}.
\label{Eqn:intro-1}
\end{align}
He used a fixed point argument to prove local well-posedness (LWP) of KdV (\eqref{Eqn:gKdV-p} with $p=2$) in $L^2(\mathbb{T})$, and automatically obtained global well-posedness (GWP) by conservation of the $L^2(\mathbb{T})$-norm.

The study of well-posedness for the periodic quartic gKdV \eqref{Eqn:gKdV} was also initiated in \cite{B1}; a fixed point argument was used to establish LWP in $H^{s}(\mathbb{T})$, for $s>\frac{3}{2}$.  This was improved to LWP in $H^{s}(\mathbb{T})$ for $s\geq 1$ by Staffilani \cite{S1}, and then to $s\geq \frac{1}{2}$ by Colliander-Keel-Staffilani-Takaoka-Tao \cite{CKSTT1}.  In \cite{CKSTT1}, they also proved analytic ill-posedness of \eqref{Eqn:gKdV} below $H^{\frac{1}{2}}(\mathbb{T})$.  That is, the data-to-solution map for \eqref{Eqn:gKdV} is not analytic in $H^{s}(\mathbb{T})$ for $s<\frac{1}{2}$.  In fact, it is not $C^{4}$ (see also \cite{B2}).


\smallskip

Bourgain \cite{B3} rigorously proved the invariance of the Gibbs measure for KdV and mKdV, but to the knowledge of the author, this problem remains open for \eqref{Eqn:gKdV-p} with $p=4$ and $p=5$.  For KdV and mKdV, he used a deterministic fixed point argument to establish well-posedness in the support of the Gibbs measure.  Recall that the evolution of KdV is well-defined for all $u_{0}\in L^{2}(\mathbb{T})$ \cite{B1}, so it is certainly well-defined in $H^{\frac{1}{2}-}(\mathbb{T})$ (globally-in-time). 
For mKdV, he proved LWP below $H^{\frac{1}{2}}(\mathbb{T})$ (in a modified Besov-type space), but he could not use conservation of the $L^2(\mathbb{T})$-norm to extend solutions globally-in-time.

The main new idea implemented in \cite{B3} was to use the invariance of the Gibbs measure under the flow of the finite-dimensional system of ODEs obtained by the projecting mKdV\footnote{In \cite{B3}, Bourgain also proved the invariance of the Gibbs measure for periodic nonlinear Schr\"{o}dinger equations, but we will focus on \eqref{Eqn:gKdV-p} in this discussion.} to the first $N>0$ modes of the trigonometric basis (and an approximation argument) as a \textit{substitute for a conservation law}, extending the local solutions of mKdV to global solutions (almost surely in the support of the Gibbs measure), and subsequently proving the invariance of the (infinite-dimensional) Gibbs measure $\mu$ under the flow.  In this way, Bourgain showed that an invariant Gibbs for a Hamiltonian PDE provides two benefits: (i) dynamical information about the flow (for example, an evolutionary PDE equipped with an invariant measure can be regarded as an infinite-dimensional dynamical system, and it follows from the Poincar\'{e} recurrence theorem that almost all points of the phase space are stable according to Poisson) and (ii) a tool for extending local solutions to global solutions at low regularities in space (in the support of the Gibbs measure).

\smallskip

We are interested in proving the invariance of the Gibbs measure under the flow of \eqref{Eqn:gKdV}.  Following the strategy developed in \cite{B3}, the crucial ingredient is local well-posedness (and good approximation to the finite-dimensional ODEs) in the support of the Gibbs measure.  Unfortunately, the $C^4$-failure of the data-to-solution map below $H^{\frac{1}{2}}(\mathbb{T})$ \cite{CKSTT1} indicates that one \textit{cannot use the contraction mapping principle} to establish LWP of \eqref{Eqn:gKdV} in $H^{s}(\mathbb{T})$ for $s<\frac{1}{2}$, as this necessitates analyticity of the data-to-solution map.  However, to establish local-in-time dynamics for \eqref{Eqn:gKdV} in the support of the Gibbs measure, it suffices to prove something weaker: that \eqref{Eqn:gKdV} is locally well-posed \textit{almost surely} with randomized initial data given by
\begin{align}
u_{0,\omega}(x)=\sum_{n\in\mathbb{Z}\setminus\{0\}}
\frac{g_{n}(\omega)}{|n|}e^{inx},
\label{Eqn:initialdata}
\end{align}
where $\{g_{n}\}_{n=1}^{\infty}$ is a sequence of complex-valued Gaussian random variables of mean 0 and variance 1 on a probability space $(\Omega,\mathcal{F},\mathbb{P})$, and $g_{-n}=\overline{g_{n}}$ (in order for $u_{0,\omega}$ to be real-valued).  The expression \eqref{Eqn:initialdata} represents a typical element in the support of the Gaussian part of the Gibbs measure, also known as the Wiener measure (see \eqref{Eqn:wiener} below).

\smallskip

The analysis of well-posedness for \eqref{Eqn:gKdV} is simplified by a gauge transformation.  This transformation preserves the initial data, and it is invertible.  A function $u$ satisfies \eqref{Eqn:gKdV} if and only if its gauge transformation $v:=\mathcal{G}(u)$ (see \eqref{Eqn:gaugetransform-a} below) satisfies
\begin{align}
\left\{
\begin{array}{ll} \partial_{t}v + \partial_{x}^{3}v = \frac{1}{4}\mathbb{P}(v^{3})\partial_{x}v  , \ \
 x\in \mathbb{T}, t\in\mathbb{R},
\\
v(x,0) = u_{0}(x),
\end{array} \right.
\label{Eqn:gKdV-zeromean}
\end{align}
where $\mathbb{P}(u)=u-\frac{1}{2\pi}\int_{\mathbb{T}}udx$ is the projection to functions with mean zero.  The analysis of well-posedness for \eqref{Eqn:gKdV-zeromean} is simpler than for \eqref{Eqn:gKdV}, but the data-to-solution map still fails to be $C^4$ below $H^{\frac{1}{2}}(\mathbb{T})$ \cite{CKSTT1}.

\smallskip

To properly state our results, we need one more definition. Let $\Phi^N(t)$ denote the flow map of the finite-dimensional system of ODEs obtained by projecting \eqref{Eqn:gKdV-zeromean} to the first $N>0$ modes of the trigonometric basis:
\begin{align}
\left\{
\begin{array}{ll} \partial_{t}u^{N} + \partial_{x}^{3}u^{N} = \mathbb{P}_N(\mathbb{P}((u^{N})^{3})\partial_{x}u^{N})  , \ \
 x\in \mathbb{T}, t\in\mathbb{R},
\\
u^{N}(x,0) = \mathbb{P}_{N}(u_{0}(x)),  \ \ u_{0} \ \text{mean zero.}
\end{array} \right.
\label{Eqn:gKdV-ZMN}
\end{align}
Here $\mathbb{P}_{N}$ denotes Dirichlet projection to $E_{N}=\text{span}\{\sin(nx),\cos(nx):1\leq n\leq N\}$.

\smallskip

In this paper, we exhibit \textit{nonlinear smoothing when the initial data are randomized} (according to \eqref{Eqn:initialdata}), and use this to prove that \eqref{Eqn:gKdV-zeromean} is locally well-posed almost surely in $H^{\frac{1}{2}-}(\mathbb{T})$.  Indeed, the first theorem we obtain is almost sure local well-posedness of \eqref{Eqn:gKdV-zeromean} with initial data given by \eqref{Eqn:initialdata}, and a good approximation to the dynamics of \eqref{Eqn:gKdV-ZMN}.
In the statement below, $S(t):=e^{it\partial_{x}^{3}}$ is the evolution operator for the linear part of gKdV.

\begin{theorem}[Almost sure local well-posedness]
The gauge-transformed periodic quartic gKdV \eqref{Eqn:gKdV-zeromean} is locally well-posed almost surely with randomized data $u_{0,\omega}$ (given by \eqref{Eqn:initialdata}).  More precisely, for all $0<\delta_{1}<\delta$, with $\delta$ sufficiently small, there exists $0<\beta<\delta-\delta_1$, and $c>0$ such that for each $0<T\ll 1$, there is a set $\Omega_{T}\in\mathcal{F}$ with the following properties:
\begin{enumerate}[(i)]
\item The complemental measure of $\Omega_{T}$ is small.  More precisely, we have $$P(\Omega_{T}^{c})= \rho\circ u_{0}(\Omega_{T}^{c})<e^{-\frac{c}{T^{\beta}}},$$ where $\rho$ is the Wiener measure (see \eqref{Eqn:wiener} below), and the initial data (given by \eqref{Eqn:initialdata}) is viewed as a map $u_{0}:\Omega\rightarrow H^{1/2-}(\mathbb{T})$.

\item For each $\omega\in\Omega_{T}$ there exists a solution $u$ to \eqref{Eqn:gKdV-zeromean} with data $u_{0,\omega}$ satisfying
\begin{align*}
u \in S(t)u_{0,\omega} + C([0,T];H^{1/2+\delta}(\mathbb{T}))\subset
C([0,T];H^{1/2-}(\mathbb{T})).
\end{align*}

\item This solution is unique in
$\displaystyle
\{S(t)u_{0,\omega} + B_{K}\}$, for some $K>0$,
where $B_{K}$ denotes a ball of radius $K$ in the space $X^{\frac{1}{2}+\delta,\frac{1}{2}-\delta}_{T}$.

\item The solution $u$ depends continuously on the initial data, in the sense that, for each $\omega\in\Omega_{T}$, the solution map $$\Phi:
    \Big\{u_{0,\omega} + \{\|\cdot\|_{H^{\frac{1}{2}+\delta}}\leq R\}\Big\}\rightarrow  \Big\{S(t)u_{0,\omega} + \{\|\cdot\|_{C([0,T];H^{\frac{1}{2}+\delta})} \leq \tilde{R}\}
    \Big\}$$ is well-defined and Lipschitz, for some fixed $R,\tilde{R}\sim 1$.

\item The solution $u$ is well-approximated by the solution of \eqref{Eqn:gKdV-ZMN}.  More precisely,
    \begin{align}
    \|u-S(t)u_{0,\omega}
    -(\Phi^{N}(t)-S(t))\mathbb{P}_{N}u_{0,\omega}
    \|_{C([0,T];H^{\frac{1}{2}+\delta_1})}
    \lesssim {N}^{-\beta}.
    \label{Eqn:finapp}
    \end{align}

\end{enumerate}
\label{Thm:LWP}
\end{theorem}

\noindent For the definition of the $X^{s,b}_{T}$ space, see Section \ref{Sec:lin} below.

Following the method developed in \cite{B3}, we use the invariance of finite-dimensional Gibbs measures under the flow of \eqref{Eqn:gKdV-ZMN} and an approximation argument, to extend the local solutions of \eqref{Eqn:gKdV-zeromean} (obtained from Theorem \ref{Thm:LWP}) to global solutions, almost surely, and to prove the invariance of the Gibbs measure under the flow.
\begin{theorem}[Invariance of the Gibbs measure]
The gauge-transformed periodic quartic gKdV \eqref{Eqn:gKdV-zeromean} is globally well-posed almost surely with randomized data $u_{0,\omega}$ (given by \eqref{Eqn:initialdata}).  More precisely, for $\delta_{2}>0$ sufficiently small, it holds that given any $T>0$, for almost every $\omega\in\Omega$, there is a (unique) solution $u$ to \eqref{Eqn:gKdV-zeromean} with data $u_{0,\omega}$ (given by \eqref{Eqn:initialdata}) satisfying
\begin{align*}
u \in S(t)u_{0,\omega} + C([0,T];H^{1/2+\delta_2}(\mathbb{T}))\subset C([0,T];H^{1/2-}(\mathbb{T})).
\end{align*}
Furthermore, the Gibbs measure $\mu$ (given by \eqref{Eqn:gibbs} below) is invariant under the flow.
\label{Thm:GWP}
\end{theorem}
\noindent By inverting the Gauge transformation, we obtain the following corollary.
\begin{corollary}[Almost sure global well-posedness]
The periodic quartic gKdV \eqref{Eqn:gKdV} is globally well-posed almost surely in $H^{1/2-}(\mathbb{T})$.  More precisely, given any $T>0$, for almost every $\omega\in\Omega$, there exists a solution $u$ to \eqref{Eqn:gKdV} for $t\in[0,T]$ with randomized data $u_{0,\omega}$ (given by \eqref{Eqn:initialdata}).
\label{Cor:GWP}
\end{corollary}

\begin{remark}
In terms of global theory, GWP of \eqref{Eqn:gKdV} in $H^{s}(\mathbb{T})$ for $s>\frac{5}{6}$ was established in \cite{CKSTT1} using the $I$-method.  This is mentioned to emphasize that, to the knowledge of the author, Corollary \ref{Cor:GWP} is the first result to provide global-in-time solutions to \eqref{Eqn:gKdV} below $H^{\frac{5}{6}}(\mathbb{T})$.  We further note that these solutions evolve from data at a spatial regularity where even local theory is unavailable at present (below $H^{\frac{1}{2}}(\mathbb{T})$).
\end{remark}
\begin{remark}
The solution of \eqref{Eqn:gKdV-zeromean} produced by Theorem \eqref{Thm:GWP} is unique in a mild sense only.  For the technical statement, see Remark \ref{Rem:un1} in Section \ref{Sec:GWP}.  For the solution of \eqref{Eqn:gKdV} produced by Corollary \ref{Cor:GWP}, we have an existence result only.
\end{remark}
\begin{remark}
By composing with a modified and time-dependent gauge transformation $\widetilde{\mathcal{G}}=\widetilde{\mathcal{G}}_{t}$, we can obtain a time-dependent measure $\nu_{t} := \mu \circ \widetilde{\mathcal{G}}_{t}$, supported on $ H^{s}(\mathbb{T})$ for $s<\frac{1}{2}$, which (due to Theorem \ref{Thm:GWP}) satisfies $\Psi(t)^{*}\nu_{t} = \mu$ for each $t\geq 0$, where $\Psi(t)$ is the evolution operator for \eqref{Eqn:gKdV} (well-defined in the support of the Gibbs measure by Corollary \ref{Cor:GWP}).  This leads to a natural question for future investigation: how is the time-dependent measure $\nu_{t}=\mu \circ \widetilde{\mathcal{G}}_{t}$ related to the Gibbs measure $\mu$?  Do we in fact have invariance of the Gibbs measure for the ungauged quartic gKdV \eqref{Eqn:gKdV}?  This type of issue was recently explored for the periodic derivative NLS \cite{NRSS}.
\end{remark}





\noindent For the remainder of the introduction we provide more background on this problem, then outline the methods involved and the challenges confronted in the proofs of Theorem \ref{Thm:LWP} and Theorem \ref{Thm:GWP}.


\subsection{Background}

Lebowitz-Rose-Speer \cite{LRS} initiated the study of invariant Gibbs measures for Hamiltonian PDEs.  They constructed the Gibbs measure as a weighted Wiener measure.  Recall that the Wiener measure\footnote{This is the mean zero Wiener measure, but we restrict attention to measures, data, and solutions with spatial mean zero throughout this paper, and will often omit the pre-fix ``mean zero".}, $\rho$, is the probability measure supported on $\bigcap_{s<\frac{1}{2}}H^{s}(\mathbb{T})$ with density
\begin{align}
d\rho = Z_{0}^{-1}e^{-\frac{1}{2}\int u_{x}^{2}dx}\prod_{x\in \mathbb{T}}du(x), \ \ u \ \text{mean zero}.
\label{Eqn:wiener}
\end{align}
This is a purely formal expression, but it provides intuition.  We can in fact define $\rho$ as the weak limit of a sequence of finite-dimensional Wiener measures $\rho_{N}$.  Each $\rho_{N}$ is the probability measure on $\mathbb{C}^{N}$ (the space of Fourier coefficients) with density
\begin{align}
d\rho_{N} = Z_{N}^{-1}
e^{-\frac{1}{2}\sum_{0<n\leq N}|n|^{2}|\widehat{u}_{n}|^{2}}
\prod_{0< n\leq N}d\widehat{u}_{n}, \ \ \ ,
\label{Eqn:wiener2}
\end{align}
pushed forward to $E_{N}=\text{span}\{\sin(nx),\cos(nx):0<|n|\leq N\}$ by the map
\begin{align}
\{\widehat{u}_{n}\}_{0< n \leq N}\longmapsto
\sum_{0<|n|\leq N}\widehat{u}_{n}e^{inx},
\ \ \ \text{with}\ \hat{u}_{-n}= \overline{\hat{u}_{n}},
\label{Eqn:pushfor}
\end{align}
and then by extension to $\bigcap_{s<\frac{1}{2}}H^{s}(\mathbb{T})$.
In the expressions \eqref{Eqn:wiener} and \eqref{Eqn:wiener2}, $Z_{0}$ and $Z_N$ are normalizing constants, and $\prod_{0< n \leq N}d\widehat{u}_{n}$ denotes the Lebesgue measure on $\mathbb{C}^{N}$.  We can interpret $\rho_{N}$ and $\rho$ as the probability measures on $\bigcap_{s<\frac{1}{2}}H^{s}(\mathbb{T})$ induced by the maps
$\omega \mapsto \sum_{0<|n|\leq N}\frac{g_{n}(\omega)}{|n|}$ and $\omega \mapsto \sum_{n\in \mathbb{Z}\setminus{0}}\frac{g_{n}(\omega)}{|n|}$, respectively, where $\{g_{n}\}_{n=1}^{\infty}$ is a sequence of complex-valued Gaussian random variables of mean 0 and variance 1 on a probability space $(\Omega,\mathcal{F},\mathbb{P})$, and $g_{-n}=\overline{g_{n}}$.

In \cite{LRS}, it was shown that the \textit{Gibbs measure}\footnote{The Gibbs measure was constructed for NLS in \cite{LRS}, but the same method applies to gKdV, see \cite{B3}.}, given by
\begin{align}
d\mu &= {Z_{0}}^{-1}\chi_{\{\|u\|_{2}\leq B\}}e^{-H(u)}\prod_{x\in \mathbb{T}}du(x) \notag \\
&= {Z_{0}}^{-1}\chi_{\{\|u\|_{2}\leq B\}}e^{-\frac{1}{2}\int_{\mathbb{T}} u_{x}^{2}dx -
\frac{1}{p(p+1)}\int_{\mathbb{T}} u^{p+1}dx}\prod_{x\in \mathbb{T}}du(x) \notag \\
&:= \chi_{\{\|u\|_{2}\leq B\}}e^{-\frac{1}{p(p+1)}\int_{\mathbb{T}}u^{p+1}dx}d\rho,
\label{Eqn:gibbs}
\end{align}
is a measure (for $p\leq 5$ only, and with restrictions on $B$ for $p=5$) that is absolutely continuous with respect to the Wiener measure $\rho$.  That is, Lebowitz-Rose-Speer provided a definition for the Gibbs measure $\mu$ for \eqref{Eqn:gKdV-p}.

\smallskip

Let us briefly explain why Hamiltonian structure (e.g. the reformulation \eqref{Eqn:HAM} of \eqref{Eqn:gKdV}) provokes the question of Gibbs measure invariance.  Consider the finite-dimensional Hamiltonian system given by
\begin{align}
\left\{
\begin{array}{ll} \partial_{t}u^{N} + \partial_{x}^{3}u^{N} = \mathbb{P}_{N}(\frac{1}{p}\partial_{x}((u^{N})^{p}))  , \ \
x\in \mathbb{T}, t\in\mathbb{R},
\\
u^{N}(0,x) = \mathbb{P}_{N}(u_{0}(x)),  \ \ u_{0} \ \text{mean zero.}
\end{array} \right.
\label{Eqn:gKdV-N}
\end{align}
where $\mathbb{P}_{N}$ denotes Dirichlet projection to $E_{N}$.
The flow of \eqref{Eqn:gKdV-N} leaves the following quantities invariant:
\begin{enumerate}[(i)]
\item The Hamiltonian
    $H(u^{N})$.

\item The $L^{2}$-norm
    $\|u^N\|_{L^{2}}=(\sum_{ 0< |n| \leq N}|\widehat{u}_{n}|^{2})^{\frac{1}{2}}$.

\item The Lebesgue measure $\prod_{0< n \leq N}d\widehat{u}_{n}$ on $\mathbb{C}^{N}$ pushed forward to $E_{N}$ by the map \eqref{Eqn:pushfor} (by Liouville's Theorem).
\end{enumerate}
The system \eqref{Eqn:gKdV-N} therefore preserves the finite-dimensional Gibbs measure $\mu_{N}$, defined as the push-forward to $E_{N}$ of the measure on $\mathbb{C}^{N}$ with density
\begin{align}
d\mu_{N}=\frac{1}{Z_{N}}\chi_{\{\sum_{0< |n| \leq N}|\widehat{u}_{n}|^{2}\leq B\}}e^{-H(\sum_{0<|n|\leq N}\hat{u}_{n}e^{inx})}
\prod_{0< n \leq N}d\widehat{u}_{n}, \ \ \text{with}\ \hat{u}_{-n}=\overline{\hat{u}_{n}}.
\label{Eqn:intro-6}
\end{align}
Taking $N\rightarrow \infty$, comparing \eqref{Eqn:wiener}, \eqref{Eqn:gibbs} and \eqref{Eqn:intro-6}, we can interpret the Gibbs measure $\mu$ as the limit of the truncated Gibbs measures $\mu_{N}$ (for a rigorous formulation of these convergence results, see Section \ref{Sec:GWP}).  By comparison with finite-dimensions, the Gibbs measure \eqref{Eqn:gibbs} is formally invariant under the flow of \eqref{Eqn:gKdV}.

\subsection{Nonlinear smoothing for the second iteration}

As discussed above, we can construct the Gibbs measure for the quartic gKdV \eqref{Eqn:gKdV} as in \cite{LRS,B3}.  However, the Gibbs measure is supported below $H^{\frac{1}{2}}(\mathbb{T})$, and local well-posedness of the quartic gKdV (both gauged and ungauged) cannot be established in $H^{s}(\mathbb{T})$ for $s<\frac{1}{2}$ by applying the contraction principle to an equivalent integral equation; the data-to-solution map is not $C^4$ \cite{CKSTT1}.

The same issue was confronted by Bourgain in \cite{B4}.  He considered the Wick-ordered cubic NLS on $\mathbb{T}^{2}$ with randomized data given by
\begin{align}
\tilde{u}_{0,\omega}(x)
=\sum_{n\in\mathbb{Z}^{2}}\frac{g_{n}(\omega)}{\sqrt{1+|n|^2}}e^{in\cdot x},
\label{Eqn:initialdata2}
\end{align}
where $\{g_{n}\}_{n\in\mathbb{Z}^{2}}$ is a collection of complex-valued independent and identically distributed Gaussian random variables of mean 0 and variance 1 on a probability space $(\Omega,\mathcal{F},\mathbb{P})$.  The data \eqref{Eqn:initialdata2} represents a typical element in the support of the 2-dimensional Wiener measure (the Gaussian part of the Gibbs measure for 2-dimensional Wick-ordered cubic NLS).

Bourgain quantified the nonlinear smoothing effect by proving that, with high probability, the nonlinear part of the solution to the Wick-ordered cubic NLS lies in a smoother space - $C([0,T];H^{s}(\mathbb{T}^{2}))$ for some $s>0$ - than the linear evolution of the randomized data, which, in contrast, almost surely stays below $L^{2}(\mathbb{T}^{2})$ for all time.  More precisely, for all $T>0$ sufficiently small, he constructed a set $\Omega_{T}\subset \Omega$ corresponding to ``good" randomized data $\tilde{u}_{0,\omega}$, such that $\Omega_{T}$ is exponentially likely as a function of $T\downarrow 0$, and such that for each $\omega\in\Omega_{T}$, he could prove local existence and uniqueness of the solution to the Wick-ordered cubic NLS with data $\tilde{u}_{0,\omega}$ for $t\in[0,T]$ by performing a contraction argument in the space $\{e^{it\Delta}\tilde{u}_{0,\omega} +  B\}$, where $B$ is a ball in $Z^{s,\frac{1}{2}}_{T}\subset C([0,T];H^{s}(\mathbb{T}^{2}))$ for some $s>0$ (for the definition of the function space $Z^{s,\frac{1}{2}}_{T}$, consult section \ref{Sec:lin}).  By taking an appropriate union over sets of this type (with $T\downarrow 0$), he obtained local well-posedness almost surely for the Wick-ordered cubic NLS below $L^{2}(\mathbb{T}^{2})$.  For other works that have used nonlinear smoothing to establish local dynamics in the support of measures on phase space, see (for example) Burq-Tzvetkov \cite{BT1,BT2,BT3}, Oh \cite{O5}, Colliander-Oh \cite{COH}, and Nahmod-Pavlovi\'{c}-Staffilani \cite{NPS}.

\smallskip

For \eqref{Eqn:gKdV-zeromean}, we consider randomized initial data of the form \eqref{Eqn:initialdata}.  We could not follow the method of \cite{B4} directly, and perform a contraction argument for \eqref{Eqn:gKdV-zeromean} (with exponential likelihood) in $\{S(t)u_{0,\omega} + B_{R}\}$, where $B_{R}$ is a ball of radius $R$ in the Banach space $Z^{s,\frac{1}{2}}_{T}$.  In particular, we found that the square of the $Z^{s,\frac{1}{2}}_{T}$-norm of the first Picard iterate for \eqref{Eqn:gKdV-zeromean} has infinite expectation.  This is due to the temporal regularity $b=\frac{1}{2}$ of the $Z^{s,\frac{1}{2}}_{T}$ space.

To avoid this obstruction, we establish estimates on the nonlinear part of the Duhamel formulation for the gauge-transformed quartic gKdV \eqref{Eqn:gKdV-zeromean} in $X^{s,b}_{T}$, with $s>\frac{1}{2}$ and $b<\frac{1}{2}$.  Unfortunately, with $b<\frac{1}{2}$, we cannot perform (with exponential likelihood) a contraction argument for \eqref{Eqn:gKdV-zeromean} in $\{S(t)u_{0,\omega} + B_{R}\}$, where $B_{R}$ is a ball of radius $R$ in the Banach space $X^{s,b}_{T}$.  Indeed, by taking $b<\frac{1}{2}$, there are regions of frequency space which (produce contributions that) make the nonlinear estimates required for a contraction argument impossible.

This is resolved by establishing a priori estimates on
the \textit{second iteration} of the Duhamel formulation of \eqref{Eqn:gKdV-zeromean}.  More precisely, the local-in-time solution $u$ to \eqref{Eqn:gKdV-zeromean} will be constructed as the limit in $X^{\frac{1}{2}-\delta,\frac{1}{2}-\delta}_{T}$ (with $0<\delta\ll 1$) of a sequence of smooth solutions $u^{N}$ evolving from frequency truncated data
$u^{N}_{0,\omega}=\mathbb{P}_{N}(u_{0,\omega})$.
Each $u^{N}$ will satisfy the Duhamel formulation
\begin{align}
u^{N}(t)= S(t)u^{N}_{0,\omega}
+ \int_{0}^{t}S(t-s)
\mathcal{N}(u^{N}(s))ds,
\label{Eqn:intro-4b}
\end{align}
where $\mathcal{N}(u)=u_{x}\Big(u^{3}-\frac{1}{2\pi}\int_{\mathbb{T}} u^{3}dx\Big)$ is the gauge-transformed nonlinearity.  We will estimate $\|u^N\|_{X^{\frac{1}{2}-\delta,\frac{1}{2}-\delta}_{T}}, \|u^N-u^M\|_{X^{\frac{1}{2}-\delta,\frac{1}{2}-\delta}_{T}}$ in order to establish convergence to a solution $u$ as $N\rightarrow \infty$.  Moreover, we will establish convergence of the nonlinear part of the smooth solutions $u^N$ in $X^{\frac{1}{2}+\delta,\frac{1}{2}-\delta}_{T}$ (notice the increase in spatial regularity from $\frac{1}{2}-\delta$ to $\frac{1}{2}+\delta$ due to nonlinear smoothing).  During the nonlinear estimates, in the troublesome regions of frequency space (created by taking $b<\frac{1}{2}$) we will substitute \eqref{Eqn:intro-4b} into an appropriately chosen factor of the nonlinearity.

This will resolve the technical obstruction (due to $b<\frac{1}{2}$), but by considering a second iteration of \eqref{Eqn:intro-4b} into just one of the factors, the multilinearity in our analysis expands from quartic to heptic.  Indeed, we will establish probabilistic heptilinear estimates on the second iteration of \eqref{Eqn:intro-4b}.  This is the trade-off involved in proving estimates on the second iteration with $b<\frac{1}{2}$: \textit{we can take $b<\frac{1}{2}$ at the cost of conducting a higher order multilinear analysis}.

This approach (using $b<\frac{1}{2}$ and the second iteration) was pioneered by Bourgain \cite{B2} in the analysis of KdV with measures as initial data.  The argument was adapted to the setting of randomized initial data by Oh \cite{O5}, who proved invariance of the white noise measure for the periodic KdV (see also \cite{QV}).  Our approach is similar, but we consider the quartic gKdV and the Gibbs measure (as opposed to the KdV and white noise).

The primary source of difficulty for well-posedness of quartic gKdV (including the $C^{4}$-failure in \cite{CKSTT1}) is the existence of distinct frequencies $n,n_{1},\cdots,n_{4}\in\mathbb{Z}$, such that $n = n_{1} + \cdots + n_{4}$, $|n|\sim|n_{1}|\sim N\gg 0$, but such that $|n^{3}-n_{1}^{3}-\cdots - n_{4}^{3}|\ll N^{2}$.  This does not occur for KdV and mKdV, which have dispersion relations with cubic factorizations, and it makes the nonlinear analysis for quartic gKdV more labor intensive.  Indeed, the regions of frequency space where these conditions are satisfied required us to use $b<\frac{1}{2}$ (and thus to consider a second iteration of \eqref{Eqn:intro-4b}).  Furthermore, it is in these regions of frequency space where our nonlinear estimates will rely most heavily on certain probabilistic lemmata (specifically involving hyper-contractivity properties of the Ornstein-Uhlenbeck semigroup, see Lemmas \ref{Lemma:prob1}-\ref{Lemma:prob2} in section \ref{Sec:NLproof}) and estimates involving a matrix norm (see Lemma \ref{Lemma:matrixnorm}).

\subsection{Global-in-time solutions}

To establish Theorem \ref{Thm:GWP}, we follow the scheme of \cite{B3}.
Using the invariance of the (finite-dimensional) Gibbs measure under the flow of the system \eqref{Eqn:gKdV-ZMN}, we extend the local solutions of \eqref{Eqn:gKdV-zeromean} (from Theorem \ref{Thm:LWP}) to global solutions, almost surely, and prove the invariance of the Gibbs measure under the flow.

To show that the system \eqref{Eqn:gKdV-ZMN} preserves the (finite-dimensional) Gibbs measure $\mu_{N}$ \eqref{Eqn:intro-6}, we need to prove that the $L^{2}$-norm of $u^N$, the Hamiltonian $H(u^{N})$ and the Lebesgue measure on phase space are all invariant under the flow.  The invariance of the Lebesgue measure is trivial (by Liousville's Theorem) for a Hamiltonian system, but the Hamiltonian formulation of \eqref{Eqn:gKdV} is disrupted by the gauge transformation (and the same is true in finite dimensions).  Instead, we will verify the invariance of the Lebesgue measure under the flow of \eqref{Eqn:gKdV-ZMN} directly (see Proposition \ref{Prop:trunc-inv} in section \ref{Sec:GWP}).  In particular, we will write \eqref{Eqn:gKdV-ZMN} as a system of ODEs in the space of Fourier coefficients, and verify that this system is driven by a divergence-free vector field.

\subsection{The gauge transformation}

Following the standard reductions of \cite{S1}, \cite{CKSTT1}, we apply the gauge transformation
\begin{align}
v(x,t)=\mathcal{G}(u(x,t)):= u\bigg(x-\bigg(\int_{0}^{t}\int_{\mathbb{T}}u^{3}(x',t')dx'dt'\bigg),t\bigg).
\label{Eqn:gaugetransform-a}
\end{align}
This transformation preserves the initial data $u_{0}$.  Also, it is invertible:
\begin{align}
u(x,t)=\mathcal{G}^{-1}(v(x,t))=v\bigg(x+\bigg(\int_{0}^{t}\int_{\mathbb{T}}v^{3}(x',t')dx'dt'\bigg),t\bigg).
\label{Eqn:gaugetransform}
\end{align}
\noindent
Note that \eqref{Eqn:gaugetransform} is well-defined for $v\in X^{1/2-,1/2-}$ by the embedding $X^{1/2-,1/2-} \subset L^{3}_{x,t}$.  Then $u$ solves \eqref{Eqn:gKdV} if and only if $v$ solves \eqref{Eqn:gKdV-zeromean}.  Since $v^{3}v_{x}=\partial_{x}(v^{4})$ and $v_{x}$ both have mean zero, so does $\mathbb{P}(v^{3})v_{x}$.  We can rewrite \eqref{Eqn:gKdV-zeromean} as
\begin{align}
\left\{
\begin{array}{ll} \partial_{t}v + \partial_{x}^{3}v = \mathbb{P}(\mathbb{P}(v^{3})v_{x})  , \ \
t\geq 0, x\in \mathbb{T}
\\
v(x,0) = u_{0}(x).
\end{array} \right.
\label{Eqn:gKdV-zeromean2}
\end{align}


In addition, note that since $\int_{\mathbb{T}}v^{2}v_{x}=\frac{1}{3}
\int_{\mathbb{T}}(v^{3})_{x}=0$, and $\int_{\mathbb{T}}v v_{x}=\frac{1}{2}
\int_{\mathbb{T}}(v^{2})_{x}=0$, we can subtract
$3\mathbb{P}(v)\int_{\mathbb{T}}v^{2}v_{x} + 3\mathbb{P}(v^{2})\int_{\mathbb{T}}v v_{x}$ from the right hand side of \eqref{Eqn:gKdV-zeromean2}, with no effect, and rewrite \eqref{Eqn:gKdV-zeromean} as
\begin{align}
\left\{
\begin{array}{ll} \partial_{t}v + \partial_{x}^{3}v = \mathbb{P}(\mathbb{P}(v^{3})v_{x})
- 3\mathbb{P}(v)\int_{\mathbb{T}}v^{2}v_{x} - 3\mathbb{P}(v^{2})\int_{\mathbb{T}}v v_{x}  , \ \
t\in\mathbb{R}, x\in \mathbb{T}
\\
v(x,0) = u_{0}(x).
\end{array} \right.
\label{Eqn:gKdV-zeromean3}
\end{align}
The reformulation \eqref{Eqn:gKdV-zeromean3} of \eqref{Eqn:gKdV-zeromean}
will be needed during the proof of crucial nonlinear estimates.
The equations \eqref{Eqn:gKdV} and \eqref{Eqn:gKdV-zeromean} leave the same Hamiltonian
\begin{align}
H(u)= \frac{1}{2}\int u_{x}^{2}dx +  \frac{1}{20}\int u^{5}dx = \frac{1}{2}\int v_{x}^{2}dx +  \frac{1}{20}\int v^{5}dx =:H(v),
\end{align}
invariant under the flow.

\smallskip

The remainder of this paper is organized as follows.
In Section \ref{Sec:lin} we present the basic linear estimates related to the  propagator $S(t) := e^{- \partial_{x}^{3} t }$ of the linear part of gKdV.  In Section \ref{Sec:NLest} we present the nonlinear estimates to be used in the proof of local well-posedness (Theorem \ref{Thm:LWP}).  In Section \ref{Sec:LWP} we will prove Theorem \ref{Thm:LWP}.  Section \ref{Sec:GWP} contains the proof of Theorem \ref{Thm:GWP}.  It is divided into parts: 5.1 construction of the Gibbs measure, 5.2. invariance of the Gibbs measure for the projection of \eqref{Eqn:gKdV-zeromean} to the first $N>0$ modes of the trigonometric system, 5.3. global well-posedness (almost surely) for \eqref{Eqn:gKdV-zeromean}, and 5.4. invariance of the Gibbs measure for \eqref{Eqn:gKdV-zeromean}.
Section \ref{Sec:NLproof-main} is devoted to the proof of the crucial nonlinear estimates.  The proofs of various lemmata that will be used throughout this paper, as well as an expanded discussion of the nonlinear smoothing effect, can be found in the Ph.D. thesis of the author \cite{R-Th}.


\section{Linear estimates}
\label{Sec:lin}

In \cite{B1}, Bourgain introduced a weighted space-time Sobolev space $X^{s,b}$ whose norm is
given by
\begin{align*}
\|u\|_{X^{s,b}}=\|\langle n\rangle^{s}\langle \tau-n^{3}\rangle^{b}\hat{u}(n,\tau)\|_{l^{2}_{n}L^{2}_{\tau}}.
\end{align*}
Since the $X^{s, \frac{1}{2}}$ norm fails to control $L^\infty_t H^s_x$ norm, a smaller space
$Z^{s, b}(\mathbb{T} \times \mathbb{R})$ was also introduced, whose norm is given by
\begin{equation} \label{Zsb}
\| u\|_{Z^{s, b}(\mathbb{T} \times \mathbb{R})}
:= \| u\|_{X^{s, b}(\mathbb{T} \times \mathbb{R})}
+ \| u\|_{Y^{s, b-\frac{1}{2}}(\mathbb{T} \times \mathbb{R})},
\end{equation}
\noindent
where $\langle\, \cdot\, \rangle = 1 + |\cdot|$ and
$\| u\|_{Y^{s, b}(\mathbb{T} \times \mathbb{R})}
= \|\langle n\rangle ^s \langle\tau - n^3\rangle^b \hat{u}(n, \tau)\|_{l^2_n L^1_\tau(\mathbb{Z} \times\mathbb{R})}.$
One also defines the local-in-time version
$Z^{s, b}_{ T }$  on $\mathbb{T} \times [0, T]$,
by
\[ \|u\|_{Z^{s, b}_{T}} =  \inf{ \big\{ \|\tilde{u} \|_{Z^{s, b}(\mathbb{T} \times \mathbb{R})}: {\tilde{u}|_{[0, T]} = u}\big\}}.\]
\noindent
The local-in-time versions of other function spaces are defined analogously.

In this section we present the basic linear estimates related to gKdV.  Let $S(t) := e^{- \partial_{x}^{3} t }$ and $T\leq 1$ in the following.
We first state the homogeneous and nonhomogeneous linear estimates.  See \cite{B1, KPV1} for details.

\begin{lemma}
For any $s \in \mathbb{R}$ and $b < \frac{1}{2}$, we have
$\|  S(t) u_0\|_{X^{s, b}_{T}} \lesssim T^{\frac{1}{2}-b}\|u_0\|_{H^s}$.
\label{Lemma:lin1}
\end{lemma}

\begin{lemma}
For any $s,b \in \mathbb{R}$, we have
$\| \eta_{T}(t)S(t) u_0\|_{X^{s, b}} \leq C(T)\|u_0\|_{H^s}$.
\label{Lemma:lin1b}
\end{lemma}

\begin{lemma}
For any $s \in \mathbb{R}$ and $b \leq \frac{1}{2}$, we have
\begin{align*}
 \bigg\|  \int_0^t S(t-t') F(x, t') dt'\bigg\|_{X^{s, b}_{ T}}
\lesssim \| F \|_{Z^{s, b-1}_{ T}}.
\end{align*}

\noindent
Also, we have
$ \big\|  \int_0^t S(t-t') F(x, t') dt'\big\|_{X^{s, b}_{ T}}
\lesssim \| F \|_{X^{s, b-1}}$
for $b > \frac{1}{2}$.
\label{Lemma:lin2}
\end{lemma}

We will also require the following Lemma concerning the $X^{s,b}_{T}$ spaces, which allows us to gain a small power of $T$ by raising the temporal exponent $b$.
\begin{lemma}
Let $0<b<\frac{1}{2}$, $s\in\mathbb{R}$, then
\[ \|u\|_{X^{s,b}_{T}}\lesssim T^{\frac{1}{2}-b-}\|u\|_{X^{s,\frac{1}{2}}_{T}}.\]
\label{Lemma:gainpowerofT}
\end{lemma}

The proof of Lemma \ref{Lemma:gainpowerofT} can be found in \cite{COH}; it is based on the following property of the $X^{s,b}_{T}$ spaces, which will be exploited throughout this paper.  For any $b<\frac{1}{2}$, letting $\chi_{[0,T]}$ denote the characteristic function of the interval $[0,T]$, we have
\begin{align}
\|u\|_{X^{s,b}_{T}} \sim  \|\chi_{[0,T]}u\|_{X^{s,b}}.
\label{Eqn:cutoff}
\end{align}

Most of our probabilistic lemmata will not be needed until the appendix, where we will prove the crucial nonlinear estimates.  The proof of Theorem \ref{Thm:LWP} will, however, make use of the following lemma regarding large deviations.
\begin{lemma}
Fix $\gamma>0$, then for $K>0$ sufficiently large, $\exists\, c>0$ such that
\begin{align*}
P\big(\|u_{0,\omega}\|_{H^{\frac{1}{2}-\gamma}}
\leq K\big) \leq e^{-cK^2},
\end{align*}
where $u_{0,\omega}$ is given by \eqref{Eqn:initialdata}.
\label{Lemma:largedev}
\end{lemma}

Next we list some embeddings involving the $X^{s,b}$ spaces, to be used throughout this paper.  We will use the trivial embedding
\begin{align}
X^{s,b}\subset X^{s',b'}
\label{Eqn:X-triv}
\end{align}
for $s\geq s'$, $b\geq b'$.  The spatial Sobolev embedding gives
\begin{align}
X^{s,0} = L^{2}_{t}H^{s}_{x} \subset L^{2}_{t}L^{p}_{x}
\label{Eqn:X-sobolev}
\end{align}
where $0 \leq s < 1/2$ and $2\leq p \leq \frac{2}{1-2s}$, or where $s>1/2$ and $2\leq p \leq \infty$.
Also recall the energy estimate
\begin{align}
X^{s,1/2+} \subset L^{\infty}_{t}H^{s}_{x} \subset L^{\infty}_{t}L^{p}_{x}
\label{Eqn:X-energy}
\end{align}
under the same conditions on $s$ and $p$.  This gives
\begin{align}
X^{1/2+,1/2+} \subset L^{\infty}_{x,t}.
\label{Eqn:X-infty}
\end{align}
Interpolating \eqref{Eqn:X-energy} with \eqref{Eqn:X-infty}, for $s>1/2$, we have $\displaystyle X^{1/2+,1/2+} \subset L^{q}_{t}L^{r}_{x}$ for all $2\leq q,r \leq \infty$.  Interpolating this estimate with \eqref{Eqn:X-sobolev} for $s=0$, $p=2$, we find
\begin{align}
X^{1/2-\delta,1/2-\delta} \subset L^{q}_{t}L^{r}_{x}
\label{Eqn:X-onehalfminus}
\end{align}
whenever $0<\delta<1/2$ and $2\leq q,r < 1/\delta$.

Recall the following Strichartz estimates from \cite{B1},
\begin{align}
X^{0,1/3} \subset L^{4}_{t,x},
\label{Eqn:X-Str-1}
\end{align}
and
\begin{align}
X^{0+,1/2+} \subset L^{6}_{t,x}.
\label{Eqn:X-Str-2}
\end{align}
We can interpolate \eqref{Eqn:X-Str-1} with \eqref{Eqn:X-Str-2} to obtain
\begin{align}
X^{0+,1/2-\sigma} \subset L^{q}_{x,t},
\label{Eqn:X-Str-interp}
\end{align}
whenever $4<q<6$ and $\sigma<2(\frac{1}{q}-\frac{1}{6})$.

Lastly we recall the following embeddings for the $Y^{s,b}$ space: for $s\in\mathbb{R}$, we have
\begin{align}
X^{s,\frac{1}{2}+}_{T}\subset Y^{s,0}_{T}\subset C([0,T];H^{s}(\mathbb{T})).
\label{Eqn:linear-Y}
\end{align}

\section{Nonlinear estimates}
\label{Sec:NLest}

In this section we will formulate and state two key propositions (see Proposition \ref{Prop:nonlin} and Proposition \ref{Prop:nonlin2} below).  These propositions provide multilinear estimates to be used in the proof of local well-posedness (Theorem \ref{Thm:LWP}).  The proof of Theorem \ref{Thm:LWP} can be found in the next section (Section \ref{Sec:LWP}).  The proofs of Proposition \ref{Prop:nonlin} and Proposition \ref{Prop:nonlin2} can be found in Section \ref{Sec:NLproof-main}.


We begin with an outline of the role of Proposition \ref{Prop:nonlin} and Proposition \ref{Prop:nonlin2} in the proof of almost sure local well-posedness (Theorem \ref{Thm:LWP}).  In the proof of Theorem \ref{Thm:LWP}, the local solution $u$ to \eqref{Eqn:gKdV-zeromean} will be constructed as the limit of a sequence of solutions $u^{N}$ which evolve from frequency truncated initial data.  This sequence will converge in the space $X^{s,b}_{T}$ of functions of space-time (for certain values of $s, b \in \mathbb{R}$) for $T>0$ sufficiently small, and this convergence will follow from a priori estimates satisfied by solutions to \eqref{Eqn:gKdV-zeromean} (see \eqref{Eqn:LWP-1a}-\eqref{Eqn:LWP-1b} in section \ref{Sec:LWP}).  Proposition \ref{Prop:nonlin} and Proposition \ref{Prop:nonlin2} provide multilinear estimates in $X^{s,b}_{T}$ which are designed to produce the a priori estimates \eqref{Eqn:LWP-1a}-\eqref{Eqn:LWP-1b}.  That is, the sequence of solutions $u^{N}$ will be inserted into the multilinear estimates of Propositions \ref{Prop:nonlin} and \ref{Prop:nonlin2} (in various ways) during the proof of Theorem \ref{Thm:LWP} to establish the a priori estimates \eqref{Eqn:LWP-1a}-\eqref{Eqn:LWP-1b}.

Let us explicitly formulate the multilinear functions which appear in Proposition \ref{Prop:nonlin} and Proposition \ref{Prop:nonlin2}.
In this paper, we solve the integral formulation of \eqref{Eqn:gKdV-zeromean} with data $u_{0,\omega}$ (given by \eqref{Eqn:initialdata}),
\begin{align}
u = S(t)u_{0,\omega} + \mathcal{D}(u).
\label{Eqn:fixedpoint}
\end{align}
Here $\mathcal{D}(u):=\mathcal{D}(u,u,u,u)$, where, using the nonlinearity in \eqref{Eqn:gKdV-zeromean3},
$$\mathcal{D}(u_{1},u_{2},u_{3},u_{4}):=\int_{0}^{t}S(t-t')\mathcal{N}(u_{1},u_{2},u_{3},u_{4})(t')dt',$$
\noindent with
\begin{align*}
\Big(\mathcal{N}(u_{1},u_{2},u_{3},u_{4})\Big)
^{\wedge}(n,t) = \sum_{\zeta(n)}
(in_{1})\widehat{u_{1}}(n_{1},t)\widehat{u_{2}}(n_{2},t)
\widehat{u_{3}}(n_{3},t)\widehat{u_{4}}(n_{4},t),
\end{align*}
where $\zeta(n)$ is the set of frequencies satisfying certain restrictions (dictated by the nonlinearity of \eqref{Eqn:gKdV-zeromean3}).  The definition of $\zeta(n)$ is slightly cumbersome, and we avoid it here.  See Section \ref{Sec:NLproof} for details.


Taking the Fourier transform in time, we have
\begin{align}
\Big(\mathcal{N}(u_{1},u_{2},u_{3},u_{4})\Big)
^{\wedge}(n,\tau) = \sum_{\zeta(n)} \int_{\tau=\tau_{1}+\cdots +\tau_{4}}
(in_{1})\hat{u_{1}}(n_{1},\tau_{1})
\cdots \hat{u_{4}}(n_{4},\tau_{4}).
\label{Eqn:NLstate-3}
\end{align}
Consider the domain $A$ of frequency space given by $$A:= \{(n,n_{1},\ldots,n_{4},\tau,
\tau_{1},\ldots,\tau_{4})\in\mathbb{Z}^{5}\times \mathbb{R}^{5}: (n_{1},n_{2},n_{3},n_{4})\in \zeta(n),\tau=\tau_{1}+\cdots +\tau_{4} \}$$
depending on the relative sizes of the dispersive weights $\sigma:= \tau-n^{3}$, $\sigma_{k} := \tau_{k}-n_{k}^{3}$, and the spatial frequencies $n$, $n_{k}$, for $k=1,\ldots,4$.
Specifically, letting $|\sigma_{\text{max}}|:=\max(|\sigma|,|\sigma_{1}|,|\sigma_{2}|,|\sigma_{3}|,|\sigma_{4}|)$ and
$|n_{\text{max}}|:=\max(|n|,|n_{1}|,|n_{2}|,
|n_{3}|,|n_{4}|)$, we express $A=A_{-1}\cup A_{0}\cup\cdots\cup
A_{4}$ by letting
\begin{align}
A_{-1} &:= A\cap \{|\sigma_{\text{max}}| \ll |n_{\text{max}}|^{2} \},
\notag \\
A_{0} &:= A\cap \{|\sigma| \gtrsim |n_{\text{max}}|^{2}\},  \label{Eqn:NL-part} \\
A_{k} &:= A\cap \{|\sigma_{k}| \gtrsim |n_{\text{max}}|^{2}\}, \notag
\end{align}
for $k=1,2,3,4$.  We will use $\mathcal{N}_{j}(u_{1},u_{2},u_{3},u_{4})$
($\mathcal{D}_{j}(u_{1},u_{2},u_{3},u_{4})$) to denote the contribution to $\mathcal{N}(u_{1},u_{2},u_{3},u_{4})$ ($\mathcal{D}(u_{1},u_{2},u_{3},u_{4})$) coming from $A_{j}$.

The partition of type \eqref{Eqn:NL-part} is standard; see for example \cite{KPV2} in the context of KdV.  However, in the analysis of KdV, the region $A_{-1}$ has no analogue.  Indeed, due to the quadratic nonlinearity in KdV, we have the convolution restriction $n=n_{1}+n_{2}$ in frequency space, and this leads to the factorization $n^{3}-n_{1}^{3}-n_{2}^{3} = 3nn_{1}n_{2}$.  Then using $\tau=\tau_{1}+\tau_{2}$, with $\sigma$, $\sigma_{1}$ and $\sigma_{2}$ defined as above, we have for $0\neq n\neq n_{1} \neq 0$ that
\begin{align}
\max(|\sigma|,|\sigma_{1}|,|\sigma_{2}|)
\gtrsim |\sigma - \sigma_{1} - \sigma_{2}|
= |n^{3}-n_{1}^{3}-n_{2}^{3}| = |nn_{1}n_{2}| \geq |n_{\text{max}}|^{2}.
\label{Eqn:KdVmax}
\end{align}
By \eqref{Eqn:KdVmax} the configuration $\max(|\sigma|,|\sigma_{1}|,|\sigma_{2}|)\ll |n_{\text{max}}|^{2}$ is impossible, and the domain of integration in frequency space can be partitioned into regions of the type $A_{0}$, $A_{1}$ and $A_{2}$.  A similar factorization takes place in the analysis of mKdV (with cubic nonlinearity), and the region $A_{-1}$ is not included, see for example \cite{TaTs}.

In the analysis of \eqref{Eqn:gKdV-zeromean} (with quartic nonlinearity), there are nontrivial contributions from the region $A_{-1}$; indeed, there are combinations of distinct non-zero frequencies $n,n_{1},\ldots,n_{4}\in \mathbb{Z}$ such that $|n|\sim|n_{1}|\sim N$, with $N\gg0$ large, but such that $n^{3}-n_{1}^{3}-\cdots -n_{4}^{3} \ll N^{2}$.  In fact, the sequence of initial data which produces $C^{4}$-failure of the data-to-solution map for \eqref{Eqn:gKdV-zeromean} (when posed in $H^{s}(\mathbb{T})$ for $s<\frac{1}{2}$, see \cite{CKSTT1}) is concentrated on frequency combinations of precisely this type.
We are forced to include $A_{-1}$ in our analysis of \eqref{Eqn:gKdV-zeromean}, and we proceed with the
knowledge that this region is responsible for the failure of the deterministic fixed point method below $H^{\frac{1}{2}}(\mathbb{T})$.

With the failure of deterministic methods in the region $A_{-1}$, the multilinear estimates we establish in this region will use a probabilistic analysis: these estimates will rely on nonlinear smoothing induced by initial data randomization and dispersion.  That is, in Proposition \ref{Prop:nonlin} below, we will establish \textit{probabilistic quadrilinear estimates} on $\mathcal{D}_{-1}(u_{1},\ldots,u_{4})$; estimates which involve the randomized data and nonlinear smoothing.  We will also need to use a probabilistic analysis (although it will be somewhat simpler) for estimates in the regions $A_{1}$.  In contrast, our estimates in the regions $A_{0},A_{2},A_{3}$ and $A_4$ will be purely deterministic.

In order to describe the estimates in $A_{0},\ldots,A_{4}$ in more detail, let us digress to discuss the temporal regularity of the $X^{s,b}$ spaces used in our analysis.  Recall from the introduction that (in the proof of Theorem \ref{Thm:LWP}) we will study the convergence of a sequence of smooth solutions $u^N$ (evolving from frequency truncated data) in $X^{s,b}$ for some $s,b<\frac{1}{2}$.  In fact, we will take $s=b=\frac{1}{2}-\delta$, for some $0<\delta\ll 1$.  The choice of temporal regularity $b=\frac{1}{2}-\delta<\frac{1}{2}$ is helpful (analytically) for the nonlinear estimates in some regions of frequency space, but it is cumbersome in others.  For example, in the region $A_{0}$, the choice of $b<\frac{1}{2}$ is beneficial, and we will establish \textit{deterministic quadrilinear estimates} on $\mathcal{D}_{0}(u_{1},\ldots,u_{4})$ (see Proposition \ref{Prop:nonlin2} below).  However, when estimating $\mathcal{D}_{i}(u)$, $i=1,2,3,4$, the selection of $b<\frac{1}{2}$ backfires, and the analysis becomes more challenging.  In these troublesome regions ($A_{1},\ldots,A_{4}$)
we consider a \textit{second iteration} of \eqref{Eqn:fixedpoint}; we replace $u^N$ (in the appropriately chosen factor) with the right-hand side of \eqref{Eqn:fixedpoint}.  More precisely, when estimating $\mathcal{N}_{i}(u^N)=\mathcal{N}_{i}(u^N,u^N,u^N,u^N)$, $i=1,\ldots,4$, we will substitute \eqref{Eqn:fixedpoint} into the $i^{\text{th}}$ factor of $u^N$; the linear part of the solution makes no contribution in this region, and this factor is replaced by $u^N\sim \mathcal{D}(u^N)
=\mathcal{D}(u^N,u^N,u^N,u^N)$.  This substitution resolves the difficulty created by taking $b<\frac{1}{2}$, but what started as a \textit{quadrilinear} estimate (for example, of $\mathcal{D}_{1}(u^N)=\mathcal{D}_{1}(u^N,u^N,u^N,u^N)$ in $A_{1}$) becomes a \textit{heptilinear} estimate (of $\mathcal{D}_{1}(\mathcal{D}(u^N,u^N,u^N,u^N)
,u^N,u^N,u^N)$).  Furthermore, we will still depend on probabilistic methods (that is, on nonlinear smoothing) in the region $A_{1}$.  For these reasons, we will establish \textit{probabilistic heptilinear estimates} on $\mathcal{D}_{1}(\mathcal{D}(u_{5},u_{6},u_{7},u_{8}),u_{2},u_{3},u_{4})$ (see Proposition \ref{Prop:nonlin} below).  In the regions $A_{2},A_{3},A_{4}$, we can proceed with deterministic estimates.  That is, we establish \textit{deterministic heptilinear estimates} on
$\mathcal{D}_{2}(u_{1},\mathcal{D}(u_{5},u_{6},u_{7},u_{8}),u_{3},u_{4})$ (and the analogous expressions for $\mathcal{D}_{3}$ and $\mathcal{D}_{4}$).




We proceed to state the multilinear estimates to be used in the proof of Theorem \ref{Thm:LWP}.  The probabilistic nonlinear estimates in $A_{-1}$ and $A_{1}$ are grouped into the following proposition.

\begin{proposition} [Probabilistic nonlinear estimates] For $\delta>0$ sufficiently small, any $\delta_0\geq0$ such that $\delta>\delta_0$, and any $0<T\ll 1$, there exists $\varepsilon, \beta, c,C>0$ with $\beta,\varepsilon\ll\delta,\delta_0$ and a measurable set $\Omega_{T}\subset \Omega$ satisfying $P(\Omega_{T}^{c})<e^{-\frac{c}{T^{\beta}}}$ and the following conditions: if $\omega \in\Omega_{T}$, then for every quadruple of Fourier multipliers $\Lambda_{1},\ldots,\Lambda_{4}$ defined by
\begin{align*}
\widehat{\Lambda_{i}f}(n) = \chi_{N_{i}\leq |n|\leq M_{i}}\hat{f}(n),
\end{align*}
for some dyadic $N_{i},M_{i}$,
we have the following estimate
\begin{align}
\|\mathcal{D}_{-1}&(u_{1},u_{2},u_{3},u_{4})\|_
{\frac{1}{2}+\delta-\delta_0,\frac{1}{2}+\delta,T}
\notag \\ &\leq CT^{-2\beta}\prod_{j=1}^{4}\big(N_{j}^{-\varepsilon} + \|u_{j}\|_{\frac{1}{2}-\delta-\delta_j,\frac{1}{2}-\delta,T} + \|u_{j} - S(t)\Lambda_{j}(u_{0,\omega})\|_{\frac{1}{2}+\delta-\delta_j,\frac{1}{2}-\delta,T}
\big),
\label{Eqn:NL-neg1}
\end{align}
for all quadruples $(\delta_1,\delta_2,\delta_3,\delta_4)
\in\{(\delta_0,0,0,0),(0,\delta_0,0,0),(0,0,\delta_0,0),
(0,0,0,\delta_0)\}$.  We also have
\begin{align}
\|\mathcal{D}_{1}(\mathcal{D}&(u_{5},u_{6},u_{7},u_{8}),
u_{2},u_{3},u_{4})\|_{\frac{1}{2}+\delta-\delta_0,\frac{1}{2}+\delta,T}
  \notag \\ &\leq CT^{-\beta}
\big(N_{5}^{-\varepsilon} + \|u_{5}\|_{\frac{1}{2}-\delta-\delta_5,\frac{1}{2}-\delta,T} + \|u_{5} - S(t)\Lambda_{1}(u_{0,\omega})\|_{\frac{1}{2}+\delta-\delta_5,\frac{1}{2}-\delta,T}
\big) \notag \\ & \ \ \ \ \ \ \ \ \ \ \ \cdot\prod_{j=2,j\neq 5}^{8}
\|u_{j}\|_{\frac{1}{2}-\delta-\delta_j,\frac{1}{2}-\delta,T},
\label{Eqn:NL-1}
\end{align}
\begin{align}
\|\mathcal{D}_{1}(\mathcal{D}_{0}&(u_{5},u_{6},u_{7},u_{8}),
u_{2},u_{3},u_{4})\|_{\frac{1}{2}+\delta-\delta_0,\frac{1}{2}+\delta,T}
  \notag \\ &\leq CT^{-\beta}
\big(N_{5}^{-\varepsilon} + \|u_{5}\|_{\frac{1}{2}-\delta-\delta_5,\frac{1}{2}-\delta,T} + \|u_{5} - S(t)\Lambda_{1}(u_{0,\omega})\|_{\frac{1}{2}+\delta-\delta_5,\frac{1}{2}-\delta,T}
\big)
\notag \\ & \ \ \ \ \ \ \ \ \ \ \ \cdot\prod_{j=2,j\neq 5}^{8}
\|u_{j}\|_{\frac{1}{2}-\delta-\delta_j,\frac{1}{2}-\delta,T}
\big),
\label{Eqn:NL-1b}
\end{align}
for all septuples $(\delta_1,\delta_2,\delta_3,\delta_4,\delta_5,\delta_6,\delta_7)
\in\{(\delta_0,0,0,0,0,0,0),\ldots,(0,0,0,0,0,0,\delta_0)\}$.
\label{Prop:nonlin}
\end{proposition}

The estimates of Proposition \ref{Prop:nonlin} are based on the nonlinear smoothing due to initial data randomization.  However, in some regions (e.g. $A_{0}$), the choice of $b=\frac{1}{2}-\delta<\frac{1}{2}$ allows us to establish deterministic estimates.

\begin{proposition}[Deterministic nonlinear estimates]
For $\delta>0$ sufficiently small, any $\delta_0\geq0$ such that $\delta>\delta_0$, and any $T>0$, there exists $\theta,C>0$ such that
\begin{align}
\|\mathcal{D}_{0}&(u_{1},u_{2},u_{3},u_{4})\|_
{\frac{1}{2}+\delta-\delta_0,\frac{1}{2}-\delta,T}
\leq CT^{\theta}
\prod_{j=1}^{4} \|u_{j}\|_{\frac{1}{2}-\delta-\delta_j,\frac{1}{2}-\delta,T},
\label{Eqn:NL-0-a}
\end{align}
\begin{align}
\|\mathcal{N}_{0}(u_{1},u_{2},u_{3},u_{4})
\|_{Y^{\frac{1}{2}+\delta-\delta_0,-1}_{T}}
\leq CT^{\theta}\prod_{j=1}^{4}\|u_{j}\|_{\frac{1}{2}-\delta-\delta_j,\frac{1}{2}-\delta,T},
\label{Eqn:NL-0-b}
\end{align}
\begin{align}
\|\mathcal{D}_{1}&(u_{1},u_{2},u_{3},u_{4})\|_
{\frac{1}{2}+\delta-\delta_0,\frac{1}{2}+\delta,T}
\leq CT^{\theta}\|u_{1}\|_{\frac{1}{2}-\delta-\delta_1,\frac{1}{2}+\delta}
\prod_{\substack{j=2}}^{4} \|u_{j}\|_{\frac{1}{2}-\delta-\delta_j,\frac{1}{2}-\delta,T},
\label{Eqn:NL-k-1}
\end{align}
and
\begin{align}
\|\mathcal{D}_{2}&(u_{1},u_{2},u_{3},u_{4})\|_
{\frac{1}{2}-\delta_0,\frac{1}{2}+\delta,T}
\leq CT^{\theta}\|u_{2}\|_{\frac{1}{2}-\delta-\delta_2,\frac{1}{2}+\delta}
\prod_{\substack{j=1\\j\neq 2}}^{4} \|u_{j}\|_{\frac{1}{2}-\delta-\delta_j,\frac{1}{2}-\delta,T},
\label{Eqn:NL-k-2}
\end{align}
for all quadruples $(\delta_1,\delta_2,\delta_3,\delta_4)
\in\{(\delta_0,0,0,0),(0,\delta_0,0,0),(0,0,\delta_0,0),
(0,0,0,\delta_0)\}$.
We also have
\begin{align}
\|\mathcal{D}_{2}(u_{1},
\mathcal{D}(u_{5},u_{6},u_{7},u_{8}),
u_{3},u_{4})\|_{\frac{1}{2}+\delta-\delta_0,\frac{1}{2}+\delta,T} \leq
CT^{\theta}\prod_{\substack{j=1\\j\neq 2}}^{8} \|u_{j}\|_{\frac{1}{2}-\delta-\delta_j,\frac{1}{2}-\delta,T},
\label{Eqn:NL-2}
\end{align}
and
\begin{align}
\|\mathcal{D}_{2}(u_{1},
\mathcal{D}_{0}(u_{5},u_{6},u_{7},u_{8}),
u_{3},u_{4})\|_{\frac{1}{2}+\delta-\delta_0,\frac{1}{2}+\delta,T}
  \leq CT^{\theta}\prod_{\substack{j=1\\j\neq 2}}^{8} \|u_{j}\|_{\frac{1}{2}-\delta-\delta_j,\frac{1}{2}-\delta,T},
\label{Eqn:NL-2b}
\end{align}
for all heptuples $(\delta_1,\delta_2,\delta_3,\delta_4,\delta_5,\delta_6,\delta_7)
\in\{(\delta_0,0,0,0,0,0,0),\ldots,(0,0,0,0,0,0,\delta_0)\}$.
\label{Prop:nonlin2}
\end{proposition}

\begin{remark}
Observe that the function
$\mathcal{D}(u_{1},u_{2},u_{3},u_{4})$ is symmetric with respect to its last three slots.  It follows that inequalities analogous to \eqref{Eqn:NL-k-2}-\eqref{Eqn:NL-2b}
for $\mathcal{D}_{3}$ and $\mathcal{D}_{4}$ also hold.
\end{remark}

\begin{remark}  The estimate \eqref{Eqn:NL-0-a} will be used during the proof of Theorem \ref{Thm:LWP} to establish the existence of a (local-in-time) solution to \eqref{Eqn:gKdV-zeromean} (with data given by \eqref{Eqn:initialdata}), while the estimate \eqref{Eqn:NL-0-b} is needed to prove that (almost surely) the nonlinear part of this solution is continuous in time with values in a Sobolev space of higher regularity than the data (condition (ii) in the statement of Theorem \ref{Thm:LWP}).  The $Y^{s,b}$-space estimate \eqref{Eqn:NL-0-b} is required for the region $A_{0}$, but it is not required elsewhere (ie. in the regions $A_{-1}$ and $A_{1},\ldots,A_{4}$).  This is due to the difference in temporal regularity (of the norms appearing) on the left-hand sides of \eqref{Eqn:NL-neg1}-\eqref{Eqn:NL-2} ($b=\frac{1}{2}+\delta$) and \eqref{Eqn:NL-0-a} ($b=\frac{1}{2}-\delta$).  Indeed, by using nonlinear estimates with $b=\frac{1}{2}+\delta>\frac{1}{2}$ (as in \eqref{Eqn:NL-neg1}-\eqref{Eqn:NL-2}), and combining Lemma \ref{Lemma:lin2} with the embedding \eqref{Eqn:linear-Y}, the corresponding contribution to (the nonlinear part of) the solution is automatically continuous in time.
However, using estimates on the nonlinearity with $b=\frac{1}{2}-\delta<\frac{1}{2}$ (as in \eqref{Eqn:NL-0-a}), we must establish continuity with a separate argument, and this is where the estimate \eqref{Eqn:NL-0-b} will be needed.
\end{remark}

\begin{remark}
There is one region of frequency space, produced by using the second iteration, which appears lethal, at first glance, to our nonlinear analysis.  Luckily there is a cancelation in this region which allows us to prove the estimates we require.  A technicality emerges, due to this cancelation, that needs to be addressed in this section.  In Proposition \eqref{Prop:nonlin} we establish multilinear estimates on $\mathcal{D}_{1}$ with different input functions $u_{2},u_{3},\ldots,u_{8}$, but the cancelation that we need to invoke in the troublesome region of frequency space requires that all input functions are the same.

This is not problematic, however, as we only need multilinear estimates with different input functions in order to bound the difference of two expressions, each given by $\mathcal{D}_{1},\ldots,\mathcal{D}_{4}$ evaluated with all input functions the same; that is, we bound this difference using multilinearity of the functions $\mathcal{D}_{1},\ldots,\mathcal{D}_{4}$ and a telescoping series (see \eqref{Eqn:LWP-5}-\eqref{Eqn:LWP-6a} in the proof of Theorem \ref{Thm:LWP}).  To incorporate the cancelation with different input functions, we simply define $\mathcal{D}_{1}(\mathcal{D}(u_{5},u_{6},u_{7},u_{8}),u_{2},u_{3},u_{4})$ with the cancelation \textit{imposed} in the troublesome region.  In other words, when we use multilinearity and a telescoping series to control these differences, we add and subtract terms with the cancelation built in.  For a more precise discussion of this cancelation (and the proper definition of the multilinear functions appearing in Proposition \ref{Prop:nonlin} above) see the discussion on heptilinear estimates that precedes the proof of Proposition \ref{Prop:NL-1-nlpart}, found in Section \ref{Sec:NLproof-main}.
\label{Rem:cancel}
\end{remark}

We postpone the proofs of Proposition \ref{Prop:nonlin} and Proposition \ref{Prop:nonlin2} to Section \ref{Sec:NLproof-main}.

\section{Local well-posedness}
\label{Sec:LWP}

In this section we present the proof of Theorem \ref{Thm:LWP}.  The key inputs for this proof are Proposition \ref{Prop:nonlin} and Proposition \ref{Prop:nonlin2} (presented in the last section).


\begin{proof}[Proof of Theorem \ref{Thm:LWP}:]

We will construct the local solution to \eqref{Eqn:gKdV-zeromean} as the limit of a sequence of solutions $u^{N}$ which evolve from frequency truncated data.  Consider initial data of the form
\begin{align*}
u_{0,\omega}^{N}(x) = \mathbb{P}_{N}(u_{0,\omega}(x)),
\end{align*}
where $u_{0,\omega}$ is given by \eqref{Eqn:initialdata}, and $\mathbb{P}_{N}$ is the Dirichlet projection to $$E_{N}=\text{span}
\{\sin(nx),\cos(nx):|n|\leq N\}$$.  Notice that $u^{N}_{0,\omega}\in H^{s}(\mathbb{T})$ almost surely, for every $s\in\mathbb{R}$.  By Theorem 2 in \cite{CKSTT1}, for each $N$, almost surely, there exists a unique global-in-time solution $u^{N}$ to \eqref{Eqn:gKdV-zeromean} with data $u_{0,\omega}^{N}$.  Define $\Gamma^{N}=\Gamma^{N}_{\omega}$ by
\begin{align}
\Gamma^{N}(v):= S(t)u_{0,\omega}^{N}(x) + \mathcal{D}(v),
\label{Eqn:LWP-0a}
\end{align}
where $$\mathcal{D}(u_{1},u_{2},u_{3},u_{4}):=\int_{0}^{t}S(t-t')\mathcal{N}(u_{1},u_{2},u_{3},u_{4})(t')dt'$$
and $\mathcal{N}(u_{1},u_{2},u_{3},u_{4})$ is defined by \eqref{Eqn:nonlin}.  We will use the notation $\mathcal{N}(u)=\mathcal{N}(u,u,u,u)$, $\mathcal{D}(u)=\mathcal{D}(u,u,u,u)$ and $\mathcal{D}^{N}=\mathcal{D}(u^{N})$.
With these definitions, due to the reformulation
\eqref{Eqn:gKdV-zeromean3} of \eqref{Eqn:gKdV-zeromean},
$u^{N}$ satisfies $u^{N}=\Gamma^{N}(u^{N})$.

Here is an outline of the proof of Theorem \ref{Thm:LWP}: we show that for $\delta>0$ sufficiently small, $\exists\,c,\beta>0$ and
$\Omega_{T}\subset \Omega$ with $P(\Omega_{T}^{c})<e^{-\frac{c}{T^{\beta}}}$ such that
$\forall \,\omega \in\Omega_{T}$, the sequence $u^{N}$ converges in
$X^{\frac{1}{2}-\delta,\frac{1}{2}-\delta}_{T}$ to a solution $u$ of \eqref{Eqn:gKdV-zeromean} with initial data $u_{0,\omega}$.  In particular, we prove that for $\delta,T>0$ sufficiently small, $\forall\, \omega \in\Omega_{T} $, we have the following: $\exists \,\epsilon >0$ such that for every $N>M>0$,
\begin{align}
\|u^{N}-u^{M}\|_{\frac{1}{2}-\delta,\frac{1}{2}-\delta,T} \lesssim M^{-\epsilon},
\label{Eqn:LWP-1a}
\end{align}
\begin{align}
\|\mathcal{D}^{N}-\mathcal{D}^{M}\|_{\frac{1}{2}+\delta,\frac{1}{2}-\delta,T} \lesssim M^{-\epsilon}.  \label{Eqn:LWP-1b}
\end{align}
These estimates show that $u^{N}$ and $\mathcal{D}^{N}$
are Cauchy in $X^{\frac{1}{2}-\delta,\frac{1}{2}-\delta}_{T}$
and $X^{\frac{1}{2}+\delta,\frac{1}{2}-\delta}_{T}$, respectively.  Then we show that the convergent $u$ (of $u^{N}$) is a solution to \eqref{Eqn:gKdV-zeromean}, and proceed to prove uniqueness, continuity and stability properties of this solution.

We begin by constructing a set $\Omega_{T}\subset \Omega$
with $P\big((\Omega_{T})^{c}\big)<e^{-\frac{c}{T^{\beta}}}$ such that for $T>0$ sufficiently small, if
$\omega \in\Omega_{T}$, we have
\begin{align}
\|u^{N}\|_{\frac{1}{2}-\delta,\frac{1}{2}-\delta,T} \leq  R,
\label{Eqn:LWP-16}
\end{align}
\begin{align}
\|\mathcal{D}^{N}\|_{\frac{1}{2}+\delta,\frac{1}{2}-\delta,T} \leq \tilde{R},
\label{Eqn:LWP-17}
\end{align}
for some constants $R,\tilde{R}\sim 1$ (independent of $N$).  Then using the estimates \eqref{Eqn:LWP-16} and \eqref{Eqn:LWP-17}, and imposing additional constraints on $T$, we will show that if $\omega\in\Omega_{T}$, then \eqref{Eqn:LWP-1a} and \eqref{Eqn:LWP-1b} hold true.

By $u^{N}=\Gamma(u^{N})$,
\eqref{Eqn:LWP-0a} and Lemma \ref{Lemma:lin1}, we find
\begin{align}
\|u^{N}\|_{\frac{1}{2}-\delta,\frac{1}{2}-\delta,T} &\leq  \|S(t)u_{0,\omega}^{N}\|_{\frac{1}{2}-\delta,\frac{1}{2}-\delta,T}
+ \|\mathcal{D}(u^{N})\|_{\frac{1}{2}-\delta,\frac{1}{2}-\delta,T}  \notag \\
&\leq  C_{0}T^{\delta}\|u_{0,\omega}^{N}\|_{H^{\frac{1}{2}-\delta}} + \|\mathcal{D}(u^{N})\|_{\frac{1}{2}+\delta,\frac{1}{2}-\delta,T}.
\label{Eqn:LWP-1}
\end{align}
Observe that we have applied the trivial embedding \eqref{Eqn:X-triv} to the last term in \eqref{Eqn:LWP-1} by raising the spatial regularity from $s=\frac{1}{2}-\delta$ to
$s=\frac{1}{2}+\delta$.  Next we use the triangle inequality and Lemma \ref{Lemma:gainpowerofT} to find
\begin{align}
\|\mathcal{D}(u^{N})\|_{\frac{1}{2}+\delta,\frac{1}{2}-\delta,T} &\leq  \|\mathcal{D}_{-1}(u^{N})\|_{\frac{1}{2}+\delta,\frac{1}{2}-\delta,T}+ \cdots
+ \|\mathcal{D}_{4}(u^{N})\|_{\frac{1}{2}+\delta,\frac{1}{2}-\delta,T}
\notag \\
& \|\mathcal{D}_{0}(u^{N})\|_{\frac{1}{2}+\delta,\frac{1}{2}-\delta,T}
+ T^{\delta-}\sum_{k=-1,k\neq 0}^{4}
\|\mathcal{D}_{k}(u^{N})\|_{\frac{1}{2}+\delta,\frac{1}{2}+\delta,T}.
\label{Eqn:LWP-1e}
\end{align}
We proceed to estimate each term on the right-hand side of \eqref{Eqn:LWP-1e} using Propositions \ref{Prop:nonlin} and \ref{Prop:nonlin2}.
Notice that by $u^{N}=\Gamma(u^{N})$ and
\eqref{Eqn:LWP-0a} we have
\begin{align}
u^{N} - S(t)u^{N}_{0,\omega}=\mathcal{D}(u^{N}).
\label{Eqn:LWP-1d}
\end{align}
By Proposition \ref{Prop:nonlin} and Proposition \ref{Prop:nonlin2}
$\exists\,\theta,\varepsilon, \beta, c,C>0$ with $\beta,\varepsilon\ll \delta,\delta-\delta_1$ and a measurable set $\Omega_{T}^{0}\subset \Omega$ with $P((\Omega_{T}^{0})^{c})<e^{-\frac{c}{T^{\beta}}}$ such that
$\forall\, \omega \in\Omega_{T}^{0} $, the estimates \eqref{Eqn:NL-neg1}-\eqref{Eqn:NL-1b} hold true.
In particular, using the estimates \eqref{Eqn:NL-neg1} and \eqref{Eqn:NL-0-a} (with $\delta_0= 0$, each $u_{i}=u^{N}$, and each  $\Lambda_{i}=\mathbb{P}_{N}$, $i=1,2,3,4$)
we have
\begin{align}
T^{\delta-}\|\mathcal{D}_{-1}(u^{N})&\|_{\frac{1}{2}+\delta,\frac{1}{2}+\delta,T} + \|\mathcal{D}_{0}(u^{N})\|_{\frac{1}{2}+\delta,\frac{1}{2}-\delta,T}\notag \\
& \lesssim  T^{\theta}
\big(1+\|u^{N}\|_{\frac{1}{2}-\delta,\frac{1}{2}-\delta,T}
+
\|u^{N} - S(t)u^{N}_{0,\omega}\|_{\frac{1}{2}+\delta,\frac{1}{2}-\delta,T}
\big)^{4} \notag \\
& =
T^{\theta}
\big(1+\|u^{N}\|_{\frac{1}{2}-\delta,\frac{1}{2}-\delta,T}
+
\|\mathcal{D}(u^{N})\|_{\frac{1}{2}+\delta,\frac{1}{2}-\delta,T}
\big)^{4},
\label{Eqn:LWP-2-a}
\end{align}
where we have used \eqref{Eqn:LWP-1d} in the last line.  When we estimate $\mathcal{D}_{k}$, for $k=1,2,3,4$, we consider a second iteration of
\eqref{Eqn:LWP-0a} in the $k^{\text{th}}$ factor.
For example, in the region $A_{1}$ we substitute
\begin{align}
u^{N} = S(t)u^{N}_{0,\omega} + \mathcal{D}(u^{N})
\label{Eqn:LWP-2b}
\end{align}
into the first slot of $\mathcal{D}_{1}$, and estimate the linear and nonlinear contributions from \eqref{Eqn:LWP-2b} separately.  To estimate $\|\mathcal{D}_{1}(S(t)u_{0,\omega}^{N},u^{N},u^{N}
,u^{N})\|_{\frac{1}{2}+\delta,\frac{1}{2}+\delta,T}$, we find
by \eqref{Eqn:NL-k-1}, $\delta_0=0$, and Lemma \ref{Lemma:lin1b},
\begin{align}
T^{\delta-}\|\mathcal{D}_{1}(S(t)u^{N}_{0,\omega}
,u^{N},&u^{N},u^{N})\|_
{\frac{1}{2}+\delta,\frac{1}{2}+\delta,T}
\notag \\ &\lesssim T^{\theta}
\|\eta(t)S(t)u^{N}_{0,\omega}\|_{\frac{1}{2}-\delta,\frac{1}{2}+\delta} \|u^{N}\|_{\frac{1}{2}-\delta,\frac{1}{2}-\delta,T}^{3}
\notag \\
&\lesssim T^{\theta}
\|u^{N}_{0,\omega}\|_{H^{\frac{1}{2}-\delta}} \|u^{N}\|_{\frac{1}{2}-\delta,\frac{1}{2}-\delta,T}^{3}.
\label{Eqn:NL-1-3}
\end{align}
Next we estimate $\|\mathcal{D}_{1}(\mathcal{D}(u^{N}),
u^{N},u^{N},u^{N})\|_{\frac{1}{2}+\delta,\frac{1}{2}-\delta,T}$.
Using \eqref{Eqn:NL-1} (with $\delta_0=0$, each $u_{i}=u^{N}$, and  $\Lambda_{1}=\mathbb{P}_{N}$)
we have
\begin{align}
T^{\delta-}\|\mathcal{D}_{1}(\mathcal{D}(u^{N}),u^{N},&
u^{N},u^{N})\|_{\frac{1}{2}+\delta,\frac{1}{2}-\delta}
\notag \\ &\lesssim
T^{\theta}\big(1+\|u^{N}\|_{\frac{1}{2}-\delta,\frac{1}{2}-\delta,T}
+
\|\mathcal{D}^{N}\|_{\frac{1}{2}+\delta,\frac{1}{2}-\delta,T}
\big)^{7}.
\label{Eqn:LWP-2-c}
\end{align}
Using \eqref{Eqn:NL-2}-\eqref{Eqn:NL-2b} we obtain estimates analogous to \eqref{Eqn:NL-1-3} and \eqref{Eqn:LWP-2-c} for $k=2,3,4$.  Combining these estimates with \eqref{Eqn:LWP-2-a} we have
\begin{align}
\|\mathcal{D}^{N}\|_{\frac{1}{2}+\delta,\frac{1}{2}-\delta,T} &\leq  C_{1}T^{\theta} Z\big(\|u_{0,\omega}^{N}\|_{H^{\frac{1}{2}-\delta}},\|u^{N}\|_{\frac{1}{2}-\delta,\frac{1}{2}-\delta,T},
\|\mathcal{D}^{N}\|_{\frac{1}{2}+\delta,\frac{1}{2}-\delta,T}\big),
\label{Eqn:LWP-2}
\end{align}
for some $\theta>0$, where $Z(x,y,z)$ is a polynomial of degree 7 with positive coefficients.

Now, for some fixed $0<\varepsilon<\delta$, $C>0$, let
$$\displaystyle\Omega_{T}:= \Omega_{T}^{0}\cap\Bigg\{\omega\in\Omega \Bigg| \Big\|\sum_{n\in\mathbb{Z}\setminus\{0\}}\frac{g_{n}(\omega)e^{inx}}{
|n|^{1-\epsilon}}\Big\|_{H^{\frac{1}{2}-\delta}}
\leq \frac{C}{T^{\frac{\beta}{2}}} \Bigg\}.$$
We have by Lemma \ref{Lemma:largedev} that, for $T>0$ sufficiently small, $P\big((\Omega_{T})^{c}\big)< P\big((\Omega_{T}^{0})^{c}\big) + e^{-(K(T))^{2}} \leq e^{-\frac{c}{T^{\beta}}}$.  Next we will show that, if $\omega\in\Omega_{T}$, then \eqref{Eqn:LWP-16} and \eqref{Eqn:LWP-17} are satisfied.  Given $R,\tilde{R}>0$, for $0<\alpha<\delta - \frac{\beta}{2},
0<\tilde{\alpha}<\theta-\frac{7\beta}{2}$, let
\begin{align*}
T_{1}&:= \inf\Big\{t>0 \Big| t^{\theta}
Z\Big(\frac{C}{t^{\frac{\beta}{2}}},R,\tilde{R}\Big)
\geq t^{\tilde{\alpha}}\tilde{R}\Big\}, \\
T_{2}&:= \inf\Big\{t>0 \Big|C_{0}t^{\alpha}
+ C_{1}t^{\tilde{\alpha}}\tilde{R}
\geq \frac{1}{2}R\Big\}, \\
T&:=T_{1}\wedge T_{2}\wedge\Big(\frac{1}{2C_{1}}\Big)^{\frac{1}{\tilde{\alpha}}}.
\end{align*}
Then $T>0$ by definition, and $T$ is also independent of $N$.  We claim that \eqref{Eqn:LWP-16} and \eqref{Eqn:LWP-17} hold for $\omega\in\Omega_{T}$.

The $X^{\frac{1}{2}\pm \delta,\frac{1}{2}-\delta}_{t}$ norms of $u^{N}$ and $\mathcal{D}^{N}$ are finite, continuous and increasing functions of $t$, and $\|u^{N}\|_{\frac{1}{2}-\delta,\frac{1}{2}-\delta,t}$,
$\|D^{N}\|_{\frac{1}{2}+\delta,\frac{1}{2}-\delta,t}
\rightarrow 0$ as $t\downarrow 0$ (almost surely, for each fixed $N>0$).  In particular, letting
$$ T^*:= \inf\big\{t>0\big| \|u^{N}\|_{\frac{1}{2}-\delta,\frac{1}{2}-\delta,t}\geq R\big\}\wedge \inf\big\{t>0\big|
\|D^{N}\|_{\frac{1}{2}-\delta,\frac{1}{2}-\delta,t}\geq \tilde{R} \big\},$$
we have $T^*=T^*(N)>0$ almost surely.  We claim that, if $\omega\in\Omega_{T}$, then $T^*\geq T$ for each $N>0$;  \eqref{Eqn:LWP-16} and \eqref{Eqn:LWP-17} follow from the definition of $T^*$.  Suppose not, and $T^{*}<T$ for some $\omega\in\Omega_{T}$.  Then we have
\begin{align}
\|\mathcal{D}^{N}\|_{\frac{1}{2}+\delta,
\frac{1}{2}-\delta,T^*} &\leq
C_{1}(T^*)^{\theta}
Z\big(\|u_{0,\omega}^{N}\|_{H^{\frac{1}{2}-\delta}},
\|u^{N}\|_{\frac{1}{2}-\delta,\frac{1}{2}-\delta,T^*},
\|\mathcal{D}^{N}\|_{\frac{1}{2}+\delta,
\frac{1}{2}-\delta,T^*}\big) \notag \\
&\leq
C_{1}(T^*)^{\theta}
Z\Big(\Big\|\sum_{n\neq 0}\frac{g_{n}(\omega)e^{inx}}{
|n|^{1-\epsilon}}\Big\|_{H^{\frac{1}{2}-\delta}},
R,\tilde{R}\Big) \ \ \text{by the definition of }T^*,  \notag \\
&\leq
C_{1}(T^*)^{\theta}
Z\Big(\frac{C}{T^{\frac{\beta}{2}}},
R,\tilde{R}\Big) \ \ \text{if}\ \omega\in\Omega_{T}, \notag \\
&\leq
C_{1}(T^*)^{\theta}
Z\Big(\frac{C}{(T^*)^{\frac{\beta}{2}}},
R,\tilde{R}\Big)
  \ \ \text{since}\ T^*<T, \notag \\
&\leq C_{1}(T^*)^{\tilde{\alpha}}\tilde{R} \label{Eqn:LWP-18}
 \\
&\leq \frac{1}{2}\tilde{R}. \notag
\end{align}
The last two lines follow from the definition of $T$, and the assumption $T^*<T$.
Also, by \eqref{Eqn:LWP-1}, \eqref{Eqn:LWP-18}, $T^*<T$, and the definition of $T$,
\begin{align}
\|u^{N}\|_{\frac{1}{2}-\delta,
\frac{1}{2}-\delta,T^*} &\leq
C_{0}(T^{*})^{\delta}\|u_{0,\omega}^{N}\|_{H^{\frac{1}{2}-\delta}} + \|\mathcal{D}(u^{N})\|_{\frac{1}{2}+\delta,\frac{1}{2}-\delta,T^*}
 \notag \\
&\leq
C_{0}(T^*)^{\alpha}+ C_{1}(T^*)^{\tilde{\alpha}}\tilde{R}
\notag \\
&\leq
C_{0}T^{\alpha}+ C_{1}T^{\tilde{\alpha}}\tilde{R}
\notag \\
&\leq \frac{1}{2}R.
\label{Eqn:LWP-19}
\end{align}
Since $\|u^{N}\|_{\frac{1}{2}-\delta,
\frac{1}{2}-\delta,t}$ and $\|\mathcal{D}^{N}\|_{\frac{1}{2}+\delta,
\frac{1}{2}-\delta,t}$ are increasing and continuous functions of $t$, by \eqref{Eqn:LWP-18} and \eqref{Eqn:LWP-19}, for each fixed $N>0$, $\exists\, T^{*}_{0}=T^{*}_{0}(N)>T^*$ such that $\|u^{N}\|_{\frac{1}{2}-\delta,
\frac{1}{2}-\delta,t}<R$ and $\|\mathcal{D}^{N}\|_{\frac{1}{2}+\delta,
\frac{1}{2}-\delta,t}<\tilde{R}$ for all $t\in[0,T^{*}_{0}]$.  This violates the definition of $T^*$ and we conclude that, if $\omega\in\Omega_{T}$, then $T^{*}\geq T$ for each $N>0$; \eqref{Eqn:LWP-16} and \eqref{Eqn:LWP-17} follow.

To establish the convergence of $u^{N}$ (and
$\mathcal{D}^{N}$), we obtain further restrictions on $T$ by considering,
for $\omega \in \Omega_{T}$,
\begin{align}
\|u^{N}-u^{M}\|_{\frac{1}{2}-\delta,\frac{1}{2}-\delta,T} &\leq \|S(t)(u_{0}^{N}-u_{0}^{M})\|_{\frac{1}{2}-\delta,\frac{1}{2}-\delta} +
\|\mathcal{D}^{N}-\mathcal{D}^{M}\|_{\frac{1}{2}-\delta,\frac{1}{2}-\delta,T}  \notag \\
&\leq M^{-\epsilon}\frac{\tilde{C}}{T^{\frac{\beta}{2}}} +
\|\mathcal{D}^{N}-\mathcal{D}^{M}\|_{\frac{1}{2}+\delta,\frac{1}{2}-\delta,T}.
\label{Eqn:LWP-4}
\end{align}
Then, using the multilinearity of $\mathcal{D}$,
\begin{align}
\|\mathcal{D}^{N}-\mathcal{D}^{M}\|_{\frac{1}{2}+\delta,\frac{1}{2}-\delta,T}  &\leq  \|\mathcal{D}(u^{N}-u^{M},u^{N},u^{N},u^{N})\|_{\frac{1}{2}+\delta,\frac{1}{2}-\delta,T}  \notag \\
&\ \ \ \ \ \ \ +\|\mathcal{D}(u^{M},u^{N}-u^{M},u^{N},u^{N})\|_{\frac{1}{2}+\delta,\frac{1}{2}-\delta,T}  \notag \\
&\ \ \ \ \ \ \ +\|\mathcal{D}(u^{M},u^{M},u^{N}-u^{M},u^{N})\|_{\frac{1}{2}+\delta,\frac{1}{2}-\delta,T} \notag \\
&\ \ \ \ \ \ \ +\|\mathcal{D}(u^{M},u^{M},u^{M},u^{N}-u^{M})\|_{\frac{1}{2}+\delta,\frac{1}{2}-\delta,T}.
\label{Eqn:LWP-5}
\end{align}
Each of the terms in \eqref{Eqn:LWP-5} will be bounded in a similar way.  We bound the first term explicitly using \eqref{Eqn:NL-neg1}-\eqref{Eqn:NL-2b}.  Consider \eqref{Eqn:NL-neg1} and \eqref{Eqn:NL-0-a} applied with $\delta_0=0$, $u_{1}=u^{N}-u^{M}$,
$\Lambda_{1}=\mathbb{P}_{N}-\mathbb{P}_{M}$, and
$u_{k}=u^{N}$,
$\Lambda_{k}=\mathbb{P}_{N}$ for $k=2,3,4$.
This gives
\begin{align}
\|\mathcal{D}_{0}(u^{N}-u^{M},&u^{N},u^{N},u^{N})\|_{\frac{1}{2}+\delta,\frac{1}{2}-\delta,T}  + \|\mathcal{D}_{-1}(u^{N}-u^{M},u^{N},u^{N},u^{N})\|_{\frac{1}{2}+\delta,\frac{1}{2}+\delta,T}  \notag \\
&\lesssim
T^{\theta}\big(M^{-\epsilon}+\|u^{N}-u^{M}\|_{\frac{1}{2}-\delta,\frac{1}{2}-\delta,T}
+
\|\mathcal{D}^{N}-\mathcal{D}^{M}\|_{\frac{1}{2}+\delta,\frac{1}{2}-\delta,T}
\big) \notag \\
&\ \ \ \ \ \ \ \cdot
\big(1+\|u^{N}\|_{\frac{1}{2}-\delta,\frac{1}{2}-\delta,T}
+
\|\mathcal{D}(u^{N})\|_{\frac{1}{2}+\delta,\frac{1}{2}-\delta,T}
\big)^{3}.
\label{Eqn:LWP-5a}
\end{align}
To estimate $\mathcal{D}_{k}(u^{N}-u^{M},u^{N},u^{N},u^{N})$, for $k=1,2,3,4$, we again consider a second iteration of
\eqref{Eqn:LWP-2b} in the $k^{\text{th}}$ factor.
This argument requires modification when we consider $\mathcal{D}_{1}(u^{N}-u^{M},u^{N},u^{N},u^{N})$.
We substitute
\begin{align}
u^{N}-u^{M} = S(t)(u^{N}_{0,\omega}-u^{M}_{0,\omega}) + \mathcal{D}(u^{N})-\mathcal{D}(u^{M}).
\label{Eqn:LWP-5b}
\end{align}
Then to estimate $\mathcal{D}_{1}(S(t)(u_{0,\omega}^{N}-u^{M}_{0,\omega}),u^{N},u^{N}
,u^{N})$, we proceed as in \eqref{Eqn:NL-1-3} above.
To be precise, by \eqref{Eqn:NL-k-1}, Lemma \ref{Lemma:lin1b}, and the definition of $\Omega_{T}$, we have
\begin{align}
\|\mathcal{D}_{1}(S(t)(u^{N}_{0,\omega}-u^{M}_{0,\omega})
,u^{N}&,u^{N},u^{N})\|_
{\frac{1}{2}+\delta,\frac{1}{2}-\delta,T}
\notag \\
&\lesssim T^{\theta}
\|\eta(t)S(t)(u^{N}_{0,\omega}-u^{M}_{0,\omega})\|
_{\frac{1}{2}-\delta,\frac{1}{2}+\delta} \|u_{j}\|_{\frac{1}{2}-\delta,\frac{1}{2}-\delta,T}^{3}
\notag \\
&\lesssim T^{\theta-\frac{\beta}{2}}
M^{-\varepsilon}
 \|u_{j}\|_{\frac{1}{2}-\delta,\frac{1}{2}-\delta,T}^{3}.
\label{Eqn:LWP-20}
\end{align}
Next we estimate $\|\mathcal{D}_{1}(\mathcal{D}^{N}-\mathcal{D}^{M},
u^{N},u^{N},u^{N})\|_{\frac{1}{2}+\delta,\frac{1}{2}-\delta,T}$.
We find
\begin{align}
\|\mathcal{D}_{1}(\mathcal{D}^{N}-\mathcal{D}^{M},u^{N},&u^{N},u^{N})
\|_{\frac{1}{2}+\delta,\frac{1}{2}-\delta,T}
\notag \\ &\lesssim \|\mathcal{D}_{1}(\mathcal{D}(u^{N}-u^{M},u^{N},u^{N},u^{N})
,u^{N},u^{N},u^{N})
\|_{\frac{1}{2}+\delta,\frac{1}{2}-\delta,T}
\notag \\
& \ \ \ \ \ \ + \|\mathcal{D}_{1}(\mathcal{D}(u^{M},u^{N}-u^{M},u^{N},u^{N})
,u^{N},u^{N},u^{N})
\|_{\frac{1}{2}+\delta,\frac{1}{2}-\delta,T}
\notag \\
& \ \ \ \ \ \ + \|\mathcal{D}_{1}(\mathcal{D}(u^{M},u^{M},u^{N}-u^{M},u^{N})
,u^{N},u^{N},u^{N})
\|_{\frac{1}{2}+\delta,\frac{1}{2}-\delta,T}
\notag \\
& \ \ \ \ \ \ + \|\mathcal{D}_{1}(\mathcal{D}(u^{M},u^{M},u^{M},u^{N}-u^{M})
,u^{N},u^{N},u^{N})
\|_{\frac{1}{2}+\delta,\frac{1}{2}-\delta,T}.
\label{Eqn:LWP-6a}
\end{align}
Each of the terms in \eqref{Eqn:LWP-6a} will be bounded in a similar way, we bound the first term explicitly.  Applying \eqref{Eqn:NL-1} with $\delta_0=0$, $u_{5}=u^{N}-u^{M}$, and
$u_{k}=u^{N}$ for $k=2,3,4,6,7,8$, we find
\begin{align*}
\|\mathcal{D}_{1}(\mathcal{D}(u^{N}-u^{M},u^{N},&u^{N},u^{N})
,u^{N},u^{N},u^{N})
\|_{\frac{1}{2}+\delta,\frac{1}{2}+\delta,T} \notag \\
&\lesssim
T^{\theta}\big(M^{-\epsilon}+\|u^{N}-u^{M}\|_{\frac{1}{2}-\delta,\frac{1}{2}-\delta,T}
+
\|\mathcal{D}^{N}-\mathcal{D}^{M}\|_{\frac{1}{2}+\delta,\frac{1}{2}-\delta,T}
\big) \notag \\
&\ \ \ \ \ \ \ \cdot
\big(1+\|u^{N}\|_{\frac{1}{2}-\delta,\frac{1}{2}-\delta,T}
+
\|\mathcal{D}(u^{N})\|_{\frac{1}{2}+\delta,\frac{1}{2}-\delta,T}
\big)^{6}.
\end{align*}
With \eqref{Eqn:LWP-6a}, this leads to the bound
\begin{align*}
\|\mathcal{D}_{1}(u^{N}-u^{M}&,u^{N},u^{N},u^{N})
\|_{\frac{1}{2}+\delta,\frac{1}{2}+\delta,T} \notag \\
&\lesssim
T^{\theta}\big(M^{-\epsilon}+\|u^{N}-u^{M}\|_{\frac{1}{2}-\delta,\frac{1}{2}-\delta,T}
+
\|\mathcal{D}^{N}-\mathcal{D}^{M}\|_{\frac{1}{2}+\delta,\frac{1}{2}-\delta,T}
\big)\notag \\
&\ \ \ \ \ \ \ \cdot
\big(1+\|u^{N}\|_{\frac{1}{2}-\delta,\frac{1}{2}-\delta,T}
+ \|u^{M}\|_{\frac{1}{2}-\delta,\frac{1}{2}-\delta,T}
\notag \\
&\ \ \ \ \ \ \ \ \ \ \ \ \ \ \ \ \ \ \ \ +
\|\mathcal{D}(u^{N})\|_{\frac{1}{2}+\delta,\frac{1}{2}-\delta,T}
+ \|\mathcal{D}(u^{M})\|_{\frac{1}{2}+\delta,\frac{1}{2}-\delta,T}
\big)^{6}.
\end{align*}
With similar arguments, using the inequalities \eqref{Eqn:NL-2}-\eqref{Eqn:NL-2b}, we find
\begin{align}
\|\mathcal{D}_{k}(u^{N}-u^{M}&,u^{N},u^{N},u^{N})
\|_{\frac{1}{2}+\delta,\frac{1}{2}+\delta,T} \notag \\
&\lesssim
T^{\theta}\big(M^{-\epsilon}+\|u^{N}-u^{M}\|_{\frac{1}{2}-\delta,\frac{1}{2}-\delta,T}
+
\|\mathcal{D}^{N}-\mathcal{D}^{M}\|_{\frac{1}{2}+\delta,\frac{1}{2}-\delta,T}
\big)\notag \\
&\ \ \ \ \ \ \ \cdot
\big(1+\|u^{N}\|_{\frac{1}{2}-\delta,\frac{1}{2}-\delta,T}
+ \|u^{M}\|_{\frac{1}{2}-\delta,\frac{1}{2}-\delta,T}
\notag \\
&\ \ \ \ \ \ \ \ \ \ \ \ \ \ \ \ \ \ \ \ +
\|\mathcal{D}(u^{N})\|_{\frac{1}{2}+\delta,\frac{1}{2}-\delta,T}
+ \|\mathcal{D}(u^{M})\|_{\frac{1}{2}+\delta,\frac{1}{2}-\delta,T}
\big)^{6}.
\label{Eqn:LWP-6b}
\end{align}
for all $k=1,2,3,4$.  Combining \eqref{Eqn:LWP-5a} and \eqref{Eqn:LWP-6b} we have
\begin{align}
\|\mathcal{D}(u^{N}-u^{M}&,u^{N},u^{N},u^{N})\|_{\frac{1}{2}+\delta,\frac{1}{2}-\delta,T}  \notag \\
&\lesssim T^{\theta}\big(M^{-\epsilon}+\|u^{N}-u^{M}\|_{\frac{1}{2}-\delta,\frac{1}{2}-\delta,T}
+
\|\mathcal{D}^{N}-\mathcal{D}^{M}\|_{\frac{1}{2}+\delta,\frac{1}{2}-\delta,T}
\big)\notag \\
&\ \ \cdot
Z_{0}\big(\|u^{N}\|_{\frac{1}{2}-\delta,\frac{1}{2}-\delta,T},
\|u^{M}\|_{\frac{1}{2}-\delta,\frac{1}{2}-\delta,T},
\|\mathcal{D}^{N}\|_{\frac{1}{2}+\delta,\frac{1}{2}-\delta,T},
\|\mathcal{D}^{M}\|_{\frac{1}{2}+\delta,\frac{1}{2}-\delta,T}\big).
\label{Eqn:LWP-6c}
\end{align}
where $Z_{0}(x,y,z,w)$ is a polynomial of degree 6 with positive coefficients.  Each of the terms in \eqref{Eqn:LWP-5} can be bounded with similar arguments.  This leads to
\begin{align}
\|\mathcal{D}^{N}&-\mathcal{D}^{M}\|_{\frac{1}{2}+\delta,\frac{1}{2}-\delta,T}
\notag \\ &\leq C_{1}T^{\theta}\big(M^{-\epsilon}+\|u^{N}-u^{M}\|_{\frac{1}{2}-\delta,\frac{1}{2}-\delta,T}
+
\|\mathcal{D}^{N}-\mathcal{D}^{M}\|_{\frac{1}{2}+\delta,\frac{1}{2}-\delta,T}
\big)\notag \\
&\ \ \ \ \ \ \ \ \ \  \cdot Z_{1}\big(\|u^{N}\|_{\frac{1}{2}-\delta,\frac{1}{2}-\delta,T},
\|u^{M}\|_{\frac{1}{2}-\delta,\frac{1}{2}-\delta,T},
\|\mathcal{D}^{N}\|_{\frac{1}{2}+\delta,\frac{1}{2}-\delta,T},
\|\mathcal{D}^{M}\|_{\frac{1}{2}+\delta,\frac{1}{2}-\delta,T}\big),
\label{Eqn:LWP-6d}
\end{align}
where $Z_{1}(x,y,z,w)$ is a polynomial of degree 6 with positive coefficients.  If we choose $T>0$ sufficiently small such that
\begin{align}
C_{1}T^{\theta}Z_{1}(R,R,\tilde{R},\tilde{R}) \leq \frac{1}{4},
\label{Eqn:LWP-21}
\end{align}
we find from \eqref{Eqn:LWP-6d}, \eqref{Eqn:LWP-16} \eqref{Eqn:LWP-17} and \eqref{Eqn:LWP-21},
\begin{align}
\|\mathcal{D}^{N}-\mathcal{D}^{M}\|_{\frac{1}{2}+\delta,\frac{1}{2}-\delta,T} \leq \frac{1}{2}M^{-\epsilon}  + \frac{1}{2}\|u^{N}-u^{M}\|_{\frac{1}{2}-\delta,\frac{1}{2}-\delta,T}.
\label{Eqn:LWP-8}
\end{align}
Then combining \eqref{Eqn:LWP-4} and \eqref{Eqn:LWP-8}, we have
\begin{align*}
\|u^{N}-u^{M}\|_{\frac{1}{2}-\delta,\frac{1}{2}-\delta,T} \leq (\tilde{C}T^{\delta-\frac{\beta}{2}}+1)M^{-\epsilon}
\lesssim M^{-\epsilon},
\end{align*}
by taking $T<1$.  With \eqref{Eqn:LWP-8} this gives
\begin{align*}
\|\mathcal{D}^{N}-\mathcal{D}^{M}\|_{\frac{1}{2}+\delta,\frac{1}{2}-\delta,T} \lesssim M^{-\epsilon},
\end{align*}
and we conclude that \eqref{Eqn:LWP-1a} and \eqref{Eqn:LWP-1b} hold for $\omega\in\Omega_{T}$.
By \eqref{Eqn:LWP-1a} and \eqref{Eqn:LWP-1b}, $u^{N}$
and $\mathcal{D}(u^{N})$ converge in $X^{\frac{1}{2}-\delta,\frac{1}{2}-\delta,T}$ and
$X^{\frac{1}{2}+\delta,\frac{1}{2}-\delta,T}$, respectively, for $\omega\in\Omega_{T}$.  It remains to be shown that, for $\omega\in\Omega_{T}$,
\begin{enumerate}[(i)]
\item The convergent $u$ is indeed a solution to \eqref{Eqn:gKdV-zeromean} with initial data $u_{0,\omega}$.
\item $u-S(t)u_{0,\omega}\in C([0,T];H^{\frac{1}{2}+\delta}(\mathbb{T}))$.
\item $u$ is unique in $\Big\{S(t)u_{0,\omega}
    + \{\|\cdot\|_{\frac{1}{2}+\delta,\frac{1}{2}-\delta,T} \leq \tilde{R}\}\Big\}$.
\item $u$ depends continuously on the initial data, in the sense that the solution map $\Phi:
    \Big\{u_{0,\omega} + \{\|\cdot\|_{H^{\frac{1}{2}+\delta}}\leq K\}\Big\}\rightarrow  \Big\{S(t)u_{0,\omega} + \{\|\cdot\|_{C([0,T];H^{\frac{1}{2}+\delta})} \leq \tilde{K}\}
    \Big\}$ is Lipschitz.
\item The solution $u$ is well-approximated by the solution of \eqref{Eqn:gKdV-ZMN}.  More precisely,
    \begin{align*}
    \|u-S(t)u_{0,\omega}
    -(\Phi^{N}(t)-S(t))\mathbb{P}_{N}u_{0,\omega}
    \|_{C([0,T];H^{\frac{1}{2}+\delta_1})}
    \lesssim {N}^{-\beta}.
    \end{align*}
\end{enumerate}
To establish (i), we need to prove that $u= \lim_{N\rightarrow \infty}u^{N}$
satisfies
\begin{align}
u=S(t)u_{0,\omega}+\mathcal{D}(u),
\label{Eqn:LWP-10}
\end{align}
in the sense of distributions.  That is, we show that for all test functions
$\varphi\in \mathcal{D}(\mathbb{T}\times[0,T])$, we have
$$ \langle u,\varphi\rangle
= \langle S(t)u_{0,\omega}+\mathcal{D}(u),\varphi \rangle. $$

Clearly $S(t)u_{0,\omega}^{N} \rightarrow S(t)u_{0,\omega}$
in $X^{\frac{1}{2}-\delta,\frac{1}{2}-\delta}_{T}$, for $\omega\in\Omega_{T}$, and
$ \mathcal{D}_{0}(u^{N}) \rightarrow \mathcal{D}_{0}(u)$
in $X^{\frac{1}{2}+\delta,\frac{1}{2}-\delta}_{T}$,
by multilinearity of $\mathcal{D}_{0}$ and the deterministic estimate \eqref{Eqn:NL-0-a}.  Then from \eqref{Eqn:LWP-5a}, \eqref{Eqn:LWP-6b}, \eqref{Eqn:LWP-1a}, \eqref{Eqn:LWP-1b}, and multilinearity
we have that for $\omega \in \Omega_{T}$, $\mathcal{D}_{j}(u^{N})$ is Cauchy in $X^{\frac{1}{2}+\delta,\frac{1}{2}+\delta}_{T}$.  That is,
\begin{align}
\mathcal{D}_{j}(u^{N}) \rightarrow v_{j} \ \ \text{in}\ X^{\frac{1}{2}+\delta,\frac{1}{2}+\delta}_{T}
\label{Eqn:LWP-11}
\end{align}
for some $v_{j}\in X^{\frac{1}{2}+\delta,\frac{1}{2}+\delta}_{T}$, $j=-1,1,2,3,4$.  We can then express
\begin{align}
u=S(t)u_{0,\omega} + \mathcal{D}_{0}(u) + v_{-1} + v_{1} + \cdots + v_{4}.
\label{Eqn:LWP-12}
\end{align}
It remains to be shown that $v_{j}= \mathcal{D}_{j}(u)$
for each $j=-1,1,2,3,4$.  For $j=-1,2,3,4$, we consider, using the multilinearity of $\mathcal{D}_{j}$,
\begin{align}
\|\mathcal{D}_{j}(u)-\mathcal{D}_{j}(u^{N})\|_{\frac{1}{2},\frac{1}{2}+\delta,T}  &\leq  \|\mathcal{D}_{j}(u-u^{N},u,u,u)\|_{\frac{1}{2},\frac{1}{2}+\delta,T}  \notag \\
&\ \ \ \ \ \ \ +\|\mathcal{D}_{j}(u^{N},u-u^{N},u,u)\|_{\frac{1}{2},\frac{1}{2}+\delta,T}  \notag \\
&\ \ \ \ \ \ \ +\|\mathcal{D}_{j}(u^{N},u^{N},u-u^{N},u)\|_{\frac{1}{2},\frac{1}{2}+\delta,T} \notag \\
&\ \ \ \ \ \ \ +\|\mathcal{D}_{j}(u^N,u^N,u^N,u-u^N)\|_{\frac{1}{2},\frac{1}{2}+\delta,T}.
\label{Eqn:LWP-13b}
\end{align}
We bound the first term in \eqref{Eqn:LWP-13b} for each $j=-1,2,3,4$.  For $j=-1$, applying \eqref{Eqn:NL-1} with $\delta_0=0$, $u_{1}=u-u^{N}$, $\Lambda_{1}=\text{Id}-\mathbb{P}_{N}$ and
$u_{k}=u$, $\Lambda_{k}=\text{Id}$, for $k=2,3,4$, we find for $\omega\in\Omega_{T}$, that
\begin{align*}
\|\mathcal{D}_{-1}&(u-u^N,u,u,u)\|_{\frac{1}{2}+\delta,\frac{1}{2}+\delta,T}
\notag \\
&\lesssim
T^{\theta}\big(N^{-\epsilon}+\|u-u^{N}\|_{\frac{1}{2}-\delta,\frac{1}{2}-\delta,T}
+
\|\mathcal{D}_{0}(u)+\sum_{\substack{k=-1 \\ k\neq 0}}^{4}v_{k}-\mathcal{D}(u^{N})\|_{\frac{1}{2}+\delta,\frac{1}{2}-\delta,T}
\big) \notag \\
&\ \ \ \ \ \ \ \cdot
\big(1+\|u\|_{\frac{1}{2}-\delta,\frac{1}{2}-\delta,T}
+
\|\mathcal{D}_{0}(u)+\sum_{\substack{k=-1 \\ k\neq 0}}^{4}v_{k}\|_{\frac{1}{2}+\delta,\frac{1}{2}-\delta,T}
\big)^{3} \notag \\
&\lesssim
T^{\theta}\big(N^{-\epsilon}+\|u-u^{N}\|_{\frac{1}{2}-\delta,\frac{1}{2}-\delta,T}
+
\sum_{\substack{k=-1 \\ k\neq 0}}^{4}\|v_{k}-\mathcal{D}_{k}(u^{N})\|_{\frac{1}{2}+\delta,\frac{1}{2}-\delta,T}
\notag\\
&\ \ \ \ \ \ \ \ \ \ \ \ \ \ \ \ \ \ \ \ \ \ \ \ \ \ \ \ \ \ \ \ \ \ \ \ \ +
\|\mathcal{D}_{0}(u)-\mathcal{D}_{0}(u^N)\|_{\frac{1}{2}+\delta,\frac{1}{2}-\delta,T}
\big) \big(1+R + \tilde{R} \big)^{3} \notag \\
&\rightarrow 0, \ \ \ \ \ \text{as}\ N\rightarrow \infty,
\end{align*}
by \eqref{Eqn:NL-0-a} and \eqref{Eqn:LWP-11}.

For $j=2,3,4$, we will show how to bound
$\|\mathcal{D}_{2}(u^N,u-u^N,u,u)\|_{\frac{1}{2},\frac{1}{2}+\delta,T}$ (where the difference appears in the second factor).  Treating other types of terms will follow from the same estimates (or simpler ones).  Inserting \eqref{Eqn:LWP-2b} and \eqref{Eqn:LWP-12} into the second factor, we find
\begin{align}
\|\mathcal{D}_{2}(u^N,&u-u^N,u,u)\|_{\frac{1}{2},\frac{1}{2}+\delta,T}
\notag  \\
&\leq \|\mathcal{D}_{2}(u^N,(\text{Id}-
\mathbb{P}_{N})S(t)u_{0,\omega},u,u)\|_{\frac{1}{2},\frac{1}{2}+\delta,T}
\notag  \\ &\ \ \ \ \ \ \ \ +
\|\mathcal{D}_{2}(u^N,\mathcal{D}_{0}(u)+\sum_{\substack{k=-1 \\ k\neq 0}}v_{k}-\mathcal{D}(u^{N}),u,u)\|_{\frac{1}{2},\frac{1}{2}+\delta,T}
\label{Eqn:LWP-22-new}
\end{align}
The first term in \eqref{Eqn:LWP-22-new} is bounded (and decays to zero) exactly as in \eqref{Eqn:LWP-20} above, by using \eqref{Eqn:NL-k-2}.  For the second term, we consider
\begin{align}
\|\mathcal{D}_{2}(u^N,
\mathcal{D}_{0}(u)+&\sum_{\substack{k=-1 \\ k\neq 0}}v_{k}-\mathcal{D}(u^{N}),u,u)
\|_{\frac{1}{2},\frac{1}{2}+\delta,T}
\notag \\
&\leq
\|\mathcal{D}_{2}(u^N,\mathcal{D}_{0}(u)
-\mathcal{D}_{0}(u^{N}),u,u)\|_{\frac{1}{2},\frac{1}{2}+\delta,T}
\notag \\
&\ \ \ \ \ \ \ \ \ \ \ + \sum_{\substack{k=-1 \\ k\neq 0}}^{4}\|\mathcal{D}_{2}(u^N,v_{k}-\mathcal{D}_{k}(u^{N}),
u,u)\|_{\frac{1}{2},\frac{1}{2}+\delta,T}.
\label{Eqn:LWP-15-new}
\end{align}
The first term in \eqref{Eqn:LWP-15-new} is bounded
using multilinearity of $\mathcal{D}_{0}$ and
a telescoping sum, as considered in \eqref{Eqn:LWP-6a}.
This reduces to estimating terms of the form
\begin{align*}
\|\mathcal{D}_{2}(u^N,\mathcal{D}_{0}(u-u^N,u,u,u),u,u)\|_{\frac{1}{2}+\delta,\frac{1}{2}+\delta,T}.
\end{align*}
By \eqref{Eqn:NL-2b} we have
\begin{align}
\|\mathcal{D}_{2}(u^N,\mathcal{D}_{0}&(u-u^N,u,u,u)
,u,u)
\|_{\frac{1}{2}+\delta,\frac{1}{2}+\delta,T} \notag \\
&\lesssim
T^{\theta}\|u-u^N\|_{\frac{1}{2}-\delta,\frac{1}{2}-\delta,T}\|u^N\|_{\frac{1}{2}-\delta,\frac{1}{2}-\delta,T}
\|u\|_{\frac{1}{2}-\delta,\frac{1}{2}-\delta,T}^{5}
\rightarrow 0,
\label{Eqn:LWP-24-new}
\end{align}
as $N\rightarrow \infty$.  The remaining terms in \eqref{Eqn:LWP-15-new} are bounded using the deterministic estimate \eqref{Eqn:NL-k-2}.  That is, for each $k\in\{-1,1,2,3,4\}$, \eqref{Eqn:NL-k-2} (with $\delta_0=0$) gives
\begin{align}
\|\mathcal{D}_{2}(u^N,v_{k}-\mathcal{D}_{k}(u^N),u,u)
\|_{\frac{1}{2},\frac{1}{2}+\delta,T}
&\lesssim  \|v_{k}-\mathcal{D}_{k}(u^{N})\|_{\frac{1}{2}-\delta,\frac{1}{2}+\delta,T}
\|u^N\|_{\frac{1}{2}-\delta,\frac{1}{2}-\delta,T}\|u\|^{2}_{\frac{1}{2}-\delta,\frac{1}{2}-\delta,T}
\notag \\
&\rightarrow 0,
\label{Eqn:LWP-23-new}
\end{align}
as $N\rightarrow \infty$, by \eqref{Eqn:LWP-11}.  Thus, with \eqref{Eqn:LWP-15-new}, \eqref{Eqn:LWP-24-new}, \eqref{Eqn:LWP-23-new} and an expansion similar to \eqref{Eqn:LWP-6a}, we have
\begin{align*}
\|\mathcal{D}_{2}(u^N,u-u^N,u,u)\|_{\frac{1}{2}+\delta,\frac{1}{2}+\delta,T}
\rightarrow 0,
\end{align*}
as $N\rightarrow \infty$.

It remains to show that $v_{1}=\mathcal{D}_{1}(u)$, in the sense of distributions.
In other words, we still need to prove that,
for all test functions
$\varphi\in \mathcal{D}(\mathbb{T}\times[0,T])$, we have
\begin{align}
\langle \mathcal{D}_{1}(u)-\mathcal{D}_{1}(u^{N}),\varphi\rangle \rightarrow 0, \ \ \text{as} \  N\rightarrow \infty.
\label{Eqn:LWP-13-1}
\end{align}

Once again using a telescoping sum, this is reduced to establishing that
\begin{align}
\langle \mathcal{D}_{1}(u-u^N,u,u,u),\varphi\rangle \rightarrow 0, \ \ \text{as} \  N\rightarrow \infty.
\label{Eqn:LWP-13-2}
\end{align}
Inserting \eqref{Eqn:LWP-2b} and \eqref{Eqn:LWP-12} into the first factor, we find
\begin{align}
\langle\mathcal{D}_{1}(u-u^N,&u,u,u),\varphi\rangle
\notag  \\
&= \langle\mathcal{D}_{1}((\text{Id}-
\mathbb{P}_{N})S(t)u_{0,\omega},u,u,u),\varphi\rangle
\notag  \\ &\ \ \ \ \ \ \ \ +
\langle\mathcal{D}_{1}(\mathcal{D}_{0}(u)+\sum_{\substack{k=-1 \\ k\neq 0}}v_{k}-\mathcal{D}(u^{N}),u,u,u)
,\varphi\rangle.
\label{Eqn:LWP-22-2}
\end{align}
The first term in \eqref{Eqn:LWP-22-2} is bounded (and decays to zero) exactly as in \eqref{Eqn:LWP-20} above (using H\"{o}lder).  For the second term, we expand
\begin{align}
\langle\mathcal{D}_{1}(
\mathcal{D}_{0}(u)+&\sum_{\substack{k=-1 \\ k\neq 0}}v_{k}-\mathcal{D}(u^{N}),u,u,u),
\varphi\rangle
\notag \\
&=
\langle\mathcal{D}_{1}(\mathcal{D}_{0}(u)
-\mathcal{D}_{0}(u^{N}),u,u,u),\varphi\rangle
\notag \\
&\ \ \ \ \ \ \ \ \ \ \ + \sum_{\substack{k=-1 \\ k\neq 0}}^{4}\langle \mathcal{D}_{1}(v_{k}-\mathcal{D}_{k}(u^{N}),
u,u,u),\varphi\rangle.
\label{Eqn:LWP-15-2}
\end{align}
Each of the last five terms in \eqref{Eqn:LWP-15-2} is bounded (and goes to zero) using H\"{o}lder and \eqref{Eqn:NL-k-1}, as in \eqref{Eqn:LWP-23-new} above.

For the first term in \eqref{Eqn:LWP-15-2}, we use a telescoping sum and reduce to establishing statements of the type
\begin{align}
\langle\mathcal{D}_{1}(\mathcal{D}_{0}(u-u^N,u,u,u),u,u,u),
\varphi\rangle \rightarrow 0, \ \ \text{as} \  N\rightarrow \infty.
\label{Eqn:LWP-13-2}
\end{align}
Here is where we take advantage of the distributional formulation of equivalence.
More precisely, as described in Section \ref{Sec:probhept} (of the appendix), the expression $$\mathcal{D}_{1}(\mathcal{D}_{0}(u_5,u_6,u_7,u_8),u_2,u_3,u_4)(x,t),$$
as appearing in the statement of Proposition \ref{Prop:nonlin}, is defined (in the region $A_{1,c}$) by the right-hand side of \eqref{Eqn:NLdef2} (after imposing a certain cancellation).  On the other hand, since we are attempting to demonstrate here that $\mathcal{D}(u^N)\rightarrow \mathcal{D}(u)$, the expressions
\begin{align}\mathcal{D}_{1}(\mathcal{D}_{0}(u,u,u,u),u,u,u)(x,t)
\label{Eqn:LWP-14-2b}\end{align}
and
\begin{align}\mathcal{D}_{1}(\mathcal{D}_{0}(u^N,u^N,u^N,u^N),u^N,u^N,u^N)(x,t)
\label{Eqn:LWP-14-2}\end{align}
are implicitly defined (in $A_{1,c}$) through the second last term in \eqref{Eqn:NLdef3} (before the cancellation), and we must justify that this can be replaced with \eqref{Eqn:NLdef2}, in order for Proposition \ref{Prop:nonlin} to be applied.  This can be done by showing that the contribution to \eqref{Eqn:LWP-14-2b} and \eqref{Eqn:LWP-14-2} from $A_{1,c}$ is equivalent in distribution whether we use the right-hand side of \eqref{Eqn:NLdef3} or \eqref{Eqn:NLdef2} for the definition of $K_{1}(n,\tau)$.

Let us slow down to explain this point more carefully.  When we estimated \eqref{Eqn:LWP-14-2} above, and used Proposition \ref{Prop:nonlin}, this was justified since the solutions $u^N$, which evolve from frequency truncated data, are smooth, and the cancellation leading to \eqref{Eqn:NLdef2}, witnessed for each fixed $(n,\tau)\in(\mathbb{Z}\setminus\{0\})\times \mathbb{R}$, extends to the full contribution to \eqref{Eqn:LWP-14-2} from the region $A_{1,c}$.  That is, when estimating \eqref{Eqn:LWP-14-2} above, we were able to define the contribution to \eqref{Eqn:LWP-14-2} from $A_{1,c}$ according to \eqref{Eqn:NLdef2} (taking advantage of the cancellation).  However, when the input is $u$ (the limit of the sequence $u^N$), it is less obvious that the cancellation holds, and any application of Proposition \ref{Prop:nonlin} with these ``rough'' inputs needs to be justified carefully.  In other words, Proposition \ref{Prop:nonlin} assumes that the nonlinearity is defined through \eqref{Eqn:NLdef2} (in the region $A_{1,c}$), but this is not obvious if the inputs are the limit $u$ of the sequence $u^N$.

We will verify directly that the cancellation leading to \eqref{Eqn:NLdef2} holds (when the input is $u$) in the sense of distributions.  More precisely, if $K_{1}(n,\tau)$ is defined by the right-hand side of \eqref{Eqn:NLdef3}, and $J_1(n,\tau)$ is defined using \eqref{Eqn:NLdef2} instead, we will verify that for all $\varphi\in \mathcal{D}(\mathbb{T}\times[0,T])$, we have
\begin{align}
\langle (K_{1})^{\vee}(\cdot,\cdot),\varphi\rangle = \langle (J_{1})^{\vee}(\cdot,\cdot),\varphi\rangle.
\label{Eqn:LWP-15-new-2}
\end{align}
Indeed, we can use H\"{o}lder to find
\begin{align}
\langle (K_{1})^{\vee}(\cdot,\cdot)-(J_{1})^{\vee}(\cdot,\cdot),\varphi \rangle
\leq \|(K_{1})^{\vee}-(J_{1})^{\vee}\|_{L^{2}_{x,t\in[0,T]}}\|\varphi\|_{L^{2}_{x,t\in[0,T]}}.
\label{Eqn:LWP-16-2}
\end{align}
By the cancelation that led to \eqref{Eqn:NLdef2} for each fixed $(n,\tau)$, if the right-hand side of \eqref{Eqn:LWP-16-2} is finite, it is necessarily zero.  It is shown in line \eqref{Eqn:NL-1-7} of the Appendix (and the lines that precede it) that the contribution to
\begin{align*}
\|\mathcal{D}_{1}(\mathcal{D}_{0}(u,u,u,u),u,u,u)\|_{\frac{1}{2}+\delta,\frac{1}{2}+\delta,T}
\end{align*}
from $J_{1}(n,\tau)$ (after cancelation) is finite (since $u\in X^{\frac{1}{2}-\delta,\frac{1}{2}-\delta}_{T}$).  If, instead, we wish to estimate the contribution to
\begin{align*}
\|\mathcal{D}_{1}(\mathcal{D}_{0}(u,u,u,u),u,u,u)\|_{L^2_{x,t\in[0,T]}}
\end{align*}
from $K_{1}(n,\tau)$ (that is, to estimate $\|(K_{1})^{\vee}\|_{L^{2}_{x,t\in[0,T]}}$), it turns out that we can proceed with the exact same analysis, because we are using a weaker norm.

The benefit using the reformulation \eqref{Eqn:NLdef2} is the introduction of an extra power of $N^0$ in the denominator of the estimate \eqref{Eqn:NL-1-7}.  More precisely, using the definition \eqref{Eqn:NLdef3} instead of \eqref{Eqn:NLdef2}, we could try to use
\begin{align}
\frac{|n_1|}{|\sigma_{1}|}\lesssim \frac{1}{(N^0)^{1-\gamma}}
\label{Eqn:LWP-17-2}
\end{align}
to replace \eqref{Eqn:Ac1}-\eqref{Eqn:Ac2}, but the resulting estimate as in line \eqref{Eqn:NL-1-7} would not be dyadically summable.  This is why we need to use the definition \eqref{Eqn:NLdef2} during the proof of Proposition \ref{Prop:nonlin}; we exploit the estimates  \eqref{Eqn:Ac1}-\eqref{Eqn:Ac2}, which have an extra power of $N^0$ in the denominator, in order to produce the estimate \eqref{Eqn:NLdef2}.

However, if we are trying to control $\|(K_{1})^{\vee}\|_{L^{2}_{x,t\in[0,T]}}$ (that is, using the $L^{2}_{x,t\in[0,T]}$-norm instead of the $X^{\frac{1}{2}+\delta,\frac{1}{2}+\delta}_{T}$-norm), we can use \eqref{Eqn:LWP-17-2} and Lemma \ref{Lemma:lin2} to modify the analysis that led to \eqref{Eqn:NL-1-7}, and find
\begin{align*}
\|(K_{1})^{\vee}\|_{L^{2}_{x,t\in[0,T]}} &= \|(K_{1})^{\vee}\|_{0,0,T} \sim \|\mathcal{D}_{1}(\mathcal{D}_{0}(u,u,u,u),u,u,u)\|_{0,0,T} \\
&\lesssim \|\mathcal{N}_{1}(\mathcal{D}_{0}(u,u,u,u),u,u,u)\|_{0,-1,T}
\\
&\lesssim \frac{T^{-\frac{\beta}{2}}}{(N^0)^{\frac{1}{2}-3\gamma-6\delta^{\frac{3}{2}}-2\beta}}
\|u\|_{\frac{1}{2}-\delta,\frac{1}{2}-\delta,T}^{7}.
\end{align*}
Since $u\in X^{\frac{1}{2}-\delta,\frac{1}{2}-\delta,T}$, we conclude that $(K_{1})^{\vee}\in L^{2}_{x,t\in[0,T]}$, and the equality \eqref{Eqn:LWP-15-new-2} holds.  Therefore, in the definition of \eqref{Eqn:LWP-14-2b} in the region $A_{1,c}$, we can interpret $K_{1}(n,\tau)$ using either \eqref{Eqn:NLdef3} or \eqref{Eqn:NLdef2}.  In particular, we can establish \eqref{Eqn:LWP-13-2} as we did \eqref{Eqn:LWP-24-new}.  Applying Proposition \ref{Prop:nonlin} with
$u_{5}=u-u^N$, $\Lambda_{5}=\text{Id}-\mathbb{P}_{N}$, and $u_{k}=u$, $\Lambda_{k}=\text{Id}$ for $k=2,3,4,6,7,8$, and $\delta_0=0$, we have
\begin{align}
\|\mathcal{D}_{1}(\mathcal{D}_{0}&(u-u^N,u,u,u)
u,u,u)
\|_{\frac{1}{2}+\delta,\frac{1}{2}+\delta,T} \notag \\
&\lesssim
T^{\theta}\|u-u^N\|_{\frac{1}{2}-\delta,\frac{1}{2}-\delta,T}
\|u\|_{\frac{1}{2}-\delta,\frac{1}{2}-\delta,T}^{6}
\rightarrow 0,
\label{Eqn:LWP-30-new}
\end{align}
as $N\rightarrow \infty$.  Having established \eqref{Eqn:LWP-13-2}, we have $v_1=\mathcal{D}_{1}(u)$, and \eqref{Eqn:LWP-10} follows.  We conclude that $u$ is indeed a solution to \eqref{Eqn:gKdV-zeromean} with data $u_{0,\omega}$ for $t\in[0,T]$.  The discussion of point (i) is complete.

Let us emphasize that thanks to the preceding discussion (and in particular \eqref{Eqn:LWP-15-new-2}) we may now use the definition \eqref{Eqn:NLdef2} (for $K_{1}(n,\tau)$) when we consider \eqref{Eqn:LWP-14-2}; i.e. we have justified that the cancellation holds for the limit $u$, and this may be used in the future.
\medskip

To address point (ii), we remark that by \eqref{Eqn:LWP-10}, \eqref{Eqn:LWP-11} and \eqref{Eqn:linear-Y}, if $\omega\in\Omega_{T}$, then $\mathcal{D}_{j}(u)\in C([0,T];H^{\frac{1}{2}+\delta}(\mathbb{T}))$ for all
$j\in\{-1,1,2,3,4\}$.  For $j=0$, we have by \eqref{Eqn:linear-Y} and \eqref{Eqn:NL-0-b} that
\begin{align*}
\|\mathcal{D}_{0}(u)-\mathcal{D}_{0}(u^N)\|
_{C([0,T];H^{\frac{1}{2}+\delta}(\mathbb{T}))}  &\lesssim
\|\mathcal{D}_{0}(u)-\mathcal{D}_{0}(u^N)\|
_{Y^{\frac{1}{2}+\delta,0}_{T}} \\ &\sim \|\mathcal{N}_{0}(u)-\mathcal{N}_{0}(u^N)\|
_{Y^{\frac{1}{2}+\delta,-1}_{T}}
T^{\theta}\|u-u^N\|_{\frac{1}{2}-\delta,\frac{1}{2}-\delta,T}.
\end{align*}
Then with \eqref{Eqn:linear-Y} we conclude that $\mathcal{D}_{0}(u)\in C([0,T];H^{\frac{1}{2}+\delta}(\mathbb{T}))$. Therefore, if $\omega\in\Omega_{T}$, we have $$u-S(t)u_{0,\omega}=\mathcal{D}(u)\in C([0,T];H^{\frac{1}{2}+\delta}(\mathbb{T})).$$

\medskip

Turning to point (iii) (uniqueness), we establish that, for $\omega\in\Omega_{T}$, the solution $u$ to \eqref{Eqn:gKdV-zeromean} with data $u_{0,\omega}$ (obtained as the limit of $u^N$ given by \eqref{Eqn:LWP-2b}) is unique in $\Big\{S(t)u_{0,\omega} + \{\|\cdot\|_{\frac{1}{2}+\delta,\frac{1}{2}-\delta,T}\leq R\}\Big\}$.  Suppose $\tilde{u}$ is another solution to \eqref{Eqn:gKdV-zeromean} with data $u_{0,\omega}$ in this function space.  With the methods used above, if $\omega\in\Omega_{T}$, then
\begin{align}
\|&\mathcal{D}(\tilde{u})-\mathcal{D}(u)\|
_{\frac{1}{2}+\delta,\frac{1}{2}-\delta,T} \notag \\
&\leq T^{\theta}(\|\tilde{u}-u\|
_{\frac{1}{2}-\delta,\frac{1}{2}-\delta,T} + \|\mathcal{D}(\tilde{u})-\mathcal{D}(u)\|
_{\frac{1}{2}+\delta,\frac{1}{2}-\delta,T})
\notag \\
&\ \ \ \cdot Z_{2}(\|\tilde{u}\|
_{\frac{1}{2}-\delta,\frac{1}{2}-\delta,T},
\|u\|_{\frac{1}{2}-\delta,\frac{1}{2}-\delta,T},
\|\mathcal{D}(\tilde{u})\|
_{\frac{1}{2}+\delta,\frac{1}{2}-\delta,T}
,\|\mathcal{D}(u)\|
_{\frac{1}{2}+\delta,\frac{1}{2}-\delta,T}),
\label{Eqn:LWP-26}
\end{align}
where $Z_{2}(x,y,z,w)$ is a polynomial of degree 6 with positive coefficients.  With the definition of $\Omega_{T}$, we have
\begin{align}
\|\mathcal{D}(\tilde{u})-\mathcal{D}(u)\|
_{\frac{1}{2}+\delta,\frac{1}{2}-\delta,T}
&\leq T^{\theta} \|\mathcal{D}(\tilde{u})-\mathcal{D}(u)\|
_{\frac{1}{2}+\delta,\frac{1}{2}-\delta,T}
\notag \\
&\ \ \ \ \ \ \ \ \ \ \cdot Z_{2}(CT^{\alpha}+R,CT^{\alpha}+R,R,R)
\notag \\
&\leq \frac{1}{2}\|\mathcal{D}(\tilde{u})-\mathcal{D}(u)\|
_{\frac{1}{2}+\delta,\frac{1}{2}-\delta,T},
\label{Eqn:LWP-27}
\end{align}
for $T>0$ sufficiently small.  We conclude that $\mathcal{D}(\tilde{u})=\mathcal{D}(u)$ in $X^{\frac{1}{2}+\delta,\frac{1}{2}-\delta}_{T}$, and thus $u=\tilde{u}$ in $X^{\frac{1}{2}-\delta,\frac{1}{2}-\delta}_{T}$, for $\omega\in\Omega_{T}$.  The proof of uniqueness is complete.
\bigskip

Next we discuss point (iv).  We will show that, for $\omega\in\Omega_{T}$, the solution map $\Phi:
\Big\{u_{0,\omega} + \{\|\cdot\|_{H^{\frac{1}{2}+\delta}}\leq R\}\Big\}\rightarrow  \Big\{S(t)u_{0,\omega} + \{\|\cdot\|_{C([0,T];H^{\frac{1}{2}+\delta})} \leq \tilde{R}\}
\Big\}$ for \eqref{Eqn:gKdV-zeromean} is well-defined and Lipschitz.  That is, given $v_{0}$ such that $\|u_{0,\omega}-v_{0}\|_{H^{\frac{1}{2}+\delta}}\leq R$, we will demonstrate that:
\begin{enumerate}[(a)]
\item The solution to \eqref{Eqn:gKdV-zeromean} with data $v_{0}$ exists, is unique in the sense described above, and satisfies $$\|v\|_{\frac{1}{2}-\delta,\frac{1}{2}-\delta,T}\leq R, \ \ \|\mathcal{D}(v)\|_{\frac{1}{2}-\delta,
    \frac{1}{2}-\delta,T}\leq \tilde{R}.$$
\item We have $\|u-v\|_{C([0,T];H^{\frac{1}{2}+\delta}(\mathbb{T}))}
    \lesssim
    \|u_{0,\omega}-v_{0}\|_{H^{\frac{1}{2}+\delta}(\mathbb{T})}$.
\end{enumerate}
To establish point (a), for $N>0$ we let $v^{N}_{0}:=\mathbb{P}_{N}v_{0}$.  By Theorem 2 in \cite{CKSTT1} the solution $v^{N}$ to \eqref{Eqn:gKdV-zeromean} with data $v_{0}^{N}$ exists for all $t\in\mathbb{R}$.  We will show that, if $\omega\in\Omega_{T}$, then for all $N>M>0$,
\begin{align}
\|v^N\|_{\frac{1}{2}-\delta,\frac{1}{2}-\delta,T} \leq R,
\label{Eqn:LWP-28}
\end{align}
\begin{align}
\|\mathcal{D}(v^N)\|_{\frac{1}{2}-\delta,\frac{1}{2}-\delta,T} \leq \tilde{R},
\label{Eqn:LWP-29}
\end{align}
and
\begin{align}
\|v^N-v^M\|_{\frac{1}{2}-\delta,\frac{1}{2}-\delta,T},
\|\mathcal{D}(v^N)-\mathcal{D}(v^M)\|
_{\frac{1}{2}+\delta,\frac{1}{2}-\delta,T} \rightarrow 0, \ \ \text{as}\ M\rightarrow \infty.
\label{Eqn:LWP-30}
\end{align}
The existence of a convergent $v\in X^{\frac{1}{2}-\delta,\frac{1}{2}-\delta}_{T}$ of the $v^N$ follows from \eqref{Eqn:LWP-28}-\eqref{Eqn:LWP-30}.  The justification of points (i)-(iii) (continuity and uniqueness) for the convergent $v$ follows the discussion above (for $u$ with data $u_{0,\omega}$) very closely, and we omit details.  We proceed to justify \eqref{Eqn:LWP-28}-\eqref{Eqn:LWP-30}.

The solution $v^N$ to \eqref{Eqn:gKdV-zeromean} with data $v^{N}_{0}$ satisfies
\begin{align}
v^N &= S(t)v^{N}_{0} + \mathcal{D}(v^N) \notag \\
&= S(t)u^{N}_{0,\omega} + S(t)(v^{N}_{0}-u^{N}_{0,\omega}) + \mathcal{D}(v^N),
\label{Eqn:LWP-31}
\end{align}
and
\begin{align}
v^N &-v^M \notag \\ &= S(t)(\mathbb{P}_{N}-\mathbb{P}_{M})u_{0,\omega} +
S(t)(\mathbb{P}_{N}-\mathbb{P}_{M})(v_{0}-u_{0,\omega}) + \mathcal{D}(v^N)- \mathcal{D}(v^M).
\label{Eqn:LWP-32}
\end{align}
The new contributions to \eqref{Eqn:LWP-31} and \eqref{Eqn:LWP-32} (ie. contributions which were absent in the analysis of the sequence $u^N$ above) satisfy, for any $b\in\mathbb{R}$,
\begin{align}
\|S(t)(v^{N}_{0}-u^{N}_{0,\omega})\|_{\frac{1}{2}+\delta,b,T} &\lesssim  \|v_{0}-u_{0,\omega}\|_{H^{\frac{1}{2}+\delta}(\mathbb{T})},
\label{Eqn:LWP-33}
\end{align}
and
\begin{align}
\|S(t)(\mathbb{P}_{N}-\mathbb{P}_{M})(v_{0}-u_{0,\omega})\|_{\frac{1}{2}+\delta,b,T} &\lesssim  \|(\text{Id}-\mathbb{P}_{M})(v_{0}-u_{0,\omega})\|_{H^{\frac{1}{2}+\delta}(\mathbb{T})}
\notag \\
&\leq C_{M} \ \rightarrow 0,
\label{Eqn:LWP-34}
\end{align}
as $M\rightarrow \infty$.  With the estimates \eqref{Eqn:LWP-33} and \eqref{Eqn:LWP-34}, we can treat the contributions from $(v_{0}-u_{0,\omega})$ in \eqref{Eqn:LWP-31}-\eqref{Eqn:LWP-32} as ``type (II)''; that is, to be estimated in the space $X^{\frac{1}{2}+\delta,\frac{1}{2}-\delta}_{T}$ (with spatial regularity $s=\frac{1}{2}+\delta$).  In fact, we can do better, and estimate these contributions in $X^{\frac{1}{2}+\delta,\frac{1}{2}+\delta}_{T}$ (with temporal regularity $b=\frac{1}{2}+\delta$ as well).  This improvement is crucial, as the proofs of inequalities \eqref{Eqn:LWP-28}-\eqref{Eqn:LWP-30} will require modification (from the proofs of \eqref{Eqn:LWP-1a}-\eqref{Eqn:LWP-17}), when we consider the second iteration of the Duhamel formulation (as in lines \eqref{Eqn:LWP-2b} and \eqref{Eqn:LWP-5b} for $u^N$).  The problem is that cannot expand the contributions from $(v_{0}-u_{0,\omega})$ in \eqref{Eqn:LWP-31} and \eqref{Eqn:LWP-32} into septilinear expression (as we can for $\mathcal{D}(v^{N})$); instead these contributions are bounded using \eqref{Eqn:LWP-33} and \eqref{Eqn:LWP-34} with $b=\frac{1}{2}+\delta$, and the deterministic estimates \eqref{Eqn:NL-k-1} and \eqref{Eqn:NL-k-2}.  Using this approach, we establish
\begin{align*}
\|v^N\|_{\frac{1}{2}-\delta,\frac{1}{2}-\delta,T} \leq C_{0}T^{\delta-}\|u_{0,\omega}^{N}\|_{H^{\frac{1}{2}-\delta}}
 + \tilde{C}T^{\delta-}\|v^{N}_{0}-u^{N}_{0,\omega}\|_{H^{\frac{1}{2}+\delta}}
+ \|\mathcal{D}(v^N)\|_{\frac{1}{2}+\delta,\frac{1}{2}-\delta,T},
\end{align*}
and if $\omega\in\Omega_{T}$, then
\begin{align}
\|\mathcal{D}&(v^N)\|_{\frac{1}{2}+\delta,\frac{1}{2}-\delta,T}\notag \\ &\leq C_{1}T^{\theta}Z_{2}(\|u_{0,\omega}^{N}\|_{H^{\frac{1}{2}-\delta}},\|v_{0}^{N}-u_{0,\omega}^{N}
\|_{H^{\frac{1}{2}+\delta}},
\notag \\ &\quad\quad\quad\quad\quad\quad
\|v^N\|_{\frac{1}{2}-\delta,\frac{1}{2}-\delta,T},
\|\mathcal{D}(v^N)\|_{\frac{1}{2}+\delta,\frac{1}{2}-\delta,T}),
\label{Eqn:LWP-34b}
\end{align}
where $Z_{2}(x,y,z,w)$ is a polynomial of degree 7 with positive coefficients.  Under the assumption
$\|v_{0}^{N}-u_{0,\omega}^{N}\|_{H^{\frac{1}{2}+\delta}} \leq R$, we can repeat the analysis done for $u^N$, and \eqref{Eqn:LWP-28}-\eqref{Eqn:LWP-29} follows for $T>0$ sufficiently small.  To prove \eqref{Eqn:LWP-30}, we proceed as above, using \eqref{Eqn:LWP-31}-\eqref{Eqn:LWP-34}, Proposition \ref{Prop:nonlin} and Proposition \ref{Prop:nonlin2} to establish, if $\omega\in\Omega_{T}$, then for all $N>M>0$ we have
\begin{align*}
\|\mathcal{D}&(v^N)-\mathcal{D}(v^M)\|_{\frac{1}{2}+\delta,\frac{1}{2}-\delta,T}  \\
&\leq C_{1}T^{\theta}(M^{-\varepsilon}
+ \|S(t)(\mathbb{P}^N-\mathbb{P}^M)(v_{0}-u_{0,\omega})\|_{\frac{1}{2}+\delta,\frac{1}{2}+\delta,T}
\\
&\ \ \ \ \ \ \ \ \ \ \ \ \ \ \ + \|v^N-v^M\|_{\frac{1}{2}-\delta,\frac{1}{2}-\delta,T}
+ \|\mathcal{D}(v^N)-\mathcal{D}(v^M)\|_{\frac{1}{2}+\delta,\frac{1}{2}-\delta,T})
\\
&\ \ \ \
\cdot Z_{3}(\|u_{0,\omega}^{N}\|_{H^{\frac{1}{2}-\delta}},\|u_{0,\omega}^{M}\|_{H^{\frac{1}{2}-\delta}},
\|v_{0}^{N}-u_{0,\omega}^{N}\|_{H^{\frac{1}{2}+\delta}},\|v_{0}^{M}-u_{0,\omega}^{M}\|_{H^{\frac{1}{2}+\delta}},
\\
&\ \ \ \ \ \ \ \ \ \ \ \ \ \ \
\|v^N\|_{\frac{1}{2}-\delta,\frac{1}{2}-\delta,T},\|v^M\|_{\frac{1}{2}-\delta,\frac{1}{2}-\delta,T},
\|\mathcal{D}(v^N)\|_{\frac{1}{2}+\delta,\frac{1}{2}-\delta,T},
\|\mathcal{D}(v^M)\|_{\frac{1}{2}+\delta,\frac{1}{2}-\delta,T}),
\end{align*}
where $Z_{3}(s,t,u,v,w,x,y,z)$ is a polynomial of degree 6 with positive coefficients.  Then using \eqref{Eqn:LWP-28} and \eqref{Eqn:LWP-29} we find
\begin{align*}
\|\mathcal{D}(v^N)-\mathcal{D}&(v^M)\|_{\frac{1}{2}+\delta,\frac{1}{2}-\delta,T}  \\ &\leq C_{1}T^{\theta}(M^{-\varepsilon} + C_{M}
+ \|\mathcal{D}(v^N)-\mathcal{D}(v^M)\|_{\frac{1}{2}+\delta,\frac{1}{2}-\delta,T})
\\
&\ \ \ \ \ \ \ \ \
\cdot Z_{3}(CT^{-\frac{\beta}{2}},CT^{-\frac{\beta}{2}},R,R,R,R,
\tilde{R},\tilde{R}).
\end{align*}
By taking $T>0$ sufficiently small (subsequently rearranging the last inequality), and using $C_M \rightarrow 0$ as $M\rightarrow 0$, we conclude that \eqref{Eqn:LWP-30} holds true for $\omega\in\Omega_{T}$.  This completes the justification of point (a): for $\omega\in\Omega_{T}$, the local solution $v$ to \eqref{Eqn:gKdV-zeromean} with data $v_{0}\in \Big\{u_{0,\omega} + \{\|\cdot\|_{H^{\frac{1}{2}+\delta}}\leq R\}\Big\}$ exists and is unique (in the sense described above).

We proceed to establish point (b).  That is, we show that the solution map
$$\Phi:
\Big\{u_{0,\omega} + \{\|\cdot\|_{H^{\frac{1}{2}+\delta}}\leq R\}\Big\}\rightarrow  \Big\{S(t)u_{0,\omega} + \{\|\cdot\|_{C([0,T];H^{\frac{1}{2}+\delta})} \leq \tilde{R}\}
\Big\}$$
for \eqref{Eqn:gKdV-zeromean} is Lipschitz.  Using Proposition \ref{Prop:nonlin} and Proposition \ref{Prop:nonlin2}, (note that we can apply these Proposition \ref{Prop:nonlin} to the nonlinearity evaluated at the limits $u$ and $v$ because of the justification of \eqref{Eqn:LWP-15-new-2} found above), if $\omega\in\Omega_{T}$, then we have
\begin{align*}
\|u-v\|_{\frac{1}{2}-\delta,\frac{1}{2}-\delta,T}
&= \|\mathcal{D}(u)-
\mathcal{D}(v)\|_{\frac{1}{2}-\delta,\frac{1}{2}-\delta,T}
 \\
&\leq \|\mathcal{D}(u)-
\mathcal{D}(v)\|_{\frac{1}{2}+\delta,\frac{1}{2}-\delta,T}
\\
&\leq
\sum_{j=-1}^{4}\|\mathcal{D}_{j}(u)-
\mathcal{D}_{j}(v)\|_{\frac{1}{2}+\delta,\frac{1}{2}-\delta,T}
\\
&\leq C_{1}T^{\theta}(\|u_{0,\omega}-v_{0}\|
_{H^{\frac{1}{2}+\delta}} + \|u-v\|_{\frac{1}{2}-\delta,\frac{1}{2}-\delta,T}
+ \|\mathcal{D}(u)-\mathcal{D}(v)\|_{\frac{1}{2}+\delta,\frac{1}{2}-\delta,T})
\\
&\ \ \ \
\cdot Z_{4}(\|u_{0,\omega}\|_{H^{\frac{1}{2}-\delta}},
\|v_{0}-u_{0,\omega}\|_{H^{\frac{1}{2}+\delta}},
\|u\|_{\frac{1}{2}-\delta,\frac{1}{2}-\delta,T},
\|v\|_{\frac{1}{2}-\delta,\frac{1}{2}-\delta,T},
\\
&\ \ \ \ \ \ \ \ \ \ \ \
\|\mathcal{D}(u)\|_{\frac{1}{2}+\delta,\frac{1}{2}-\delta,T},
\|\mathcal{D}(v)\|_{\frac{1}{2}+\delta,\frac{1}{2}-\delta,T}),
\end{align*}
where $Z_{4}(u,v,w,x,y,z)$ is a polynomial of degree 6 with positive coefficients.  Repeating the arguments above, we conclude that, if $T>0$ is sufficiently small, then for $\omega\in\Omega_{T}$, we have
\begin{align}
\|\mathcal{D}(u)-
\mathcal{D}(v)\|_{\frac{1}{2}+\delta,\frac{1}{2}-\delta,T}
\lesssim \|u_{0,\omega}-v_{0}\|_{H^{\frac{1}{2}+\delta}}.
\label{Eqn:LWP-36}
\end{align}
Then using \eqref{Eqn:linear-Y}, Proposition \ref{Prop:nonlin} and \eqref{Eqn:NL-0-b} (instead of \eqref{Eqn:NL-0-a}), we have
\begin{align}
\|u-v\|_{C([0,T];H^{\frac{1}{2}+\delta}(\mathbb{T}))}
&= \|\mathcal{D}(u)-
\mathcal{D}(v)\|_{C([0,T];H^{\frac{1}{2}+\delta}(\mathbb{T}))}
\notag \\
&\leq \|\mathcal{D}(u)-
\mathcal{D}(v)\|_{Y^{\frac{1}{2}+\delta,0}_{T}}
\notag \\
&\leq
\sum_{j=-1}^{4}\|\mathcal{D}_{j}(u)-
\mathcal{D}_{j}(v)\|_{Y^{\frac{1}{2}+\delta,0}_{T}}
\notag \\
&\leq \|\mathcal{N}_{0}(u)-
\mathcal{N}_{0}(v)\|_{Y^{\frac{1}{2}+\delta,0}_{T}}
+ \sum_{j=-1,j\neq 0}^{4}\|\mathcal{D}_{j}(u)-
\mathcal{D}_{j}(v)\|_{X^{\frac{1}{2}+\delta,\frac{1}{2}+\delta}_{T}}
\notag \\
&\leq C_{1}T^{\theta}(\|u_{0,\omega}-v_{0}\|
_{H^{\frac{1}{2}+\delta}} + \|u-v\|_{\frac{1}{2}-\delta,\frac{1}{2}-\delta,T}
\notag \\ &\ \ \ \ \ \ \ \ \ \ \ \ \ \ \ \ \ + \|\mathcal{D}(u)-\mathcal{D}(v)\|_{\frac{1}{2}+\delta,\frac{1}{2}-\delta,T})
\notag \\
&\ \ \ \ \ \ \cdot
Z_{4}(\|u_{0,\omega}\|_{H^{\frac{1}{2}-\delta}},
\|v_{0}-u_{0,\omega}\|_{H^{\frac{1}{2}+\delta}},
\|u\|_{\frac{1}{2}-\delta,\frac{1}{2}-\delta,T},
\|v\|_{\frac{1}{2}-\delta,\frac{1}{2}-\delta,T},
\notag \\
&\ \ \ \ \ \ \ \ \ \ \ \ \|\mathcal{D}(u)\|_{\frac{1}{2}+\delta,\frac{1}{2}-\delta,T},
\|\mathcal{D}(v)\|_{\frac{1}{2}+\delta,\frac{1}{2}-\delta,T}),
\notag \\
&\leq \tilde{C}_{1}T^{\theta}
\|u_{0,\omega}-v_{0}\|_{H^{\frac{1}{2}+\delta}} Z_{4}(CT^{-\frac{\beta}{2}},R,R,R,\tilde{R},\tilde{R})
\notag \\
&\lesssim
\|u_{0,\omega}-v_{0}\|_{H^{\frac{1}{2}+\delta}},
\label{Eqn:LWP-37}
\end{align}
for $T>0$ sufficiently small.  From \eqref{Eqn:LWP-37} we conclude that the solution map $\Phi$ for \eqref{Eqn:gKdV-zeromean} is Lipschitz.  This completes the discussion of point (iv).

Lastly, we need to address point (v).
We compare solutions $\tilde{u}^{N}$ of the truncated system \eqref{Eqn:gKdV-ZMN} to the local solution $u$ of \eqref{Eqn:gKdV-zeromean} constructed above.  Let us be clear that we are using $\tilde{u}^{N}$ to denote the solution to the frequency truncated PDE \eqref{Eqn:gKdV-ZMN} to avoid confusion with the solution $u^N$ to \eqref{Eqn:gKdV-zeromean} with frequency truncated data.  Avoiding frequency truncation of the nonlinearity was useful above, but we will need to study the finite-dimensional dynamics of \eqref{Eqn:gKdV-trunc} (in particular to exploit the invariance of the Gibbs measure) in order to extend the local solutions of Theorem \ref{Thm:LWP} to global solutions.  We will also use the notation $\widetilde{\mathcal{D}}^N:=\mathbb{P}_{N}\mathcal{D}(\tilde{u}^{N},\ldots,\tilde{u}^{N})$.

We remark that the analysis applied to the sequence $u^{N}$ above, for \textit{fixed} $N>0$, applies to the frequency truncated sequence $\tilde{u}^{N}$ as well (the $X^{s,b}$-norm ``behaves nicely'' with respect to frequency truncation).  It was only when we estimated differences of solutions in the proof of Theorem \ref{Thm:LWP} above that avoiding frequency truncation for the sequence $u^N$ became useful.
In particular, the estimates \eqref{Eqn:LWP-16}-\eqref{Eqn:LWP-17} with $\tilde{u}^{N}$ (instead of $u^N$) imply that for $\omega\in \Omega_{T}$ we have
\begin{align}
\|\tilde{u}^{N}\|_{\frac{1}{2}-\delta,\frac{1}{2}-\delta,T} &\leq R, \ \
\|u\|_{\frac{1}{2}-\delta,\frac{1}{2}-\delta,T} \leq R, \notag \\
\|\mathcal{D}(\tilde{u}^{N})\|_{\frac{1}{2}+\delta,\frac{1}{2}-\delta,T} &\leq \tilde{R}, \ \
\|\mathcal{D}(u)\|_{\frac{1}{2}+\delta,\frac{1}{2}-\delta,T} \leq \tilde{R},
\label{Eqn:conv-bdd}
\end{align}
for each $N>0$.  Recall $0<\delta_{1}<\delta$ as given in the statement of Theorem \ref{Thm:LWP}.  We claim that the following estimates hold for $\omega\in \Omega_{T}$:
\begin{align}
\|\tilde{u}^{N}-u\|_{\frac{1}{2}-2\delta+\delta_{1},\frac{1}{2}-\delta,T}
&\lesssim N^{-\varepsilon}, \label{Eqn:conv-est1}\\
\|\widetilde{\mathcal{D}}^{N}-\mathcal{D}\|_{\frac{1}{2}+\delta_{1},\frac{1}{2}-\delta,T}
&\lesssim N^{-\varepsilon}.
\label{Eqn:conv-est2}
\end{align}
We proceed to justify \eqref{Eqn:conv-est1} and \eqref{Eqn:conv-est2}.
Using the equations \eqref{Eqn:gKdV-zeromean}
and \eqref{Eqn:gKdV-ZMN} we have
\begin{align}
\|\tilde{u}^{N}-u\|_{\frac{1}{2}-2\delta+\delta_{1},\frac{1}{2}-\delta,T}
&\leq
\|S(t)(I-\mathbb{P}_{N})u_{0,\omega}\|_{\frac{1}{2}-2\delta+\delta_{1},\frac{1}{2}-\delta,T}
+
\|\widetilde{\mathcal{D}}^{N}-\mathcal{D}\|_{\frac{1}{2}-2\delta+\delta_{1},\frac{1}{2}-\delta,T}
\notag \\
&\leq \frac{N^{-\epsilon}}{T^{\frac{\beta}{2}}} +
\|\widetilde{\mathcal{D}}^{N}-\mathcal{D}\|_{\frac{1}{2}+\delta_{1},\frac{1}{2}-\delta,T}.
\label{Eqn:preconv-1}
\end{align}
Then
\begin{align*}
\|\widetilde{\mathcal{D}}^{N}-\mathcal{D}\|_{\frac{1}{2}+\delta_{1},\frac{1}{2}-\delta,T}
\leq
\|(I-\mathbb{P}_{N})\mathcal{D}\|_
{\frac{1}{2}+\delta_{1},\frac{1}{2}-\delta,T}
+
\|\mathbb{P}_{N}(\widetilde{\mathcal{D}}^{N}-\mathcal{D})\|
_{\frac{1}{2}+\delta_{1},\frac{1}{2}-\delta,T}.
\end{align*}
We find
\begin{align}
\|(I-\mathbb{P}_{N})\mathcal{D}\|_
{\frac{1}{2}+\delta_{1},\frac{1}{2}-\delta,T}
\leq N^{-(\delta-\delta_{1})}\|\mathcal{D}\|_
{\frac{1}{2}+\delta,\frac{1}{2}-\delta,T}
\leq
N^{-(\delta-\delta_{1})}R,
\label{Eqn:preconv-1b}
\end{align}
and
\begin{align}
\|\mathbb{P}_{N}(\widetilde{\mathcal{D}}^{N}-\mathcal{D})\|
_{\frac{1}{2}+\delta_{1},\frac{1}{2}-\delta,T}
&\leq \|\mathbb{P}_{N}\mathcal{D}(\tilde{u}^{N}-u,\tilde{u}^{N},\tilde{u}^{N},\tilde{u}^{N})
\|_{\frac{1}{2}+\delta_{1},\frac{1}{2}-\delta,T} \notag \\
&\ \ \ \ \ \ \ +
\|\mathbb{P}_{N}\mathcal{D}(u,\tilde{u}^{N}-u,\tilde{u}^{N},\tilde{u}^{N})\|
_{\frac{1}{2}+\delta_{1},\frac{1}{2}-\delta,T}
\notag \\
&\ \ \ \ \ \ \ +
\|\mathbb{P}_{N}\mathcal{D}(u,u,\tilde{u}^{N}-u,\tilde{u}^{N})\|
_{\frac{1}{2}+\delta_{1},\frac{1}{2}-\delta,T}
\notag \\
&\ \ \ \ \ \ \ +
\|\mathbb{P}_{N}\mathcal{D}(u,u,u,\tilde{u}^{N}-u)\|
_{\frac{1}{2}+\delta_{1},\frac{1}{2}-\delta,T}.
\label{Eqn:preconv-2}
\end{align}
Each term on the right-hand side of \eqref{Eqn:preconv-2} will be bounded in a similar way.  We bound the first term explicitly.
Using the decomposition of frequency space from section \ref{Sec:NLest}, we expand
\begin{align*}
\|\mathbb{P}_{N}\mathcal{D}(\tilde{u}^{N}-u,\tilde{u}^{N},\tilde{u}^{N},\tilde{u}^{N})
\|_{\frac{1}{2}+\delta_{1},\frac{1}{2}-\delta,T}
&\leq
\sum_{k=-1}^{4}
\|\mathbb{P}_{N}\mathcal{D}_{k}(\tilde{u}^{N}-u,\tilde{u}^{N},\tilde{u}^{N},\tilde{u}^{N})
\|_{\frac{1}{2}+\delta_{1},\frac{1}{2}-\delta,T}.
\end{align*}
For $k=-1$ and $k=0$ we have by \eqref{Eqn:NL-neg1} and  \eqref{Eqn:NL-0-a}, with $\delta_0=\delta-\delta_1>0$,
\begin{align}
&\sum_{k=-1}^{0}
\|\mathbb{P}_{N}\mathcal{D}_{k}(\tilde{u}^{N}-u,\tilde{u}^{N},\tilde{u}^{N},\tilde{u}^{N})
\|_{\frac{1}{2}+\delta_{1},\frac{1}{2}-\delta,T}
\notag \\
&\leq
T^{\theta}\bigg(N^{-\epsilon} + \|\tilde{u}^{N}-u\|_{\frac{1}{2}-2\delta +\delta_{1},\frac{1}{2}-\delta,T} +
\|\widetilde{\mathcal{D}}^{N}-\mathcal{D}\|_{\frac{1}{2}
+\delta_{1},\frac{1}{2}-\delta,T}\bigg)
\notag  \\
&\ \ \ \ \ \ \ \ \cdot\bigg(1+\|\tilde{u}^{N}\|_{\frac{1}{2}-\delta,
\frac{1}{2}-\delta,T} +
\|\widetilde{\mathcal{D}}^{N}\|_{\frac{1}{2}
+\delta,\frac{1}{2}-\delta,T}\bigg)^{3}
\notag  \\
&\lesssim
T^{\theta}(1+R+\tilde{R})^{3}\bigg(N^{-\epsilon} + \|\tilde{u}^{N}-u\|_{\frac{1}{2}-2\delta +\delta_{1},\frac{1}{2}-\delta,T} +
\|\widetilde{\mathcal{D}}^{N}-\mathcal{D}\|_{\frac{1}{2}
+\delta_{1},\frac{1}{2}-\delta,T}\bigg).
\label{Eqn:pre-conv3}
\end{align}
Then for $k=1$ we find
\begin{align*}
\|\mathbb{P}_{N}\mathcal{D}_{1}(\tilde{u}^{N}-u,\tilde{u}^{N},\tilde{u}^{N},\tilde{u}^{N})
\|_{\frac{1}{2}+\delta_{1},\frac{1}{2}-\delta,T}
&\leq
\|\mathbb{P}_{N}\mathcal{D}_{1}((I-\mathbb{P}_{N})u,\tilde{u}^{N},\tilde{u}^{N},\tilde{u}^{N})
\|_{\frac{1}{2}+\delta_{1},\frac{1}{2}-\delta,T}
\\
&\ \ \ \ \
+\|\mathbb{P}_{N}\mathcal{D}_{1}(\mathbb{P}_{N}(\tilde{u}^{N}-u),\tilde{u}^{N},\tilde{u}^{N},\tilde{u}^{N})
\|_{\frac{1}{2}+\delta_{1},\frac{1}{2}-\delta,T}.
\end{align*}
Then we have
\begin{align}
\|\mathbb{P}_{N}\mathcal{D}_{1}(&(I-\mathbb{P}_{N})u,\tilde{u}^{N},\tilde{u}^{N},\tilde{u}^{N})
\|_{\frac{1}{2}+\delta_{1},\frac{1}{2}-\delta,T}
\notag \\
&\sim
\|\mathbb{P}_{N}\mathcal{D}_{1}((I-\mathbb{P}_{N})\mathcal{D},\tilde{u}^{N},\tilde{u}^{N},\tilde{u}^{N})
\|_{\frac{1}{2}+\delta_{1},\frac{1}{2}-\delta,T}
\notag\\
&\leq \|\mathbb{P}_{N}\mathcal{D}_{1}(\mathcal{D}((I-\mathbb{P}_{N/4})u,u,u,u),\tilde{u}^{N},\tilde{u}^{N},\tilde{u}^{N})
\|_{\frac{1}{2}+\delta_{1},\frac{1}{2}-\delta,T} \notag \\
&\ \ \ \ \ \ \ +
\|\mathbb{P}_{N}\mathcal{D}_{1}(\mathcal{D}
(u,(I-\mathbb{P}_{N/4})u,u,u),\tilde{u}^{N},\tilde{u}^{N},\tilde{u}^{N})
\|_{\frac{1}{2}+\delta_{1},\frac{1}{2}-\delta,T}
\notag \\
&\ \ \ \ \ \ \ +
\|\mathbb{P}_{N}\mathcal{D}_{1}(\mathcal{D}
(u,u,(I-\mathbb{P}_{N/4})u,u),\tilde{u}^{N},\tilde{u}^{N},\tilde{u}^{N})
\|_{\frac{1}{2}+\delta_{1},\frac{1}{2}-\delta,T}
\notag \\
&\ \ \ \ \ \ \ +
\|\mathbb{P}_{N}\mathcal{D}_{1}(\mathcal{D}
(u,u,u,(I-\mathbb{P}_{N/4})u),\tilde{u}^{N},\tilde{u}^{N},\tilde{u}^{N})
\|_{\frac{1}{2}+\delta_{1},\frac{1}{2}-\delta,T}.
\label{Eqn:preconv-3}
\end{align}
Once again, each term on the right-hand side of \eqref{Eqn:preconv-3} will be bounded in a similar way, and we proceed to bound the first term explicitly.  Using \eqref{Eqn:NL-1} with $\delta_0=\delta-\delta_1$, we find
\begin{align}
\|\mathbb{P}_{N}&\mathcal{D}_{1}(\mathcal{D}(
(I-\mathbb{P}_{N/4})u,u,u,u),\tilde{u}^{N},\tilde{u}^{N},\tilde{u}^{N})
\|_{\frac{1}{2}+\delta_{1},\frac{1}{2}-\delta,T}
\notag \\
&\lesssim
T^{\theta}\bigg(N^{-\varepsilon} + \|(I-\mathbb{P}_{N/4})u\|_{\frac{1}{2}-2\delta +\delta_{1},\frac{1}{2}-\delta,T} +
\|(I-\mathbb{P}_{N/4})\mathcal{D}\|_{\frac{1}{2}
+\delta_{1},\frac{1}{2}-\delta,T}\bigg)
\notag  \\
&\ \ \cdot\bigg(1+\|u\|_{\frac{1}{2}-\delta,
\frac{1}{2}-\delta,T} +
\|\mathcal{D}\|_{\frac{1}{2}
+\delta,\frac{1}{2}-\delta,T}\bigg)^{3}
\bigg(1+\|\tilde{u}^{N}\|_{\frac{1}{2}-\delta,
\frac{1}{2}-\delta,T} +
\|\widetilde{\mathcal{D}}^{N}\|_{\frac{1}{2}
+\delta,\frac{1}{2}-\delta,T}\bigg)^{3}
\notag \\
&\leq
T^{\theta}\bigg(N^{-\varepsilon} + N^{-(\delta-\delta_{1})}\big(\|u\|_{\frac{1}{2}-\delta
,\frac{1}{2}-\delta,T} +
\|\mathcal{D}\|_{\frac{1}{2}
+\delta,\frac{1}{2}-\delta,T}\big)\bigg)
\notag  \\
&\ \ \cdot\bigg(1+\|u\|_{\frac{1}{2}-\delta,
\frac{1}{2}-\delta,T} +
\|\mathcal{D}\|_{\frac{1}{2}
+\delta,\frac{1}{2}-\delta,T}\bigg)^{3}
\bigg(1+\|\tilde{u}^{N}\|_{\frac{1}{2}-\delta,
\frac{1}{2}-\delta,T} +
\|\widetilde{\mathcal{D}}^{N}\|_{\frac{1}{2}
+\delta,\frac{1}{2}-\delta,T}\bigg)^{3}
\notag \\
&\lesssim N^{-\varepsilon}T^{\theta}(1+R+\tilde{R})^{7}.
\label{Eqn:preconv-4}
\end{align}
Next we find
\begin{align}
\|\mathbb{P}_{N}\mathcal{D}_{1}(\mathbb{P}_{N}&(\tilde{u}^{N}-u),\tilde{u}^{N},\tilde{u}^{N},\tilde{u}^{N})
\|_{\frac{1}{2}+\delta_{1},\frac{1}{2}-\delta,T}
\notag \\
&\sim \|\mathbb{P}_{N}\mathcal{D}_{1}(\mathbb{P}_{N}
(\widetilde{\mathcal{D}}^{N}-\mathcal{D}),\tilde{u}^{N},\tilde{u}^{N},\tilde{u}^{N})
\|_{\frac{1}{2}+\delta_{1},\frac{1}{2}-\delta,T}
\notag \\
&\leq
\|\mathbb{P}_{N}\mathcal{D}_{1}(\mathbb{P}_{N}
(\mathcal{D}(\tilde{u}^{N}-u,\tilde{u}^{N},\tilde{u}^{N},\tilde{u}^{N})),\tilde{u}^{N},\tilde{u}^{N},\tilde{u}^{N})
\|_{\frac{1}{2}+\delta_{1},\frac{1}{2}-\delta,T} \notag \\
&\ \ \ \ \ \ \ +
\|\mathbb{P}_{N}\mathcal{D}_{1}(\mathbb{P}_{N}
(\mathcal{D}(u,\tilde{u}^{N}-u,\tilde{u}^{N},\tilde{u}^{N})),\tilde{u}^{N},\tilde{u}^{N},\tilde{u}^{N})
\|_{\frac{1}{2}+\delta_{1},\frac{1}{2}-\delta,T}
\notag \\
&\ \ \ \ \ \ \ +
\|\mathbb{P}_{N}\mathcal{D}_{1}(\mathbb{P}_{N}
(\mathcal{D}(u,u,\tilde{u}^{N}-u,\tilde{u}^{N})),\tilde{u}^{N},\tilde{u}^{N},\tilde{u}^{N})
\|_{\frac{1}{2}+\delta_{1},\frac{1}{2}-\delta,T}
\notag \\
&\ \ \ \ \ \ \ +
\|\mathbb{P}_{N}\mathcal{D}_{1}(\mathbb{P}_{N}
(\mathcal{D}(u,u,u,\tilde{u}^{N}-u)),\tilde{u}^{N},\tilde{u}^{N},\tilde{u}^{N})
\|_{\frac{1}{2}+\delta_{1},\frac{1}{2}-\delta,T}.
\label{Eqn:preconv-5}
\end{align}
Each term on the right-hand side of \eqref{Eqn:preconv-5} will be bounded in a similar way, and we proceed to bound the first term explicitly.  By \eqref{Eqn:NL-1} with $\delta_0=\delta-\delta_1$, we have
\begin{align}
\|&\mathbb{P}_{N}\mathcal{D}_{1}(\mathbb{P}_{N}
(\mathcal{D}(\tilde{u}^{N}-u,\tilde{u}^{N},\tilde{u}^{N},\tilde{u}^{N})),\tilde{u}^{N},\tilde{u}^{N},\tilde{u}^{N})
\|_{\frac{1}{2}+\delta_{1},\frac{1}{2}-\delta,T}
\notag \\
&\leq
T^{\theta}\bigg(N^{-\epsilon} + \|\tilde{u}^{N}-u\|_{\frac{1}{2}-2\delta +\delta_{1},\frac{1}{2}-\delta,T} +
\|\widetilde{\mathcal{D}}^{N}-\mathcal{D}\|_{\frac{1}{2}
+\delta_{1},\frac{1}{2}-\delta,T}\bigg)
\notag  \\
&\ \ \ \ \ \ \ \ \cdot\bigg(1+\|\tilde{u}^{N}\|_{\frac{1}{2}-\delta,
\frac{1}{2}-\delta,T} +
\|\widetilde{\mathcal{D}}^{N}\|_{\frac{1}{2}
+\delta,\frac{1}{2}-\delta,T}\bigg)^{6}
\notag \\
&\leq
T^{\theta}(1+R+\tilde{R})^{6}\bigg(N^{-\epsilon} + \|\tilde{u}^{N}-u\|_{\frac{1}{2}-2\delta +\delta_{1},\frac{1}{2}-\delta,T} +
\|\widetilde{\mathcal{D}}^{N}-\mathcal{D}\|_{\frac{1}{2}
+\delta_{1},\frac{1}{2}-\delta,T}\bigg).
\end{align}
With a continuity argument, as in the proof of \eqref{Eqn:LWP-1a}-\eqref{Eqn:LWP-1b} above, we arrive at \eqref{Eqn:conv-est1}-\eqref{Eqn:conv-est2}.

Given the convergence results \eqref{Eqn:conv-est1} and \eqref{Eqn:conv-est2}, by following the approach taken in \eqref{Eqn:LWP-37} (using an estimate of the type \eqref{Eqn:preconv-1b} to control high frequencies), we can establish the a posteriori estimate
\begin{align}
\|u-S(t)u_{0,\omega}
    -(\Phi^{N}(t)-S(t))\mathbb{P}_{N}u_{0,\omega}
    \|_{C([0,T];H^{\frac{1}{2}+\delta_1})}
    \lesssim N^{-\varepsilon}.
\label{Eqn:conv-est3}
\end{align}
This completes the discussion of point (v), and the proof of Theorem \ref{Thm:LWP} is complete.

\end{proof}

\section{Global well-posedness and invariance of the Gibbs measure}
\label{Sec:GWP}

In this section we will extend the local solutions produced by Theorem \ref{Thm:LWP} to global solutions, and prove our main result: the invariance of the Gibbs measure under the flow of \eqref{Eqn:gKdV-zeromean}, the gauge-transformed quartic gKdV (Theorem \ref{Thm:GWP}).

\subsection{Construction of the Gibbs measure}

Let $\{g_{n}(\omega)\}_{n=1}^{\infty}$ be a sequence of standard independent complex-valued Gaussian random variables on a probability space $(\Omega,\mathcal{F},P)$.  For each $N>0$, define $$E_{N}=\text{span}\{\sin(nx),\cos(nx):1 \leq |n| \leq N\}.$$  The finite-dimensional Wiener measure $\rho_{N}$ on $E_{N}$ is the push-forward of $P$ under the map from $(\Omega,\mathcal{F},P)$ to $E_{N}$ (equipped with the Borel sigma algebra) given by
\begin{align}
\omega \longmapsto \sum_{1\leq |n|\leq N}\frac{g_{n}(\omega)}{|n|}e^{inx},
\label{Eqn:GWP-A1}
\end{align}
where $g_{-n}:=\overline{g_n}$.

We proceed to extend the Wiener measure to infinite dimensions\footnote{We will state the results required for the proof of Theorem \ref{Thm:GWP}.  For more details about the Wiener measure, and Gaussian measures on Banach spaces, see \cite{Kuo}.}.  Fix $\delta >0$ (smallness conditions will eventually be imposed on $\delta$), the Wiener measure $\rho$ on $H^{\frac{1}{2}-\delta}(\mathbb{T})$ is the push-forward of $P$ under the map from $(\Omega,\mathcal{F},P)$ to $H^{\frac{1}{2}-\delta}(\mathbb{T})$ (equipped with the Borel sigma algebra) given by
\begin{align}
\omega \longmapsto u_{0,\omega}:=\sum_{n\in \mathbb{Z}\setminus \{0\}}\frac{g_{n}(\omega)}{|n|}e^{inx},
\label{Eqn:GWP-A2}
\end{align}
where $g_{-n}:=\overline{g_{n}}$.

The ``finite-dimensional''  Gibbs measure $\mu_N$ on $H^{\frac{1}{2}-\delta}(\mathbb{T})$ is the push-forward under the map \eqref{Eqn:GWP-A2} of the weighted measure
\begin{align*}
e^{-\frac{1}{20}\int_{\mathbb{T}}(\sum_{1\leq |n|\leq N}\frac{g_{n}(\omega)}{|n|}e^{inx})^{5}dx}
\chi_{\big\{\|\sum_{1\leq |n|\leq N}\frac{g_{n}(\omega)}{|n|}e^{inx}\|_{2}\leq B\big\}}dP(\omega).
\end{align*}
We recall a crucial result from \cite{LRS,B3}.
\begin{theorem}[\cite{LRS,B3}]
Let $B<\infty$, then for each $r\geq1$, we have
$$
e^{-\frac{1}{20}\int_{\mathbb{T}}(u_{0,\omega}(x))^{5}dx}
\chi_{\big\{\|u_{0,\omega}\|_{2}\leq B\big\}}\in L^{r}(\Omega).$$
In particular the \textbf{Gibbs measure} $\mu$, defined as the push-forward under \eqref{Eqn:GWP-A2} of the weighted measure
\begin{align*}
e^{-\frac{1}{20}\int_{\mathbb{T}}(u_{0,\omega}(x))^{5}dx}
\chi_{\big\{\|u_{0,\omega}\|_{2}\leq B\big\}}dP(\omega),
\end{align*}
is absolutely continuous with respect to the Wiener measure $\rho$.
\label{Thm:LRS}
\end{theorem}
The proof of Theorem \ref{Thm:LRS} first appeared in \cite{LRS}, but it is clarified and expanded in \cite{B3}.
\begin{remark}
It is easily verified from the proof in \cite{B3} that we also have the following conclusion: there exists $0<C<\infty$ such that that for all $N>0$,
$$ \Big\|e^{-\frac{1}{20}\int_{\mathbb{T}}(\mathbb{P}_{N}u_{0,\omega}(x))^{5}dx}
\chi_{\big\{\|\mathbb{P}_{N}u_{0,\omega}\|_{2}\leq B\big\}}\Big\|_{L^{r}(\Omega)}
\leq C < \infty.$$
\end{remark}

Having defined these measures, we establish a convergence property to be used in the proof of Theorem \ref{Thm:GWP}.
The application of this property (and its proof) are inspired by similar arguments appearing in Burq-Tzvetkov \cite{BT3}.
\begin{proposition}
Set $$f(u)=e^{-\frac{1}{20}\int_{\mathbb{T}}u^{5}dx}
\chi_{\big\{\|u\|_{2}\leq B\big\}}\  \text{and} \ f_{N}(u)=e^{-\frac{1}{20}\int_{\mathbb{T}}(\mathbb{P}_{N}u(x))^{5}dx}
\chi_{\big\{\|\mathbb{P}_{N}(u)\|_{2}\leq B\big\}}.$$  Then
$$ \lim_{N\rightarrow \infty}
\int_{H^{\frac{1}{2}-\delta}}|f_{N}(u)-f(u)|d\rho (u)=0.$$
\label{Prop:conv}
\end{proposition}
\begin{proof}
We first claim that $f_{N}(u)\rightarrow f(u)$ in measure with respect to $\rho$, and this follows from showing that $f_{N}(u) \rightarrow f(u)$ $\rho$-almost surely (by Egorov's Theorem).  Clearly we have
$$\chi_{\big\{\|\mathbb{P}_{N}u\|_{2}\leq B\big\}}\longrightarrow
\chi_{\big\{\|u\|_{2}\leq B\big\}}$$
$\rho$-almost surely.  By continuity of the exponential function, we need only verify that $\mathbb{P}_{N}u\rightarrow u$ in $L^{5}(\mathbb{T})$, $\rho$-almost surely, and this follows easily from the Sobolev embedding $L^{5}(\mathbb{T})\hookrightarrow H^{\frac{1}{2}-\delta}(\mathbb{T})$ (for $\delta>0$ sufficiently small).  Fix $\varepsilon>0$, and let $A_{N,\varepsilon}:=\{u\in H^{\frac{1}{2}-\delta}(\mathbb{T}):|f_{N}(u)-f(u)|\leq \varepsilon\}.$  We apply Cauchy-Schwarz followed by Theorem \ref{Thm:LRS},
\begin{align*}
\int_{H^{\frac{1}{2}-\delta}(\mathbb{T})}|f_{N}(u)-f(u)|d\rho(u)
&\leq \big(\int_{A_{N,\varepsilon}} + \int_{A_{N,\varepsilon}^{c}}\big)
|f_{N}(u)-f(u)|d\rho(u) \\
&\leq \int_{A_{N,\varepsilon}}|f_{N}(u)-f(u)|d\rho(u)
+ \|f_{N}-f\|_{L^{2}(d\rho)}\big(\rho(A_{N,\varepsilon}^{c})
\big)^{\frac{1}{2}} \\
&\leq \varepsilon + 2C\big(\rho(A_{N,\varepsilon}^{c})
\big)^{\frac{1}{2}}.
\end{align*}
Then since $f_{N}(u)\rightarrow f(u)$ in measure with respect to $\rho$, we have $\rho(A_{N,\varepsilon}^{c})\rightarrow 0$ as $N\rightarrow \infty$, and the proof of Proposition \ref{Prop:conv} is complete.
\end{proof}

We have the following useful corollary of Proposition \ref{Prop:conv}.
\begin{corollary}
For any Borel set $A\subset H^{\frac{1}{2}-\delta}(\mathbb{T})$, we have
\begin{align}
\mu(A)=\lim_{N\rightarrow \infty}\mu_{N}(A).
\label{Cor:prop}
\end{align}
\label{Cor:conv}
\end{corollary}
\begin{proof}[Proof of Corollary \ref{Cor:conv}]
From the definitions above, the measure $\mu_N$ on $H^{\frac{1}{2}-\delta}(\mathbb{T})$ is defined by its density $d\mu_N(u) = f_{N}(u)d\rho(u)$.  This corollary is then automatic from Proposition \ref{Prop:conv}.  We have
\begin{align*}
|\mu_N(A)-\mu(A)|&= \Big| \int_{A}(f_N(u)-f(u))d\rho(u)\Big| \leq \int_{A}|f_N(u)-f(u)|d\rho(u)
\\
&\leq \int_{H^{\frac{1}{2}-\delta}(\mathbb{T})}|f_N(u)-f(u)|d\rho(u) \rightarrow 0,
\end{align*}
as $N\rightarrow \infty$.
\end{proof}

\subsection{Invariance of the finite-dimensional Gibbs measure}

Consider the frequency cutoff and gauge-transformed quartic gKdV
\begin{align}
\left\{
\begin{array}{ll} \partial_{t}u^{N} + \partial_{x}^{3}u^{N} = \mathbb{P}_N(\mathbb{P}((u^{N})^{3})\partial_{x}u^{N})  , \ \
t\in\mathbb{R}, x\in \mathbb{T},
\\
u^{N}(0,x) = \mathbb{P}_{N}(u_{0}(x)),  \ \ u_{0} \ \text{mean zero.}
\end{array} \right.
\label{Eqn:gKdV-trunc}
\end{align}
where $\mathbb{P}$ is the projection to mean zero functions, $\mathbb{P}_{N}$ is Dirichlet projection to frequencies $\leq N$, $u^{N} = \mathbb{P}_{N}u$, and $u^{N}_{0} =\mathbb{P}_{N}u_{0}$.  We can write \eqref{Eqn:gKdV-trunc} in coordinates as a system of $N$ complex ODEs (for the Fourier coefficients) $c_{n}:=\widehat{u^{N}}(n)$, $1\leq n\leq N$ (see \eqref{Eqn:leb-F1-2} below).  This system is locally well-posed by the Cauchy-Lipschitz Theorem, and it is easily verified that the $L^2$-norm $u^N$ of the solution to \eqref{Eqn:gKdV-trunc} is preserved under the flow (see \eqref{Eqn:L2-trunc-inv} below).  This provides an a priori bound on the $\ell^{\infty}_{n}$-norm of the Fourier coefficients $\{c_n\}_{n\in\mathbb{N}}$, and it follows that solutions $u^N$ to \eqref{Eqn:gKdV-trunc} are global-in-time.



In this subsection we explicitly compute the invariance of the ``finite-dimensional'' Gibbs measure $\mu_{N}$ under the flow of \eqref{Eqn:gKdV-trunc}.  Specifically, we establish the following proposition.
\begin{proposition}
The Gibbs measure $\mu_{N}$ on $H^{\frac{1}{2}-\delta}(\mathbb{T})$
is invariant under the flow of the frequency truncated (gauge-transformed) quartic gKdV \eqref{Eqn:gKdV-trunc}.
\label{Prop:trunc-inv}
\end{proposition}

Here we are viewing the flow of \eqref{Eqn:gKdV-trunc} extended to $H^{\frac{1}{2}-\delta}(\mathbb{T})=E_N\oplus E_{N}^{\perp}$ as the combined flow of $\eqref{Eqn:gKdV-trunc}$ on $E_N$ and the trivial flow on $E_{N}^{\perp}$.
We define the (truly) finite-dimensional Gibbs measure $\widetilde{\mu}_{N}$ on $E_N$ by the density
$d\widetilde{\mu}_N(u) = f_N(u)d\rho_N(u)$.  Proposition \ref{Prop:trunc-inv} is then a straight-forward consequence of the following lemma.

\begin{lemma}
The finite-dimensional Gibbs measure $\widetilde{\mu}_{N}$ on $E_N$
is invariant under the flow of the frequency truncated (gauge-transformed) quartic gKdV \eqref{Eqn:gKdV-trunc}.
\label{Lemma:trunc-inv-2}
\end{lemma}

We begin by proving Proposition \ref{Prop:trunc-inv} using Lemma \ref{Lemma:trunc-inv-2}, then we present the proof of Lemma \ref{Lemma:trunc-inv-2}.

\begin{proof}[Proof of Proposition \ref{Prop:trunc-inv}]
Let $\Phi_N(t)$ denote the flow map of \eqref{Eqn:gKdV-trunc} extended to $H^{\frac{1}{2}-\delta}(\mathbb{T})=E_{N}\oplus E_{N}^{\perp}$, and let $\widetilde{\Phi}_N(t)$ denote the flow of \eqref{Eqn:gKdV-trunc} restricted to $E_N$.  Then by definition we have $\Phi_N=(\widetilde{\Phi}_N,\text{Id})$, and Proposition \ref{Prop:trunc-inv} is automatic from Lemma \ref{Lemma:trunc-inv-2}, and the invariance of the Gaussian measure (on each Fourier mode) under the trivial flow on $E_N^{\perp}$.
\end{proof}

We now proceed with the proof of Lemma \ref{Lemma:trunc-inv-2}.

\begin{proof}[Proof of Lemma \ref{Lemma:trunc-inv-2}]
For the sake of clarity, let us begin by verifying conservation of the quantity
\begin{align}
H(u) = \frac{1}{2}\int(u_{x})^{2}dx + \frac{1}{20}\int u^{5}dx
\label{Eqn:GWP-1}
\end{align}
under the flow of the gauge-transformed quartic gKdV \eqref{Eqn:gKdV-zeromean} (without frequency truncation).
Suppose $u$ solves \eqref{Eqn:gKdV-zeromean}, then we find
\begin{align*}
\frac{dH(u)}{dt} &= \int u_{x}(u_{t})_{x}dx + \frac{1}{4}\int u^{4}u_{t}dx  \\
&= \int u_{x}[-u_{xxx} + \mathbb{P}(u^{3})u_{x}]_{x}dx + \frac{1}{4}\int u^{4}[-u_{xxx} + \mathbb{P}(u^{3})u_{x}]dx  \\
&= \int u_{x}[-u_{xxx} + u^{3}u_{x} - \big(\int u^{3}dx'\big)u_{x}]_{x}dx +  \frac{1}{4}\int u^{4}[-u_{xxx} + u^{3}u_{x} - \big(\int u^{3}dx'\big)u_{x}]dx
\\
&=  \int\big(u_{xx}u_{xxx} + \frac{1}{4}u^{7}u_{x}\big)dx
\\
&\ \ \ \ \ \ \ \ \ \ \ \ -
\big(\int u^{3}dx'\big)\bigg( \int \big(u_{x}u_{xx}dx + \frac{1}{4} u^{4}u_{x}\big)dx \bigg) \ \ \ \ \  \text{after integrating by parts,}\\
&= \int\big(\frac{1}{2}[(u_{xx})^{2}]_{x}
+ \frac{1}{32}[u^{8}]_{x}\big)dx
-
\big(\int u^{3}dx'\big)\bigg( \int\big( \frac{1}{2}[(u_{x})^{2}]_{x} + \frac{1}{20} [u^{5}]_{x}\big)dx \bigg) \\
&=  0.
\end{align*}
Now suppose $u^{N}$ solves the truncated gauged quartic gKdV \eqref{Eqn:gKdV-trunc}.  With a modified computation we will confirm that
$\displaystyle \frac{dH(u^{N})}{dt}=0$.  Indeed, we have
\begin{align}
\frac{dH(u^{N})}{dt} &= \int u^{N}_{x}(u^{N}_{t})_{x}dx + \frac{1}{4}\int (u^{N})^{4}u^{N}_{t}dx  \notag \\
&= \int u^{N}_{x}\bigg[-u^{N}_{xxx} + \mathbb{P}_{N}\bigg(\mathbb{P}\big((u^{N})^{3}\big)u^{N}_{x}
\bigg)\bigg]_{x}dx \notag \\
&\ \ \ \ \ \ \ \ \ \ \ \ + \frac{1}{4}\int (u^{N})^{4}\bigg[-u^{N}_{xxx} + \mathbb{P}_{N}\bigg(\mathbb{P}\big((u^{N})^{3}\big)u^{N}_{x}
\bigg)\bigg]dx  \notag \\
&= \int u^{N}_{x}\bigg[-u^{N}_{xxx} + \mathbb{P}_{N}\bigg((u^{N})^{3}u^{N}_{x}\bigg) - \mathbb{P}_{N}\bigg(\big(\int (u^{N})^{3}dx'\big)u^{N}_{x}\bigg)\bigg]_{x}dx
\notag \\
&\ \ \ \ \ \ \ \ \ \ \ \ +  \frac{1}{4}\int (u^{N})^{4}\bigg[-u^{N}_{xxx} + \mathbb{P}_{N}\bigg((u^{N})^{3}u^{N}_{x}\bigg) - \mathbb{P}_{N}\bigg(\big(\int (u^{N})^{3}dx'\big)u^{N}_{x}\bigg)\bigg]dx
\notag \\
&=  \int\bigg[u^{N}_{xx}u^{N}_{xxx} + \frac{1}{4}(u^{N})^{4}
\mathbb{P}_{N}\bigg((u^{N})^{3}u^{N}_{x}
\bigg)\bigg]dx
\notag \\
&\ \ \ \ \ \ \ \ \ \ \ \  -
\big(\int (u^{N})^{3}dx'\big)\bigg( \int \big(u^{N}_{x}u^{N}_{xx} + \frac{1}{4} (u^{N})^{4}u^{N}_{x}\big)dx \bigg)
\ \ \ \ \ \text{integrating by parts,}
\notag\\
&= \int\Bigg(\frac{1}{2}[(u^{N}_{xx})^{2}]_{x}
+ \frac{1}{32}\bigg[\bigg(\mathbb{P}_{N}\big[(u^{N})^{4}\big]
\bigg)^{2}\bigg]_{x}\Bigg)dx
\notag \\
&\ \ \ \ \ \ \ \ \ \ \ \  -
\big(\int (u^{N})^{3}dx'\big)\bigg( \int \big(\frac{1}{2}[(u^{N}_{x})^{2}]_{x} + \frac{1}{20} [(u^{N})^{5}]_{x}\big)dx \bigg)
\notag \\
&=  0.
\label{Eqn:E-trunc-inv}
\end{align}
Next we show that the $L^{2}$-norm of $u^N$ is invariant under the flow of \eqref{Eqn:gKdV-trunc}.  We compute
\begin{align}
\frac{d}{dt}\int (u^{N})^{2}dx &=
2\int u^{N}\bigg[-u^{N}_{xxx} + \mathbb{P}_{N}\bigg(\mathbb{P}
\big((u^{N})^{3}\big)u^{N}_{x}
\bigg)\bigg]dx \notag \\
&= 2\int u^{N}\bigg[-u^{N}_{xxx} + \mathbb{P}_{N}\bigg(
(u^{N})^{3}u^{N}_{x}
\bigg)-\mathbb{P}_{N}\bigg(
\big(\int(u^{N})^{3}dx'\big)u^{N}_{x}
\bigg) \bigg]dx
\notag \\
&= \int\bigg([(u^{N}_{x})^{2}]_{x} +
\frac{2}{5}[(u^{N})^{5}]_{x}\bigg)dx
\notag
\\
&\ \ \ \ \ \ \ \ \ \ \ \
-\big(\int(u^{N})^{3}dx'\big)
\bigg(\int[(u^{N})^{2}]_{x}dx\bigg) \notag
\\
&= 0. \label{Eqn:L2-trunc-inv}
\end{align}

We can write \eqref{Eqn:gKdV-trunc} in coordinates as a system of $N$ complex ODEs (for the Fourier coefficients)
$c_{n}:=\hat{u^{N}}(n)$, $1\leq n\leq N$ (recall that $c_{-n}=\overline{c_n}$).  We will verify the invariance of $\widetilde{\mu}_{N}$ in the coordinate space $\mathbb{C}^{N}$, and this extends to $E_{N}$ by its definition.
Having verified Hamiltonian and $L^{2}$-norm conservation for the solution $u^{N}(t)$ to \eqref{Eqn:gKdV-trunc}, to prove that $\widetilde{\mu}_{N}$ is invariant under the flow, it remains to show that the Lebesgue measure on $\mathbb{C}^{N}$ is preserved.

This system will take the form
$\displaystyle \frac{d}{dt}c_{n}=F_{n}(\{c_{k},\overline{c_{k}}\})$,
for a sequence of functions $F_{n}$.  Equivalently
\begin{align*}
\frac{d}{dt}a_{n}=\text{Re}(F_{n}(\{c_{k},\overline{c_{k}}\}))
\\
\frac{d}{dt}b_{n}=\text{Im}(F_{n}(\{c_{k},\overline{c_{k}}\})),
\end{align*}
when written in terms of $a_{n}$ and $b_{n}$, the real and imaginary parts of $c_{n}$.  To show that this system preserves the Lebesgue measure $\displaystyle \prod_{1<n\leq N}da_{n}db_{n}$ on $\mathbb{C}^{N}$, it suffices to verify that the corresponding vector field is divergence-free.  That is, to show that
\begin{align*}
\sum_{1\leq n\leq N}\frac{\partial \text{Re}(F_{n})}
{\partial a_{n}} + \frac{\partial \text{Im}(F_{n})}
{\partial b_{n}} = 0.
\end{align*}
This is equivalent to
\begin{align}
\sum_{0<n \leq N}\frac{\partial F_{n}}
{\partial c_{n}} + \frac{\partial \overline{F_{n}}}
{\partial \overline{c_{n}}} = 0.
\label{Eqn:div-free}
\end{align}
We will verify \eqref{Eqn:div-free} by explicit computation.  The equation \eqref{Eqn:gKdV-trunc}
is a system of $N$ coupled ODEs for each $1\leq n\leq N$,
\begin{align}
\frac{d}{dt}c_{n}(t)&= -in^{3}c_{n}(t) +
\sum_{
\substack{n_{1},\ldots, n_{4}\\n=n_{1}+\cdots n_{4}\\n\neq n_{1}}} (-in_{1})c_{n_{1}}c_{n_{2}}c_{n_{3}}c_{n_{4}}
\notag \\
&=:  F_{n}^{1} + F_{n}^{2}.
\label{Eqn:leb-F1-2}
\end{align}
For any fixed $n$, $1\leq n\leq N$, observe that
\begin{align}
\frac{\partial F_{n}^{1}}{\partial c_{n}}
+ \frac{\partial\overline{F_{n}^{1}}}{\partial \overline{c_{n}}}
=  (-in)^{3} + \overline{(-in)^{3}}
= in^{3} - in^{3}
=0.
\label{Eqn:leb-F1}
\end{align}
Then, using the Liebniz rule, for each fixed $n$, $0<|n|\leq N$,
\begin{align*}
\frac{\partial F_{n}^{1}}{\partial c_{n}} &=
\frac{\partial}{\partial c_{n}} \bigg(\sum_{
\substack{n_{1},\ldots,n_{4}\\n=n_{1}+\cdots n_{4}\\n\neq n_{1}}} (-in_{1})c_{n_{1}}c_{n_{2}}c_{n_{3}}c_{n_{4}}\bigg)
\\
&=
\sum_{
\substack{n_{1},\ldots,n_{4}\\n=n_{1}+\cdots n_{4}\\n\neq n_{1}}}
(-in_{1})
\bigg(c_{n_{1}}\frac{\partial c_{n_{2}}}{\partial c_{n}}c_{n_{3}}c_{n_{4}}
+c_{n_{1}}c_{n_{2}}\frac{\partial c_{n_{3}}}{\partial c_{n}}c_{n_{4}}
+c_{n_{1}}c_{n_{2}}c_{n_{3}}\frac{\partial c_{n_{4}}}{\partial c_{n}}\bigg)
\\
&=
\sum_{
\substack{n_{1},\ldots,n_{4}\\n=n_{1}+\cdots n_{4}\\n\neq n_{1}}}
(-in_{1})
\bigg(c_{n_{1}}\delta(n-n_{2})c_{n_{3}}c_{n_{4}}
+c_{n_{1}}c_{n_{2}}\delta(n-n_{3})c_{n_{4}}
+c_{n_{1}}c_{n_{2}}c_{n_{3}}\delta(n-n_{4})\bigg)
\\
&=
3\sum_{
\substack{n_{1},n_{3},n_{4}\\0=n_{1}+n_{3}+ n_{4}\\n\neq n_{1}}}
(-in_{1})c_{n_{1}}c_{n_{3}}c_{n_{4}}
\ \ \ \text{by symmetry in $n_{2},n_{3},n_{4}$,}
\\
&=
3\bigg(\sum_{\substack{n_{1},n_{3},n_{4}\\0=n_{1}+n_{3}+ n_{4}}}
(-in_{1})c_{n_{1}}c_{n_{3}}c_{n_{4}} + \sum_{\substack{n_{3},n_{4}\\0=n+n_{3}+ n_{4}}}(in)c_{n}c_{n_{3}}c_{n_{4}} \bigg)
\\
&=  3\bigg(\big(\partial_{x}u^{N}(u^{N})^{2}\big)
^{\wedge}(0) + \sum_{\substack{n_{3},n_{4}\\0=n+n_{3}+ n_{4}}}(in)c_{n}c_{n_{3}}c_{n_{4}} \bigg)
\\
&=  3\bigg(\big(\partial_{x}(u^{N})^{3}\big)
^{\wedge}(0) + \sum_{\substack{n_{3},n_{4}\\0=n+n_{3}+ n_{4}}}(in)c_{n}c_{n_{3}}c_{n_{4}} \bigg)
\\
&= 3\sum_{\substack{n_{3},n_{4}\\0=n+n_{3}+ n_{4}}}(in)c_{n}c_{n_{3}}c_{n_{4}}.
\end{align*}
Then
\begin{align*}
\sum_{0<|n|\leq N}\frac{\partial F_{n}^{1}}{\partial c_{n}} &= 3\sum_{\substack{n,n_{3},n_{4}\\0=n+n_{3}+ n_{4}}}(in)c_{n}c_{n_{3}}c_{n_{4}}
 \\
&=
3\sum_{\substack{n,n_{3},n_{4}\\0=n+n_{3}+ n_{4}}}
(-in)c_{n}c_{n_{3}}c_{n_{4}}
  \\
&=  3\big(\partial_{x}u^{N}(u^{N})^{2}\big)
^{\wedge}(0)
 \\
&=  3\big(\partial_{x}(u^{N})^{3}\big)
^{\wedge}(0)
\\
&= 0,
\end{align*}
and $\displaystyle \sum_{0<|n|\leq N}\frac{\partial \overline{F_{n}^{1}}}{\partial \overline{c_{n}}} =0$ with the same computation.  Using the property $c_{-n}=\overline{c_{n}}$ for each $n$, this gives
\begin{align}
\sum_{1\leq n\leq N}\frac{\partial F_{n}^{1}}{\partial c_{n}}
+ \frac{\partial \overline{F_{n}^{1}}}{\partial \overline{c_{n}}}=0.
\label{Eqn:leb-F2}
\end{align}
Combining \eqref{Eqn:leb-F1} and \eqref{Eqn:leb-F2} the justification of \eqref{Eqn:div-free} is complete.  This completes the proof of Proposition \ref{Prop:trunc-inv}.
\end{proof}

\subsection{Extending to global-in-time solutions}

In this subsection we will establish a proposition which uses the finite-dimensional Gibbs measure invariance (Proposition \ref{Prop:trunc-inv}), and an approximation argument, to extend the local solutions of \eqref{Eqn:gKdV-zeromean} (produced by Theorem \ref{Thm:LWP}) to global solutions.  Let $\Phi^{N}(t)$ denote the data-to-solution map for \eqref{Eqn:gKdV-trunc}.

\begin{proposition}
$\forall \,0<\sigma<1$ and $T^{*}>0, \exists \,\delta_1>\delta_2> \varepsilon>0$ sufficiently small and a measurable set $\Lambda_{\sigma,T^{*}}\subset H^{\frac{1}{2}-\delta_1}(\mathbb{T})$ such that
$\mu(\Lambda_{\sigma,T^{*}}^{c})<\sigma$ and
$\forall \, u_0\in\Lambda_{\sigma,T^{*}}$ there exists a (unique) solution $u\in S(t)u_0 + C([0,T^{*}];H^{\frac{1}{2}+\delta_2}(\mathbb{T}))\subset C([0,T^{*}];H^{\frac{1}{2}-\delta_1}(\mathbb{T}))$ to \eqref{Eqn:gKdV-zeromean} with data $u_0$.  Furthermore, for all $N\gg 0$, we have
\begin{align}
\big\|u(t)-S(t)u_0 - (\Phi^{N}(t)-S(t)\mathbb{P}_{N})u_0
\big\|_{C([0,T^{*}];H^{\frac{1}{2}+\delta_2}(\mathbb{T}))}
\lesssim C(\sigma)N^{-\varepsilon}.
\label{Eqn:global-app}
\end{align}
\label{Prop:global}
\end{proposition}
\begin{remark}
Regarding the uniqueness of the solution in Proposition \ref{Prop:global}, recall that, for $\omega\in \Omega_{T}$, the local solution produced by Theorem \ref{Thm:LWP} is unique in a ball in $X^{\frac{1}{2}+\delta_1,\frac{1}{2}-\delta_1}_{[0,T]}$ centered at the randomized data $S(t)u_{0,\omega}$.  For the solution produced by Proposition \ref{Prop:global}, this characterization is extended to $\frac{T^{*}}{T_{0}}$ intervals of size $T_{0}$ for some $T_{0}>0$ sufficiently small.  That is, for each $j=1,\ldots,\frac{1}{T_{0}}$, $u$ is the unique solution to \eqref{Eqn:gKdV-zeromean} for $t\in[jT_{0},(j+1)T_{0}]$ (with data $u(jT_{0})$) in a ball in $X^{\frac{1}{2}+\delta_1,\frac{1}{2}-\delta_1}_{[jT_{0},(j+1)T_{0}]}$ centered at $S(t-jT_{0})u(jT_{0})$.
\label{Rem:un1}
\end{remark}

\begin{proof}[Proof of Proposition \ref{Prop:global}:]
We will prove this proposition with $T^{*}=1$.  It will be obvious, after the proof, to see how this argument can be generalized to any value of $T^{*}>0$.  Let $\delta>\delta_1>\delta_2>0$ be sufficiently small such that we can apply Theorem \ref{Thm:LWP} twice, first with $\delta$ and $\delta_1$ matching here and in the statement of Theorem \ref{Thm:LWP}, and second with $\delta_1$ and $\delta_2$ playing the roles of $\delta$ and $\delta_1$ within Theorem \ref{Thm:LWP}.  Let $T>0$ be sufficiently small such that Theorem \ref{Thm:LWP} applies with both sets of parameters.  Denote by $\Omega_{1,T}\subset \Omega$ and $\Omega_{2,T}\subset \Omega$ the sets of ``good data'' produced by our first and second applications of Theorem \ref{Thm:LWP}, respectively.  Also let $\Omega_{T}=\Omega_{1,T}\cap \Omega_{2,T}$.
Then, in particular, $P(\Omega_{T}^c)<e^{-\frac{c}{T^{\beta}}}$, and the conclusion \eqref{Eqn:finapp} holds true for $\omega\in\Omega_{T}$.  By the triangle inequality, for each $\omega\in\Omega_{T}$, and every $M\geq N$, we have
\begin{align}
\|\big[\Phi^{M}(t)\mathbb{P}_{M}
-\Phi^{N}(t)\mathbb{P}_{N}-S(t)(\mathbb{P}_{M
}-\mathbb{P}_{N})\big]u_{0,\omega}\|
_{H^{\frac{1}{2}+\delta_{1}}} \leq N^{-\varepsilon},
\label{Eqn:ext-1}
\end{align}
for $t\in[0,T]$.

Our next objective is to extend \eqref{Eqn:ext-1} to the interval $[T,2T]$.  Let
\begin{align*}
\widetilde{\Omega}_{M,T}:=\bigg\{\sum_{1\leq|n|\leq N}\frac{g_{n}(\omega)}{|n|}e^{inx}
\big| \,\omega \in \Omega_{T} \bigg\}\oplus E_{M}^{\perp}\subset H^{\frac{1}{2}-\delta}(\mathbb{T}).
\end{align*}
Then we have $\rho(\widetilde{\Omega}_{M,T}^{c}) \leq P(\Omega_{T}^{c})<e^{-\frac{c}{T^{\beta}}}$, and by Cauchy-Schwarz and Theorem \ref{Thm:LRS},
\begin{align}
\mu_{M}(\widetilde{\Omega}_{M,T}^{c}) &=
\int_{\widetilde{\Omega}_{M,T}^{c}}f_{M}(u) d\rho
\notag \\
&\leq
\| f_{N}(u)\|_{L^{2}(d\rho)}
\Big(\rho(\widetilde{\Omega}_{M,T}^{c})\Big)^{\frac{1}{2}}
\notag\\
&\leq e^{-\frac{c}{T^{\beta}}},
\end{align}
for the appropriate constant $c>0$.  It follows that \eqref{Eqn:ext-1} holds if $u_{0}
\in\widetilde{\Omega}_{M,T}$ with $\mu_{M}(\widetilde{\Omega}_{M,T}^{c})<
e^{-\frac{1}{T^{\tilde{\beta}}}}$.  In particular, letting $u_{1}=\Phi^{M}(T)\mathbb{P}_{M}u_{0}$, we have
\begin{align}
\|u_{1}
-\big[\Phi^{N}(T)\mathbb{P}_{N}+
S(T)(\mathbb{P}_{M
}-\mathbb{P}_{N})\big]u_{0}\|
_{H^{\frac{1}{2}+\delta_{1}}} \leq N^{-\varepsilon}.
\label{Eqn:ext-2}
\end{align}
Separating \eqref{Eqn:ext-2}
into low and high frequencies,
\begin{align}
\|(I-\mathbb{P}_{N})u_{1}
-S(T)(\mathbb{P}_{M
}-\mathbb{P}_{N})u_{0}\|
_{H^{\frac{1}{2}+\delta_{1}}} \leq N^{-\varepsilon},
\label{Eqn:ext-3}
\end{align}
and
\begin{align}
\|\mathbb{P}_{N}u_{1}
-\Phi^{N}(T)\mathbb{P}_{N}u_{0}\|
_{H^{\frac{1}{2}+\delta_{1}}} \leq N^{-\varepsilon}.
\label{Eqn:ext-4}
\end{align}
For future use, notice that \eqref{Eqn:ext-3} easily gives, for $t\in[0,T]$,
\begin{align}
\|S(t)(I-\mathbb{P}_{N})u_{1}
-S(t+T)(\mathbb{P}_{M
}-\mathbb{P}_{N})u_{0}\|
_{H^{\frac{1}{2}+\delta_{1}}} \leq N^{-\varepsilon}.
\label{Eqn:ext-5}
\end{align}
Next suppose that $u_{1}\in\tilde{\Omega}_{M,T}$ (i.e. suppose that $u_{1}$ is ``good data" up to frequency $M$).
Then we can apply \eqref{Eqn:ext-1} to find, for $t\in[0,T]$,
\begin{align*}
\|\big[\Phi^{M}(t)\mathbb{P}_{M}-\Phi^{N}(t)\mathbb{P}_{N}-S(t)(\mathbb{P}_{M
}-\mathbb{P}_{N})\big]u_{1}\|
_{H^{\frac{1}{2}+\delta_{1}}} \leq N^{-\varepsilon},
\end{align*}
Or, equivalently
\begin{align}
\|\Phi^{M}(T+t)\mathbb{P}_{M}
u_{0}
-\big[\Phi^{N}(t)\mathbb{P}_{N}+S(t)(\mathbb{P}_{M
}-\mathbb{P}_{N})\big]u_{1}\|
_{H^{\frac{1}{2}+\delta_{1}}} \leq N^{-\varepsilon},
\label{Eqn:ext-7}
\end{align}
Furthermore, since $u_{1}$ is ``good data", we have (by adapting the justification of property (iv) in the proof of Theorem \ref{Thm:LWP} to the truncated system \eqref{Eqn:gKdV-trunc}) that $\Phi^{N}(t)$ is Lipschitz - with the same Lipschitz constant for all $N>0$ - on a ball in $H^{\frac{1}{2}+\delta_{1}}$ centered at $u_{1}$, and \eqref{Eqn:ext-5} gives, for $t\in[0,T]$,
\begin{align}
\|\Phi^{N}(t)u_{1}
-\Phi^{N}(T+t)\mathbb{P}_{N}u_{0}\|
_{H^{\frac{1}{2}+\delta_{1}}} \lesssim N^{-\varepsilon}.
\label{Eqn:ext-8}
\end{align}
Combining \eqref{Eqn:ext-5}, \eqref{Eqn:ext-7}
and \eqref{Eqn:ext-8}, we have for $t\in[0,T]$
that
\begin{align}
\|\big[\Phi^{M}(T+t)\mathbb{P}_{M}-
&\Phi^{N}(T+t)\mathbb{P}_{N}
-S(T+t)(\mathbb{P}_{M
}-\mathbb{P}_{N})\big]u_{0}\|
_{H^{\frac{1}{2}+\delta_{1}}} \notag\\
&\leq
\|\Phi^{M}(T+t)\mathbb{P}_{M}
u_{0}-
\big[\Phi^{N}(t)\mathbb{P}_{N}
-S(t)(\mathbb{P}_{M
}-\mathbb{P}_{N})\big]u_{1}\|
_{H^{\frac{1}{2}+\delta_{1}}} \notag \\
&\ \ \ \ \ \ \ +
\|\Phi^{N}(t)\mathbb{P}_{N}u_{1}-
\Phi^{N}(T+t)\mathbb{P}_{N}u_{0}\|
_{H^{\frac{1}{2}+\delta_{1}}}
\notag \\
&\ \ \ \ \ \ \ +
\|S(t)(\mathbb{P}_{M}-
\mathbb{P}_{N})u_{1}
-S(T+t)(\mathbb{P}_{M
}-\mathbb{P}_{N})u_{0}\|
_{H^{\frac{1}{2}+\delta_{1}}}
\notag \\
&\lesssim N^{-\varepsilon}.
\label{Eqn:ext-9}
\end{align}

This process can be iterated to extend
\eqref{Eqn:ext-1} to the interval $[0,1]$.
More precisely, if
$$u_{0}
\in\Lambda_{M,T}:=
\widetilde{\Omega}_{M,T}\cap
(\Phi^{M}(T))^{-1}\widetilde{\Omega}_{M,T}
\cap\cdots\cap
(\Phi^{M}(T))^{-k}\widetilde{\Omega}_{M,T},$$
where $k\sim \frac{1}{T}$, we can repeat this procedure $k$ times to obtain
\begin{align}
\|\big[\Phi^{M}(t)\mathbb{P}_{M}
-\Phi^{N}(t)\mathbb{P}_{N}-S(t)(\mathbb{P}_{M
}-\mathbb{P}_{N})\big]u_{0}\|
_{H^{\frac{1}{2}+\delta_{1}}} \leq C^{\frac{1}{T}}N^{-\varepsilon},
\label{Eqn:ext-10}
\end{align}
for all $t\in[0,1]$.

By the invariance of the truncated Gibbs measure $\mu_{M}$ under the flow $\Phi^{M}(t)$ (Proposition \ref{Prop:trunc-inv}), we have
\begin{align}
\mu_{M}(\Lambda_{M,T}^{c})
&\leq
\sum_{j=1}^{k}
\mu_{M}\big(((\Phi^{M}(T))^{-k}
\widetilde{\Omega}_{M,T})^{c}\big)
=
\sum_{j=1}^{k}
\mu_{M}\big(\widetilde{\Omega}_{M,T}^{c}\big)
\notag \\
&\lesssim
\frac{1}{T}e^{-\frac{1}{T^{\beta}}}
\rightarrow 0,
\label{Eqn:need}
\end{align}
as $T\rightarrow 0+$.

Following the definition found in Remark \ref{Rem:time} (of the Appendix), we let
\begin{align}
M_T := \Omega_{T}(0)\cap \cdots \cap \Omega_{T}(kT)\subset \Omega,
\label{Eqn:times}
\end{align}
where $k\sim \frac{1}{T}$.  Then
\begin{align*}
P(M_{T}^{c})\leq \sum_{j=1}^{k}P((\Omega_{T}(jT))^{c})\leq \frac{1}{T}e^{-\frac{c}{T^{\beta}}}.
\end{align*}
If we take
$\tilde{\Omega}_{T}:=\big\{\sum_{n\neq 0}
\frac{g_{n}(\omega)}{|n|}e^{inx}:\omega \in M_{T}\big\}$, we have by Theorem \ref{Thm:LRS} that $\mu(\tilde{\Omega}_{T}^{c})<\frac{1}{\sqrt{T}}e^{-\frac{c}{T^{\beta}}}$.  Then for any $\sigma>0$, by \eqref{Eqn:need}, we can pick
$T_{0}=T_{0}(\sigma)$
sufficiently small such that
$\displaystyle\mu_{M}(\Lambda_
{M,T_{0}}^{c})< \frac{\sigma}{2}$ for each $M>0$, and $\displaystyle\mu(\tilde{\Omega}_{T_{0}}^{c})<\frac{\sigma}{2}$.


Then consider
\begin{align*}
\Lambda_{\sigma}:&=\tilde{\Omega}_{T_{0}}\cap\limsup_{N\rightarrow \infty}\Lambda_{N,T_{0}}  \\ &=\tilde{\Omega}_{T_{0}}\cap\bigcap_{N=1}^{\infty}\bigcup_{M=N}^{\infty}
\Lambda_{M,T_{0}}.
\end{align*}
We will show that $\mu(\Lambda_{\sigma}^{c})< \sigma$.  By Fatou's Lemma we have
\begin{align}
\mu(\limsup_{N\rightarrow \infty}\Lambda_{N,T_{0}})
\geq \limsup_{N\rightarrow \infty}\mu(\Lambda_{N,T_{0}}).
\label{Eqn:GWP-A6}
\end{align}
We also have
$\mu(\Lambda_{N,T_{0}})=\int_{\Lambda_{N,T_{0}}}
f(u)d\rho(u)$
and
$\mu_{N}(\Lambda_{N,T_{0}})=\int_{\Lambda_{N,T_{0}}}
f_{N}(u)d\rho(u)$.
Applying Proposition \ref{Prop:conv}, we have
$$ \lim_{N\rightarrow \infty}(\mu(\Lambda_{N,T_{0}})
-\mu_{N}(\Lambda_{N,T_{0}}))=0.$$
Then using Corollary \ref{Cor:conv}, we have
\begin{align}
 \limsup_{N\rightarrow \infty}
\mu\Big(\Lambda_{N,T_{0}}\Big)
&= \limsup_{N\rightarrow \infty}
\mu_{N}\Big(\Lambda_{N,T_{0}}\Big)
= \limsup_{N\rightarrow \infty}
\Big[\mu_{N}\big(H^{\frac{1}{2}-\delta}(\mathbb{T})\big)
-
\mu_{N}\Big(\Lambda_{N,T_{0}}^{c}\Big)\Big]
\notag \\
&\geq  \limsup_{N\rightarrow \infty}
\Big[\mu_{N}\big(H^{\frac{1}{2}-\delta}(\mathbb{T})\big) - \frac{\sigma}{2}\Big]
=
\mu(H^{\frac{1}{2}-\delta}(\mathbb{T})) - \frac{\sigma}{2}.
\label{Eqn:GWP-A5}
\end{align}
Combining \eqref{Eqn:GWP-A6} and \eqref{Eqn:GWP-A5} we have
$$\mu(\limsup_{N\rightarrow \infty}\Lambda_{N,T_{0}})
\geq \mu(H^{\frac{1}{2}-\delta}(\mathbb{T})) - \frac{\sigma}{2},$$ so that
$$ \mu((\limsup_{N\rightarrow \infty}\Lambda_{N,T_{0}})^{c})\leq \frac{\sigma}{2}.$$  It follows that $\mu(\Lambda_{\sigma}^{c})<\sigma$.

Next, for $u_{0}\in\Lambda_{\sigma}$, we will use the estimate \eqref{Eqn:ext-10} to extend the truncated solutions $\tilde{u}^{N}$ to a full solution $u$ (of \eqref{Eqn:gKdV-zeromean}) for all $t\in[0,1]$. Indeed, for $\omega\in \Lambda_{\sigma}$, given any $N_{2}\geq N_{1}>0$, we have
$\omega\in\Lambda_{M,T_{0}}$ for some $M>N_{2}$.  Then by the triangle inequality and \eqref{Eqn:ext-10}, we have
\begin{align}
\|\big[(\Phi^{N_{2}}(t)\mathbb{P}_{N_{2}}
-S(t)\mathbb{P}_{N_{2}})
-(\Phi^{N_{1}}(t)\mathbb{P}_{N_{1}}&
-S(t)\mathbb{P}_{N_{1}})\big]u_{0}\|
_{C([0,1];H^{\frac{1}{2}+\delta_{1}})} \notag \\
&\leq  C(\sigma)N_{1}^{-\varepsilon}.
\label{Eqn:ext-11}
\end{align}
That is, the sequence $(\Phi^{N}(t)\mathbb{P}_{N}
-S(t)\mathbb{P}_{N})u_{0}$ is Cauchy in $C([0,1];H^{\frac{1}{2}+\delta_{1}})$ and converges to some $\tilde{\mathcal{D}}\in C([0,1];H^{\frac{1}{2}+\delta_{1}})$.

Let $\Phi(t)u_{0} := S(t)u_{0} + \tilde{\mathcal{D}}$.  Then we have
\begin{align}
\|\big[\Phi(t)
-S(t)
-(\Phi^{N}(t)
-S(t))\mathbb{P}_{N}\big]u_{0}\|
_{C([0,1];H^{\frac{1}{2}+\delta_{1}})} \leq  C(\sigma)N^{-\varepsilon}.
\label{Eqn:ext-11b}
\end{align}
We claim that $\tilde{\mathcal{D}}=\mathcal{D}(\Phi(t)u_{0},
\Phi(t)u_{0},\Phi(t)u_{0},\Phi(t)u_{0})$, and $\Phi(t)u_{0}$ is a solution to \eqref{Eqn:gKdV-zeromean} with data $u_{0}$.  To reach this conclusion we repeat the analysis of section \ref{Sec:LWP} on the interval $[jT_{0},(j+1)T_{0}]$ for each $j=1,2,\ldots\frac{1}{T_{0}}$.  We obtain $u$, a solution to \eqref{Eqn:gKdV-zeromean} for $t\in[jT_{0},(j+1)T_{0}]$ with data $\Phi(jT_{0})u_{0}$, as the limit of a sequence of solutions to \eqref{Eqn:gKdV-zeromean} with truncated data $\Phi^{N_k}(jT_{0})\mathbb{P}_{N_{k}}u_{0}$.  Then we prove that solutions to the truncated equation \eqref{Eqn:gKdV-ZMN} converge to the same limit, and therefore $u=\Phi(t)u_{0}$ satisfies \eqref{Eqn:gKdV-zeromean}.

More precisely, given $u_{0}\in \Lambda_{\sigma}$, there exists
a subsequence $N_{k}\rightarrow\infty$, as $k\rightarrow \infty$, such that for each $k\geq 0$,
$\Phi^{N_{k}}(jT_{0})\mathbb{P}_{N_{k}}u_{0}
\in \tilde{\Omega}_{N_{k},T_{0}}$, for all $j=1,2,\ldots\frac{1}{T_{0}}$.  That is, for every $k\geq 0$, $j=1,2,\ldots\frac{1}{T_{0}}$,
$\Phi^{N_{k}}(jT_{0})\mathbb{P}_{N_{k}}u_{0}$ is ``good data'' (up to frequency $N_k$). Consider the sequence of solutions $u^{N_{k}}_{j}$ to  \eqref{Eqn:gKdV-zeromean} with truncated data
$\Phi_{N_{k}}(jT_{0})\mathbb{P}_{N_k}u_{0}$ evolving from time $t=jT_{0}$.
These solutions exists globally in time by the results of \cite{CKSTT1,S1}.  In particular, for each $j$, the solution $u^{N_{k}}_{j}(t)$ exists for $t\in[jT_{0},(j+1)T_{0}]$.
We first claim that for each $j=1,2,\ldots,\frac{1}{T_{0}}$, there exists $\varepsilon>0$ such that for every $k> l \geq K_{0}(\sigma)$ sufficiently large,
\begin{align}
\|u^{N_{k}}_{j}
\|_{\frac{1}{2}-\delta_{1},\frac{1}{2}-\delta_{1},[jT_{0},(j+1)T_{0}]}
&\leq   R,
\label{Eqn:GWP-sol-2a}\\
\|\mathcal{D}(u^{N_{k}}_{j})
\|_{\frac{1}{2}-\delta_{1},\frac{1}{2}-\delta_{1},[jT_{0},(j+1)T_{0}]}
&\leq   \tilde{R},
\label{Eqn:GWP-sol-2b}\\
\|u^{N_{k}}_{j}-u^{N_{l}}_{j}
\|_{\frac{1}{2}-\delta_{1},\frac{1}{2}-\delta_{1},[jT_{0},(j+1)T_{0}]}
&\leq  \tilde{C}_1(\sigma)N_{l}^{-\varepsilon},
\label{Eqn:GWP-sol-2}\\
\|\mathcal{D}(u^{N_{k}}_{j})-
\mathcal{D}(u^{N_{l}}_{j})\|_{\frac{1}{2}+\delta_{1},\frac{1}{2}-\delta_{1},[jT_{0},(j+1)T_{0}]}
&\leq   \tilde{C}_2(\sigma)N_{l}^{-\varepsilon}.
\label{Eqn:GWP-sol-3}
\end{align}
These inequalities are established in a manner very similar to the justification of \eqref{Eqn:LWP-28}-\eqref{Eqn:LWP-30} found in the proof of Theorem \ref{Thm:LWP}.  In particular, each solution $u_j^{N_k}$ is evolving from ``good data'', and \eqref{Eqn:GWP-sol-2a}-\eqref{Eqn:GWP-sol-2b} follow automatically for $u_0\in\Lambda_{\sigma}$ following the justification of \eqref{Eqn:LWP-28}-\eqref{Eqn:LWP-29} (using Proposition \ref{Prop:nonlin} and Proposition \ref{Prop:nonlin2}).  To prove \eqref{Eqn:GWP-sol-2}-\eqref{Eqn:GWP-sol-3} we have to work a bit harder, since although each solution $u^{N_k}_{j}$ is evolving from ``good data'' at time $t=jT_{0}$, for different values of $k$, these solutions are not necessarily evolving from frequency truncation of the \textit{same} data.

In particular, we cannot use Proposition \ref{Prop:nonlin} directly to control an expression of (for example) the type $\mathcal{N}(u^{N_k}_{j},u^{N_l}_{j},u^{N_l}_{j},u^{N_l}_{j})$ with $k>l$, as Proposition \ref{Prop:nonlin} invokes the ``type (I)-type (II)'' decomposition with respect to the \textit{same} randomized data in each factor.  However, we can avoid this issue by using a telescopic summation
of $\mathcal{D}(u^{N_k}_{j})-\mathcal{D}(u^{N_l}_{j})$, for $k>l$, where $u^{N_l}_{j}$ only appears in factors of the form $u^{N_k}_{j}-u^{N_l}_{j}$.  That is, any factor of the type $u_{j}^{N_l}$ can always be decomposed as $u_{j}^{N_l}=u_{j}^{N_k}-(u_{j}^{N_k}-u_{j}^{N_l})$ to produce two terms, one with the factor $u_{j}^{N_k}$, the other with $u_{j}^{N_k}-u_{j}^{N_l}$.  Furthermore, by beginning with the telescopic summation(s) used in the proof of Theorem \ref{Thm:LWP}, we can ensure that each term has at least one factor of the type $u_{j}^{N_k}-u_{j}^{N_l}$.

Then we can decompose
\begin{align}
u_{j}^{N_k}(t)-&u_{j}^{N_l}(t)=S(t-jT_{0})
(\Phi^{N_{k}}(jT_{0})\mathbb{P}_{N_{k}}u_{0}-\Phi^{N_{l}}(jT_{0})\mathbb{P}_{N_{l}}u_{0})
+ \mathcal{D}_{j}^{N_k}(t)-\mathcal{D}_{j}^{N_l}(t)\notag \\
&= S(t-jT_{0})(\mathbb{P}_{N_k}-\mathbb{P}_{N_l})\Phi^{N_{k}}(jT_{0})
 +S(t-jT_{0})(\mathbb{P}_{N_l}(\Phi^{N_k}(jT_{0})u_{0}
- \Phi(jT_{0})u_{0}))
\notag \\
&\ \ \ +S(t-jT_{0})(\mathbb{P}_{N_l}(\Phi^{N_l}(jT_{0})u_{0}
- \Phi(jT_{0})u_{0}))
+ \mathcal{D}_{j}^{N_k}(t)-\mathcal{D}_{j}^{N_l}(t),
\label{Eqn:GWP-A6}
\end{align}
where
$\mathcal{D}_{j}^{N_k}(t)=\int_{jT_{0}}^{t}S(t-t')
\mathcal{N}(u_{j}^{N_k}(t'))dt'$.  The first term in \eqref{Eqn:GWP-A6} can be treated as type (I), matching the contributions from other factors that are either produced from $u_{j}^{N_k}-u_{l}^{N_k}$, or $u_{j}^{N_k}$, so that Proposition \ref{Prop:nonlin} is applicable.  The last three terms in \eqref{Eqn:GWP-A6} can be treated as type (II) by using \eqref{Eqn:ext-11b}.  With these modifications, we can use Proposition \ref{Prop:nonlin} and Proposition \ref{Prop:nonlin2} to obtain that for
$u_{0}\in \Lambda_{\sigma}$, we have
\begin{align}
\|u_{j}^{N_k}-&u_{j}^{N_l}\|_{\frac{1}{2}-\delta_1,\frac{1}{2}-\delta_1,
[jT_{0},(j+1)T_{0}]}
\notag \\
&\leq T_{0}^{\delta_{1}-}\tilde{C}(\sigma)N_{l}^{-\varepsilon}
+ \|\mathcal{D}_{j}^{N_k}-\mathcal{D}_{j}^{N_l}\|
_{\frac{1}{2}+\delta_1,\frac{1}{2}-\delta_1,
[jT_{0},(j+1)T_{0}]},
\label{Eqn:GWP-A8}
\end{align}
and
\begin{align}
\|&\mathcal{D}_{j}^{N_k}-\mathcal{D}_{j}^{N_l}\|
_{\frac{1}{2}+\delta_1,\frac{1}{2}-\delta_1,
[jT_{0},(j+1)T_{0}]}
\notag \\
&\leq T_{0}^{\theta}\big(\tilde{C}(\sigma)N_{l}^{-\varepsilon}
+ \|u_{j}^{N_k}-u_{j}^{N_l}\|_{\frac{1}{2}-\delta_1,\frac{1}{2}-\delta_1,
[jT_{0},(j+1)T_{0}]}
 + \|\mathcal{D}_{j}^{N_k}-\mathcal{D}_{j}^{N_l}\|
_{\frac{1}{2}+\delta_1,\frac{1}{2}-\delta_1,
[jT_{0},(j+1)T_{0}]}\big) \notag \\
&\ \ \ \ \ \ \cdot Z_{6}(\tilde{C}(\sigma)N_{l}^{-\varepsilon},2R,2\tilde{R}),
\label{Eqn:GWP-A8b}
\end{align}
where $Z_{6}(x,y,z)$ is a polynomial of degree 6.

The constants $R,\tilde{R}>0$ are fixed (to our choosing) in the statement of Theorem \ref{Thm:LWP}.  We can impose on $T_{0}>0$, in a manner that depends only on $R$ and $\tilde{R}$, such that if \eqref{Eqn:GWP-A8} holds, and we further have
\begin{align}
\|&\mathcal{D}_{j}^{N_k}-\mathcal{D}_{j}^{N_l}\|
_{\frac{1}{2}+\delta_1,\frac{1}{2}-\delta_1,
[jT_{0},(j+1)T_{0}]}
\notag \\
&\leq T_{0}^{\theta}\big(\tilde{C}(\omega)N_{l}^{-\varepsilon}
+ \|u_{j}^{N_k}-u_{j}^{N_l}\|_{\frac{1}{2}-\delta_1,\frac{1}{2}-\delta_1,
[jT_{0},(j+1)T_{0}]}
 + \|\mathcal{D}_{j}^{N_k}-\mathcal{D}_{j}^{N_l}\|
_{\frac{1}{2}+\delta_1,\frac{1}{2}-\delta_1,
[jT_{0},(j+1)T_{0}]}\big) \notag \\
&\ \ \ \ \ \ \cdot Z_{6}(R,2R,2\tilde{R}),
\label{Eqn:GWP-A9}
\end{align}
then by rearrangement it follows that \eqref{Eqn:GWP-sol-2}-\eqref{Eqn:GWP-sol-3} hold true.  Then we choose $K_{0}=K_{0}(T_{0})$ sufficiently large such that
$\tilde{C}(\sigma)N_{K_0}^{-\varepsilon}\leq R$.  For $k>l\geq K_0$,  \eqref{Eqn:GWP-A8b} implies \eqref{Eqn:GWP-A9}, and \eqref{Eqn:GWP-sol-2}-\eqref{Eqn:GWP-sol-3} are satisfied.

It follows that $u^{N_k}_{j}$ and $\mathcal{D}^{N_k}_{j}$ converge in $X^{\frac{1}{2}-\delta_1,\frac{1}{2}-\delta_1}_{jT_{0},(j+1)T_{0}}$ and $X^{\frac{1}{2}+\delta_1,\frac{1}{2}-\delta_1}_{jT_{0},(j+1)T_{0}}$, respectively, as $k\rightarrow \infty$.  As in the proof of Theorem \ref{Thm:LWP}, we proceed to show that, for each $j=1,2,\ldots,\frac{1}{T_{0}}$,
\begin{enumerate}[(i)]
\item  $u_{j}^{N_k}\rightarrow u\in C([jT_{0},(j+1)T_{0}];H^{\frac{1}{2}-\delta_{1}})$, a solution to \eqref{Eqn:gKdV-zeromean} with data $\Phi(jT_{0})u_0$.
\item The estimate
\begin{align}
\big\|u(t)-S(t)u_0 - (\Phi^{N_k}(t)-S(t)\mathbb{P}_{N_k})u_0
\big\|_{C([0,jT_{0}];H^{\frac{1}{2}+\delta_2}(\mathbb{T}))}
\lesssim (C(\sigma))^{j}N_k^{-\varepsilon}
\end{align}
 is satisfied.
\end{enumerate}
From (i) and (ii) (and \eqref{Eqn:ext-11b}) it will follow that $u(t)=\Phi(t)u_0$ solves \eqref{Eqn:gKdV-zeromean} for $t\in[0,1]$.
We proceed to justify (i) and (ii) by induction on $j=1,\ldots,\frac{1}{T_{0}}$.  The base case ($j=0$) follows from the conclusions of Theorem 1, which hold since $u_0\in\Lambda_{\sigma}\subset \tilde{\Omega}_{T_{0}}$ by construction.  Now suppose (i) and (ii) hold up to some $j$, and we wish to extend these results to $t\in[jT_{0},(j+1)T_{0}]$.

By (i) and (ii) we have
\begin{align*}
\Phi(t)u_0 = S(t)u_0 + \int_{0}^{t}S(t-t')\mathcal{N}(\Phi(t')u_0)dt'
\end{align*}
for $t\in [0,j T_{0}]$.  We also have, for $t\in [0,j T_{0}]$, that
$\int_{0}^{t}S(t-t')\mathcal{N}(\Phi(t')u_0)dt'=\tilde{D}(t)$
is the convergent of $(\Phi^{N}(t)-S(t))\mathbb{P}_{N}u_{0}$ in $C([0,1];H^{\frac{1}{2}+\delta_{1}})$, and we therefore have
\begin{align}
\|\int_{0}^{t}S(t-t')\mathcal{N}(\Phi(t')u_0)dt'\|_{C([0,jT_{0}];H^{\frac{1}{2}+\delta_{1}})}
\leq \|\tilde{D}\|_{C([0,1];H^{\frac{1}{2}+\delta_{1}})}\leq C_4,
\label{Eqn:GWP-A10b}
\end{align}
for some constant $0<C_4<\infty $.
We also remark, that if (ii) holds, then we have
\begin{align}
\|(I-P_{N_k})S(t-jT_{0})\Phi(jT_{0})u_0\|_{X^{\frac{1}{2}-\delta_1,\frac{1}{2}-\delta_{1}}_{T}}
\rightarrow 0,
\label{Eqn:GWP-A10}
\end{align}
as $k\rightarrow \infty$.

Let $u_j$ denote the convergent of $u_j^{N_k}$ in  $X^{\frac{1}{2}-\delta_{1},\frac{1}{2}-\delta_{1}}_{[jT_{0},(j+1)T_{0}]}$.  We wish to show that
\begin{align}
u_{j}(t)=S(t-jT_{0})\Phi(jT_{0})u_0
+ \int_{jT_{0}}^{t}S(t-t')\mathcal{N}(u_j(t'))dt'.
\label{Eqn:GWP-A10c}
\end{align}
for $t\in[jT_{0},(j+1)T_{0}]$.

As in the justification of this property during the proof of Theorem \ref{Thm:LWP} (where we established \eqref{Eqn:LWP-12}), we have
$$ \int_{jT_{0}}^{t}S(t-t')\mathcal{N}_{0}(u^{N_k}_j(t'))dt'\rightarrow
\int_{jT_{0}}^{t}S(t-t')\mathcal{N}_{0}(u_j(t'))dt' \ \ \text{in }\ X^{\frac{1}{2}-\delta_1,\frac{1}{2}-\delta_1}_{[jT_{0},(j+1)T_{0}]},$$
and
$$ \int_{jT_{0}}^{t}S(t-t')\mathcal{N}_{r}(u^{N_k}_j(t'))dt'\rightarrow
v_{r,j}, \ \ \text{in }\ X^{\frac{1}{2}+\delta_1,\frac{1}{2}-\delta_1}_{[jT_{0},(j+1)T_{0}]},$$
for $r=-1,1,2,3,4$, where each $v_{r,j}\in X^{\frac{1}{2}+\delta_1,\frac{1}{2}-\delta_1}_{[jT_{0},(j+1)T_{0}]}$.
We can then express
\begin{align*}
u_j(t)-S(t-jT_{0})\Phi(jT_{0})u_0
&=  u_{j}(t)-u_{j}^{N_k}(t) + (I-P_{N_k})S(t-jT_{0})\Phi(jT_{0})u_0 \\
&\ \ \ \ - S(t-jT_{0})
(\mathbb{P}_{N_k}\Phi(jT_{0})-\Phi^{N_k}(jT_{0})\mathbb{P}_{N_k})u_0
\\&\ \ \ \ + \sum_{r=-1}^{4}\int_{jT_{0}}^{t}S(t-t')\mathcal{N}_{r}(u^{N_k}_j(t'))dt'
\\
&\rightarrow  \int_{jT_{0}}^{t}S(t-t')\mathcal{N}_{0}(u^{N_k}_j(t'))dt'
+ \sum_{r=-1,r\neq 0}^{4}v_{r,j},
\end{align*}
in $X^{\frac{1}{2}-\delta_1,\frac{1}{2}-\delta_1}_{[jT_{0},(j+1)T_{0}]}$, as $k\rightarrow \infty$.  It remains to show that $v_{r,j}=\mathcal{D}_{r}(u_j)$
for $r=-1,1,2,3,4$.  We follow the approach taken in the justification of \eqref{Eqn:LWP-12}.  This argument requires modification in several ways.  As above, we use a telescopic summation that only includes factors of the form $u_j-u_j^{N_k}$ or $u_j$.  If $u_j^{N_k}$ appears in a factor by itself, it can be expanded as $u_{j}^{N_k}=u_j-(u_j-u_j^{N_k})$.  For the type (I)-type (II) decomposition of factors of the type $u_j-u_j^{N_k}$, we consider
\begin{align}
u_j(t)-u_{j}^{N_k}(t)
&= (I-P_{N_k})S(t-jT_{0})\Phi(jT_{0})u_0 \notag \\
&\ \ \ \ + S(t-jT_{0})
(\mathbb{P}_{N_k}\Phi(jT_{0})-\Phi^{N_k}(jT_{0})\mathbb{P}_{N_k})u_0
\notag \\&\ \ \ \ +
\int_{jT_{0}}^{t}S(t-t')(\mathcal{N}_{0}(u_j(t'))-\mathcal{N}_{0}(u^{N_k}_j(t')))dt'
\notag \\
&\ \ \ \ + \sum_{r=-1,r\neq 0}^{4}(v_{r,j}-\int_{jT_{0}}^{t}S(t-t')\mathcal{N}_{r}(u^{N_k}_j(t'))dt')
\notag \\
&= (I-P_{N_k})\big(S(t)u_0+ S(t-jT_{0})\int_{0}^{jT_{0}}S(jT_{0}-t')\mathcal{N}(\Phi(t')u_0)dt'\big)
\notag \\ &\ \ \ \ + (I-P_{N_k})S(t-jT_{0})\Phi(jT_{0})u_0 \notag \\
&\ \ \ \ + S(t-jT_{0})
(\mathbb{P}_{N_k}\Phi(jT_{0})-\Phi^{N_k}(jT_{0})\mathbb{P}_{N_k})u_0
\notag \\&\ \ \ \ +
\int_{jT_{0}}^{t}S(t-t')(\mathcal{N}_{0}(u_j(t'))-\mathcal{N}_{0}(u^{N_k}_j(t')))dt'
\notag \\
&\ \ \ \ + \sum_{r=-1,r\neq 0}^{4}(v_{r,j}-\int_{jT_{0}}^{t}S(t-t')\mathcal{N}_{r}(u^{N_k}_j(t'))dt').
\label{Eqn:GWP-A11}
\end{align}
Similarly we decompose contributions from factors of the type $u_j$ as follows
\begin{align}
u_j(t)&=S(t)u_0 + S(t-jT_{0})\int_{0}^{jT_{0}}S(jT_{0}-t')\mathcal{N}(\Phi(t')u_0)dt'
\notag \\&\ \ \ \
+ \int_{jT_{0}}^{t}S(t-t')\mathcal{N}_{0}(u_j(t'))dt' + \sum_{r=-1,r\neq 0}^{4}v_{r,j}.
\label{Eqn:GWP-A12}
\end{align}
The first term on the right-hand side of both \eqref{Eqn:GWP-A11} and \eqref{Eqn:GWP-A12} can be treated as type (I) on the interval $[jT_{0},(j+1)T_{0}]$.  This is not entirely obvious, because the initial data is not propagating from ``good data'' at time $t=jT_{0}$, but rather from data at time $t=0$.  This is not a problem, however, because we defined $M_T$, $\Omega_{T}$ and $\Lambda_{\sigma}$ in a way that manages this complication (see \eqref{Eqn:times} above, and Remark \ref{Rem:time} for more discussion).  The second term on the right-hand side of both \eqref{Eqn:GWP-A11} and \eqref{Eqn:GWP-A12} can be treated as type (II), and estimated with the higher temporal regularity $b=\frac{1}{2}+\delta_1$, since by \eqref{Eqn:GWP-A10b}, we have
\begin{align*}
\|S(t-jT_{0})\int_{0}^{jT_{0}}&S(jT_{0}-t')\mathcal{N}(\Phi(t')u_0)dt'
\|_{X^{\frac{1}{2}+\delta_1,\frac{1}{2}+\delta_1}_{[jT_{0},(j+1)T_{0}]}}\\
&\leq
\|\int_{0}^{jT_{0}}S(jT_{0}-t')\mathcal{N}(\Phi(t')u_0)dt'
\|_{H^{\frac{1}{2}-\delta_1}}\leq B_{N_k}\rightarrow 0,
\end{align*}
as $k\rightarrow \infty$.  The remaining terms in \eqref{Eqn:GWP-A11} and \eqref{Eqn:GWP-A12} can (as in the proof of Theorem \ref{Thm:LWP}) be treated as type (II) and estimated in the higher regularity $b=\frac{1}{2}+\delta_1$, or for the contribution from $A_0$, through the second iteration.  With this strategy, we have by Proposition \ref{Prop:nonlin} and Proposition \ref{Prop:nonlin2}, that for $r=-1,1,2,3,4$,
\begin{align*}
\|\int_{jT_{0}}^{t}S(t-t')&(\mathcal{N}_{r}(u_j(t'))-\mathcal{N}_{r}(u^{N_k}_j(t')))dt'
\|_{\frac{1}{2}+\delta_{1},\frac{1}{2}+\delta_1,[jT_{0},(j+1)T_{0}]}
\\
&\lesssim T^{\theta}\Big(C_{4}(\sigma)N_{k}^{-\varepsilon} + B_{N_k}
 + \|u_j-u_j^{N_k}\|_{\frac{1}{2}-\delta_{1},\frac{1}{2}-\delta_1,[jT_{0},(j+1)T_{0}]}
\\
&\ \ \ \ \quad \quad + \|\int_{jT_{0}}^{t}S(t-t')(\mathcal{N}_{0}(u_j(t'))-\mathcal{N}_{0}(u^{N_k}_j(t')))dt'
\|_{\frac{1}{2}+\delta_{1},\frac{1}{2}-\delta_1,[jT_{0},(j+1)T_{0}]}
\\
&\ \ \ \ \quad \quad +
\sum_{r=-1,r\neq 0}^{4}
\|v_{r,j}-\int_{jT_{0}}^{t}S(t-t')\mathcal{N}_{r}(u^{N_k}_j(t'))dt'
\|_{\frac{1}{2}+\delta_{1},\frac{1}{2}+\delta_1,[jT_{0},(j+1)T_{0}]}
 \Big)\\
&\rightarrow 0,
\end{align*}
as $k\rightarrow \infty$, where the implied constant in the last expression depends on $R$ and $\tilde{R}$.  It follows that $v_{r,j}=\int_{jT_{0}}^{t}S(t-t')\mathcal{N}_{r}(u_j(t'))dt'$ for $r=-1,1,2,3,4$, and we have justified \eqref{Eqn:GWP-A10c}.  The conclusion of point (i) has been established for $t\in[jT_{0},(j+1)T_{0}]$.

We turn to the justification of point (ii) for $t\in[jT_{0},(j+1)T_{0}]$.  Again we closely follow the argument used in the proof of Theorem \ref{Thm:LWP}, specifically we follow the justification of \eqref{Eqn:conv-est3} in lines \eqref{Eqn:conv-bdd}-\eqref{Eqn:conv-est3}.  In fact, the only modifications required are very similar those described above in the discussion of point (i).  For $j=1,2,\ldots,\frac{1}{T_{0}}$, let $\tilde{u}_{j}^{N_k}(t)=\Phi^{N_k}(t)\mathbb{P}_{N_k}(u_0)$.  We first claim that for $u_0\in\Lambda_{\sigma}$, we have
\begin{align}
\|\tilde{u}^{N_{k}}_{j}
\|_{\frac{1}{2}-\delta_{1},\frac{1}{2}-\delta_{1},[jT_{0},(j+1)T_{0}]}
&\leq   R,\notag \\
\|\tilde{\mathcal{D}}_{j}^{N_k}
\|_{\frac{1}{2}+\delta_{1},\frac{1}{2}-\delta_{1},[jT_{0},(j+1)T_{0}]}
&\leq   \tilde{R}.
\label{Eqn:GWP-A12b}
\end{align}
where $\mathcal{D}_{j}(u)(t)=\int_{jT_{0}}^{t}S(t-t')\mathcal{N}(u(t'))dt'$, $\mathcal{D}_{j}:= \mathcal{D}_{j}(u)$, and $\tilde{\mathcal{D}}_{j}^{N_k}:=\mathbb{P}_{N_k}\mathcal{D}_{j}(\tilde{u}^{N_k}_{j})$.
In fact this is automatic from the same justification as \eqref{Eqn:GWP-sol-2a}-\eqref{Eqn:GWP-sol-2b}, since each $\tilde{u}^{N_k}$ is evolving from ``good data'' $\Phi^{N_k}(jT_{0})\mathbb{P}_{N_{k}}u_0\in \tilde{\Omega}_{N_k,T_{0}}$.  Next we claim that
\begin{align}
\|u-\tilde{u}^{N_{k}}_{j}
\|_{\frac{1}{2}-2\delta_{1}+\delta_2,\frac{1}{2}-\delta_{1},[jT_{0},(j+1)T_{0}]}
&\lesssim   C_6(\sigma) N_{k}^{-\varepsilon},\notag \\
\|\mathcal{D}_{j}-\tilde{\mathcal{D}}^{N_k}_{j}
\|_{\frac{1}{2}+\delta_{2},\frac{1}{2}-\delta_{1},[jT_{0},(j+1)T_{0}]}
&\lesssim   C_6(\sigma) N_{k}^{-\varepsilon},
\label{Eqn:GWP-A13}
\end{align}
where the implied constants in \eqref{Eqn:GWP-A13} depend on $R$ and $\tilde{R}$.
The justification of \eqref{Eqn:GWP-A13} follows that of \eqref{Eqn:conv-est1} closely, but once again we use two modifications: a telescopic summation of $\mathcal{D}_{j}-\tilde{\mathcal{D}}^{N_k}_{j}$ which only includes factors of the form
$u$ or $u - \tilde{u}^{N_{k}}_{j}$, and a decomposition of $u$ and $u - \tilde{u}^{N_{k}}_{j}$ as follows:
\begin{align}
u(t)-\tilde{u}_{j}^{N_k}(t)
&= (I-P_{N_k})\big(S(t)u_0+ S(t-jT_{0})\int_{0}^{jT_{0}}S(jT_{0}-t')\mathcal{N}(\Phi(t')u_0)dt'\big)
\notag \\
&\ \ \ \ + S(t-jT_{0})
(\mathbb{P}_{N_k}\Phi(jT_{0})-\Phi^{N_k}(jT_{0})\mathbb{P}_{N_k})u_0
\notag \\&\ \ \ \ +
\mathcal{D}_{j}(t)-\tilde{\mathcal{D}}^{N_k}_{j}(t),
\label{Eqn:GWP-A14}
\end{align}
and
\begin{align}
u(t)=S(t)u_0 + S(t-jT_{0})\int_{0}^{jT_{0}}S(jT_{0}-t')\mathcal{N}(\Phi(t')u_0)dt'
+ \mathcal{D}_{j}.
\label{Eqn:GWP-A15}
\end{align}
With \eqref{Eqn:GWP-A13} in place, it follows that
\begin{align}
\|u(t)-&S(t-jT_{0})\Phi(jT_{0})u_0 + \Phi^{N_k}(t)\mathbb{P}_{N_k}u_0 -S(t-jT_{0})\Phi^{N_k}(jT_{0})u_0 \|_{C([jT_{0},(j+1)T_{0}];H^{\frac{1}{2}+\delta_{2}}(\mathbb{T}))} \notag \\
&= \|\mathcal{D}_{j}-\tilde{\mathcal{D}}^{N_k}_{j}\|
_{C([jT_{0},(j+1)T_{0}];H^{\frac{1}{2}+\delta_{2}}(\mathbb{T}))}
 \notag \\
&\lesssim \|\mathcal{N}_{j}-\tilde{\mathcal{N}}^{N_k}_{j}\|_{Y^{\frac{1}{2}+\delta_{2},-1}_{T}}
\notag \\
&\leq
C_7(\sigma)\big(N_{k}^{-\varepsilon}+ \|u-\tilde{u}^{N_{k}}_{j}
\|_{\frac{1}{2}-2\delta_{1}+\delta_2,\frac{1}{2}-\delta_{1},[jT_{0},(j+1)T_{0}]}
\notag \\
&\ \ \ \ +
\|\mathcal{D}_{j}-\tilde{\mathcal{D}}_{j}\mathbb{P}_{N_{k}}\mathcal{D}_{j}(\tilde{u}^{N_{k}}_{j})
\|_{\frac{1}{2}+\delta_{2},\frac{1}{2}-\delta_{1},[jT_{0},(j+1)T_{0}]}
\big)
\leq C_8(\sigma)N_{k}^{-\varepsilon}.
\label{Eqn:GWP-A16}
\end{align}
By \eqref{Eqn:GWP-A16} and the triangle inequality,
\begin{align}
&\|u(t)-S(t)u_0 + (\Phi^{N_k}(t)-S(t))\mathbb{P}_{N_k}u_0
\|_{C([jT_{0},(j+1)T_{0}];H^{\frac{1}{2}+\delta_{2}}(\mathbb{T}))}
\notag \\
&\leq
\|u(t)-S(t-jT_{0})\Phi(jT_{0})u_0 + \Phi^{N_k}(t)\mathbb{P}_{N_k}u_0 -S(t-jT_{0})\Phi^{N_k}(jT_{0})u_0
\|_{C([jT_{0},(j+1)T_{0}];H^{\frac{1}{2}+\delta_{1}}(\mathbb{T}))}
\notag \\
&\ \ \ \ + \|\Phi(jT_{0})-S(jT_{0})u_0 + (\Phi^{N_k}(jT_{0})-S(jT_{0}))\mathbb{P}_{N_k}u_0
\|_{H^{\frac{1}{2}+\delta_{2}}(\mathbb{T})}\notag \\
&\leq C_{9}(\sigma)N_{k}^{-\varepsilon}.
\label{Eqn:GWP-A17}
\end{align}
Finally by our inductive hypothesis this gives
\begin{align}
\|u(t)-S(t)u_0 &+ (\Phi^{N_k}(t)-S(t))\mathbb{P}_{N_k}u_0
\|_{C([0,(j+1)T_{0}];H^{\frac{1}{2}+\delta_{2}}(\mathbb{T}))}
\notag \\
&\leq
\|u(t)-S(t)u_0 + (\Phi^{N_k}(t)-S(t))\mathbb{P}_{N_k}u_0
\|_{C([0,jT_{0}];H^{\frac{1}{2}+\delta_{2}}(\mathbb{T}))}
\notag \\
&\ \ \ \
+\|u(t)-S(t)u_0 + (\Phi^{N_k}(t)-S(t))\mathbb{P}_{N_k}u_0
\|_{C([jT_{0},(j+1)T_{0}];H^{\frac{1}{2}+\delta_{2}}(\mathbb{T}))}
\notag \\
&\leq C_{10}(\sigma)N_{k}^{-\varepsilon}.
\label{Eqn:GWP-A18}
\end{align}
With \eqref{Eqn:GWP-A18} we have justified point (ii) on the interval $[jT_{0},(j+1)T_{0}]$.  The proof by induction of points (i) and (ii) is complete.
The justification of uniqueness as described in Remark \ref{Rem:un1} is easily established following the proof of Theorem \ref{Thm:LWP} (and the modifications outlined above).
The proof of Proposition \ref{Prop:global} is complete.
\end{proof}

\subsection{Invariance of the Gibbs measure}

In this subsection we prove Theorem \ref{Thm:GWP}.  The main ingredients of this proof are: (i) weak convergence of the finite-dimensional Gibbs measures (Proposition \ref{Prop:conv}), (ii) invariance of the Gibbs measure under the flow of \eqref{Eqn:gKdV-trunc} (Proposition \ref{Prop:trunc-inv}), and (iii) the existence of global-in-time solutions to \eqref{Eqn:gKdV-zeromean} in the support of the Gibbs measure with a good approximation to the finite-dimensional dynamics (Proposition \ref{Prop:global}).

\begin{proof}[Proof of Theorem \ref{Thm:GWP}]
Given $n,j\in\mathbb{N}$, let $T_{j}=2^{j}$ and $\displaystyle\sigma_{n,j}=\frac{1}{n 2^{j}}$.  Also let $\Lambda_{\sigma_{n,j},T_j}$ be the subset of $H^{\frac{1}{2}-\delta_1}(\mathbb{T})$ produced by Proposition \ref{Prop:global} with $\sigma=\sigma_{n,j}$ and $T^*=T_j$.  Define $\Sigma_{n}:=\cap_{j=1}^{\infty}\Lambda_{\sigma_{n,j},T_j}$, so that $\mu(\Sigma_{n}^{c})<\frac{1}{n}$.  By taking $\Sigma:=\cup_{n=1}^{\infty}\Sigma_{n}$, it follows that $\mu(\Sigma^{c})=0$.  Moreover if $u_0\in\Sigma$, we have
$u_0 \in \cap_{j=1}^{\infty}\Lambda_{\sigma_{n,j},T_j}$, and \eqref{Eqn:gKdV-zeromean} is globally well-posed by the conclusions of Proposition \ref{Prop:global}.

Next we prove that the Gibbs measure $\mu$ is invariant under the flow.  For $u_0\in\Sigma$, let $\Phi(t)$ denote the data-to-solution map of \eqref{Eqn:gKdV-zeromean}.  One formulation of invariance is the following: for all $F\in L^{1}(H^{\frac{1}{2}-\delta_1}(\mathbb{T}),d\mu)$, we have
\begin{align}
\int_{\Sigma}F(\Phi(t)(u)) d\mu(u) = \int_{\Sigma}F(u) d\mu(u)
\label{Eqn:GWP-B1}
\end{align}
for all $t\geq 0$.  It suffices to establish \eqref{Eqn:GWP-B1}
on a dense set in $L^{1}(H^{\frac{1}{2}-\delta_1}(\mathbb{T}),d\mu)$, in particular we choose $\mathcal{H}\subset L^{1}(H^{\frac{1}{2}-\delta_1}(\mathbb{T}),d\mu)$ given by
\begin{align}
\mathcal{H}=\bigcup_{N\in\mathbb{N}}
\{F=F(\hat{u}_{-N},\ldots,\hat{u}_{N}) \ \text{bounded and continuous}\}.
\label{Eqn:GWP-B2}
\end{align}
Fix $F\in \mathcal{H}$, and $\kappa>0$.  For $N>0$ sufficiently large we have
$$ \int_{\Sigma}F(u) d\mu_N(u) = \int_{\Sigma} F(u) f_{N}(u)d\rho.$$
We also have $\int_{\Sigma}F(u) d\mu(u) = \int_{\Sigma}F(u)f(u)d\rho(u)$, and
by boundedness of $F$ combined with Proposition \ref{Prop:conv}, it follows that $\exists\, N_1>0$ such that for $N\geq N_1$, we have
\begin{align}
\big|\int_{\Sigma}&F(u) d\mu(u) - \int_{\Sigma}F(u) d\mu_N(u)\big| \notag \\
& + \big|\int_{\Sigma}F(\Phi(t)(u)) d\mu(u) - \int_{\Sigma}F(\Phi(t)(u)) d\mu_N(u)\big|
<\frac{\kappa}{2}.
\label{Eqn:GWP-B3}
\end{align}
Let $n>0$ be sufficiently large such that $\displaystyle \frac{1}{n}<\frac{\kappa}{32\|F\|_{L^{\infty}}}$.  Then we have
\begin{align}
\big|\int_{\Sigma\setminus\Sigma_{n}}&F(\Phi(t)(u)) d\mu_{N}(u)
- \int_{\Sigma\setminus\Sigma_{n}}F(\Phi_{N}(t)(u)) d\mu_N(u)\big|\notag \\
&\leq  2\|F\|_{L^{\infty}}\mu_N(\Sigma\setminus\Sigma_{n})
\leq  2\|F\|_{L^{\infty}}\mu_{N}(\Sigma_{n}^{c})
<\frac{\kappa}{4},
\label{Eqn:GWP-B4}
\end{align}
for $N$ sufficiently large, where we have used Corollary \ref{Cor:conv} in the last line.

By continuity of $F$, there exists $\gamma>0$ such that if $\|\Phi(t)u_0-\Phi^{N}(t)\mathbb{P}_{N}u_0\|_{H^{\frac{1}{2}-\delta_1}(\mathbb{T})}<\gamma$, then $\displaystyle |F(\Phi(t)u_0)-F(\Phi^{N}(t)\mathbb{P}_{N}u_0)|<\frac{\kappa}{8\mu(H^{\frac{1}{2}-\delta_1}(\mathbb{T}))}$.
For $u_0 \in \Sigma_n$, we project \eqref{Eqn:global-app} to $E_{N}$ to obtain, for all $N\gg 0$,
\begin{align}
\|\Phi(t)u_0-\Phi^{N}(t)\mathbb{P}_{N}u_0\|_{H^{\frac{1}{2}+\delta_2}(\mathbb{T})}
\leq C(n)N^{-\varepsilon}.
\label{Eqn:GWP-B5}
\end{align}
Taking $N\geq N_2$ sufficiently large, we have $\|\Phi(t)u_0-\Phi^{N}(t)\mathbb{P}_{N}u_0\|_{H^{\frac{1}{2}-\delta_1}(\mathbb{T})}<\gamma$, and $\displaystyle|F(\Phi(t)u_0)-F(\Phi^{N}(t)\mathbb{P}_{N}u_0)|\leq \frac{\kappa}{8\mu(H^{\frac{1}{2}-\delta_1}(\mathbb{T}))}$ is satisfied.  This gives
\begin{align}
\big|\int_{\Sigma_{n}}&F(\Phi(t)(u)) d\mu_{N}(u)
- \int_{\Sigma_{n}}F(\Phi_{N}(t)(u)) d\mu_N(u)\big| \leq \frac{\kappa}{4}.
\label{Eqn:GWP-B6}
\end{align}
We also have, by Proposition \ref{Prop:trunc-inv}, for all $t\geq 0$,
\begin{align}
\int_{\Sigma_{n}}F(\Phi_{N}(t)(u))d\mu_{N}(u) = \int_{\Sigma_{n}}F(u)d\mu_{N}(u).
\label{Eqn:GWP-B7}
\end{align}
By combining \eqref{Eqn:GWP-B3}, \eqref{Eqn:GWP-B4}, \eqref{Eqn:GWP-B6} and \eqref{Eqn:GWP-B7}, we conclude that for $N\gg 0$ sufficiently large, we have
\begin{align*}
\big|\int_{\Sigma}F(\Phi(t)(u)) d\mu(u)
- \int_{\Sigma}F(u) d\mu(u)\big| < \kappa.
\end{align*}
Since $\kappa$ was arbitrary, we conclude that $\int_{\Sigma}F(\Phi(t)(u)) d\mu(u)
= \int_{\Sigma}F(u) d\mu(u)$, and the Gibbs measure $\mu$ is invariant under the flow of \eqref{Eqn:gKdV-zeromean}.

We have now established global well-posedness of \eqref{Eqn:gKdV-zeromean} on a set $\Sigma\subset H^{\frac{1}{2}-\delta_{1}}(\mathbb{T})$ of full $\mu$-measure, and invariance of the Gibbs measure under the flow.  Let us now describe how this extends to global well-posedness almost surely, with randomized initial data given by \eqref{Eqn:initialdata}.  From the definition of $\mu$, it follows that
\begin{align*}
P(\{\omega\in\Omega:u_{0,\omega}\in \Sigma^{c},\|u_0\|_{L^2}\leq B\})=0.
\end{align*}
Recall that our initial choice of $B>0$ (in the definition of the Gibbs measure $\mu$) was arbitrary.  By the large deviation estimate (Lemma \ref{Lemma:largedev}), for all $n\in\mathbb{N}$, there exists $B_{n}>0$ sufficiently large such that
\begin{align}
P(\{\omega\in\Omega:\|u_0\|_{L^2}\geq B_n\})<\frac{1}{n}.
\label{Eqn:GWP-C1}
\end{align}
By the arguments above, for each fixed $n\in\mathbb{N}$, there exists a corresponding sequence of Gibbs measures $\mu_{B_n}$ and $\mu_{B_n}$-measurable sets $\Sigma_{B_n}\subset H^{\frac{1}{2}-\delta_1}(\mathbb{T})$ such that \eqref{Eqn:gKdV-zeromean} is globally well-posed for $u_0\in \Sigma_{B_n}$, and such that
\begin{align}
P(\{\omega\in\Omega:u_{0,\omega}\in \Sigma_{B_n}^{c},\|u_0\|_{L^2}\leq B_n\})=0,
\label{Eqn:GWP-C2}
\end{align}
for all $n\in\mathbb{N}$.  Let $\tilde{\Omega}:=\bigcup_{n=1}^{\infty}\{\omega\in\Omega:u_{0,\omega}\in\Sigma_{B_{n}}\}$, then $P(\tilde{\Omega})=1$, and for each $\omega\in\tilde{\Omega}$, we have $u_{0,\omega}\in\Sigma_{B_n}$ for some $n$; global well-posedness of \eqref{Eqn:gKdV-zeromean} with data $u_{0,\omega}$ follows.  The proof of Theorem \ref{Thm:GWP} is complete.

\end{proof}

\section{Proof of nonlinear estimates}
\label{Sec:NLproof-main}

In this Section we prove the crucial nonlinear estimates (Proposition \ref{Prop:nonlin} and Proposition \ref{Prop:nonlin2}).  We follow the definitions and notations of Section \ref{Sec:NLest}.  In Subsections \ref{Sec:NLproof}-\ref{Sec:probhept} we establish Proposition \ref{Prop:nonlin}.  In Section \ref{Sec:NLproof2} we present the proof of Proposition \ref{Prop:nonlin2}.  These proofs will rely on certain lemmata, and the proofs of these lemmata can either be found in Section 5.2.5 of \cite{R-Th}, or elsewhere in the literature, as indicated below.

\subsection{Setup}
\label{Sec:NLproof}

In this section we prove Proposition \ref{Prop:nonlin} using two sets of estimates: quadrilinear probabilistic estimates (see Proposition \ref{Prop:NL-neg1-local} below), and heptilinear probabilistic estimates (see Proposition \ref{Prop:NL-1-nlpart}).  The proof of Proposition \ref{Prop:NL-neg1-local} can be found in Section \ref{Sec:probquad}, and the proof of Proposition \ref{Prop:NL-1-nlpart} is in Section \ref{Sec:probhept}.

We begin by identifying the exact form of the multilinear expressions appearing in Proposition \ref{Prop:nonlin}.  Following the reformulation \eqref{Eqn:gKdV-zeromean3}, we consider the multilinear function
\begin{align}
\mathcal{N}(u_{1},u_{2},u_{3},u_{4}) &:= \mathbb{P}\big[(u_{1})_{x}\mathbb{P}(u_{2}u_{3}u_{4})\big] - \mathbb{P}(u_{2})\int_{\mathbb{T}}(u_{1})_{x}u_{3}u_{4}
- \mathbb{P}(u_{3})\int_{\mathbb{T}}(u_{1})_{x}u_{2}u_{4}
\notag \\
&\ \ \ \ \ \ \ - \mathbb{P}(u_{4})\int_{\mathbb{T}}(u_{1})_{x}u_{2}u_{3}
- \mathbb{P}(u_{3}u_{4})\int_{\mathbb{T}}(u_{1})_{x}u_{2}
- \mathbb{P}(u_{2}u_{4})\int_{\mathbb{T}}(u_{1})_{x}u_{3}
\notag  \\
&\ \ \ \ \ \ \ - \mathbb{P}(u_{2}u_{4})\int_{\mathbb{T}}(u_{1})_{x}u_{4}.
\label{Eqn:nonlin}
\end{align}
We also let $$\mathcal{D}(u_{1},u_{2},u_{3},u_{4}):=\int_{0}^{t}S(t-t')\mathcal{N}(u_{1},u_{2},u_{3},u_{4})(t')dt'.$$
We will use the notation $\mathcal{N}(u):=\mathcal{N}(u,u,u,u)$ and $\mathcal{D}(u):=\mathcal{D}(u,u,u,u)$.
With these definitions (due to the reformulation
\eqref{Eqn:gKdV-zeromean3} of \eqref{Eqn:gKdV-zeromean}) $u$ solves
\eqref{Eqn:gKdV-zeromean} for $t\in[0,T]$ with data $u_{0,\omega}$ given by \eqref{Eqn:initialdata} if and only if
\begin{align}
u = S(t)u_{0,\omega} + \mathcal{D}(u)
\label{Eqn:fixedpoint2}
\end{align}
for all $t\in[0,T]$.  For fixed $n\in \mathbb{Z}\setminus \{0\}$, $t\in\mathbb{R}$, consider the $n^{\text{th}}$ Fourier coefficient of $\mathcal{N}(u_{1},u_{2},u_{3},u_{4})(t)$ (we suppress the dependance on time below)
\begin{align}
\big(\mathcal{N}(u_{1},u_{2},u_{3}&,u_{4})
\big)^{\wedge}(n)
= \bigg(\mathbb{P}\big[(u_{1})_{x}\mathbb{P}(u_{2}u_{3}u_{4})\big]
 - \mathbb{P}(u_{2})\int_{\mathbb{T}}(u_{1})_{x}u_{3}u_{4} \notag \\
& - \mathbb{P}(u_{3})\int_{\mathbb{T}}(u_{1})_{x}u_{2}u_{4}
- \mathbb{P}(u_{4})\int_{\mathbb{T}}(u_{1})_{x}u_{2}u_{3}
- \mathbb{P}(u_{3}u_{4})\int_{\mathbb{T}}(u_{1})_{x}u_{2}
\notag  \\
&
- \mathbb{P}(u_{2}u_{4})\int_{\mathbb{T}}(u_{1})_{x}u_{3}
- \mathbb{P}(u_{2}u_{3})\int_{\mathbb{T}}(u_{1})_{x}u_{4}
\bigg)^{\wedge}(n)
\notag \\
=  \sum_{n=n_{1}+m_{1}} & (in_{1})\widehat{u_{1}}(n_{1})\big(\mathbb{P}(u_{2}u_{3}u_{4})\big)^{\wedge}(m_{1})
-(\mathbb{P}(u_{2}))^{\wedge}(n)\int_{\mathbb{T}}(u_{1})_{x}u_{3}u_{4} \notag
\\
\
-(\mathbb{P}(u_{3}))^{\wedge}&(n)\int_{\mathbb{T}}(u_{1})_{x}u_{2}u_{4}
-(\mathbb{P}(u_{4}))^{\wedge}(n)\int_{\mathbb{T}}(u_{1})_{x}u_{2}u_{3}
-(\mathbb{P}(u_{3}u_{4}))^{\wedge}(n)\int_{\mathbb{T}}(u_{1})_{x}u_{2} \notag \\
&\ \ \ \ \ \ \ \ \ \ -(\mathbb{P}(u_{2}u_{4}))^{\wedge}(n)\int_{\mathbb{T}}(u_{1})_{x}u_{3}
-(\mathbb{P}(u_{2}u_{3}))^{\wedge}(n)\int_{\mathbb{T}}(u_{1})_{x}u_{4}.
\label{Eqn:NLstate-1}
\end{align}
Using  $\int_{\mathbb{T}}w=\hat{w}(0)$, we find
\begin{align}
\eqref{Eqn:NLstate-1}
=
\Bigg(&\sum_{\substack{n=n_{1}+\cdots + n_{4}\\n\neq 0 }} - \sum_{k=1}^{4}\sum_{\substack{n=n_{1}+\cdots + n_{4}\\0\neq n=n_{k} }}
-\sum_{k=2}^{4}\sum_{\substack{n=n_{1}+\cdots + n_{4}\\ n\neq 0, n_{1}=-n_{k} }}\Bigg) n_{1}\widehat{u_{1}}(n_{1})
\widehat{u_{2}}(n_{2})\widehat{u_{3}}(n_{3})\widehat{u_{4}}(n_{4}) \notag
\\
&= \Bigg(\sum_{\zeta_{1}(n)} - \sum_{\zeta_{2}(n)}\Bigg)n_{1}\widehat{u_{1}}(n_{1})
\widehat{u_{2}}(n_{2})\widehat{u_{3}}(n_{3})\widehat{u_{4}}(n_{4}),
\label{Eqn:NLstate-2a}
\end{align}
where
\begin{align*}
\zeta_{1}(n)=\Big\{(n_{1},n_{2},n_{3},n_{4})\in \mathbb{Z}^{4}:
n=&n_{1}+ n_{2} +n_{3} +n_{4}, \ n\neq n_{k}\
\text{for each} \ k\in\{1,2,3,4\},
\\& \text{and}\ n_{1}\neq -n_{j}
\ \text{for each} \ j\in\{2,3,4\} \Big\},
\end{align*}
and
\begin{align*}
\zeta_{2}(n)=\Big\{(n_{1},n_{2},n_{3},n_{4})\in \mathbb{Z}^{4}:
n= & n_{1}+n_{2}+n_{3} +n_{4},
\text{with}\ n_{k}, n_{j}\in\{n,-n_{1}\},
\\
&\
\ \text{for some}\  k\neq j, \ k,j\in\{1,2,3,4\},
\\
&\
\ \text{where} \ n_{k}=n  \ \text{if} \ k=1 \
(\text{and}\ n_{j}=n \ \text{if}\ j=1)  \Big\}.
\end{align*}
\noindent We define $\zeta(n):=\zeta_{1}(n)\cup \zeta_{2}(n)$, and abuse notation by taking
$\displaystyle \sum_{\zeta(n)}:= \sum_{\zeta_{1}(n)}-\sum_{\zeta_{2}(n)}$.
With \eqref{Eqn:NLstate-1} and \eqref{Eqn:NLstate-2a} (reinserting the dependance on time) this gives
\begin{align}
\Big(\mathcal{N}(u_{1},u_{2},u_{3},u_{4})\Big)
^{\wedge}(n,t) = \sum_{\zeta(n)}
(in_{1})\widehat{u_{1}}(n_{1},t)\widehat{u_{2}}(n_{2},t)
\widehat{u_{3}}(n_{3},t)\widehat{u_{4}}(n_{4},t).
\label{Eqn:NLform}
\end{align}

We remark on one more restriction in frequency space.  All of the factors $u_j$ we will consider will be solutions to \eqref{Eqn:gKdV-zeromean} (equivalently \eqref{Eqn:gKdV-zeromean3}) or the truncation of these systems to finite dimensions.  In all cases, the input factors are functions with mean zero for all time.  Thus, we may assume that:
\begin{align} \text{Each} \ n_{k}\neq 0, \ \text{for}\ k=1,2,3,4.
\label{Eqn:meanzero}
\end{align}
To avoid cumbersome notation, we will not carry this restriction with us in notation, but we will rely on this property from time to time.

We will now present the proof of Proposition \ref{Prop:nonlin}.  That is, for $\delta>0$ sufficiently small, any $\delta_0\geq0$ such that $\delta>\delta_0$, and any $0<T\ll 1$, we prove there exists $\varepsilon, \beta, c,C>0$ with $\beta,\varepsilon\ll\delta,\delta_0$ and a measurable set  $\Omega_{T}\subset \Omega$ satisfying $P(\Omega_{T}^{c})<e^{-\frac{c}{T^{\beta}}}$ such that if $\omega \in\Omega_{T}$, then \eqref{Eqn:NL-neg1}-\eqref{Eqn:NL-1b} hold true.  To simplify presentation, we will prove the estimates \eqref{Eqn:NL-neg1}-\eqref{Eqn:NL-1b} with $\delta_0=0$, and provide a discussion (see Remark \ref{Rem:deltazero} below) for generalizing to $0<\delta_0<\delta$.  
The estimates \eqref{Eqn:NL-neg1}-\eqref{Eqn:NL-1b} will follow from standard linear estimates (Lemmas \ref{Lemma:lin1}-\ref{Lemma:gainpowerofT}) and the probabilistic nonlinear estimates given by the following propositions.
\begin{proposition}
For $\delta>0$ sufficiently small, and any $0<T\ll 1$, there exists $\varepsilon, \beta, C, c>0$ and a measurable set $\Omega_{T}\subset \Omega$ satisfying $P(\Omega_{T}^{c})<e^{-\frac{c}{T^{\beta}}}$ and the following conditions: if $\omega \in\Omega_{T}$, then for every quadruple of Fourier multipliers $\Lambda_{1},\ldots,\Lambda_{4}$ defined by
\begin{align*}
\widehat{\Lambda_{i}f}(n) = \chi_{M_{i}\leq |n|\leq K_{i}}\hat{f}(n),
\end{align*}
for some dyadic $M_{i},K_{i}$, we have the following estimate:
\begin{align}
\|\mathcal{N}_{-1}&(\chi_{[0,T]}u_{1},\chi_{[0,T]}u_2,\chi_{[0,T]}u_3,\chi_{[0,T]}u_{4})\|
_{\frac{1}{2}+\delta,-\frac{1}{2}+\delta}
\notag \\
&\leq
CT^{-\beta}\prod_{j=1}^{4}\big(M_{j}^{-\varepsilon} + \|u_{j}\|_{\frac{1}{2}-\delta,\frac{1}{2}-\delta,T} + \|u_{j} - S(t)\Lambda_{j}(u_{0,\omega})\|_{\frac{1}{2}+\delta,\frac{1}{2}-\delta,T}
\big).
\label{Eqn:NL-neg1-nlpart}
\end{align}
\label{Prop:NL-neg1-nlpart}
\end{proposition}

\begin{proposition}
For $\delta>0$ sufficiently small, and any $0<T\ll 1$, there exists $\varepsilon, \beta, C, c>0$ and a measurable set $\Omega_{T}\subset \Omega$ satisfying $P(\Omega_{T}^{c})<e^{-\frac{c}{T^{\beta}}}$ and the following conditions: if $\omega \in\Omega_{T}$, then for every Fourier multipliers $\Lambda_{5}$ defined by
\begin{align*}
\widehat{\Lambda_{5}f}(n) = \chi_{M_{5}\leq |n|\leq K_{5}}\hat{f}(n),
\end{align*}
for some dyadic $M_{5},K_{5}$, we have the following estimates:
\begin{align}
\|\mathcal{N}_{1}(\mathcal{D}&(\chi_{[0,T]}u_{5},\chi_{[0,T]}u_{6},
\chi_{[0,T]}u_{7},\chi_{[0,T]}u_{8}),
\chi_{[0,T]}u_{2},\chi_{[0,T]}u_{3},\chi_{[0,T]}u_{4})
\|_{\frac{1}{2}+\delta,-\frac{1}{2}+\delta,T}
  \notag \\ &\leq CT^{-\beta}
\big(N_{5}^{-\varepsilon} + \|u_{5}\|_{\frac{1}{2}-\delta-\delta_5,\frac{1}{2}-\delta,T} + \|u_{5} - S(t)\Lambda_{5}(u_{0,\omega})\|_{\frac{1}{2}+\delta-\delta_5,\frac{1}{2}-\delta,T}
\big) \notag \\ & \ \ \ \ \ \ \ \ \ \ \ \cdot\prod_{j=2,j\neq 5}^{8}
\|u_{j}\|_{\frac{1}{2}-\delta-\delta_j,\frac{1}{2}-\delta,T},
\label{Eqn:NL-1-nlpart}
\end{align}
\begin{align}
\|\mathcal{N}_{1}(\mathcal{D}_{0}&(\chi_{[0,T]}u_{5},\chi_{[0,T]}u_{6},
\chi_{[0,T]}u_{7},\chi_{[0,T]}u_{8}),
\chi_{[0,T]}u_{2},\chi_{[0,T]}u_{3},\chi_{[0,T]}u_{4})\|
_{\frac{1}{2}+\delta,-\frac{1}{2}+\delta,T}
  \notag \\ &\leq CT^{-\beta}
\big(N_{5}^{-\varepsilon} + \|u_{5}\|_{\frac{1}{2}-\delta-\delta_5,\frac{1}{2}-\delta,T} + \|u_{5} - S(t)\Lambda_{5}(u_{0,\omega})\|_{\frac{1}{2}+\delta-\delta_5,\frac{1}{2}-\delta,T}
\big) \notag \\ & \ \ \ \ \ \ \ \ \ \ \ \cdot\prod_{j=2,j\neq 5}^{8}
\|u_{j}\|_{\frac{1}{2}-\delta-\delta_j,\frac{1}{2}-\delta,T}.
\label{Eqn:NL-1b-nlpart}
\end{align}
\label{Prop:NL-1-nlpart}
\end{proposition}


\begin{remark}
Notice that in Proposition \ref{Prop:NL-neg1-nlpart} and  \ref{Prop:NL-1-nlpart} we have taken $\delta_0=0$ (compared with Proposition \ref{Prop:nonlin}).  It is not hard to prove
that, for the set $\Omega_T$ produced by these theorems, if $\omega\in\Omega_{T}$, then the inequalities \eqref{Eqn:NL-neg1-nlpart}
and \eqref{Eqn:NL-1-nlpart}-\eqref{Eqn:NL-1b-nlpart} hold for any fixed $0<\delta_0<\delta$.  This is because, as can be observed a posteriori, the proofs of Proposition \ref{Prop:NL-neg1-nlpart} (or rather, of Proposition \ref{Prop:NL-neg1-local} found below) and Proposition \ref{Prop:NL-1-nlpart} will be flexible with respect to this particular manipulation.

If we wish to prove the statement analogous to \eqref{Eqn:NL-neg1-nlpart} with $0<\delta_0<\delta$ (see the statement of Proposition \ref{Prop:nonlin}), then on the left-hand side of the inequality, we will have lowered the spatial Sobolev regularity to $s=\frac{1}{2}+\delta-\delta_0$ from $s=\frac{1}{2}+\delta$.  This amounts to having the factor $|n|^{\frac{1}{2}+\delta-\delta_0}$ in the nonlinear estimates below, instead of $|n|^{\frac{1}{2}+\delta}$ (e.g. in the lines \eqref{Eqn:NL-neg1-26}).

In every case of each nonlinear estimate we establish below (excluding Case 1.b. in the proof of Proposition \ref{Prop:NL-neg1-local}, which we discuss in the next paragraph), we control the factor $|n|^{\frac{1}{2}+\delta}$ using the estimate $|n|\leq N^0$.  That is, we control this factor using terms in the denominator that are known to be of the size $N^0$ (see, for example, \eqref{Eqn:NL-neg1-38}).  This means that we can replace $|n|^{\frac{1}{2}+\delta}$ with $|n|^{\frac{1}{2}+\delta-\delta_0}|n_k|^{\delta_0}$, for any $k=1,2,3,4$ (or $k=2,3,4,5,6,7,8$ for the heptilinear estimate \eqref{Eqn:NL-1-nlpart}).  This allows us to lower the spatial Sobolev regularity of any one of the factors on the right-hand side by the same amount $\delta_0>0$.  That is, we can establish the nonlinear estimate with $0<\delta_0<\delta$ as stated in Proposition \ref{Prop:nonlin}.

We should comment that, in Case 1.b. during the proof of Proposition \ref{Prop:NL-neg1-local}, we did not use the estimate $|n|\leq N^0$.  However, it is easily verified that we can still lower the spatial regularity of any one of the factors on the right-hand side by a small amount $\delta_0>0$ (the estimates in this case have some room to spare).
\label{Rem:deltazero}
\end{remark}

\begin{remark}
There is another flexibility implicit to the nonlinear estimates of Proposition \ref{Prop:NL-neg1-nlpart} and Proposition \ref{Prop:NL-1-nlpart} (and thus Proposition \ref{Prop:nonlin}), which was, in fact, already used in the globalizing estimates of Proposition \ref{Prop:global} (in Section \ref{Sec:GWP}).  Specifically, the time interval $[0,T]$ can be replaced with an interval $I$ of length $T$.  Furthermore, we do not need the randomized data $S(t)u_{0,\omega}$ to evolve from time $t=0$.  In particular, we can prove Proposition \ref{Prop:nonlin} replacing $S(t)u_{0,\omega}$ with $S(t+t_0)u_{0,\omega}$, for any $t_0\in\mathbb{R}$, as the linear evolution of gKdV preserves the Gaussian probability densities of the (independent) randomized Fourier coefficients in \eqref{Eqn:initialdata}.  However, by varying $t_0$ the probabilistic set $\Omega_{T}=\Omega_{T}(t_0)$ varies as well.  We can use this flexibility (varying $t_0\in\mathbb{R}$), but the measurable set of good data produced by Proposition \ref{Prop:nonlin} changes.

We will stick to the following notation: $\Omega_{T}(t_0)$ is the set satisfying the conclusions of Proposition \ref{Prop:nonlin} on the time interval $[t_0,t_0+T]$ (instead of $[0,T]$) with initial data $u_{0,\omega}$ evolving from time $t=0$.
\label{Rem:time}
\end{remark}

\noindent Before we prove Proposition \ref{Prop:NL-neg1-nlpart} and Proposition \ref{Prop:NL-1-nlpart}, let us use them to establish Proposition \ref{Prop:nonlin}.

\begin{proof}[Proof of Proposition \ref{Prop:nonlin}:]
Apply Proposition \ref{Prop:NL-1-nlpart}, and suppose $\omega\in\Omega_{T}$ so that the estimate \eqref{Eqn:NL-neg1-nlpart} holds true.  Note that by the equivalence
$$\chi_{[0,T]}\mathcal{D}(u_{1},\ldots,u_{4})
=\chi_{[0,T]}\mathcal{D}(\chi_{[0,T]}u_{1},\ldots,
\chi_{[0,T]}u_{4}),$$ we have
$$ \|\mathcal{D}_{-1}(u_{1},u_{2},u_{3},u_{4})\|_
{\frac{1}{2}+\delta,\frac{1}{2}+\delta,T}\|\leq
\|\mathcal{D}_{-1}(\chi_{[0,T]}u_{1},\chi_{[0,T]}u_{2},\chi_{[0,T]}u_{3},
\chi_{[0,T]}u_{4})\|_{\frac{1}{2}+\delta,\frac{1}{2}+\delta,T}.$$
Applying Lemma \ref{Lemma:gainpowerofT}, Lemma \ref{Lemma:lin2}, and \eqref{Eqn:NL-neg1-nlpart}, we find
\begin{align*}
\|&\mathcal{D}_{-1}(\chi_{[0,T]}u_{1},\ldots,
\chi_{[0,T]}u_{4})\|_
{\frac{1}{2}+\delta,\frac{1}{2}+\delta,T}
\\
&\lesssim \|\mathcal{D}_{-1}(\chi_{[0,T]}u_{1},\ldots,
\chi_{[0,T]}u_{4})\|_
{\frac{1}{2}+\delta,\frac{1}{2}+\delta,T}
\\
&\lesssim \|\mathcal{N}_{-1}(\chi_{[0,T]}u_{1},\ldots,
\chi_{[0,T]}u_{4})\|_
{\frac{1}{2}+\delta,-\frac{1}{2}+\delta}
\\
&\lesssim
T^{-\beta}\prod_{j=1}^{4}\big(M_{j}^{-\varepsilon} + \|u_{j}\|_{\frac{1}{2}-\delta,\frac{1}{2}-\delta,T} + \|u_{j} - S(t)\Lambda_{j}(u_{0,\omega})\|_{\frac{1}{2}+\delta,\frac{1}{2}-\delta,T}
\big).
\end{align*}
The proof of \eqref{Eqn:NL-neg1} is complete.  The justification of \eqref{Eqn:NL-1}-\eqref{Eqn:NL-1b} follows from \eqref{Eqn:NL-1-nlpart}-\eqref{Eqn:NL-1b-nlpart} using the same type of argument.  This completes the proof of Proposition \ref{Prop:nonlin}.
\end{proof}

For the proof of Proposition \ref{Prop:NL-neg1-nlpart}, we will use a dyadically localized estimate (this will not be necessary for the proof of Proposition \ref{Prop:NL-1-nlpart}).  That is, we will establish probabilistic quadrilinear estimates which are independent of the Fourier multipliers $\Lambda_{1},\ldots,\Lambda_{4}$ appearing in the statement of Proposition \ref{Prop:NL-neg1-nlpart}.  In the following, subscripts with capital letters denote dyadic localization; i.e. $u_{N_j}=(\chi_{|n_j|\sim N_{j}}\widehat{u_{j}})^{\vee}$ for $N_{j}$ dyadic.  Let
\begin{align*}
f_{0,j}&:= \|u_{N_j}\|
_{\frac{1}{2}+\delta,\frac{1}{2}-\delta,T}, \\
f_{1,j}&:= N_{j}^{-\varepsilon}
+ \|u_{N_j}\|
_{\frac{1}{2}-\delta,\frac{1}{2}-\delta,T}
+ \|u_{N_j}-(S(t)u_{0,\omega})_{N_j}\|
_{\frac{1}{2}+\delta,\frac{1}{2}-\delta,T}.
\end{align*}
Here is the dyadically localized probabilistic quadrilinear estimate.
\begin{proposition}
For $\delta>0$ sufficiently small, and any $0<T\ll 1$, there exists $\alpha, \beta, \kappa, C, c>0$ with $\beta,\alpha,\kappa \ll \delta$ such that for every quintuple of dyadic frequencies $N,N_{1},\ldots,N_{4}$,
$\exists\, \Omega_{N,N_{1},\ldots,N_{4},T}\subset\Omega$
with $P(\Omega_{N,N_{1},\ldots,N_{4},T}^{c})<
\frac{1}{(NN_{1}\cdots N_{4})^{\kappa}}e^{-\frac{\tilde{c}}{T^{\beta}}}$
such that for all $\omega\in\widetilde{\Omega}_{T}\cap
\Omega_{N,N_{1},\ldots,N_{4},T}$ we have
\begin{align}
\|\mathcal{N}_{-1}|_{|n|\sim N}(u_{N_1}&,u_{N_2},u_{N_3},u_{N_4})\|
_{\frac{1}{2}+\delta,-\frac{1}{2}+\delta}
\notag \\
&\leq
\frac{CT^{-\beta}}{(NN_{1}\cdots N_{4})^{\alpha}}
\prod_{j=1}^{4}\min(f_{0,j},f_{1,j}),
\label{Eqn:NL-neg1-local}
\end{align}
where $\widetilde{\Omega}_{T}$ is the set obtained from Lemma \ref{Lemma:prob1}.
\label{Prop:NL-neg1-local}
\end{proposition}

\noindent We proceed to prove Proposition \ref{Prop:NL-neg1-nlpart} using Proposition \ref{Prop:NL-neg1-local}.  Then we present the proof of Proposition \ref{Prop:NL-neg1-local}, followed by the proof of Proposition \ref{Prop:NL-1-nlpart}.

\begin{proof}[Proof of Proposition \ref{Prop:NL-neg1-nlpart}:]
Fix any dyadic $N_{j}$ for $j\in\{1,2,3,4\}$.  Observe that
\begin{align}
\min(f_{0,j},f_{1,j}) \leq
M_{j}^{-\varepsilon}+ \|u_{j}\|
_{\frac{1}{2}-\delta,\frac{1}{2}-\delta,T}
+ \|\widehat{u_{j}}-S(t)\Lambda_{j}(u_{0,\omega})\|
_{\frac{1}{2}+\delta,\frac{1}{2}-\delta,T}.
\label{Eqn:NL-neg1-3}
\end{align}
Indeed, suppose
$N_{j}\in [M_{j},K_{j}]=\text{supp}(\Lambda_{j})$, then we have
\begin{align}
f_{1,j}&=N_{j}^{-\varepsilon}
+ \|u_{N_j}\|_{\frac{1}{2}-\delta,\frac{1}{2}-\delta,T}
+ \|u_{N_j}-(S(t)u_{0,\omega})_{N_j}\|
_{\frac{1}{2}+\delta,\frac{1}{2}-\delta,T} \notag \\
&\leq M_{j}^{-\varepsilon}+ \|u_{j}\|
_{\frac{1}{2}-\delta,\frac{1}{2}-\delta,T}
+ \|u_{j}-S(t)\Lambda_{j}(u_{0,\omega})\|
_{\frac{1}{2}+\delta,\frac{1}{2}-\delta,T}.
\label{Eqn:NL-neg1-1}
\end{align}
On the other hand, if $N_{j}\not\in[M_{j},K_{j}]$, we have
\begin{align}
f_{0,j}&=\|u_{N_j}\|
_{\frac{1}{2}+\delta,\frac{1}{2}-\delta,T}
 \notag \\
&\leq \|u_{N_j}-(S(t)\Lambda_{j}(u_{0,\omega}))_{N_j}\|
_{\frac{1}{2}+\delta,\frac{1}{2}-\delta,T}
 \notag \\
&\leq
M_{j}^{-\varepsilon}+ \|u_{j}\|
_{\frac{1}{2}-\delta,\frac{1}{2}-\delta,T}
+ \|u_{j}-S(t)\Lambda_{j}(u_{0,\omega})\|
_{\frac{1}{2}+\delta,\frac{1}{2}-\delta,T}.
\label{Eqn:NL-neg1-2}
\end{align}
Combining \eqref{Eqn:NL-neg1-1} and \eqref{Eqn:NL-neg1-2} we conclude that for each $j=1,2,3,4$ and every dyadic $N_{j}$, the inequality \eqref{Eqn:NL-neg1-3} holds true.

We proceed to build a set $\Omega_{T}\subset\Omega$ (satisfying the necessary conditions) where the estimate \eqref{Eqn:NL-neg1-nlpart} is satisfied.
Consider a dyadic decomposition of the nonlinearity,
\begin{align}
\|\mathcal{N}_{-1}(u_{1},u_{2},&u_{3},u_{4})\|_
{\frac{1}{2}+\delta,-\frac{1}{2}+\delta}
&\leq
\sum_{N,N_{1},\ldots,N_{4}}
\|\mathcal{N}_{-1}|_{|n|\sim N}(u_{N_1},u_{N_2},&u_{N_3},u_{N_4})\|_
{\frac{1}{2}+\delta,-\frac{1}{2}+\delta}.
\label{Eqn:NL-neg1-5}
\end{align}
Now let
$\displaystyle\Omega_{T}:=
\widetilde{\Omega}_{T}\cap_{\text{dyadic }N,N_{1},\ldots,N_{4}}
\Omega_{N,N_{1},\ldots,N_{4},T}$.  Then $$P(\Omega_{T}^{c})\leq \sum_{N,N_{1},\ldots,N_{4}}
P(\Omega_{N,N_{1},\ldots,N_{4},T}^{c})
<\sum_{N,N_{1},\ldots,N_{4}}\frac{1}{(NN_{1}\cdots N_{4})^{\kappa}}e^{-\frac{\tilde{c}}{T^{\beta}}}
\leq e^{-\frac{c}{T^{\beta}}},$$
where $c=c(\tilde{c},\kappa)>0$.
Furthermore, if $\omega \in \Omega_{T}$, then for every combination of dyadic scales $N,N_{1},\ldots,N_{4}$, the conclusion \eqref{Eqn:NL-neg1-local} holds true.  With \eqref{Eqn:NL-neg1-3}, this gives
\begin{align}
\eqref{Eqn:NL-neg1-5}
&\lesssim
\sum_{N,N_{1},\ldots,N_{4}}\frac{T^{-\beta}}
{(NN_{1}\cdots N_{4})^{\alpha}}
\prod_{j=1}^{4}\min(f_{0,j}(N_{j}),f_{1,j}(N_{j}))
\notag \\
&\leq
\sum_{N,N_{1},\ldots,N_{4}}\frac{T^{-\beta}}
{(NN_{1}\cdots N_{4})^{\alpha}}
\notag \\
&\ \ \ \ \ \ \ \ \ \ \ \ \ \ \ \prod_{j=1}^{4}
\Big(M_{j}^{-\varepsilon} + \|u_{j}\|
_{\frac{1}{2}-\delta,\frac{1}{2}-\delta,T}
+ \|u_{j}-S(t)\Lambda_{j}(u_{0,\omega})\|
_{\frac{1}{2}+\delta,\frac{1}{2}-\delta,T}
\Big)
\notag \\
&\lesssim
T^{-\beta}\prod_{j=1}^{4}
\Big(M_{j}^{-\varepsilon}+ \|u_{j}\|
_{\frac{1}{2}-\delta,\frac{1}{2}-\delta,T}
+ \|u_{j}-S(t)\Lambda_{j}(u_{0,\omega})\|
_{\frac{1}{2}+\delta,\frac{1}{2}-\delta,T}
\Big).
\label{Eqn:NL-neg1-6b}
\end{align}

\end{proof}

\noindent Next we present with the proof of Proposition \ref{Prop:NL-neg1-local}, followed by the proof of Proposition \ref{Prop:NL-1-nlpart}.

\subsection{Probabilistic quadrilinear estimates}
\label{Sec:probquad}

We begin by presenting some probabilistic lemmata to be used in the proof of Proposition \ref{Prop:NL-neg1-local}.  In each lemma, we are considering the probability space $(\Omega,\mathcal{F},P)$ with $P = \rho\circ u_{0,\omega}$, where $\rho$ is the Wiener measure defined in \eqref{Eqn:wiener}, and the initial data (given by \eqref{Eqn:initialdata}) is viewed as a map $u_{0,\omega}:\Omega\rightarrow H^{1/2-}(\mathbb{T})$.

\begin{lemma}
Let $\varepsilon,\beta>0$ and $T\ll 1$.  Then there exists $\widetilde{\Omega}_{T}\subset \Omega$ with
$P(\widetilde{\Omega}_{T}^{c})<e^{-\frac{1}{T^{\beta}}}$, such that for $\omega\in\widetilde{\Omega}_{T}$, we have
\begin{align*}
|g_{n}(\omega)| \leq CT^{-\frac{\beta}{2}}\langle n\rangle^{\varepsilon}
\end{align*}
for all $n\in\mathbb{Z}$.
\label{Lemma:prob1}
\end{lemma}

\begin{proof}[Proof of Lemma \ref{Lemma:prob1}]
Recall from \cite{O1} that
$$ P(\sup_{n\in\mathbb{Z}\setminus\{0\}}\langle n\rangle^{-\varepsilon}|g_{n}(\omega)|>K)\leq e^{-cK^2}$$ for $K$ sufficiently large.  Lemma \ref{Lemma:prob1} follows by taking $K\sim T^{-\frac{\beta}{2}}$.
\end{proof}

\begin{lemma}[Thomann-Tzvetkov,\cite{TT}]
Let $d\geq 1$ and $c(n_{1},\ldots,n_{k})\in \mathbb{C}$.  Let $\{\gamma_{n}(\omega)\}_{1\leq n\leq d}$ be a sequence of $\mathbb{R}$-valued standard Gaussian random variables.  For $k\geq 1$, denote by
\newline $A(k,d)=\{(n_{1},\ldots,n_{k})\in\{1,\ldots,d\}^{k}:
n_{1}\leq \cdots \leq n_{k}\}$, and
\begin{align*}
S_{k}(\omega) = \sum_{A(k,d)}c(n_{1},\ldots,n_{k})\gamma_{n_{1}}(\omega)\cdots \gamma_{n_{k}}(\omega).
\end{align*}
Then, for each $p\geq 1$, we have
\begin{align*}
\|S_{k}\|_{L^{p}(\Omega)}\leq \sqrt{k+1}(p-1)^{\frac{k}{2}}\|S_{k}\|_{L^{2}(\Omega)}.
\end{align*}
\label{Lemma:prob2}
\end{lemma}

The proof of Lemma \ref{Lemma:prob2} can be found in \cite{TT}; it relies on hypercontractivity of the Ornstein-Uhlenbeck semigroup.

\begin{lemma}[Tzvetkov, \cite{Tz3}]
Let $F:H^{\frac{1}{2}-}(\mathbb{T})\rightarrow \mathbb{R}$
be measurable.  Assume there exists $\alpha>0, \tilde{N}>0, k\geq 1$ and $C>0$ such that for all $p\geq 2$,
\begin{align}
\|F\|_{L^{p}(d\rho)}\leq C\tilde{N}^{-\alpha}p^{\frac{k}{2}}.
\label{Eqn:L3-1}
\end{align}
Then there exists $\delta>0$, $c_{1}$ independent of $N$ and $\alpha$ such that
\begin{align*}
\int_{H^{\frac{1}{2}-}(\mathbb{T})}e^{\delta \tilde{N}^{\frac{2\alpha}{k}}|F(u)|^{\frac{2}{k}}}
d\rho(u) \leq c_{1}.
\end{align*}
As a consequence, for all $\lambda>0$,
\begin{align}
P(\omega\in\Omega:|F(u_{0,\omega})|>\lambda)
\leq c_{1}e^{-\delta\tilde{N}^{\frac{2\alpha}{k}}
\lambda^{\frac{2}{k}}}
\label{Eqn:L3-2}
\end{align}
\label{Lemma:prob3}
\end{lemma}
The proof of Lemma \ref{Lemma:prob3} can be found in \cite{Tz3}.

We will also need the following basic observation from linear algebra.
\begin{lemma}
Let $A=\{a_{i,j}\}_{1\leq i,j \leq N}$ be a square ($N\times N$) matrix with complex entries.  Then
$$ \|A\|\leq \sup_{1\leq n \leq N}|a_{n,n}| + \Big(\sum_{n\neq n'}|a_{n,n'}|^{2}\Big)^{\frac{1}{2}}.$$
\label{Lemma:matrixnorm}
\end{lemma}

\noindent The proof of Lemma \ref{Lemma:matrixnorm} is omitted (the analysis required is straightforward).  We proceed to prove Proposition \ref{Prop:NL-neg1-local}.

\begin{proof}[Proof of Proposition \ref{Prop:NL-neg1-local}:]

For the remainder of this proof, all factors $u_{N_{j}}$ are dyadically localized, and we simplify notation by taking $u_{j}=u_{N_{j}}$.
Also, we typically drop the $\chi_{[0,T]}$ from in front of each factor $u_j$, but may reintroduce them as needed.  This proof is based on multiple decompositions of frequency space.  In some regions, we will use what will be referred to as a type (I) - type (II) decomposition, which leads to additional subcases.
More precisely, in each region of frequency space we impose one of the following two conditions: for each $j\in\{1,2,3,4\}$, either

\vspace{0.1in}

(i)  $u_{j}\in X^{\frac{1}{2}-\delta,\frac{1}{2}-\delta}_{T}$,

\vspace{0.1in}

\noindent or

\vspace{0.1in}

(ii) $\displaystyle u_{j} - \gamma_{j}
\sum_{|n_{j}|\sim N_{j}}\frac{g_{n_{j}}(\omega)}
{|n_{j}|}e^{in_{j}x+in_{j}^{3}t}
\in X^{\frac{1}{2}+\delta,\frac{1}{2}-\delta}_{T}$, for each $\gamma_{j} \in\{0,1\}$.  When $\gamma_{j}=1$, this is a
nonlinear smoothing hypothesis.

\vspace{0.1in}

\noindent The additional parameters $\gamma_{j}\in\{0,1\}$ are introduced in order to establish a single result for variable $\gamma_{j}$, which produces factors of $f_{0,j}$ with $\gamma_{j}=0$ and factors of $f_{1,j}$ with $\gamma_{j}=1$.  That is, by keeping each $\gamma_{j}$ variable, we will produce the right hand side of \eqref{Eqn:NL-neg1-local}.

Contributions to the left-hand side of \eqref{Eqn:NL-neg1-local} from a region where $u_{j}$ satisfies condition (i) produce a corresponding factor of $\|u_{j}\|_{\frac{1}{2}-\delta,\frac{1}{2}-\delta,T}$ on the right-hand side of the inequality.  For contributions from regions where $u_{j}$ satisfies condition (ii), we establish probabilistic bounds, by writing
$$ u_{j} = \underbrace{\gamma_{j}
\sum_{|n_{j}|\sim N_{j}}\frac{g_{n_{j}}(\omega)}
{|n_{j}|}e^{in_{j}x+in_{j}^{3}t}}
_{\text{type (I)}} + \underbrace{u_{j} - \gamma_{j}
\sum_{|n_{j}|\sim N_{j}}\frac{g_{n_{j}}(\omega)}
{|n_{j}|}e^{in_{j}x+in_{j}^{3}t}}
_{\text{type (II)}}.$$
We show that each type (I) contribution produces a factor of $\gamma_{j}N_{j}^{-\varepsilon}$ on the right-hand side of the inequality, for $\omega\in\Omega_{N,\ldots,N_{4},T}$.
The type (II) contribution will produce a factor of
$\|u_{j}- \gamma_{j}
\sum_{|n_{j}|\sim N_{j}}\frac{g_{n_{j}}(\omega)}
{|n_{j}|}e^{in_{j}x+in_{j}^{3}t}\|_{\frac{1}{2}+\delta,\frac{1}{2}-\delta,T}$
on the right-hand side.    Combining the contributions from (i) and (ii), each $u_{j}$ will produce a factor of
\begin{align}\gamma_{j}N_{j}^{-\varepsilon} +
\|u_{j}\|_{\frac{1}{2}-\delta,\frac{1}{2}-\delta,T}
+ \|u_{j}- \gamma_{j}
\sum_{|n_{j}|\sim N_{j}}\frac{g_{n_{j}}(\omega)}
{|n_{j}|}e^{in_{j}x+in_{j}^{3}t}\|_{\frac{1}{2}+\delta,\frac{1}{2}-\delta,T}.
\label{Eqn:NL-neg1-23}
\end{align}
Notice that \eqref{Eqn:NL-neg1-23}$\leq f_{0,j}$ for $\gamma_{j}=0$, and \eqref{Eqn:NL-neg1-23}$= f_{1,j}$ for $\gamma_{j}=1$.  By establishing these estimates for all combinations of $\gamma_{j}\in\{0,1\}$, $j=1,2,3,4$, we can always choose the smaller of the two bounds, and each $u_{j}$ contributes a factor of $\min(f_{0,j},f_{1,j})$ to the right-hand side of our inequality.

\medskip

Summarizing the previous paragraphs, we prove Proposition \ref{Prop:NL-neg1-local} by constructing $\Omega_{N,\ldots,N_{4},T}$ $\subset \Omega_{T}$ with $P(\Omega_{N,N_{1},\ldots,N_{4},T}^{c})<
\frac{1}{(NN_{1}\cdots N_{4})^{\kappa}}e^{-\frac{\tilde{c}}{T^{\beta}}}$
such that for all $\omega\in\widetilde{\Omega}_{T}\cap
\Omega_{N,N_{1},\ldots,N_{4},T}$ we can, throughout frequency space, either bound each $u_{j}$ deterministically, using condition (i), or probabilistically, using condition (ii) and Lemmas \ref{Lemma:prob1} - \ref{Lemma:prob3} (the type (I)-type (II) decomposition).

In the break down of cases that follows, as we estimate the left-hand side of \eqref{Eqn:NL-neg1-local} using the method just described, each factor $u_{j}$ may be declared to be of the following types
\begin{itemize}
\item type (I) (rough but random): $u_j= \sum_{|n_{j}|\sim N_{j}}\frac{g_{n_{j}}(\omega)}{|n_{j}|}e^{in_{j}x+in_{j}^{3}t}$,

\item type (II) (smooth and deterministic): $u_{j}\in X^{\frac{1}{2}+\delta,\frac{1}{2}-\delta}_{T}$.
\end{itemize}
In a given case, if $u_j$ is declared to be of type (I) or type (II), this means that we are choosing to use condition (ii) in this factor, and according to the decomposition above, we must consider each case of $u_j$ type (I) and $u_j$ type (II).  If we make no declaration about a particular factor $u_{j}$ in a given case, it means that we are imposing condition (i) in that factor: $u_{j}\in X^{\frac{1}{2}-\delta,\frac{1}{2}-\delta}_{T}$.

We will use superscripts $n^{k}$ ($N^{k}$) $k=0,1,\ldots,4$ to indicate frequencies (and corresponding dyadic blocks) which have been ordered from largest to smallest.  That is, $|n^{0}|\geq|n^{1}|\geq \cdots \geq |n^{4}|$ (and $N^{0}\geq N^{1}\geq \cdots \geq N^{4}$).  Let us remark that we order the frequency $n$ as $-n$ (that is, if $n$ is the frequency of largest magnitude, then $n^{0}=-n$).  Also, by symmetry of $\mathcal{N}(u_{1},u_{2},u_{3},u_{4})$ in $(u_{2},u_{3},u_{4})$, we can assume that $|n_{2}|\geq |n_{3}| \geq |n_{4}|$.  Let us begin with an overview of each case to be considered.

\vspace{0.1in}

\noindent $\bullet$ \textbf{CASE 1.}
$n^{0}=-n^{1}$.




\vspace{0.1in}

\noindent $\bullet$ \textbf{CASE 2.}
$n^{0}\neq -n^{1}$.

\vspace{0.1in}

\hspace{0.4in} $\bullet$ \textbf{CASE 2.a.}
$N^{3}\ll N^{0}$ and $N^{2}N^{3}N^{4}\ll
N^{0}N^{1}|n^{0}+n^{1}|$.

\hspace{0.4in} $\bullet$ \textbf{CASE 2.b.}
$N^{3}\sim N^{0}$.

\hspace{0.5in} We use a type (I) - type (II) decomposition for $k=1,2,3$.

\vspace{0.1in}

\hspace{0.8in} $\bullet$ \textbf{CASE 2.b.i.} At least two $u_{i}$ of type (I), $i=1,2,3$.

\hspace{0.9in}That is, $u_1$,$u_2$, $u_3$ of types (I)(I)(I), (I)(I)(II),(I)(II)(I) and (II)(I)(I).

\hspace{0.8in} $\bullet$ \textbf{CASE 2.b.ii.} One of $u_{i}$ of type (I), $i=1,2,3$, others type (II).

\hspace{0.9in} Types (I)(II)(II), (II)(I)(II) and (II)(II)(I).

\hspace{0.8in} $\bullet$ \textbf{CASE 2.b.iii.}
$u_{1}$,$u_{2}$,$u_{3}$ all type (II).

\vspace{0.1in}

\hspace{0.4in} $\bullet$ \textbf{CASE 2.c.}
$N^{3}\ll N^{0}$ and $N^{2}N^{3}N^{4}\gtrsim
N^{0}N^{1}|n^{0}+n^{1}|$.

\hspace{0.5in} We use a type (I) - type (II) decomposition for $k=1,2,3,4$.

\vspace{0.1in}

\hspace{0.8in} $\bullet$ \textbf{CASE 2.c.i.} $u_{1}$ type (I) and at least two of $u_{2},u_{3},u_{4}$ type (I).

\hspace{0.9in} That is, $u_1,u_{2},u_{3},u_{4}$ of types (I)(I)(I)(I), (I)(I)(I)(II),

\hspace{0.9in} (I)(I)(II)(I) and (I)(II)(I)(I).

\hspace{0.8in} $\bullet$ \textbf{CASE 2.c.ii.}
$u_{1}$ type (II) and $u_{2},u_{3},u_{4}$ type (I).

\hspace{0.8in} $\bullet$ \textbf{CASE 2.c.iii.}
Two of $u_{1},u_{2},u_{3},u_{4}$ type (I) and two type (II).

\hspace{0.9in} Types (I)(I)(II)(II), (I)(II)(I)(II), (I)(II)(II)(I), (II)(I)(II)(I),

\hspace{0.9in}(II)(II)(I)(I) and (II)(I)(I)(II).

\hspace{0.8in} $\bullet$ \textbf{CASE 2.c.iv.}
At least three of $u_{1},u_{2},u_{3},u_{4}$ type (II).

\hspace{0.9in} Types (II)(II)(II)(I), (II)(II)(I)(II), (II)(I)(II)(II),

\hspace{0.9in}
(I)(II)(II)(II), and (II)(II)(II)(II).

\vspace{0.1in}

Before we proceed with the analysis of each case, let us remark on an important property of our frequency space restrictions:
\begin{align}
\text{If }\
 (n,n_1,n_2,&n_3,n_4,\tau,\tau_1,\tau_2,\tau_3,\tau_4)\in A_{-1},\ \
 \notag \\ &\text{then }\
(n_{1},n_{2},n_{3},n_{4})\in \zeta_{1}(n).
\label{Eqn:NL-neg1-33}
\end{align}
For the estimate \eqref{Eqn:NL-neg1-local} we are restricted to the region $A_{-1}$ of frequency space (defined in \eqref{Eqn:NL-part}), and the condition $|n^{3}-n_{1}^{3}-\cdots -n_{4}^{3}|\ll |n_{\text{max}}|^{2}$ is satisfied.  To establish \eqref{Eqn:NL-neg1-33} we show that this condition necessitates   $(n_{1},n_{2},n_{3},n_{4})\in \zeta_{1}(n)$ (see \eqref{Eqn:NLstate-2a} for the definition of $\zeta_{1}(n)$).  In fact, we show the contrapositive; that $(n_{1},n_{2},n_{3},n_{4})\not\in \zeta_{1}(n)$ implies $|n^{3}-n_{1}^{3}-\cdots -n_{4}^{3}|\gtrsim |n_{\text{max}}|^{2}$.
Recall from \eqref{Eqn:NLstate-3} that in the domain of integration we have
$(n_{1},n_{2},n_{3},n_{4})\in \zeta(n)=\zeta_{1}(n)\cup \zeta_{2}(n)$, and $(n_{1},n_{2},n_{3},n_{4})\not\in \zeta_{1}(n)$ is therefore equivalent to $(n_{1},n_{2},n_{3},n_{4})\in \zeta_{2}(n)$.  We show that if $(n_{1},n_{2},n_{3},n_{4})\in \zeta_{2}(n)$, then $|n^{3}-n_{1}^{3}-\cdots -n_{4}^{3}|\gtrsim |n_{\text{max}}|^{2}$.  Suppose $(n_{1},n_{2},n_{3},n_{4}) \in \zeta_{2}(n)$, then there are six possibilities (up to permutations of $(n_{2},n_{3},n_{4})$):
\begin{enumerate}[(i)]
\item $n=n_{1}=n_{2}$
\item $n=n_{2}=n_{3}$
\item $n=n_{1}=-n_{2}$
\item $n_{1}=-n_{2}=-n_{3}$
\item $n=-n_{1}=n_{2}$
\item $n=n_{2}, n_{1}=-n_{3}$
\end{enumerate}
We proceed to show $|n^{3}-n_{1}^{3}-\cdots -n_{4}^{3}|\gtrsim |n_{\text{max}}|^{2}$ in each circumstance.  Suppose possibility (i) holds, and we have $n=n_{1}=n_{2}$. Then $n=n_{1}+\cdots+n_{4}$ gives $n_{2}+n_{3}+n_{4}=0$, and we find
\begin{align}
n^{3}-n_{1}^{3}-\cdots - n_{4}^{3}
&= -n_{2}^{3} -n_{3}^{3}-n_{4}^{3} \notag \\
&= -3n_{2}n_{3}n_{4}.
\label{Eqn:NLproof-1}
\end{align}
Recall that each $n_{i}\neq 0$ by the mean zero condition \eqref{Eqn:meanzero}.  If $|n_{3}|\sim |n_{4}|\sim|n_{\text{max}}|$, then by \eqref{Eqn:NLproof-1}, and the mean zero condition, we have $|n^{3}-n_{1}^{3}-\cdots -n_{4}^{3}|\gtrsim
|n_{\text{max}}|^{2}$, which is impossible in the region $A_{-1}$.  Therefore, we must have $|n|=|n_{1}|=|n_{2}|=|n_{\text{max}}|$.  Then $n_{2}+n_{3}+n_{4}=0$ gives (without loss of generality) that $|n_{3}|\sim|n_{2}|=|n_{\text{max}}|$, and again we arrive at $|n^{3}-n_{1}^{3}-\cdots -n_{4}^{3}|\gtrsim
|n_{\text{max}}|^{2}$.  We conclude that possibility (i) cannot occur in the region $A_{-1}$.  It is straightforward to verify that the same argument rules out (ii)-(v); only (vi) remains to be considered.  Suppose (vi) holds, and we have $n=n_{2}$, $n_{1}=-n_{3}$.  Then $n_{4}=n-n_{1}-n_{2}-n_{3}=0$, which is impossible by the mean zero condition \eqref{Eqn:meanzero}.  Therefore, in the region $A_{-1}$, we cannot have $(n_{1},n_{2},n_{3},n_{4})\in \zeta_{2}(n)$, and we conclude that \eqref{Eqn:NL-neg1-33} holds true.

\vspace{0.1in}

We now proceed with the analysis of each case listed above.

\vspace{0.1in}

\noindent $\bullet$ \textbf{CASE 1.}
$n^{0}=-n^{1}$.

\vspace{0.1in}

With $(n_{1},n_{2},n_{3},n_{4})\in \zeta_{1}(n)$ , we have $n\neq n_{i}$ for all $i\in\{1,2,3,4\}$, and $n_{1}\neq -n_{k}$ for all $k\in \{2,3,4\}$.
Therefore, if $n^{0}=-n^{1}$, we must have $n^{0}=n_{k}=-n_{j}=-n^{1}$ for some $k,j\in\{2,3,4\}$.  By the condition $|n_{2}|\geq |n_{3}| \geq |n_{4}|$ it follows that $n^{0}=n_{2}=-n_{3}=-n^{1}$.  With $n_2=-n_3$, we have $n=n_1+n_4$ and
$$ \max(|\sigma|,|\sigma_1|,|\sigma_2|,|\sigma_3|,|\sigma_4|)\gtrsim |nn_1n_4|.$$
In this case we establish:

\begin{align}
\|\mathcal{N}_{-1}|_{1.b.}&(u_{1},u_{2},u_{3},u_{4})
\|_{\frac{1}{2}+\delta,-\frac{1}{2}+\delta,T}  \notag \\
&\lesssim
\frac{1}{(NN_{1}\cdots N_{4})^{\alpha}}
\prod_{j=1}^{4}\|u_{j}\|_{\frac{1}{2}-\delta,\frac{1}{2}-\delta,T}.
\label{Eqn:NL-new-1}
\end{align}

We consider various subcases.  In each of the cases that follow, we will employ the same method.  In fact, this method will continue to appear throughout this section of the appendix.  In most places where the method is used, we will spell out the details, but for the CASE 1.b. we will establish one case (case 1.b.i.) in details only.

\vspace{0.1in}

\noindent $\bullet$ \textbf{CASE 1.b.i.} $|\sigma|\gtrsim |nn_1n_4|$.

\vspace{0.1in}

In this case we find
\begin{align*}
\frac{|n|^{\frac{1}{2}+\delta}|n_1|}
{|\sigma|^{\frac{1}{2}-\delta}|n_1|^{\frac{1}{2}-\delta}|n_2|^{5\delta}}
\lesssim \frac{1}{(NN_1N_2N_3N_4)^{\alpha}}.
\end{align*}
Using this estimate, \eqref{Eqn:NL-new-1} follows from
\begin{align}
\|f_1f_2u_3u_4\|_{L^{2}_{x,t\in[0,T]}} \leq \|f_{1}\|_{0,\frac{1}{2}-\delta,T}\|f_{2}\|_{\frac{1}{2}-6\delta,\frac{1}{2}-\delta,T} \|u_{3}\|_{\frac{1}{2}-\delta,\frac{1}{2}-\delta,T}\|u_{4}\|_{\frac{1}{2}-\delta,\frac{1}{2}-\delta,T}.
\label{Eqn:lasthope1}
\end{align}
We can establish \eqref{Eqn:lasthope1} using H\"{o}lder, \eqref{Eqn:X-onehalfminus}, \eqref{Eqn:X-Str-1} and \eqref{Eqn:X-Str-interp},
\begin{align}
\|f_1f_2u_3u_4\|_{L^{2}_{x,t\in[0,T]}} &\lesssim
\|f_1\|_{L^4_{x,t\in[0,T]}}\|f_2\|_{L^{12}_{x,t\in[0,T]}}
\|u_{3}\|_{L^{12}_{x,t\in[0,T]}}\|u_{4}\|_{L^{12}_{x,t\in[0,T]}} \notag \\
&\lesssim \|f_{1}\|_{0,\frac{1}{2}-\delta,T}\|f_{2}\|_{\frac{1}{2}-6\delta,\frac{1}{2}-\delta,T} \|u_{3}\|_{\frac{1}{2}-\delta,\frac{1}{2}-6\delta,T}\|u_{4}\|_{\frac{1}{2}-\delta,\frac{1}{2}-6\delta,T}
\notag \\ &\leq \|f_{1}\|_{0,\frac{1}{2}-\delta,T}\|f_{2}\|_{\frac{1}{2}-6\delta,\frac{1}{2}-\delta,T} \|u_{3}\|_{\frac{1}{2}-\delta,\frac{1}{2}-\delta,T}\|u_{4}\|_{\frac{1}{2}-\delta,\frac{1}{2}-\delta,T},
\label{Eqn:lasthope1}
\end{align}
for $\delta>0$ sufficiently small.

\vspace{0.1in}

\noindent $\bullet$ \textbf{CASE 1.b.ii.} $|\sigma_1|\gtrsim |nn_1n_4|$.

\vspace{0.1in}

In this case we find
\begin{align*}
\frac{|n|^{\frac{1}{2}+\delta}|n_1|}
{|\sigma_1|^{\frac{1}{2}-\delta}|n_1|^{\frac{1}{2}-\delta}|n_2|^{5\delta}}
\lesssim \frac{1}{(NN_0N_1N_2N_3N_4)^{\alpha}},
\end{align*}
and from here we establish \eqref{Eqn:NL-new-1} using H\"{o}lder, \eqref{Eqn:X-onehalfminus}, \eqref{Eqn:X-Str-1} and \eqref{Eqn:X-Str-interp}.

\vspace{0.1in}

\noindent $\bullet$ \textbf{CASE 1.b.iii.} $|\sigma_2|\gtrsim |nn_1n_4|$.

\vspace{0.1in}

In this case we find
\begin{align*}
\frac{|n|^{\frac{1}{2}+\delta}|n_1|}
{|\sigma_2|^{\frac{1}{2}-\delta}|n_2|^{\frac{1}{2}-\delta}|n_3|^{5\delta}}
\lesssim \frac{1}{(NN_0N_1N_2N_3N_4)^{\alpha}},
\end{align*}
and from here we establish \eqref{Eqn:NL-new-1} using duality, H\"{o}lder, \eqref{Eqn:X-onehalfminus}, \eqref{Eqn:X-Str-1} and \eqref{Eqn:X-Str-interp}.

\vspace{0.1in}

\noindent $\bullet$ \textbf{CASE 1.b.iv.} $|\sigma_3|\gtrsim |nn_1n_4|$.

\vspace{0.1in}

Here we proceed exactly as in Case 1.b.iii. above, swapping the roles of $n_2$ and $n_3$.

\vspace{0.1in}

\noindent $\bullet$ \textbf{CASE 1.b.v.} $|\sigma_4|\gtrsim |nn_1n_4|$.

\vspace{0.1in}

In this case we find
\begin{align*}
\frac{|n|^{\frac{1}{2}+\delta}|n_1|}
{|n_1|^{\frac{1}{2}-2\delta}|n_2|^{6\delta}|\sigma_4|^{\frac{1}{2}-\delta}}\
\lesssim \frac{1}{(NN_0N_1N_2N_3N_4)^{\alpha}},
\end{align*}
and from here we establish \eqref{Eqn:NL-new-1} using duality, H\"{o}lder, \eqref{Eqn:X-onehalfminus}, \eqref{Eqn:X-Str-1} and \eqref{Eqn:X-Str-interp}.  The analysis of Case 1 is complete.

\vspace{0.1in}

\noindent $\bullet$ \textbf{CASE 2.}  $n^{0}\neq -n^{1}$.

Before we proceed with each subcase, let us identify a useful restriction in this case: if $(n_1,n_2,n_3,n_4)\in \eta_1(n)$ with $n^0\neq -n^1$, then
\begin{align}
\text{In case 2, no two integers in the set} \ \{-n,n_1,n_2,n_3,n_4\}\
\text{sum to zero.}
\label{Eqn:NL-new-2}
\end{align}
Indeed, with $(n_1,n_2,n_3,n_4)\in \eta_1(n)$, we already have $n\neq n_k$ for all $k=1,2,3,4$ and $n_1\neq -n_k$ for all $k=2,3,4$.  The only pairs of integers that could sum to zero are within the set $\{n_2,n_3,n_4\}$.  Suppose, for example, that $n_2=-n_3$, then by the restriction $n^0\neq -n^1$ we must have $N_2,N_3\ll N^0$.  Then $n=n_1+n_4$ and we have
$$ |n^3-n_1^3-\cdots - n_4^{3}|=|n^3-n_1^3-n_4^3|=3|nn_1n_4|\gtrsim
(N^0)^{2},$$
in contradiction with restriction to the region $A_{-1}$ in this case.  The same argument applies if $n_2=-n_4$ or $n_{3}=-n_4$, and \eqref{Eqn:NL-new-2} follows.

\vspace{0.1in}

\noindent $\bullet$ \textbf{CASE 2.a.}  $N^{3}\ll N^{0}$ and $N^{2}N^{3}N^{4}\ll
N^{0}N^{1}|n^{0}+n^{1}|$.

\vspace{0.1in}

Recall that we have taken $n=-n^{k}$ for some $k\in\{0,1,\ldots,4\}$, so that $n^{0}+\cdots + n^{4}=0$ is satisfied.  Then
\begin{align*}
|n^{3}-n_{1}^{3}&-\cdots -n_{4}^{3}| =
|(n^1 + \cdots +n^4)^{3}-(n^{1})^{3}-\cdots -(n^{4})^{3}| \\
&=
3|-n^0 n^1 (n^2 + n^3 + n^4) + n^2(n^3 + n^4)(n^2 +n^3 +n^4) + n^3 n^4(n^3 + n^4)| \\
&=
3|(-n^0 n^1+ n^2(n^3 + n^4)+ n^3 n^4)(n^2 + n^3 + n^4) + n^2 n^3 n^4|
\\
&\gtrsim N^0 N^1 |n^0 + n^1|,
\end{align*}
since $N^{3}\ll N^{0}$, $N^{2}N^{3}N^{4}\ll
N^{0}N^{1}|n^{0}+n^{1}|$ and $n^0\neq -n^1$.  Then
\begin{align*}
\max(|\sigma|,|\sigma_{1}|,\ldots,|\sigma_{4}|)
\gtrsim |n^{3}-n_{1}^{3}&-\cdots -n_{4}^{3}|
\gtrsim N^0 N^1 |n^0 + n^1| \geq |n_{\text{max}}|^{2},
\end{align*}
and we cannot be in the region $A_{-1}$.  That is, there is to contribution to $\mathcal{N}_{-1}$ from this case, and we proceed to the next one.

\vspace{0.1in}

\noindent $\bullet$ \textbf{CASE 2.b.} $N^{3}\sim N^{0}$.

\vspace{0.1in}

With $N_2\geq N_3\geq N_4$, this implies, in particular, that
\begin{align}
N_{3}\sim N^{0}.
\label{Eqn:NL-neg1-6}
\end{align}
For the remainder of this case we will only use the restriction of \eqref{Eqn:NL-neg1-6} (and not the stronger condition $N^{3}\sim N^{0}$).  This way, in future subcases, once we establish \eqref{Eqn:NL-neg1-6}, we can revert to the analysis of this case.


\noindent $\bullet$ \textbf{CASE 2.b.i.}
At least two $u_{i}$ of type (I), $i=1,2,3$.

\vspace{0.1in}

Suppose $u_{1}, u_{2}$ are type (I).  We will comment on adapting these arguments to the other cases afterward.  We will use $\mathcal{N}_{-1}|_{2.b.i}$ to denote the contribution to the nonlinearity from this case.  We establish the estimate:
\begin{align}
\|\mathcal{N}_{-1}|_{2.b.i}&(u_{1},u_{2},u_{3},u_{4})\|_{\frac{1}{2}+\delta,-\frac{1}{2}+\delta,T}  \notag \\
&\lesssim \frac{T^{-\beta}}
{(NN_{1}\cdots N_{4})^{\alpha}}(N_{1}N_{2})^{-\varepsilon}
\|u_{3}\|_{\frac{1}{2}-\delta,\frac{1}{2}-\delta,T}
\|u_{4}\|_{\frac{1}{2}-\delta,\frac{1}{2}-\delta,T}.
\label{Eqn:pre-a-1-1}
\end{align}

\vspace{0.1in}

\noindent By changing variables and taking out a supremum, we find
\begin{align}
\|\mathcal{N}_{-1}|_{2.b.i}&(u_{1},u_{2},u_{3},u_{4})\|_{\frac{1}{2}+\delta,-\frac{1}{2}+\delta,T}
= \Big\|\frac{\langle n\rangle^{\frac{1}{2}+\delta}}
{\langle \sigma \rangle^{\frac{1}{2}-\delta}}
\widehat{\mathcal{N}}_{-1}|_{2.b.i}(u_{1},u_{2},u_{3},u_{4})
(n,\tau)\Big\|_{L^{2}_{\{|n|\sim N\},\tau}}
\notag \\
&\ \ \ \ \ \ = \Big\|\chi_{\{|\lambda|<(N^{0})^{2}\}}\frac{\langle n\rangle^{\frac{1}{2}+\delta}}
{\langle \lambda \rangle^{\frac{1}{2}-\delta}}
\widehat{\mathcal{N}}_{-1}|_{2.b.i}(u_{1},u_{2},u_{3},u_{4})
(n,\lambda + n^{3})\Big\|_{L^{2}_{\{|n|\sim N\},\lambda}}
\notag \\
&\ \ \ \ \ \ \leq (N^{0})^{\delta}\sup_{|\lambda|<(N^{0})^{2}}
\|\widehat{\mathcal{N}}_{-1}|_{2.b.i}(u_{1},u_{2},u_{3},u_{4})
(n,\lambda + n^{3})\|_{H^{\frac{1}{2}+\delta}_{|n|\sim N}}.
\label{Eqn:NL-neg1-26}
\end{align}

For the factors $u_{3},u_{4}\in X^{\frac{1}{2}-\delta,\frac{1}{2}-\delta}_{T}$, we will use the following standard representation for functions in $X^{s,b}$ (see, for example, Klainerman-Selberg \cite{KS}).  Given a function $v(x,t)$, we can write $v$ as
\begin{align}
v(x,t)=\int\langle\lambda\rangle^{-b}
\Big(\sum_{n}\langle n\rangle^{2s}\langle \lambda\rangle^{2b}|\hat{v}(n,n^{3}+\lambda)|^{2}
\Big)^{\frac{1}{2}}\Big\{e^{i\lambda t}\sum_{n}a_{\lambda}(n)e^{i(nx+n^{3}t)}\Big\}d\lambda
\label{Eqn:NL-neg1-24}
\end{align}
where $a_{\lambda}(n)=\frac{\hat{v}(n,n^{3}+\lambda)}
{(\sum_{n}\langle n\rangle^{2s}|\hat{v}(n,n^{3}+\lambda)|^{2})^{\frac{1}{2}}}$.
Notice that $\sum_{n}\langle n\rangle^{2s}|a_{\lambda}(n)|^{2}=1$.  For $v\in X^{s,b}$, with $b<\frac{1}{2}$, we have
\begin{align}
\int_{|\lambda|<K}\langle\lambda\rangle^{-b}
\Big(\sum_{n}\langle n\rangle^{2s}\langle \lambda\rangle^{2b}|\hat{v}(n,n^{3}+\lambda)|^{2}
\Big)^{\frac{1}{2}}d\lambda
\lesssim K^{\frac{1}{2}-b}\|v\|_{X^{s,b}},
\label{Eqn:NL-neg1-25}
\end{align}
by Cauchy-Schwarz.  In our context, for each $j=3,4$, we have $u_{j}=\chi_{[0,T]}u_j\in X^{\frac{1}{2}-\delta,\frac{1}{2}-\delta}$, and $|\tau_{j}-n_{j}^{3}|
<(N^{0})^{2}$.  Using \eqref{Eqn:NL-neg1-24} we can write
\begin{align*}
\hat{u}_{j}(n_{j},\tau_{j}) = \int_{|\lambda_{j}|<(N^{0})^{2}}\langle \lambda_{j}\rangle^{-\frac{1}{2}+\delta}
c_{j}(\lambda_{j})a_{\lambda_{j}}(n_{j})\delta(\tau_{j}-n_{j}^{3}-\lambda_{j})d\lambda_{j},
\end{align*}
with $\sum_{n}\langle n\rangle^{2s}|a_{\lambda}(n)|^{2}=1$ and $c_{j}(\lambda_{j})=\Big(\sum_{n}\langle n\rangle^{1-2\delta}\langle \lambda_{j}\rangle^{1-2\delta}
|\hat{u_j}(n,n^{3}+\lambda)|^{2}\Big)^{\frac{1}{2}}$.
Inserting this representation for $u_{3}$, $u_{4}$, we have
\begin{align}
\eqref{Eqn:NL-neg1-26}
&= (N^{0})^{\delta}\sup_{|\lambda|<(N^{0})^{2}}
\Big\|\langle n \rangle^{\frac{1}{2}+\delta}
\sum_{\{|n_{j}|\sim N_{j}\}\cap A_{-1}}
(in_{1})\prod_{i=1}^{2}\frac{g_{n_{i}}(\omega)
\delta(\tau_{i}-n_{i}^{3})}{|n_{i}|} \notag \\
&\iint_{|\lambda_{3}|,|\lambda_{4}|<(N^{0})^{2}}
\prod_{j=3}^{4}
\langle \lambda_{j}\rangle^{-\frac{1}{2}+\delta}
c_{j}(\lambda_{j})a_{\lambda_{j}}(n_{j})
\delta(\tau_{j}-n_{j}^{3}-\lambda_{j})d\lambda_{j}
\Big\|_{L^{2}_{|n|\in N}}.
\label{Eqn:NL-neg1-27}
\end{align}
By Minkowski in $\lambda_{3},\lambda_{4}$, we find
\begin{align}
&\eqref{Eqn:NL-neg1-27}
\leq (N^{0})^{\delta}
\iint_{|\lambda_{3}|,|\lambda_{4}|<(N^{0})^{2}}
\prod_{j=3}^{4}
\langle \lambda_{j}\rangle^{-\frac{1}{2}+\delta}
|c_{j}(\lambda_{j})|d\lambda_{j}
\notag \\
&\sup_{|\lambda|,|\lambda_{3}|,|\lambda_{4}|<(N^{0})^{2}}
\Bigg\|\langle n \rangle^{\frac{1}{2}+\delta}
\sum_{\{|n_{j}|\sim N_{j}\}\cap A_{-1}}
(in_{1})\frac{g_{n_{1}}(\omega)g_{n_{2}}(\omega)}
{|n_{1}||n_{2}|}
a_{\lambda_{3}}(n_{3})a_{\lambda_{4}}(n_{4})
\notag \\
&\ \ \ \ \ \ \ \ \ \ \ \ \ \ \ \ \ \ \ \ \ \iiint_{\tau_{1},\tau_{2},\tau_{3}}
\delta(\tau_{1}-n_{1}^{3})\delta(\tau_{2}-n_{2}^{3})
\delta(\tau_{3}-n_{3}^{3}-\lambda_{3})
\notag \\
&\ \ \ \ \ \ \ \ \ \ \ \ \ \ \ \ \ \ \ \ \ \delta(\lambda + n^{3}-\tau_{1} - \tau_{2} - \tau_{3} -n_{4}^{3}-\lambda_{4})d\tau_{1}
d\tau_{2}d\tau_{3}
\Bigg\|_{L^{2}_{|n|\in N}}.
\label{Eqn:NL-neg1-28}
\end{align}
For fixed $n,n_{1},n_{2},n_{3},\lambda,\lambda_{3},\lambda_{4}$, we find
\begin{align*}
\iiint_{\tau_{1},\tau_{2},\tau_{3}}&
\delta(\tau_{1}-n_{1}^{3})\delta(\tau_{2}-n_{2}^{3})
\delta(\tau_{3}-n_{3}^{3}-\lambda_{3})
\delta(\lambda + n^{3}-\tau_{1} - \tau_{2} - \tau_{3} -n_{4}^{3}-\lambda_{4})d\tau_{1}
d\tau_{2}d\tau_{3} \\
&=
\iint_{\tau_{2},\tau_{3}}
\delta(\tau_{2}-n_{2}^{3})
\delta(\tau_{3}-n_{3}^{3}-\lambda_{3})
\delta(\lambda + n^{3}-n_{1}^{3} - \tau_{2} - \tau_{3} -n_{4}^{3}-\lambda_{4})
d\tau_{2}d\tau_{3}
 \\
&=
\int_{\tau_{3}}
\delta(\tau_{3}-n_{3}^{3}-\lambda_{3})
\delta(\lambda + n^{3}-n_{1}^{3} - n_{2}^{3} - \tau_{3} -n_{4}^{3}-\lambda_{4})
d\tau_{3} \\
&= \left\{
\begin{array}{ll} 1, \ \ \text{if}\
\lambda-\lambda_{3}-\lambda_{4}+n^{3}-n_{1}^{3}
-\cdots - n_{4}^{3}=0,
\\
0,  \ \ \text{otherwise}.
\end{array} \right.
\end{align*}

Then we have
\begin{align}
\eqref{Eqn:NL-neg1-28}
\leq (N^{0})^{\delta}
&\iint_{|\lambda_{3}|,|\lambda_{4}|<(N^{0})^{2}}
\prod_{j=3}^{4}
\langle \lambda_{j}\rangle^{-\frac{1}{2}+\delta}
|c_{j}(\lambda_{j})|d\lambda_{j}
\notag \\
&\sup_{|\lambda|,|\lambda_{3}|,|\lambda_{4}|<(N^{0})^{2}}
\Bigg\|\langle n \rangle^{\frac{1}{2}+\delta}
\sum_{*(n,\lambda + \lambda_{3} + \lambda_{4})}
(in_{1})\frac{g_{n_{1}}(\omega)g_{n_{2}}(\omega)}
{|n_{1}||n_{2}|}
a_{\lambda_{3}}(n_{3})a_{\lambda_{4}}(n_{4})
\Bigg\|_{L^{2}_{|n|\in N}}
\notag \\
\leq (N^{0})^{3\delta}&\prod_{j=3}^{4}
\|u_{j}\|_{\frac{1}{2}-\delta,\frac{1}{2}-\delta,T}
\notag \\
&\sup_{|\mu|<3(N^{0})^{2}}
\Bigg\|\langle n \rangle^{\frac{1}{2}+\delta}
\sum_{*(n,\mu)}
(in_{1})\frac{g_{n_{1}}(\omega)g_{n_{2}}(\omega)}
{|n_{1}||n_{2}|}
a_{3}(n_{3})a_{4}(n_{4})
\Bigg\|_{L^{2}_{|n|\in N}},
\label{Eqn:NL-neg1-29a}
\end{align}
by \eqref{Eqn:NL-neg1-25}, where $\sum_{n_{i}}|n_{i}|^{1-2\delta}|a_{i}(n_{i})|^{2} = 1$, for $i=3,4$, and
\begin{align*}
*(n,\mu):=\Big\{(n_{1},n_{2},n_{3},n_{4})\in\mathbb{Z}^{4}\Big|
(&n,n_{1},n_{2},n_{3},n_{4})\in A_{-1},\ \text{each} \ |n_{i}|\sim N_{i}, \\ &\text{and}\
\mu = n^{3}-n_{1}^{3}-\cdots-n_{4}^{3}\Big\}.
\end{align*}  When we fix numbers other than $n,\mu$, for example $n_{1}$, we let
\begin{align*}
*(n,\mu,n_{1}):=\Big\{(n_{2},n_{3},n_{4})\in\mathbb{Z}^{3}\Big|
(&n,n_{1},n_{2},n_{3},n_{4})\in A_{-1},\ \text{each} \ |n_{i}|\sim N_{i}, \\ &\text{and}\
\mu = n^{3}-n_{1}^{3}-\cdots-n_{4}^{3}\Big\},
\end{align*}
and define $*(n,\mu,n_{2},n_{3})\subset \mathbb{Z}^2, *(n,\mu,n_{1},n_{2},n_{3})\subset \mathbb{Z}$,..., etc, similarly.

Notice that we have dropped the dependence on $\lambda_{3},\lambda_{4}$ in \eqref{Eqn:NL-neg1-29a}; this is justified a posteriori by using estimates which are independent of $\lambda_{3},\lambda_{4}$.
Now for each fixed $|n|\in N,|\mu|<3(N^{0})^{2}$, we write
\begin{align*}
\Big| \sum_{*(n,\mu)}(in_{1})\frac{g_{n_{1}}(\omega)g_{n_{2}}(\omega)}{\langle n_{1}\rangle\langle n_{2}\rangle}a_{n_{3}}a_{n_{4}}\Big|^{2} &=  \Big| \sum_{|n_{4}|\sim N_{4}}|n_{4}|^{\frac{1}{2}-\delta}a_{n_{4}}
\\
&\ \ \ \ \ \ \ \ \ \ \cdot\frac{1}{|n_{4}|^{\frac{1}{2}-\delta}} \Big(\sum_{*(n,\mu,n_{4})}(in_{1})\frac{g_{n_{1}}(\omega)g_{n_{2}}(\omega)}{\langle n_{1}\rangle\langle n_{2}\rangle}a_{n_{3}}\Big)\Big|^{2} \\
&\lesssim  \sum_{|n_{4}|\sim N_{4}} \frac{1}{|n_{4}|^{1-2\delta}}\Big|\sum_{*(n,\mu,n_{4})}(in_{1})\frac{g_{n_{1}}(\omega)g_{n_{2}}(\omega)}{\langle n_{1}\rangle\langle n_{2}\rangle}a_{n_{3}}\Big|^{2},
\end{align*}
by Cauchy-Schwarz in $n_{4}$.  For each fixed $|n_{4}|\sim N_{4},\mu<3(N^{0})^{2}$, we write
\begin{align}
\sum_{|n|\sim N}\Big|\sum_{*(n,\mu,n_{4})}(in_{1})\frac{g_{n_{1}}(\omega)g_{n_{2}}(\omega)}{\langle n_{1}\rangle\langle n_{2}\rangle}a_{n_{3}}\Big|^{2}  =  \sum_{|n|\sim N} |\sum_{|n_{3}|\sim N_{3}}\sigma_{n,n_{3}}^{n_{4},\mu}|n_{3}|^{\frac{1}{2}-\delta}a_{n_{3}}|^{2}
\label{Eqn:case1-matrix}
\end{align}
where $\sigma_{n,n_{3}}^{n_{4},\mu}$ is the $(n,n_{3})^{\text{rd}}$ entry of a matrix $\sigma^{n_{4},\mu}$ (for $n_{4},\mu$ fixed) with columns indexed by $|n_{3}|\sim N_{3}$, and rows indexed by $|n|\sim N$.  These entries are given by
\begin{align}
\sigma_{n,n_{3}}^{n_{4},\mu} = \sum_{(n_{1},n_{2}) \in *(n,n_{3},n_{4},\mu)}(in_{1})\frac{ g_{n_{1}}(\omega)g_{n_{2}}(\omega)}{\langle n_{1}\rangle\langle n_{2}\rangle|n_{3}|^{\frac{1}{2}-\delta}}.
\label{Eqn:NL-neg1-39}
\end{align}
Let us also recall the following property of matrix norms: $\|A^{*}A\|=\|AA^{*}\|$.
Using Cauchy Schwarz, the condition $\sum_{n_{3}}|n_{3}|^{1-2\delta}|a_{3}(n_{3})|^{2} = 1$, and applying Lemma \ref{Lemma:matrixnorm}, we find
\begin{align}
(\ref{Eqn:case1-matrix}) &\lesssim  \|(\sigma^{n_{4},\mu})^{*}\sigma^{n_{4},\mu}\| \notag \\
&=  \|\sigma^{n_{4},\mu}(\sigma^{n_{4},\mu})^{*}\| \notag \\
&\leq  \sup_{|n|\sim N} \sum_{|n_{3}|\sim N_{3}}|\sigma_{n,n_{3}}^{n_{4},\mu}|^{2}  +  \Big(\sum_{\substack{n\neq n'\\ |n|,|n'|\sim N }}\big|\sum_{|n_{3}|\sim N_{3}}\sigma_{n,n_{3}}^{n_{4},\mu}
\overline{\sigma_{n',n_{3}}^{n_{4},\mu}}\big|^{2}\Big)^{\frac{1}{2}}  \notag \\
&=:  I_{1}(n_{4},\mu)  +  I_{2}(n_{4},\mu) .
\label{Eqn:case1-reducedtoI}
\end{align}
To recap we now have
\begin{align}
\eqref{Eqn:NL-neg1-29a}
\lesssim
(N^{0})^{3\delta}&N^{\frac{1}{2}+\delta}
\prod_{j=3}^{4}
\|u_{j}\|_{\frac{1}{2}-\delta,\frac{1}{2}-\delta,T}
\notag \\
&\sup_{|\mu|<3(N^{0})^{2}}
\Big(\sum_{|n_{4}|\sim N_{4}}\frac{1}
{|n_{4}|^{1-2\delta}}(I_{1}(n_{4},\mu)  +  I_{2}(n_{4},\mu))
\Big)^{\frac{1}{2}},
\label{Eqn:NL-neg1-30}
\end{align}
and we estimate the contributions from $I_{1}(n_{4},\mu)$ and $I_{2}(n_{4},\mu)$ separately.

We remark that the sum in \eqref{Eqn:NL-neg1-39}
has at most two terms.  Indeed, for $n$, $n_{3}$, $n_{4}$, and $\mu$ fixed, if $(n_{1},n_{2}) \in *(n,n_{3},n_{4},\mu)$, then $n_{2}$ is determined by $n_{1}$ through the condition $n=n_{1}+\cdots + n_{4}$, and $n_{1}$ satisfies the equation $\mu=n^{3}-n_{1}^{3}-\cdots -n_{4}^{3}$.  Since $n_{1}\neq -n_{2}$ (recall \eqref{Eqn:NL-new-2}), this is a non-degenerate quadratic equation in $n_{1}$:
\begin{align}
\mu &= n^{3} -n_{1}^{3} -\cdots-n_{4}^{3} \notag
\\ &= n^{3} - n_{1}^{3}-(n-n_{1}-n_{3}-n_{4})^{3} -\cdots-n_{4}^{3} \notag \\ &=  -3(n-n_{3}-n_{4})n_{1}^{2} -3(n-n_{3}-n_{4})^{2}n_{1} + n^{3} \notag \\ &\ \ \ \ \ \ \ -(n-n_{3}-n_{4})^{3} -n_{3}^{3}-n_{4}^{3},
\label{Eqn:quad}
\end{align}
with $n-n_{3}-n_{4}=n_{1}+n_{2}\neq 0$, and this equation has at most two roots $n_{1}$.

Then to estimate $I_{1}(n_{4},\mu)$, for $n,n_{3},n_{4},\mu$ fixed, we bring the absolute value inside the sum of (at most) two terms in \eqref{Eqn:NL-neg1-39} and apply Lemma \ref{Lemma:prob1} to obtain, for $\omega\in\widetilde{\Omega}_{T}$:
\begin{align}
I_{1}(n_{4},\mu) &\leq  \sup_{|n|\sim N}\sum_{\substack{|n_{3}|\sim N_{3}\\(n_{1},n_{2}) \in *(n,n_{3},n_{4},\mu)}} \frac{|n_{1}|^{2}
|g_{n_{1}}(\omega)||g_{n_{2}}(\omega)|}{\langle n_{1}\rangle^{2}
\langle n_{2}\rangle^{2}\langle n_{3}\rangle^{1-2\delta}} \notag \\
&\lesssim  T^{-\beta}\sup_{|n|\sim N}\sum_{\substack{|n_{3}|\sim N_{3}\\(n_{1},n_{2}) \in *(n,n_{3},n_{4},\mu)}} \frac{|n_{1}|^{2}}{\langle n_{1}\rangle^{2-\epsilon}
\langle n_{2}\rangle^{2-\epsilon}\langle n_{3}\rangle^{1-2\delta}} \notag \\
&\lesssim  \frac{T^{-\beta}}{(N^{0})^{2-2\epsilon-2\delta-\gamma}}
\sum_{|n_{3}|\sim N_{3}}\frac{1}{\langle n_{3}\rangle^{1+\gamma}}
\lesssim   \frac{T^{-\beta}}{(N^{0})^{2-2\epsilon-2\delta-\gamma}},
\label{Eqn:NL-neg1-44}
\end{align}
where we have used $N_{3}\gtrsim \max(N,N_{1})$ and $N_{2}\geq N_{3}\geq N_{4}$ in the second last line.
Then we can estimate the contribution to \eqref{Eqn:NL-neg1-30} coming from $I_{1}(n_{4},\mu)$ by
\begin{align}
\frac{T^{-\frac{\beta}{2}}
N^{\frac{1}{2}+\delta}}{(N^0)^{1-\epsilon-5\delta-\frac{\gamma}{2}}} &\prod_{j=3}^{4}
\|u_{j}\|_{\frac{1}{2}-\delta,\frac{1}{2}-\delta,T}
\Big(\sum_{|n_{4}| \sim N_{4}}\frac{1}{\langle n_{4}\rangle^{1-2\delta}}\Big)^{\frac{1}{2}}
\notag \\
&\lesssim
\frac{T^{-\frac{\beta}{2}}
N^{\frac{1}{2}+\delta}}
{(N^0)^{1-\epsilon-6\delta-\gamma}} \prod_{j=3}^{4}
\|u_{j}\|_{\frac{1}{2}-\delta,\frac{1}{2}-\delta,T}\Big(\sum_{|n_{4}| \sim N_{4}}\frac{1}{\langle n_{4}\rangle^{1+\gamma}}\Big)^{\frac{1}{2}}
\notag \\
&\lesssim
\frac{T^{-\beta}}
{(NN_{1}\cdots N_{4})^{\alpha}}(N_{1}N_{2})^{-\varepsilon}
\|u_{3}\|_{\frac{1}{2}-\delta,\frac{1}{2}-\delta,T}
\|u_{3}\|_{\frac{1}{2}-\delta,\frac{1}{2}-\delta,T}.
\label{Eqn:NL-neg1-31}
\end{align}
To estimate $I_{2}(n_{4},\mu)$, note that
\begin{align}
I_{2}(n_{4},\mu) &=  \Bigg(
\sum_{\substack{n\neq n'\\ |n|,|n'|\sim N}}
\Big|\sum_{|n_{3}|\sim N_{3}}
\Big(\sum_{*(n,n_{3},n_{4},\mu)}\frac{(in_{1}) g_{n_{1}}(\omega)g_{n_{2}}(\omega)}{\langle n_{1}\rangle\langle n_{2}\rangle\langle n_{3}\rangle^{\frac{1}{2}-\delta}}\Big)
\notag \\
&\ \ \ \ \ \  \ \ \ \ \ \ \ \ \ \ \Big(\sum_{*(n',n_{3},n_{4},\mu)}\frac{(-in_{1}') \overline{g_{n_{1}'}}(\omega)\overline{g_{n_{2}'}}(\omega)}{\langle n_{1}'\rangle\langle n_{2}'\rangle \langle n_{3}'\rangle ^{\frac{1}{2}-\delta}}\Big)\Big|^{2}\Bigg)^{\frac{1}{2}}.
\label{Eqn:case1-I2}
\end{align}
For each fixed $n,n',n_{4},\mu$, let
\begin{align*}
F_{n,n',n_{4},\mu}(\omega) := \sum_{\substack{|n_{3}|\sim N_{3}\\n_{1},n_{2}\in*(n,\mu,n_{3},n_{4})\\n_{1}',n_{2}'\in*(n',\mu,n_{3},n_{4})}}
\frac{n_{1}n_{1}'g_{n_{1}}(\omega)
g_{n_{2}}(\omega)\overline{g_{n_{1}'}}(\omega)
\overline{g_{n_{2}'}}(\omega)}{\langle n_{1}\rangle\langle n_{2}\rangle\langle n_{1}'\rangle\langle n_{2}'\rangle\langle n_{3}\rangle^{1-2\delta}}.
\end{align*}
Notice that $F_{n,n',n_{4},\mu}(\omega):= F_{n,n',n_{4},\mu}(u_{0,\omega})$ is $\rho$-measurable (it is a polynomial function of the randomized Fourier coefficients).  By Lemma \ref{Lemma:prob2},
\begin{align*}
\|F_{n,n',n_{4},\mu}\|_{L^{p}(\Omega)} \leq \sqrt{5}(p-1)^{2}\|F_{n,n',n_{4},\mu}\|_{L^{2}(\Omega)},
\end{align*}
for each $2<p<\infty$.  Then by Lemma \ref{Lemma:prob3} (applied with
$\tilde{N}=(\|F_{n,n',n_{4},\mu}\|_{L^{2}(\Omega)})^{-1}$, $\alpha=1$ and $k=4$) it follows that
\begin{align*}
P(|F_{n,n',n_{4},\mu}(\omega)|\geq \lambda) \leq e^{-c\|F_{n,n',n_{4},\mu}\|^{\frac{1}{2}}_{L^{2}(\Omega)}\lambda^{\frac{1}{2}}}.
\end{align*}
Taking $\displaystyle \lambda = \|F_{n,n',n_{4},\mu}\|_{L^{2}(\Omega)}(N^{0})^{2\beta}T^{-2\beta}$, we have
\begin{align*}
P(|F_{n,n',n_{4},\mu}(\omega)|\geq \|F_{n,n',n_{4},\mu}\|_{L^{2}(\Omega)}(N^{0})^{2\beta}T^{-2\beta}) \leq e^{-c\frac{(N^0)^{\beta}}{T^{\beta}}}.
\end{align*}
Let
\begin{align*}
\Omega_{N,N_{1},\ldots,N_{4},T}
:= \bigcap_{\substack{|n|\sim N, |n'|\sim N
\\ |n_{4}|\sim N_{4}, |\mu|<3(N^{0})^{2}
}}
\Big\{|F_{n,n',n_{4},\mu}(\omega)| < \|F_{n,n',n_{4},\mu}\|_{L^{2}(\Omega)}(N^{0})^{2\beta}T^{-2\beta}\Big\}.
\end{align*}
Then
\begin{align*}
P((\Omega_{N,N_{1},\ldots,N_{4},T})^{c})
&\leq \sum_{\substack{|n|\sim N, |n'|\sim N
\\ |n_{4}|\sim N_{4}, |\mu|<3(N^{0})^{2}
}}
P(|F_{n,n',n_{4},\mu}(\omega)|\geq \|F_{n,n',n_{4},\mu}\|_
{L^{2}(\Omega)}(N^{0})^{2\beta}T^{-2\beta})
\\
&\leq \sum_{\substack{|n|\sim N, |n'|\sim N
\\ |n_{4}|\sim N_{4}, |\mu|<3(N^{0})^{2}
}}
e^{-c\frac{(N^0)^{\beta}}{T^{\beta}}}
\leq
N^{2}(N^{0})^{2}N_{4}
e^{-c\frac{(N^0)^{\beta}}{T^{\beta}}}
\\
&\leq
(N^{0})^{5}
e^{-c\frac{(N^0)^{\beta}}{T^{\beta}}}
\leq
(N^{0})^{-5\kappa}
e^{-\frac{\tilde{c}}{T^{\beta}}}
\leq
(NN_{1}\cdots N_{4})^{-\kappa}
e^{-\frac{\tilde{c}}{T^{\beta}}},
\end{align*}
for some
$\tilde{c}(\beta),\kappa(\beta)>0$.
Furthermore, if $\omega\in\Omega_{N,N_{1},\ldots,N_{4},T}$, then for each $|n_{4}|\sim N_{4}, |\mu| < 3(N^0)^{2}$, we have
\begin{align}
I_{2}(n_{4},\mu) &=  \Big(\sum_{
\substack{n\neq n'\\ |n|,|n'|\sim N}}
|F_{n,n',n_{4},\mu}(\omega)|^{2}\Big)^{\frac{1}{2}}  \notag \\
&<  \Big(\sum_{
\substack{n\neq n'\\ |n|,|n'|\sim N}}
\|F_{n,n',n_{4},\mu}\|^{2}_{L^{2}(\Omega)}(N^{0})^{4\beta}T^{-4\beta}
\Big)^{\frac{1}{2}}.
\label{Eqn:NL-neg1-36}
\end{align}
Next we compute
\begin{align}
\|F_{n,n',n_{4},\mu}\|^{2}_{L^{2}(\Omega)}
&=  \mathbb{E}\Bigg(
\Big|\sum_{\substack{|n_{3}|\sim N_{3}\\n_{1},n_{2}\in
*(n,\mu,n_{3},n_{4})\\n_{1}',n_{2}'\in
*(n',\mu,n_{3},n_{4})}}
\frac{-n_{1}n_{1}'g_{n_{1}}(\omega)g_{n_{2}}(\omega)
\overline{g_{n_{1}'}}(\omega)\overline{g_{n_{2}'}}(\omega)}{\langle n_{1}\rangle\langle n_{2}\rangle\langle n_{1}'\rangle\langle n_{2}'\rangle\langle n_{3}\rangle^{1-2\delta}}\Big|^{2}\Bigg)
\notag \\
&=
\sum_{\substack{|n_{3}|,|m_{3}|\sim N_{3}\\(n_{1},n_{2})\in
*(n,\mu,n_{3},n_{4}), \,(n_{1}',n_{2}')\in
*(n',\mu,n_{3},n_{4})\\
(m_{1},m_{2})\in
*(n,\mu,m_{3},n_{4}), \,(m_{1}',m_{2}')\in
*(n',\mu,m_{3},n_{4})
}}
\notag \\
&\ \ \ \frac{(-n_{1}n_{1}')(-m_{1}m_{1}')}{\langle n_{1}\rangle \langle n_{2}\rangle
\langle n_{1}'\rangle \langle n_{2}'\rangle
\langle n_{3}\rangle^{1-2\delta}\langle m_{1}\rangle \langle m_{2}\rangle
\langle m_{1}'\rangle \langle m_{2}'\rangle
\langle m_{3}\rangle^{1-2\delta}} \notag \\
&\ \ \ \mathbb{E}\Big(g_{n_{1}}(\omega)g_{n_{2}}(\omega)\overline{g_{n_{1}'}}(\omega)
\overline{g_{n_{2}'}}(\omega)
\overline{g_{m_{1}}}(\omega)
\overline{g_{m_{2}}}(\omega)
g_{m_{1}'}(\omega)g_{m_{2}'}(\omega)\Big).
\notag
\\
&\lesssim  \frac{1}{(N^0)^{6-4\delta}}
\sum_{\substack{|n_{3}|,|m_{3}|\sim N_{3}\\(n_{1},n_{2})\in
*(n,\mu,n_{3},n_{4}), \,(n_{1}',n_{2}')\in
*(n',\mu,n_{3},n_{4})\\
(m_{1},m_{2})\in
*(n,\mu,m_{3},n_{4}), \,(m_{1}',m_{2}')\in
*(n',\mu,m_{3},n_{4})
}}\notag \\
&\ \ \ \big|\mathbb{E}\Big(g_{n_{1}}(\omega)g_{n_{2}}(\omega)
\overline{g_{n_{1}'}}(\omega)\overline{g_{n_{2}'}}(\omega)
\overline{g_{m_{1}}}(\omega)\overline{g_{m_{2}}}(\omega)
g_{m_{1}'}(\omega)g_{m_{2}'}(\omega)\Big)
\big|.
\label{Eqn:NL-neg1-34}
\end{align}
Then combining \eqref{Eqn:NL-neg1-36} and \eqref{Eqn:NL-neg1-34} we have
\begin{align}
I_{2}(n_{4},\mu) &=  \Big(\sum_{
\substack{n\neq n'\\ |n|,|n'|\sim N}}|F_{n,n',n_{4},\mu}(\omega)|^{2}\Big)^{\frac{1}{2}}  \notag \\
&<  \frac{T^{-2\beta}}{(N^0)^{3-2\delta-2\beta}}\Bigg(
\sum_{\substack{n\neq n',\, |n|,|n'|\sim N, \,|n_{3}|,|m_{3}|\sim N_{3}\\(n_{1},n_{2})\in
*(n,\mu,n_{3},n_{4})\, (n_{1}',n_{2}')\in
*(n',\mu,n_{3},n_{4}) \\
(m_{1},n_{2})\in
*(n,\mu,m_{3},n_{4})\, (m_{1}',m_{2}')\in
*(n',\mu,m_{3},n_{4})
}}\notag \\
&\ \ \ \ \ \ \ \big|\mathbb{E}\Big(g_{n_{1}}(\omega)g_{n_{2}}(\omega)
\overline{g_{n_{1}'}}(\omega)\overline{g_{n_{2}'}}(\omega)
g_{m_{1}}(\omega)g_{m_{2}}(\omega)\overline{g_{m_{1}'}}(\omega)\overline{g_{m_{2}'}}(\omega)\Big)
\big|
\Bigg)^{\frac{1}{2}}
\notag \\
&<  \frac{T^{-2\beta}}{(N^0)^{\frac{3}{2}-2\delta-2\beta}},
\label{Eqn:NL-neg1-37}
\end{align}
by the following lemma.
\begin{lemma}
Let
\begin{align*}
S(n_{4},\mu):=\Bigg\{(n,n_{1},&n_{2},n_{3},n',n_{1}',
n_{2}',m_{1},m_{2},m_{3},m_{1}',
m_{2}')\bigg| \\ &n\neq n',\, |n|,|n'|\sim N, \,|n_{3}|,|m_{3}|\sim N_{3}\\&(n_{1},n_{2})\in
*(n,\mu,n_{3},n_{4}),\, (n_{1}',n_{2}')\in
*(n',\mu,n_{3},n_{4}), \\
&(m_{1},m_{2})\in
*(n,\mu,m_{3},n_{4}),\, (m_{1}',m_{2}')\in
*(n',\mu,m_{3},n_{4}),
 \\  \text{and} \  &\mathbb{E}\Big(g_{n_{1}}(\omega)g_{n_{2}}(\omega)
 \overline{g_{n_{1}'}}(\omega)\overline{g_{n_{2}'}}(\omega)
\overline{g_{m_{1}}}(\omega)\overline{g_{m_{2}}}(\omega)
g_{m_{1}'}(\omega)g_{m_{2}'}(\omega)\Big)
\neq 0\Bigg\}.
\end{align*}
Then $\#\{S(n_{4},\mu)\}<(N^0)^{3}$.
\label{Lemma:1bi}
\end{lemma}
The proof of Lemma \ref{Lemma:1bi} can be found in Section 5.2.5 of \cite{R-Th}.  Using \eqref{Eqn:NL-neg1-37} and Lemma \ref{Lemma:1bi}, if $\omega\in\Omega_{N,N_{1},\ldots,N_{4},T}$, we can estimate the contribution to \eqref{Eqn:NL-neg1-30} coming from $I_{2}(n_{4},\mu)$ by
\begin{align}
\frac{T^{-\beta}
N^{\frac{1}{2}+\delta}}{(N^0)^{\frac{3}{4}-5\delta-\beta}} &\prod_{j=3}^{4}
\|u_{j}\|_{\frac{1}{2}-\delta,\frac{1}{2}-\delta,T}
\Big(\sum_{|n_{4}| \sim N_{4}}\frac{1}{\langle n_{4}\rangle^{1-2\delta}}\Big)^{\frac{1}{2}}
\notag \\
&\lesssim
\frac{T^{-\beta}
N^{\frac{1}{2}+\delta}}{(N^0)^{\frac{3}{4}-6\delta-\beta-\gamma}} \prod_{j=3}^{4}
\|u_{j}\|_{\frac{1}{2}-\delta,\frac{1}{2}-\delta,T}
\Big(\sum_{|n_{4}| \sim N_{4}}\frac{1}{\langle n_{4}\rangle^{1+2\gamma}}\Big)^{\frac{1}{2}}
\notag \\
&\lesssim
\frac{T^{-\beta}}
{(NN_{1}\cdots N_{4})^{\alpha}}(N_{1}N_{2})^{-\varepsilon}
\|u_{3}\|_{\frac{1}{2}-\delta,\frac{1}{2}-\delta,T}
\|u_{4}\|_{\frac{1}{2}-\delta,\frac{1}{2}-\delta,T}.
\label{Eqn:NL-neg1-38}
\end{align}
Combining \eqref{Eqn:NL-neg1-26}, \eqref{Eqn:NL-neg1-27}, \eqref{Eqn:NL-neg1-28}, \eqref{Eqn:NL-neg1-29a}, \eqref{Eqn:NL-neg1-30}, \eqref{Eqn:NL-neg1-31} and \eqref{Eqn:NL-neg1-38},  if $\omega\in
\widetilde{\Omega}_{T}\cap\Omega_{N,N_{1},\ldots,N_{4},T}$, then the estimate \eqref{Eqn:pre-a-1-1} holds true.

It is straight forward to check that the crucial inequalities in lines \eqref{Eqn:NL-neg1-44} and \eqref{Eqn:NL-neg1-37} remain true (using \eqref{Eqn:NL-neg1-6}) under permutations of the roles of $(n_{1},n_{2},n_{3})$ in the preceding analysis.  In particular by exploiting the restriction \eqref{Eqn:NL-new-2}.  The analysis of case 2.b.1. is complete.

\vspace{0.1in}

 \noindent $\bullet$ \textbf{CASE 2.b.ii.}
One of $u_{i}$ of type (I), $i=1,2,3$, others type (II).

\vspace{0.1in}
We will begin by assuming $u_{1}$ is type (I),
and $u_{2}, u_{3}$ are type (II).  We will discuss modifications for other possibilities afterwards.  In this case we establish the estimate
\begin{align}
\|\mathcal{N}_{-1}|_{2.b.ii.}&(u_{1},u_{2},u_{3},u_{4})\|_{\frac{1}{2}+\delta,-\frac{1}{2}+\delta,T}  \notag \\
&\lesssim T^{-3\beta}\frac{1}{(NN_{1}\cdots N_{4})^{\alpha}N_{1}^{\varepsilon}}\|u_{2}\|_{\frac{1}{2}+\delta,\frac{1}{2}-\delta,T}
\|u_{3}\|_{\frac{1}{2}+\delta,\frac{1}{2}-\delta,T}
\|u_{4}\|_{\frac{1}{2}-\delta,\frac{1}{2}-\delta,T}.
\label{Eqn:pre-b-1}
\end{align}
With the condition \eqref{Eqn:NL-neg1-6} we have
\begin{align}
\frac{|n|^{\frac{1}{2}+\delta}|n_{1}|}
{|n_{1}|^{\frac{1}{2}-\gamma}|n_{2}|^{\frac{1}{2}+\delta-\gamma}
|n_{3}|^{\frac{1}{2}+\delta-\gamma}} &\lesssim  \frac{1}{(N^{0})^{\delta-3\gamma}}
\notag \\
&\lesssim \frac{1}{|n|^{\gamma}}
\frac{1}{(NN_{1}\cdots N_{4})^{\alpha}(N_{1})^{\epsilon}}.
\label{Eqn:NL-neg1-45}
\end{align}
Using \eqref{Eqn:NL-neg1-45}, \eqref{Eqn:pre-b-1} follows from
\begin{align}
\|\Big(\sum_{|n_{1}|\sim N_{1}}
&\frac{|g_{n_{1}}(\omega)|e^{in_{1}x+in_{1}^{3}t}}
{|n_{1}|^{\frac{1}{2}+\gamma}}
\Big)
f_{2}f_{3}u_{4}\|_{-\gamma,-\frac{1}{2}+\delta}
\notag \\ &\lesssim  T^{-\beta}
\|f_{2}\|_{\gamma,\frac{1}{2}-\delta}
\|f_{3}\|_{\gamma,\frac{1}{2}-\delta}
\|u_{4}\|_{\frac{1}{2}-\delta,\frac{1}{2}-\delta}
\label{Eqn:NL-neg1-46}
\end{align}
To establish \eqref{Eqn:NL-neg1-46}, notice that by Lemma \ref{Lemma:lin2} and Lemma \ref{Lemma:prob1}, if $\omega\in\widetilde{\Omega}_{T}$, then
\begin{align}
\Bigg\|\sum_{|n_{1}|\sim N_{1}}
\frac{|g_{n_{1}}(\omega)|e^{in_{1}x+in_{1}^{3}t}}
{|n_{1}|^{\frac{1}{2}+\gamma}}
\Bigg\|_{\frac{\gamma}{2},\frac{1}{2}-\delta,T}
\lesssim T^{-\beta}
\Big(\sum_{|n_{1}|\sim N_{1}}\frac{1}
{|n_{1}|^{1+\gamma-2\varepsilon}}\Big)^{\frac{1}{2}} \lesssim T^{-\beta},
\label{Eqn:NL-neg1-47}
\end{align}
by taking $\varepsilon=\varepsilon(\gamma)$ sufficiently small.  Then using duality, H\"{o}lder's inequality, \eqref{Eqn:X-Str-interp} and \eqref{Eqn:NL-neg1-47},
\begin{align*}
|\int\sum v\cdot \Big(\sum_{|n_{1}|\sim N_{1}}
&\frac{|g_{n_{1}}(\omega)|e^{in_{1}x+in_{1}^{3}t}}
{|n_{1}|^{\frac{1}{2}+\gamma}}
\Big)f_{2}f_{3}u_{4} dxdt|  \\ &\leq  \|v\|_{L^{5}_{x,t}}
\Bigg\|\sum_{|n_{1}|\sim N_{1}}
\frac{|g_{n_{1}}(\omega)|e^{in_{1}x+in_{1}^{3}t}}
{|n_{1}|^{\frac{1}{2}+\gamma}}
\Bigg\|_{L^{5}_{x,t}}
\prod_{j=2,3}\|f_{j}\|_{L^{5}_{x,t}}
\|u_{4}\|_{L^{5}_{x,t}}   \\
&\lesssim
\|v\|_{\frac{\gamma}{2},\frac{1}{2}-\delta}
\Bigg\|\sum_{|n_{1}|\sim N_{1}}
\frac{|g_{n_{1}}(\omega)|e^{in_{1}x+in_{1}^{3}t}}
{|n_{1}|^{\frac{1}{2}+\gamma}}
\Bigg\|_{\frac{\gamma}{2},\frac{1}{2}-\delta}
\prod_{j=2,3}
\|f_{j}\|_{\frac{\gamma}{2},\frac{1}{2}-\delta}
\|u_{4}\|_{\frac{\gamma}{2},\frac{1}{2}-\delta}
\\
&\lesssim
T^{-\beta}
\|v\|_{\gamma,\frac{1}{2}-\delta}
\|f_{2}\|_{\gamma,\frac{1}{2}-\delta}
\|f_{3}\|_{\gamma,\frac{1}{2}-\delta}
\|u_{4}\|_{\frac{1}{2}-\delta,\frac{1}{2}-\delta},
\end{align*}
and \eqref{Eqn:NL-neg1-46} holds for $\omega\in\widetilde{\Omega}_{T}$.

It is easy to verify that the crucial inequality, \eqref{Eqn:NL-neg1-45}, remains true (by \eqref{Eqn:NL-neg1-6}) if we permute the roles of $(n_{1},n_{2},n_{3})$ in the preceding analysis.
The analysis of Case 2.b.ii. is complete.

\vspace{0.1in}

 \noindent $\bullet$ \textbf{CASE 2.b.iii.}
$u_{1}$,$u_{2}$,$u_{3}$ all type (II).

\vspace{0.1in}

In this subcase we establish the deterministic estimate
\begin{align}
\|\mathcal{N}_{-1}|_{2.b.iii.}&(u_{1},u_{2},u_{3},u_{4})
\|_{\frac{1}{2}+\delta,-\frac{1}{2}+\delta,T}  \notag \\
&\lesssim
\frac{1}{(NN_{1}\cdots N_{4})^{\alpha}}
\prod_{j=1}^{3}\|u_{j}\|_{\frac{1}{2}+\delta,\frac{1}{2}-\delta,T}
\|u_{4}\|_{\frac{1}{2}-\delta,\frac{1}{2}-\delta,T}.
\label{Eqn:pre-c-1}
\end{align}
Using \eqref{Eqn:NL-neg1-6} we find
\begin{align*}
\frac{|n|^{\frac{1}{2}+\delta}|n_{1}|}
{|n_{1}|^{\frac{1}{2}+\delta-\gamma}
|n_{2}|^{\frac{1}{2}+\delta-\gamma}|n_{3}|^{\frac{1}{2}+\delta-\gamma}}\lesssim \lesssim \frac{1}{(N^0)^{2\delta-3\gamma}}
\lesssim \frac{1}{|n|^{\gamma}(NN_1\cdots N_4)^{\alpha}}.
\end{align*}
Then \eqref{Eqn:pre-c-1} follows from
\begin{align*}
\|f_{1}f_{2}f_{3}u_{4}\|_{-\gamma,-\frac{1}{2}+\delta} \lesssim  \prod_{j=1}^{3}\|f_{j}\|_{\gamma,\frac{1}{2}-\delta}\|u_{4}\|_{\frac{1}{2}-\delta,\frac{1}{2}-\delta}.
\end{align*}
By duality, the last estimate is equivalent to
\begin{align}
|\sum\int v \cdot f_{1}f_{2}f_{3}u_{4}dxdt| \lesssim  \|v\|_{\gamma,\frac{1}{2}-\delta}\prod_{j=1}^{3}
\|f_{j}\|_{\gamma,\frac{1}{2}-\delta}
\|u_{4}\|_{\frac{1}{2}-\delta,\frac{1}{2}-\delta}.
\label{Eqn:b21}
\end{align}
We obtain \eqref{Eqn:b21} with H\"{o}lder's inequality and \eqref{Eqn:X-Str-interp}
\begin{align*}
|\sum\int v \cdot f_{1}f_{2}f_{3}u_{4}dxdt| &\leq \|v\|_{L^{5}_{x,t}}\prod_{j=1}^{3}\|f_{j}\|_{L^{5}_{x,t}}
\|u_{4}\|_{L^{5}_{x,t}}   \\
&\lesssim  \|v\|_{\gamma,\frac{1}{2}-\delta}
\prod_{j=1}^{3}\|f_{j}\|_{\gamma,\frac{1}{2}-\delta}
\|u_{4}\|_{\frac{1}{2}-\delta,\frac{1}{2}-\delta}.
\end{align*}
This concludes the justification of \eqref{Eqn:pre-c-1}, and case 2.b. is complete.

\vspace{0.1in}

\noindent $\bullet$ \textbf{CASE 2.c.}
$N^{3}\ll N^{0}$ and $N^{2}N^{3}N^{4}\gtrsim
N^{0}N^{1}|n^{0}+n^{1}|$.

\vspace{0.1in}

Observe that the assumptions of this case provide the additional condition
\begin{align}
N^{3}N^{4}\gtrsim N^0,
\label{Eqn:NL-neg1-6d}
\end{align}
otherwise we would find $(N^0)^2 \lesssim N^{0}N^{1}|n^{0}+n^{1}| \lesssim N^{2}N^{3}N^{4}\leq N^{0}N^{3}N^4 \ll (N^0)^2$, a contradiction.  In this region, we also have
\begin{align}
N_{2}N_{3}N_{4} \gtrsim N^{2}N^{3}N^{4} \gtrsim N^{0}N^{1}|n^{0}+n^{1}| \sim (N^{0})^{2}|n^{0}+n^{1}|.
\label{Eqn:NL-neg1-6c}
\end{align}
Then we find, by \eqref{Eqn:NL-neg1-6c},
\begin{align}
\frac{|n|^{\frac{1}{2}+\delta}|n_{1}|}
{|n_{1}|^{\frac{1}{2}-2\delta}|n_{2}n_{3}n_{4}|^
{\frac{1}{2}-2\delta}} &\lesssim \frac{(N^0)^{7\delta}}{|n^{0}+n^{1}|^{\frac{1}{2}-2\delta}} \lesssim  \frac{1}{(N^0)^{\delta}},
\label{Eqn:b-1}
\end{align}
unless $|n^{0}+n^{1}|\ll (N^0)^{\frac{16\delta}{1-4\delta}}$.  If \eqref{Eqn:b-1} holds, we can proceed with (a modification of) the method used in case 2.b.iii to establish
\begin{align*}
\|\mathcal{N}_{-1}|_{\text{2.c.iii.a.1.}}(u_{1},u_{2},u_{3},u_{4})
\|_{\frac{1}{2}+\delta,-\frac{1}{2}+\delta,T}   \lesssim
\frac{1}{(NN_{1}\cdots N_{4})^{\alpha}}
\prod_{j=1}^{4}\|u_{j}\|
_{\frac{1}{2}-\delta,\frac{1}{2}-\delta,T}.
\end{align*}
We therefore assume that
\begin{align}
|n^0 + n^{1}|\ll (N^0)^{\frac{16\delta}{1-4\delta}},
\label{Eqn:b-3}
\end{align}
for the remainder of case 2.c.

\vspace{0.1in}

\noindent $\bullet$ \textbf{CASE 2.c.i.}
$u_{1}$ type (I) and two of $u_{2},u_{3},u_{4}$ type (I).

Let us assume that
$u_{1},u_{2}$ and $u_{3}$ are all of type (I).  We will discuss the other possibilities afterwards.
In this case we establish
\begin{align}
\|\mathcal{N}_{-1}|_{\text{case 2.c}}(u_{1},u_{2}&,u_{3},u_{4})\|
_{\frac{1}{2}+\delta,-\frac{1}{2}+\delta}
 \notag \\
 &\lesssim   T^{-3\beta}\frac{1}{(NN_{1}\cdots N_{4})^{\alpha}}\frac{1}
 {(N_{1}N_{2}N_{3})^{\varepsilon}}
\|u_{4}\|_{\frac{1}{2}-\delta,\frac{1}{2}-\delta,T}.
\label{Eqn:NL-neg1-1cii}
\end{align}

Using the representation \eqref{Eqn:NL-neg1-24} for $u_{4}$, we apply the Minkowski inequality in $\lambda_{4}$ to find
\begin{align}
\|\mathcal{N}_{-1}|_{\text{case 2.c}}(&u_{1},u_{2},u_{3},u_{4})\|
_{\frac{1}{2}+\delta,-\frac{1}{2}+\delta}  \notag \\ &\lesssim  \|u_{4}\|_{\frac{1}{2}+\delta,\frac{1}{2}-\delta}(N^0)^{\delta}
\notag \\
&\cdot\sup_{\lambda_{4},\mu \ll (N^0)^{2}}\Big|
\sum_{|n|\sim N}|n|^{1+2\delta}\Big| \sum_{*(n,\mu+\lambda_{4})\cap\text{case 2.c.}}(in_{1})\frac{g_{n_{1}}(\omega)g_{n_{2}}(\omega)g_{n_{3}}(\omega)}{\langle n_{1}\rangle\langle n_{2}\rangle\langle n_{3}\rangle}a_{\lambda_{4}}(n_{4})\Big|^{2}\Big|^{\frac{1}{2}} \notag \\
&\lesssim  \|u_{4}\|_{\frac{1}{2}+\delta,\frac{1}{2}-\delta}
(N^0)^{\delta}N^{\frac{1}{2}+\delta}
\notag \\
&\ \ \ \ \ \cdot\sup_{\mu \ll (N^0)^{2}}\Big|
\sum_{|n|\sim N}\Big| \sum_{*(n,\mu)\cap\text{case 2.c.}}(in_{1})\frac{g_{n_{1}}(\omega)g_{n_{2}}(\omega)g_{n_{3}}(\omega)}{\langle n_{1}\rangle\langle n_{2}\rangle\langle n_{3}\rangle}a_{n_{4}}\Big|^{2}\Big|^{\frac{1}{2}}
  \label{Eqn:caseb-reducedbyMIN},
\end{align}
where $\sum_{n_{4}}|n_{4}|^{1-2\delta}|a_{n_{4}}|^{2} = 1$.  We have dropped the dependence on $\lambda_{4}$ in the previous expression; this is justified a posteriori by obtaining estimates which are uniform in $\lambda_{4}$.  For each fixed $\mu$, we consider
\begin{align}
\sum_{|n|\sim N}\Big|\sum_{*(n,\mu)\cap\text{case 2.c.}}(in_{1})&\frac{g_{n_{1}}(\omega)g_{n_{2}}(\omega)g_{n_{3}}(\omega)}{\langle n_{1}\rangle\langle n_{2}\rangle}a_{n_{4}}\Big|^{2}
 \notag \\ &=  \sum_{|n|\sim N} |\sum_{n_{4}}\sigma_{n,n_{4}}^{\mu}|n_{4}|^{\frac{1}{2}-\delta}a_{n_{4}}|^{2}
\label{Eqn:caseb-matrix}
\end{align}
where $\sigma_{n,n_{4}}^{\mu}$ is the $(n,n_{4})$ entry of a matrix $\sigma_{\mu}$ (for $\mu$ fixed) with columns indexed by $|n_{4}|\sim N_{4}$, and rows indexed by $|n|\sim N$.  That is, the entries of this matrix are given by
\begin{align*}
\sigma_{n,n_{4}}^{\mu} =  \sum_{*(n,n_{4},\mu)\cap\text{case 2.c.}}\frac{in_{1} g_{n_{1}}(\omega)g_{n_{2}}(\omega)g_{n_{3}}(\omega)}{\langle n_{1}\rangle\langle n_{2}\rangle\langle n_{3}\rangle|n_{4}|^{\frac{1}{2}-\delta}}
\end{align*}
Then by Lemma \ref{Lemma:matrixnorm}
\begin{align}
\eqref{Eqn:caseb-matrix} &\lesssim  \|(\sigma^{\mu}_{n,n_{4}})^{*}\sigma^{\mu}_{n,n_{4}}\|
=  \|\sigma^{\mu}_{n,n_{4}}(\sigma^{\mu}_{n,n_{4}})^{*}\| \notag \\
&\leq  \sup_{|n|\sim N} \sum_{n_{4}}|\sigma_{n,n_{4}}^{\mu}|^{2}  +  \Big(\sum_{n\neq n'}\big|\sum_{n_{4}}\sigma_{n,n_{4}}^{\mu}
\overline{\sigma_{n',n_{4}}^{\mu}}\big|^{2}\Big)^{\frac{1}{2}}  \notag \\
&=  I_{1}(\mu)  +  I_{2}(\mu).
\label{Eqn:caseb-reducedtoI}
\end{align}
To estimate $\displaystyle I_{1}(\mu) =  \sup_{|n|\sim N}\sum_{n_{4}}|\sigma_{n,n_{4}}^{\mu}|^{2}$,
we consider
$F_{n,n_{4},\mu}(\omega) := \sigma_{n,n_{4}}^{\mu}(\omega)$, then by Lemma \ref{Lemma:prob2},
\begin{align*}
\|F_{n,n_{4},\mu}\|_{L^{p}(\Omega)} \leq p^{\frac{3}{2}}\|F_{n,n_{4},\mu}\|_{L^{2}(\Omega)},
\end{align*}
for each $2<p<\infty$.  Applying Lemma \ref{Lemma:prob3} it follows that
\begin{align*}
P(|F_{n,n_{4},\mu}(\omega)|\geq \lambda) \leq e^{-c\|F_{n,n_{4},\mu}\|^{\frac{2}{3}}_{L^{2}(\Omega)}\lambda^{\frac{2}{3}}}.
\end{align*}
Taking $\displaystyle \lambda = \|F_{n,n_{4},\mu}\|_{L^{2}(\Omega)}(N^{0})^{\frac{3\beta}{2}}T^{-\frac{3\beta}{2}}$, we have
\begin{align*}
P(|F_{n,n_{4},\mu}(\omega)|\geq \|F_{n,n_{4},\mu}\|_{L^{2}(\Omega)}(N^0)^{\frac{3\beta}{2}}
T^{-\frac{3\beta}{2}}) \leq e^{-c\frac{(N^0)^{\beta}}{T^{\beta}}}.
\end{align*}
Then letting $\Omega_{n,N_{1},N_{2},N_{3},n_{4},\mu,T}:=
\{|F_{n,n_{4},\mu}(\omega)|\geq \|F_{n,n_{4},\mu}\|_{L^{2}(\Omega)}(N^0)^{\frac{3\beta}{2}}
T^{-\frac{3\beta}{2}}\}$ and $$\Omega_{N,N_{1},N_{2},N_{3},N_{4},T}:=\bigcap_{|n|\sim N,|n_{4}|\sim N_{4}, |\mu|<(N^0)^{2}}\Omega_{n,N_{1},N_{2},N_{3},n_{4},\mu,T},$$ we have
\begin{align*}
P(\Omega_{N,N_{1},N_{2},N_{3},N_{4},T}^{c})
\leq \sum_{|n|\sim N,|n_{4}|\sim N_{4}, |\mu|<(N^0)^{2}} e^{-c\frac{(N^0)^{\beta}}{T^{\beta}}} &\lesssim (N^0)^{4} e^{-c\frac{(N^0)^{\beta}}{T^{\beta}}} \lesssim  (N^0)^{0-}e^{-\frac{c'}{T^{\beta}}}.
\end{align*}
Then for each $|\mu|\ll (N^0)^{2}$, if
$\omega\in\Omega_{N,N_{1},N_{2},N_{3},N_{4},T}$,
\begin{align}
I_{1}(\mu) &\lesssim  \sup_{|n|\sim N}\sum_{|n_{4}|\sim N_4}|F_{n,n_{4},\mu}(\omega)|^{2} \lesssim  (N^0)^{3\beta}T^{-3\beta}\sup_{|n|\sim N}
\sum_{|n_{4}|\sim N_{4}} \|F_{n,n_{4},\mu}\|^{2}_{L^{2}(\Omega)}.
\label{Eqn:2ciiA}
\end{align}
We compute that
\begin{align}
\|F_{n,n_{4},\mu}\|^{2}_{L^{2}(\Omega)}
&= \mathbb{E}\Bigg(\Bigg|\sum_{(n_{1},n_{2},n_{3})\in
*(n,n_{4},\mu)\cap case 2(c)}
\frac{(in_{1})g_{n_{1}}(\omega)
g_{n_{2}}(\omega)g_{n_{3}}(\omega)}
{\langle n_{1}\rangle \langle n_{1}\rangle
\langle n_{1}\rangle |n_{4}|^{\frac{1}{2}-\delta}}\Bigg|^{2}\Bigg) \notag \\
&\lesssim \sum_{\substack{(n_{1},n_{2},n_{3})\in
*(n,n_{4},\mu)\cap \text{case 2.c.}\\
(m_{1},m_{2},m_{3})\in
*(n,n_{4},\mu)\cap \text{case 2.c.}}}
\frac{1}{(N_{2}N_{3})^{2}N_{4}^{1-2\delta}}
\notag \\ &\ \ \ \ \ \ \cdot|\mathbb{E}(g_{n_{1}}(\omega)
g_{n_{2}}(\omega)g_{n_{3}}(\omega)\overline{g_{m_{1}}}(\omega)
\overline{g_{m_{2}}}(\omega)\overline{g_{m_{3}}}(\omega))|
\label{Eqn:2ciiB}
\end{align}
To bound this sum we use the following lemma.

\begin{lemma}
Let
\begin{align*}
S(n,\mu):=\Bigg\{(n_{1},&n_{2},n_{3},n_{4},m_{1},
m_{2},m_{3})\bigg| \,|n_{4}|\sim N_{4}, \\ &(n_{1},n_{2},n_{3})\in *(n,n_{4},\mu), \
(m_{1},m_{2},m_{3})\in *(n,n_{4},\mu),
 \\
 &(n,n_{1},n_{2},n_{3},n_{4}),
(n,m_{1},m_{2},m_{3},n_{4}) \ \text{ordered such that} \ n^0\neq -n^1,
 \\
 &\mathbb{E}(g_{n_{1}}(\omega)g_{n_{2}}(\omega)g_{n_{3}}(\omega)
\overline{g_{m_{1}}}(\omega)\overline{g_{m_{2}}}
(\omega)\overline{g_{m_{3}}}(\omega))
\neq 0\Bigg\}.
\end{align*}
Then $\#\{S(n,\mu)\}<\min(N_{1}N_{2},N_{1}N_{3},N_{2}N_{3})$.
\label{Lemma:2ci}
\end{lemma}

The proof of Lemma \ref{Lemma:2ci} can be found in Section 5.2.5 of \cite{R-Th}.  By combining \eqref{Eqn:caseb-matrix}-\eqref{Eqn:2ciiB} and Lemma \ref{Lemma:2ci}, the contribution to  \eqref{Eqn:caseb-reducedbyMIN} from $I_{1}(\mu)$ is bounded by
\begin{align*}
\|u_{4}\|_{\frac{1}{2}-\delta,\frac{1}{2}-\delta}
&(N^0)^{\delta}N^{\frac{1}{2}+\delta} \sup_{|\mu|<(N^0)^{2}}I_{1}(\mu)
\lesssim \frac{T^{-\frac{3\beta}{2}}(N^0)^{\delta+\frac{\beta}{2}}
N^{\frac{1}{2}+\delta}}
{(N_{2}N_{3})^{\frac{1}{2}}N_{4}^{\frac{1}{2}-\delta}}\|u_{4}\|_{\frac{1}{2}-\delta,\frac{1}{2}-\delta}
\\
&\lesssim \frac{T^{-\frac{3\beta}{2}}(N^0)^{2\delta+\frac{\beta}{2}}
N^{\frac{1}{2}+\delta}}
{N^0}\|u_{4}\|_{\frac{1}{2}-\delta,\frac{1}{2}-\delta}
\\
&\lesssim \frac{T^{-\frac{3\beta}{2}}}{(NN_{1}\cdots N_{4})^{\alpha}}
(N_{1}N_{2}N_{3})^{-\varepsilon}\|u_{4}\|_{\frac{1}{2}-\delta,\frac{1}{2}-\delta}.
\end{align*}

It remains to control the contribution to \eqref{Eqn:caseb-reducedbyMIN} from $I_{2}(\mu)$.
Consider
\begin{align}
I_{2}(\mu) = \Big( \sum_{n\neq n'}\Big|\sum_{|n_{4}|\sim N_{4}}\sigma_{n,n_{4}}^{\mu}
\overline{\sigma_{n',n_{4}}^{\mu}}\Big|^{2}\Big)^{\frac{1}{2}}
= \Big( \sum_{n\neq n'}|G_{n,n',\mu}(\omega)|^{2}\Big)^{\frac{1}{2}},
\label{Eqn:caseb-I2}
\end{align}
where, for each fixed $n,n',\mu$, we have taken
\begin{align*}
G_{n,n',\mu}(\omega) := \sum_{\substack{|n_{4}|\sim N_{4}\\(n_{1},n_{2},n_{3})\in*(n,n_{4},\mu)\cap
\text{case 2.c.}\\(n_{1}',n_{2}',n_{3}')\in
*(n',n_{4},\mu)\cap
\text{case 2.c.}}}
\frac{-n_{1}n_{1}'g_{n_{1}}(\omega)
g_{n_{2}}(\omega)g_{n_{3}}(\omega)\overline{g_{n_{1}'}}(\omega)
\overline{g_{n_{2}'}}(\omega)\overline{g_{n_{3}'}}(\omega)}{\langle n_{1}\rangle\langle n_{2}\rangle\langle n_{3}\rangle\langle n_{1}'\rangle\langle n_{2}'\rangle\langle n_{3}'\rangle\langle n_{4}\rangle^{1-2\delta}}.
\end{align*}
By Lemma \ref{Lemma:prob2} we have
\begin{align*}
\|G_{n,n',\mu}\|_{L^{p}(\Omega)} \leq p^{3}\|G_{n,n',\mu}\|_{L^{2}(\Omega)},
\end{align*}
for each $2<p<\infty$.  With Lemma \ref{Lemma:prob3} it follows that
\begin{align*}
P(|G_{n,n',\mu}(\omega)|\geq \lambda) \leq e^{-c\|G_{n,n',\mu}\|^{-\frac{1}{3}}_{L^{2}(\Omega)}\lambda^{\frac{1}{3}}}.
\end{align*}
Taking $\displaystyle \lambda = \|G_{n,n',\mu}\|_{L^{2}(\Omega)}(N^0)^{3\beta}T^{-3\beta}$, we have
\begin{align*}
P(|G_{n,n',\mu}(\omega)|\geq \|G_{n,n',\mu}\|_{L^{2}(\Omega)}(N^0)^{3\beta}T^{-3\beta}) \leq e^{-c\frac{(N^0)^{\beta}}{T^{\beta}}}.
\end{align*}
Then letting $\Omega_{n,n',N_{1},N_{2},N_{3},N_{4},\mu,T}:=
\{|G_{n,n',\mu}(\omega)|\geq \|G_{n,n',\mu}\|_{L^{2}(\Omega)}(N^0)^{3\beta}
T^{-3\beta}\}$ and $$\Omega_{N,N_{1},N_{2},N_{3},N_{4},T}
:=\bigcap_{|n|,|n'|\sim N, |\mu|<(N^0)^{2}}\Omega_{n,n',N_{1},N_{2},N_{3},N_{4},\mu,T},$$ we have
\begin{align*}
P(\Omega_{N,N_{1},N_{2},N_{3},N_{4},T}^{c})
\leq \sum_{|n|,|n'|\sim N, |\mu|<(N^0)^{2}} e^{-c\frac{(N^0)^{\beta}}{T^{\beta}}} &\lesssim \
(N^0)^{4}e^{-c\frac{(N^0)^{\beta}}{T^{\beta}}} \\
&\lesssim  (N^0)^{0-}e^{-\frac{c'}{T^{\beta}}},
\end{align*}
for some $c'>0$.  Then for each $|\mu|\ll (N^0)^{2}$, if $\omega\in\Omega_{N,N_{1},N_{2},N_{3},N_{4},T}$,
\begin{align}
I_{2}(\mu) \lesssim  &\Big(\sum_{n\neq n'}|G_{n,n',\mu}(\omega)|^{2}\Big)^{\frac{1}{2}}  \notag \leq
T^{-3\beta}(N^0)^{3\beta}\Big(\sum_{n\neq n'}\|G_{n,n',\mu}\|_{L^{2}(\Omega)}^{2}\Big)^{\frac{1}{2}}
\end{align}

We compute that
\begin{align}
&\sum_{n\neq n'}\|G_{n,n',\mu}\|_{L^{2}(\Omega)}^{2}
\notag \\
 &= \sum_{n\neq n'} \mathbb{E}\Bigg(\Bigg|
\sum_{\substack{|n_{4}|\sim N_{4}\\(n_{1},n_{2},n_{3})\in*(n,n_{4},\mu)\cap
\text{case 2.c.}\\(n_{1}',n_{2}',n_{3}')\in
*(n',n_{4},\mu)\cap
\text{case 2.c.}}}
\frac{-n_{1}n_{1}'g_{n_{1}}(\omega)
g_{n_{2}}(\omega)g_{n_{3}}(\omega)\overline{g_{n_{1}'}}(\omega)
\overline{g_{n_{2}'}}(\omega)\overline{g_{n_{3}'}}(\omega)}{\langle n_{1}\rangle\langle n_{2}\rangle\langle n_{3}\rangle\langle n_{1}'\rangle\langle n_{2}'\rangle\langle n_{3}'\rangle\langle n_{4}\rangle^{1-2\delta}}\Bigg|^{2}\Bigg)
 \notag \\
&\lesssim \sum_{\substack{|n|,|n'|\sim N, |n_{4}|,|m_{4}|\sim N_{4}\\
(n_{1},n_{2},n_{3})\in *(n,n_{4},\mu)
\\(n_{1}',n_{2}',n_{3}')\in *(n',n_{4},\mu)
\\(m_{1},m_{2},m_{3})\in *(n,m_{4},\mu)
\\(m_{1}',m_{2}',m_{3}')\in *(n',m_{4},\mu)}}
\frac{1}{(N_{2}N_{3})^{4}N_{4}^{2-4\delta}}
 \mathbb{E}(g_{n_{1}}(\omega)
g_{n_{2}}(\omega)g_{n_{3}}(\omega)\overline{g_{n_{1}'}}(\omega)
\overline{g_{n_{2}'}}(\omega)\overline{g_{n_{3}'}}(\omega)
\notag  \\ &\quad \quad \quad \quad
\quad \quad
\quad \quad \quad \quad \cdot \overline{g_{m_{1}}}
(\omega)\overline{g_{m_{2}}}(\omega)
\overline{g_{m_{3}}}(\omega)g_{m_{1}'}(\omega)
g_{m_{2}'}(\omega)g_{m_{3}'}(\omega))
\label{Eqn:bounds-3}
\end{align}
To control this sum we establish the following lemma.

\begin{lemma}
Let
\begin{align*}
S(\mu):=
\Bigg\{(&n,n',n_{1},n_{2},n_{3},n_{1}',n_{2}',n_{3}',
m_{1},m_{2},m_{3},m_{1}',m_{2}',m_{3}')\Bigg| \\ &|n|,|n'|\sim N, |n_{4}|,|m_{4}|\sim N_{4}, (n_{1},n_{2},n_{3})\in *(n,n_{4},\mu), \\ &(n_{1}',n_{2}',n_{3}')\in *(n',n_{4},\mu),
(m_{1},m_{2},m_{3})\in *(n,m_{4},\mu) \\ & (m_{1}',m_{2}',m_{3}')\in *(n',m_{4},\mu), \text{with}\ n^0\neq -n^1 \ \text{in all quintuples,}\\
& |n^0+n^1|\ll (N^0)^{\frac{16\delta}{1-4\delta}}, \ \text{and}  \\
&\mathbb{E}(g_{n_{1}}(\omega)
g_{n_{2}}(\omega)g_{n_{3}}(\omega)\overline{g_{n_{1}'}}(\omega)
\overline{g_{n_{2}'}}(\omega)\overline{g_{n_{3}'}}(\omega)
\notag  \\ &\quad \quad \quad \quad
\quad \quad \quad \quad
 \cdot \overline{g_{m_{1}}}(\omega)\overline{g_{m_{2}}}(\omega)
\overline{g_{m_{3}}}(\omega)g_{m_{1}'}(\omega)
g_{m_{2}'}(\omega)g_{m_{3}'}(\omega))
\neq 0\Bigg\}.
\end{align*}
Then $\#\{S(\mu)\}\lesssim (N^0)^{3+\frac{32\delta}{1-4\delta}}$.
\label{Lemma:prob-bound-3}
\end{lemma}
The proof of Lemma \ref{Lemma:prob-bound-3} can be found in Section 5.2.5 of \cite{R-Th}.  Using \eqref{Eqn:NL-neg1-6} we have
\begin{align}
\frac{1}{(N_{2}N_{3})^{4}N_{4}^{2-4\delta}}
= \frac{N_4^{2+4\delta}}{(N_2N_3N_4)^4}
\lesssim  \frac{N_4^{2+4\delta}}{(N^0)^{8}}.
\lesssim \frac{1}{(N^0)^{6-4\delta}}
\label{Eqn:bounds-new-2}
\end{align}
By combining \eqref{Eqn:bounds-new-2} with Lemma \ref{Lemma:prob-bound-3} we have
\begin{align}
\eqref{Eqn:bounds-3}\lesssim  \frac{1}{(N^0)^{3-4\delta-\frac{32\delta}{1-4\delta}}}
\label{Eqn:bounds-new-3}
\end{align}

Then from \eqref{Eqn:bounds-new-3} we can estimate the contribution to \eqref{Eqn:caseb-reducedbyMIN} coming from $I_{2}(\mu)$ by
\begin{align*}
\|u_{4}\|_{\frac{1}{2}-\delta,\frac{1}{2}-\delta}
&(N^0)^{\delta}N^{\frac{1}{2}+\delta} \sup_{|\mu|<(N^0)^{2}}I_{2}(\mu)
\lesssim
\|u_{4}\|_{\frac{1}{2}-\delta,\frac{1}{2}-\delta} \frac{T^{-\frac{3\beta}{2}}(N^0)^{\delta+\frac{3\beta}{2}}
N^{\frac{1}{2}+\delta}}{(N^0)^{\frac{3}{4}-\delta-\frac{8\delta}{1-4\delta}}}
\\&\lesssim \frac{T^{-\frac{3\beta}{2}}}{(NN_{1}\cdots N_{4})^{\alpha}}
(N_{1}N_{2}N_{3})^{-\varepsilon}\|u_{4}\|_{\frac{1}{2}-\delta,\frac{1}{2}-\delta},
\end{align*}
for $\delta,\beta,\alpha>0$ sufficiently small.
It is clear that the previous analysis applies upon permutation of the variables $n_{2},n_{3}$ and $n_{4}$, as we did not use the ordering $N_{2}\geq N_{3}\geq N_{4}$ in this case (see Remark 5.4 in \cite{R-Th}).  The analysis of case 2.c.i. is complete.

\vspace{0.1in}

\noindent $\bullet$ \textbf{CASE 2.c.ii.} $u_{1}$ type (II), and $u_{2},u_{3},u_{4}$ type (I).

\vspace{0.1in}

In this case we proceed precisely as in case 2.c.ii, swapping the roles of $n_{1}$ and $n_{4}$.  The analysis requires modification in the lines \eqref{Eqn:2ciiB} and \eqref{Eqn:bounds-3}, where we need to include the factor
$\displaystyle \frac{N_{1}^{1-2\delta}}
{(N_{2}N_{3}N_{4})^{2}}$ instead of $\displaystyle\frac{1}
{(N_{2}N_{3})^{2}N_{4}^{1-2\delta}}$.  In order to estimate \eqref{Eqn:2ciiB}, by $N_{2}\geq N_{3}\geq N_{4}$ and \eqref{Eqn:NL-neg1-6c}, we find
\begin{align*}
\frac{1}{N_{2}^{2}N_{3}N_{4}}\leq \frac{1}{(N_{2}N_{3}N_{4})^{\frac{4}{3}}} \lesssim \frac{1}{(N^0)^{\frac{8}{3}}},
\end{align*}
and we have
\begin{align*}
\frac{(N^{0})^{2\delta}N^{1+2\delta}N_{1}^{1-2\delta}}
{N_{2}^{2}N_{3}N_{4}} \lesssim  \frac{1}{(N^0)^{\frac{2}{3}-2\delta}}.
\end{align*}
By combining this inequality with Lemma \ref{Lemma:prob-bound-3} (with $n_{1}$ and $n_{4}$ swapped) we can estimate the contribution from $I_{1}(\mu)$ as we did in case 2.c.i.

In the modification of \eqref{Eqn:bounds-3}, we consider
\begin{align*}
\frac{(N_{1})^{2-4\delta}}{(N_{2}N_{3}N_{4})^{4}}\lesssim \frac{(N_{1})^{2-4\delta}}{(N^0)^{8}}
 \lesssim \frac{1}{(N^0)^{6+4\delta}},
\end{align*}
which is precisely the conclusion we reached in case 2.c.i. These are the only modifications required to estimate the contribution from $I_{2}(\mu)$, and the analysis of case 2.c.ii. is complete.

\vspace{0.1in}

\noindent $\bullet$ \textbf{CASE 2.c.iii:} Two type (I), two type (II).

\vspace{0.1in}

We will consider further subcases.

\vspace{0.1in}

\hspace{0.4in} $\bullet$ \textbf{CASE 2.c.iii.a:} $u_{1}$ type (II).

\vspace{0.1in}

\hspace{0.8in} $\bullet$ \textbf{CASE 2.c.iii.a.1:} $u_{4}$ type (I).

\hspace{0.8in} $\bullet$ \textbf{CASE 2.c.iii.a.2:} $u_{4}$ type (II).

\vspace{0.1in}

\hspace{0.4in} $\bullet$ \textbf{CASE 2.c.iii.b:} $u_{1}$ type (I).

\vspace{0.1in}

\noindent We proceed with the analysis of each subcase.

\vspace{0.1in}

\noindent $\bullet$ \textbf{CASE 2.c.iii.a:} $u_{1}$ type (II).

\vspace{0.1in}

\noindent $\bullet$ \textbf{CASE 2.c.iii.a.1:} $u_{4}$ type (I).

Let us assume that $u_{1}, u_{2}$ are of type (II), and $u_{3}, u_{4}$ are of type (I).  It is easily verified (a posteriori) that the analysis of this subcase is symmetric with respect to the functions $u_{2}$ and $u_{3}$, and
therefore, the preceding assumption holds without loss of generality.

In this case we exploit one more condition which restricts the size of $N_{4}$.  Specifically, we notice that if $N_{4}\geq (N^0)^{\frac{2}{3}+5\delta}$, then we have, using $N_2\geq N_3 \geq N_4$,
\begin{align}
\frac{|n|^{\frac{1}{2}+\delta}|n_{1}|}
{|n_{1}|^{\frac{1}{2}-2\delta}|n_{2}n_{3}n_{4}|^
{\frac{1}{2}-2\delta}} &\lesssim \frac{(N^0)^{1+3\delta}}
{(N^0)^{3(\frac{1}{2}-2\delta)(\frac{2}{3}+5\delta)}} \lesssim  \frac{1}{(N^0)^{\frac{\delta}{2}- 30\delta^{2}}}.
\label{Eqn:case-c-2}
\end{align}
Once again, if \eqref{Eqn:case-c-2} holds, we can proceed with (a modification of) the method used in case 2.b.iii to establish
\begin{align*}
\|\mathcal{N}_{-1}|_{\text{case 2.c.}}(u_{1},u_{2},u_{3},u_{4})
\|_{\frac{1}{2}+\delta,-\frac{1}{2}+\delta,T}   \lesssim
\frac{1}{(NN_{1}\cdots N_{4})^{\alpha}}
\prod_{j=1}^{4}\|u_{j}\|
_{\frac{1}{2}-\delta,\frac{1}{2}-\delta,T}.
\end{align*}
We therefore assume for the remainder of this case that
\begin{align}
N_{4}\ll (N^0)^{\frac{2}{3}+5\delta}.
\label{Eqn:case-c-3}
\end{align}
We consider
\begin{align}
\frac{N^{\frac{1}{2}+\delta}N_{1}}
{(N_{1}N_{2})^{\frac{1}{2}+\frac{11\delta}{12}}
(N_{3}N_{4})^
{\frac{1}{2}-\frac{\delta}{12}}}
&\leq
\frac{N^{\frac{1}{2}+\delta}N_{1}
(N_{3}N_{4})^{\delta}}
{(N_{1}N_{2}N_{3}N_{4})^
{\frac{1}{2}+\frac{11\delta}{12}}}
\leq \frac{(N^0)^{1+\frac{13\delta}{12}}N_{4}^{\delta}}
{(N_{2}N_{3}N_{4})^
{\frac{1}{2}+\frac{11\delta}{12}}} \notag \\
&\leq \frac{(N^0)^{1+\frac{13\delta}{12}}(N^0)^
{(\frac{2}{3}+5\delta)\delta}}
{(N^0)^
{1+\frac{11\delta}{6}}}
\leq
\frac{1}{(N^0)^
{\frac{\delta}{12}-5\delta^{2}}}
\label{Eqn:case-c-4}
\end{align}
By using \eqref{Eqn:case-c-4} and (a modification of) the methods of case 2.b.iii. we establish
\begin{align}
\|\mathcal{N}_{-1}|_{\text{2.c.}}&(u_{1},u_{2},u_{3},u_{4})\|_{\frac{1}{2}+\delta,-\frac{1}{2}+\delta,T}  \notag \\
&\lesssim \frac{T^{-\beta}}
{(NN_{1}\cdots N_{4})^{\alpha}}(N_{3}N_{4})^{-\varepsilon}
\|u_{1}\|_{\frac{1}{2}+\delta,\frac{1}{2}-\delta,T}
\|u_{2}\|_{\frac{1}{2}+\delta,\frac{1}{2}-\delta,T},
\label{Eqn:pre-a-1}
\end{align}
and the analysis of case 2.c.iii.a.1. is complete.

\vspace{0.1in}

\noindent $\bullet$ \textbf{CASE 2.c.iii.a.2:} $u_{4}$ type (II).

In this case we can obtain a stronger restriction on $N_{4}$.  More precisely, we have
\begin{align}
\frac{|n|^{\frac{1}{2}+\delta}|n_{1}|}{|n_{1}|^{\frac{1}{2}+\delta-\gamma}
|n_{2}|^{\frac{1}{2}-\gamma}|n_{3}|^{\frac{1}{2}-\gamma}
|n_{4}|^{\frac{1}{2}+\delta-\gamma}}
=\frac{|n|^{\frac{1}{2}+\delta}|n_{1}|^{\frac{1}{2}-\delta+\gamma}}
{|n_{2}n_{3}n_{4}|^{\frac{1}{2}-\gamma}
|n_{4}|^{\delta}}
\lesssim \frac{(N^{0})^{3\gamma}}{N_{4}^{\delta}} \lesssim \frac{1}{(N^0)^{\gamma}},
\label{Eqn:2B-1}
\end{align}
unless $N_{4}^{\delta}\ll (N^0)^{4\gamma}$.  If \eqref{Eqn:2B-1} holds, we can proceed with a straightforward modification of the method in case 1.b.ii.  Therefore, by taking $\gamma=\gamma(\delta)>0$ sufficiently small, we may assume that
\begin{align}
N_{4}\ll (N^0)^{\delta},
\label{Eqn:2B-2}
\end{align}
for the remainder of this case.

\vspace{0.1in}

Given the defining condition of case 2.c.iii., we must have $u_{1}, u_{4}$ of type (II), and $u_{2}, u_{3}$ of type (I).  The analysis of this case closely follows the method of case 2.b.i.  Indeed, the analysis is identical until the line \eqref{Eqn:NL-neg1-30}, where, due to the assumption
$u_{1},u_{4} \in X^{\frac{1}{2}+\delta,\frac{1}{2}-\delta}_{T}$
 (instead of $X^{\frac{1}{2}-\delta,\frac{1}{2}-\delta}_{T}$), we obtain
\begin{align}
\|\mathcal{N}_{-1}|_{\text{case 2.c.iii}}&(u_{1},u_{2},u_{3},
u_{4})\|_{\frac{1}{2}+\delta,-\frac{1}{2}+\delta,T}   \notag \\
\lesssim (N^{0})^{\delta}&\prod_{j=3}^{4}
\|u_{j}\|_{\frac{1}{2}-\delta,\frac{1}{2}-\delta,T}
\notag \\
&\sup_{|\mu|<3(N^{0})^{2}}
\Bigg\|\langle n \rangle^{\frac{1}{2}+\delta}
\sum_{*(n,\mu)\cap \text{case 2.c.}}
(in_{1})a_{1}(n_{1})a_{4}(n_{4})
\frac{g_{n_{2}}(\omega)g_{n_{3}}(\omega)}
{|n_{2}||n_{3}|}
\Bigg\|_{L^{2}_{|n|\in N}},
\label{Eqn:NL-neg1-29}
\end{align}
where $\sum_{n_{i}}|n_{i}|^{1+2\delta}|a_{n_{i}}|^{2} = 1$, for $i=1,4$.

From here we continue to proceed as in case 1.b.i., but more precisely $n_{3}$ will play the role played by $n_{1}$ in case 1.b.i.  The analysis proceeds in this way, with only minor modifications (replacing powers of $\frac{1}{2}-\delta$ with $\frac{1}{2}+\delta$ in some places), until we estimate $I_{1}(n_{4},\mu)$ as in line \eqref{Eqn:NL-neg1-44}, and apply Lemma \ref{Lemma:prob1} to obtain, for $\omega\in\widetilde{\Omega}_{T}$:
\begin{align}
I_{1}(n_{4},\mu) &\leq  \sup_{|n|\sim N}\sum_{\substack{|n_{1}|\sim N_{1}\\(n_{2},n_{3}) \in *(n,n_{1},n_{4},\mu)\cap \text{case 2.c.}}} \frac{|n_{1}|^{2}
|g_{n_{2}}(\omega)||g_{n_{3}}(\omega)|}
{ |n_{1}|^{1+2\delta}|n_{3}|^{2}|n_{4}|^{2}} \notag \\
&\leq   T^{-\beta}\sup_{|n|\sim N}\sum_{(n_{1},n_{2},n_{3}) \in *(n,n_{4},\mu)\cap \text{case 2.c.}} \frac{|n_{1}|^{1-2\delta}}
{|n_{2}|^{2-\varepsilon}|n_{3}|^{2-\varepsilon}}.
\label{Eqn:case-c-5}
\end{align}



We will need the following lemma.
\begin{lemma}
Let
\begin{align*}
S(n,n_{4},\mu)&:=\{(n_{1},n_{2},n_{3}):(n_{1},n_{2},n_{3})\in
*(n,n_{4},\mu) \ \text{and} \ \eqref{Eqn:b-3}\  holds\}.
\end{align*}
Then we have $|S(n,n_{4},\mu)|<
(N^0)^{\frac{16\delta}{1-4\delta}}$.
\label{Lemma:prob-bound-5}
\end{lemma}
The proof of Lemma \ref{Lemma:prob-bound-5} can be found in Section 5.2.5 of \cite{R-Th}.  Combining \eqref{Eqn:case-c-5} and Lemma \ref{Lemma:prob-bound-5} we have
\begin{align}
I_{1}(n_{4},\mu)
&\leq   T^{-\beta}
\frac{(N^0)^{\frac{16\delta}{1-4\delta}}
N_{1}^{1-2\delta}}
{ N_{2}^{2-\varepsilon}N_{3}^{2-\varepsilon}} \leq
T^{-\beta}
\frac{(N^0)^{\frac{2\delta}{1-4\delta}}
N_{1}^{1-2\delta}N_{4}^{2-\varepsilon}}
{ N_{2}^{2-\varepsilon}N_{3}^{2-\varepsilon}N_{4}^{2-\varepsilon}}
\notag \\
&\leq
T^{-\beta}
\frac{N_{1}^{1-2\delta}N_{4}^{2-\varepsilon}}
{(N^0)^{4-2\varepsilon-\frac{16\delta}{1-4\delta}}}
\lesssim T^{-\beta}
\frac{N_{1}^{1-2\delta}(N^0)^{2\delta-\varepsilon\delta}}
{(N^0)^{4-2\varepsilon-\frac{16\delta}{1-4\delta}}}
\lesssim \frac{T^{-\beta}}
{(N^0)^{3-(2+\delta)\varepsilon-\frac{16\delta}{1-4\delta}}}.
\label{Eqn:2B-3}
\end{align}
Notice that we have applied \eqref{Eqn:NL-neg1-6c} and \eqref{Eqn:2B-2} in the previous lines.  From here the estimates on (the contribution from) $I_{1}(n_{4},\mu)$ proceed as in case 2.b.i.

In the analysis of the contribution from $I_{2}(n_{4},\mu)$, we have to modify our analysis once again.  In particular, in the line of inequalities in \eqref{Eqn:NL-neg1-34}, we need to obtain the same prefactor of $(N^{0})^{-(6-)}$.  This is done quite easily by following the approach used in \eqref{Eqn:2B-3} above.  We find
\begin{align*}
\frac{n_{1}n_{1}'m_{1}m_{1}'}{\langle n_{1}\rangle^{1+2\delta} \langle n_{2}\rangle
\langle n_{1}'\rangle^{1+2\delta} \langle n_{2}'\rangle
\langle n_{3}\rangle \langle m_{1}\rangle \langle m_{2}\rangle
\langle m_{1}'\rangle \langle m_{2}'\rangle
\langle m_{3}\rangle} &\lesssim \frac{N_{1}^{2-4\delta}N_{4}^{4}}
{(N_2N_3N_4)^{4}} \\
&\lesssim
\frac{N_{1}^{2-4\delta}(N^0)^{4\delta}}
{(N^0)^{8}} \lesssim  \frac{1}{(N^0)^{6}},
\end{align*}
and from here the analysis proceeds as in case 2.b.i.  This completes the analysis of case 2.c.iii.a.

\vspace{0.1in}

\noindent $\bullet$ \textbf{CASE 2.c.iii.b:} $u_1$ type (I).

\vspace{0.1in}

As in case 2.c.iii.a., we pivot on the type of the $u_4$ factor, producing two more subcases.

\vspace{0.1in}

\noindent $\bullet$ \textbf{CASE 2.c.iii.b.1:} $u_4$ type (I).

\vspace{0.1in}

In this case we have $u_1$, $u_4$ type (I), and $u_2$, $u_3$ type (II).  We find
\begin{align}
\frac{|n|^{\frac{1}{2}+\delta}|n_1|}{|n_{1}|^{\frac{1}{2}-\gamma}
|n_{2}n_3|^{\frac{1}{2}+\delta-\gamma}|n_{4}|^{\frac{1}{2}-\gamma}}
&\lesssim \frac{|n|^{\frac{1}{2}+\delta}|n_1|^{\frac{1}{2}+\gamma}|n_4|^{\delta}}
{|n_{2}n_{3}n_{4}|^{\frac{1}{2}+\delta-\gamma}} \notag \\
&\lesssim \frac{|n|^{\frac{1}{2}+\delta}|n_1|^{\frac{1}{2}+\gamma}|n_4|^{\delta}}
{(N^0)^{1+2\delta-2\gamma}}
\lesssim  \frac{{N_4}^{\delta}}
{(N^0)^{\delta-2\gamma}}
\lesssim  \frac{1}
{(N^0)^{\frac{\delta}{3}-3\gamma-5\delta^2}}.
\label{Eqn:2B-4}
\end{align}
In the preceding inequalities, we have applied both \eqref{Eqn:b-3} and \eqref{Eqn:case-c-3}.  Having established \eqref{Eqn:2B-4}, we can proceed with a straight-forward modification the method used in case 2.b.ii.

\vspace{0.1in}

\noindent $\bullet$ \textbf{CASE 2.c.iii.b.2:} $u_4$ type (II).

\vspace{0.1in}

In this case, there are two possibilities.  We have either that $u_{1}, u_{2}, u_{3}, u_{4}$ are types (I)(I)(II)(II), respectively, or that they are types (I)(II)(I)(II).  Let us consider the case (I)(I)(II)(II), and briefly describe the adaptation to (I)(II)(I)(II) throughout.

Once again, the analysis of this case closely follows the method of case 2.b.i.  Indeed, the analysis is identical until we estimate $I_{1}(n_{4},\mu)$ as in line \eqref{Eqn:NL-neg1-44}, and apply Lemma \ref{Lemma:prob1} to obtain, for $\omega\in\widetilde{\Omega}_{T}$:
\begin{align}
I_{1}(n_{4},\mu) &\leq  T^{-\beta}\sup_{|n|\sim N}\sum_{\substack{|n_{1}|\sim N_{1}\\(n_{2},n_{3}) \in *(n,n_{1},n_{4},\mu)\cap \text{case 2.c.}}} \frac{|n_{1}|^{2}
|g_{n_{1}}(\omega)||g_{n_{2}}(\omega)|}
{ |n_{1}|^{2}|n_{2}|^{2}|n_{3}|^{1+2\delta}} \notag \\
&\leq   T^{-\beta-\varepsilon}\sup_{|n|\sim N}\sum_{(n_{1},n_{2},n_{3}) \in *(n,n_{4},\mu)\cap \text{case 2.c.}} \frac{|n_{1}|^{\varepsilon}}
{|n_{2}|^{2-\varepsilon}|n_{3}|^{1+2\delta}}\notag
\\
&\lesssim \frac{|N_1|^{\varepsilon}(N^0)^{\frac{16\delta}{1-4\delta}}}
{|N_{2}|^{2-\varepsilon}|N_{3}|^{1+2\delta}}
\lesssim \frac{(N^0)^{\varepsilon + \frac{16\delta}{1-4\delta}}N_{4}^{1+2\delta}}
{|N_{2}|^{1-\varepsilon+2\delta}(N^0)^{2+4\delta}}
\lesssim
\frac{1}
{(N^0)^{\frac{4}{3}-\frac{7\delta}{3}-10\delta^2-\varepsilon
-\frac{16\delta}{1-4\delta}}}.
\label{Eqn:2c-6}
\end{align}
In the previous lines, we have applied Lemma \ref{Lemma:prob-bound-5}, \eqref{Eqn:NL-neg1-6c} and \eqref{Eqn:2B-2}.
The inequality \eqref{Eqn:2c-6} is enough to estimate the contribution from $I(n_{4},\mu)$ as in case 1.b.i.  Indeed, it is easily verified that, from the inequality \eqref{Eqn:2c-6}, we only require a negative power of $N^0$ with magnitude greater than $1+$.  Note that, by taking $\varepsilon,\delta>0$ sufficiently small, this is exactly what we have accomplished.  Let us pause to remark that the analysis above is easily accomplished with types (I)(II)(I)(II) as well.

Before we estimate the contribution from $I_{2}(n_{4},\mu)$, let us first observe that, in the case of types (I)(I)(II)(II), we can obtain a stronger restriction on the size of $N_2$.  More precisely, we have
\begin{align}
\frac{|n|^{\frac{1}{2}+\delta}|n_1|}{|n_{1}|^{\frac{1}{2}-\gamma}
|n_{2}|^{\frac{1}{2}-\gamma}|n_{3}n_4|^{\frac{1}{2}+\delta-\gamma}}
&\lesssim \frac{|n|^{\frac{1}{2}+\delta}|n_1|^{\frac{1}{2}+\gamma}|n_2|^{\delta}}
{|n_{2}n_{3}n_{4}|^{\frac{1}{2}+\delta-\gamma}} \notag \\
&\lesssim \frac{|n|^{\frac{1}{2}+\delta}|n_1|^{\frac{1}{2}+\gamma}|n_2|^{\delta}}
{(N^0)^{1+2\delta-2\gamma}}
\lesssim  \frac{{N_2}^{\delta}}
{(N^0)^{\delta-3\gamma}}
\lesssim  \frac{1}
{(N^0)^{\frac{\delta}{5}-3\gamma}},
\label{Eqn:2B-7}
\end{align}
unless $N_2\gtrsim (N^0)^{\frac{4}{5}}$.  If \eqref{Eqn:2B-7} holds, we can proceed with a modification of the analysis in case 1.b.ii.  We will therefore assume, for the remainder of this case, that
\begin{align}
N_2\gtrsim (N^0)^{\frac{4}{5}}.
\label{Eqn:2B-8}
\end{align}

Turning to the contribution from $I_{2}(n_{4},\mu)$, we proceed with the analysis of case 2.b.i. until \eqref{Eqn:NL-neg1-34}, where we find, using \eqref{Eqn:2B-8},
\begin{align}
\|F_{n,n',n_{4},\mu}\|^{2}_{L^{2}(\Omega)}
&=  \mathbb{E}\Bigg(
\Big|\sum_{\substack{|n_{3}|\sim N_{3}\\n_{1},n_{2}\in
*(n,\mu,n_{3},n_{4})\\n_{1}',n_{2}'\in
*(n',\mu,n_{3},n_{4})}}
\frac{-n_{1}n_{1}'g_{n_{1}}(\omega)g_{n_{2}}(\omega)
\overline{g_{n_{1}'}}(\omega)\overline{g_{n_{2}'}}(\omega)}{\langle n_{1}\rangle\langle n_{2}\rangle\langle n_{1}'\rangle\langle n_{2}'\rangle\langle n_{3}\rangle^{1+2\delta}}\Big|^{2}\Bigg)
\notag \\
&=
\sum_{\substack{|n_{3}|,|m_{3}|\sim N_{3}\\(n_{1},n_{2})\in
*(n,\mu,n_{3},n_{4}), \,(n_{1}',n_{2}')\in
*(n',\mu,n_{3},n_{4})\\
(m_{1},m_{2})\in
*(n,\mu,m_{3},n_{4}), \,(m_{1}',m_{2}')\in
*(n',\mu,m_{3},n_{4})
}}
\notag \\
&\ \ \ \frac{(-n_{1}n_{1}')(-m_{1}m_{1}')}{\langle n_{1}\rangle \langle n_{2}\rangle
\langle n_{1}'\rangle \langle n_{2}'\rangle
\langle n_{3}\rangle^{1+2\delta}\langle m_{1}\rangle \langle m_{2}\rangle
\langle m_{1}'\rangle \langle m_{2}'\rangle
\langle m_{3}\rangle^{1+2\delta}} \notag \\
&\ \ \ \mathbb{E}\Big(g_{n_{1}}(\omega)g_{n_{2}}(\omega)
\overline{g_{n_{1}'}}(\omega)\overline{g_{n_{2}'}}(\omega)
\overline{g_{m_{1}}}(\omega)\overline{g_{m_{2}}}(\omega)
g_{m_{1}'}(\omega)g_{m_{2}'}(\omega)\Big).
\notag
\\
&\lesssim  \frac{1}{(N^0)^{\frac{16}{5}}N_3^{2+4\delta}}
\sum_{\substack{|n_{3}|,|m_{3}|\sim N_{3}\\(n_{1},n_{2})\in
*(n,\mu,n_{3},n_{4}), \,(n_{1}',n_{2}')\in
*(n',\mu,n_{3},n_{4})\\
(m_{1},m_{2})\in
*(n,\mu,m_{3},n_{4}), \,(m_{1}',m_{2}')\in
*(n',\mu,m_{3},n_{4})
}}\notag \\
&\ \ \ \big|\mathbb{E}\Big(g_{n_{1}}(\omega)g_{n_{2}}(\omega)g_{n_{1}'}(\omega)g_{n_{2}'}(\omega)
\overline{g_{m_{1}}}(\omega)\overline{g_{m_{2}}}(\omega)\overline{g_{m_{1}'}}(\omega)\overline{g_{m_{2}'}}(\omega)\Big)
\big|.
\label{Eqn:2B-9}
\end{align}
Then combining \eqref{Eqn:NL-neg1-36} and \eqref{Eqn:2B-9}, we have
\begin{align}
I_{2}(n_{4},\mu) &=  \Big(\sum_{
\substack{n\neq n'\\ |n|,|n'|\sim N}}|F_{n,n',n_{4},\mu}(\omega)|^{2}\Big)^{\frac{1}{2}}  \notag \\
&<  \frac{T^{-2\beta}}{(N^0)^{\frac{8}{5}}N_3^{1+2\delta}}\Bigg(
\sum_{\substack{n\neq n',\, |n|,|n'|\sim N, \,|n_{3}|,|m_{3}|\sim N_{3}\\(n_{1},n_{2})\in
*(n,\mu,n_{3},n_{4})\, (n_{1}',n_{2}')\in
*(n',\mu,n_{3},n_{4}) \\
(m_{1},n_{2})\in
*(n,\mu,m_{3},n_{4})\, (m_{1}',m_{2}')\in
*(n',\mu,m_{3},n_{4})
}}\notag \\
&\ \ \ \ \ \ \ \big|\mathbb{E}\Big(g_{n_{1}}(\omega)g_{n_{2}}(\omega)g_{n_{1}'}(\omega)g_{n_{2}'}(\omega)
\overline{g_{m_{1}}}(\omega)\overline{g_{m_{2}}}(\omega)\overline{g_{m_{1}'}}(\omega)\overline{g_{m_{2}'}}(\omega)\Big)
\big|
\Bigg)^{\frac{1}{2}}
\notag \\
&\lesssim  \frac{T^{-2\beta}}{(N^0)^{\frac{8}{5}}}\sup_{|n_{3}|,|m_{3}|\sim N_3}\Bigg(
\sum_{\substack{(n,n_{1},n_{2})\in
*(\mu,n_{3},n_{4})\, (n',n_{1}',n_{2}')\in
*(\mu,n_{3},n_{4}) \\
(n,m_{1},m_{2})\in
*(\mu,m_{3},n_{4})\, (n',m_{1}',m_{2}')\in
*(\mu,m_{3},n_{4})}}
\Bigg)^{\frac{1}{2}}
\notag \\
&\lesssim  \frac{T^{-2\beta}}{(N^0)^{\frac{8}{5}-\frac{32\delta}{1-4\delta}}}
\lesssim
\frac{T^{-2\beta}}{(N^0)^{\frac{3}{2}}},
\label{Eqn:2B-10}
\end{align}
by taking $\delta>0$ sufficiently small.  Let us remark that, to obtain \eqref{Eqn:2B-10} above, we have applied Lemma \ref{Lemma:prob-bound-5}, \eqref{Eqn:NL-neg1-6c} and \eqref{Eqn:2B-8}.  With \eqref{Eqn:2B-10}, we have established an estimate superior to \eqref{Eqn:NL-neg1-37}, and the remaining analysis of this case follows case 2.b.i.

With the combination of types (I)(II)(I)(II), we can follow the same scheme to estimate the contribution from $I_{2}(n_{4},\mu)$, but the roles of $n_2$ and $n_3$ are swapped (including \eqref{Eqn:2B-8}, which in this case restricts the size of $N_3$).

This completes the analysis of case 2.c.iii.

\vspace{0.1in}

\noindent $\bullet$ \textbf{CASE 2.c.iv:} At least 3 of $u_{1},u_{2},u_{3},u_{4}$ of type (II).

\vspace{0.1in}

We consider subcases.  In each subcase, we follow the method of case 2.b.ii.

\vspace{0.1in}

\noindent $\bullet$ \textbf{CASE 2.c.iv.1:} $u_{1},u_{2},u_{3},u_{4}$ of types (I)(II)(II)(II), respectively.

\vspace{0.1in}

In this case, we find
\begin{align}
\frac{|n|^{\frac{1}{2}+\delta}|n_1|}{|n_{1}|^{\frac{1}{2}-\gamma}
|n_2n_{3}n_4|^{\frac{1}{2}+\delta-\gamma}}
\lesssim
\frac{1}{(N^0)^{\delta-3\gamma}}.
\label{Eqn:2B-11a}
\end{align}
Using \eqref{Eqn:2B-11a}, we may proceed as in case 2.b.ii.

\vspace{0.1in}

\noindent $\bullet$ \textbf{CASE 2.c.iv.2:} $u_{1},u_{2},u_{3},u_{4}$ of types (II)(I)(II)(II), (II)(II)(I)(II) or (II)(II)(II)(I), respectively.

\vspace{0.1in}

Suppose $u_{1},u_{2},u_{3},u_{4}$ are of types (II)(I)(II)(II).
In this case, we find
\begin{align}
\frac{|n|^{\frac{1}{2}+\delta}|n_1|}{|n_{1}|^{\frac{1}{2}+\delta-\gamma}
|n_2|^{\frac{1}{2}-\gamma}|n_{3}n_4|^{\frac{1}{2}+\delta-\gamma}}
&\lesssim
\frac{|n|^{\frac{1}{2}+\delta}|n_1|^{\frac{1}{2}-\delta+\gamma}
|n_2|^{\delta}}
{|n_2n_{3}n_4|^{\frac{1}{2}+\delta-\gamma}}
\notag \\
\lesssim
\frac{|n_2|^{\delta}}{(N^0)^{2\delta-3\gamma}}
\lesssim
\frac{1}{(N^0)^{\delta-3\gamma}}.
\label{Eqn:2B-12}
\end{align}
Again, using \eqref{Eqn:2B-12}, we may proceed as in case 2.b.ii.
It is trivial to verify that this approach applies with types (II)(II)(I)(II) and (II)(II)(II)(I) as well, and this case is complete.

\vspace{0.1in}

\noindent $\bullet$ \textbf{CASE 2.c.iv.3:} $u_{1},u_{2},u_{3},u_{4}$ all type (II).

\vspace{0.1in}

We consider
\begin{align}
\frac{|n|^{\frac{1}{2}+\delta}|n_1|}{|n_{1}|^{\frac{1}{2}+\delta-\gamma}
|n_2n_{3}n_4|^{\frac{1}{2}+\delta-\gamma}}
\lesssim
\frac{1}{(N^0)^{2\delta-3\gamma}},
\label{Eqn:2B-13}
\end{align}
and once again revert to the analysis of case 2.b.ii.

This completes the analysis of case 2.c.iv., our final case, and the proof of Proposition \ref{Prop:NL-neg1-local} is complete.

\vspace{0.1in}
\end{proof}


\subsection{Probabilistic heptilinear estimates}
\label{Sec:probhept}

In this subsection we prove Proposition \ref{Prop:NL-1-nlpart}.  This proof will be somewhat probabilistic in nature; it will incorporate the randomized data $u_{0,\omega}$ and make use of Lemma \ref{Lemma:prob1}.  This probabilistic analysis will be simpler, however, than the analysis used in the proof of Proposition \ref{Prop:NL-neg1-local}.  In particular we will not need Lemmas \ref{Lemma:prob2}-\ref{Lemma:prob3} (hypercontractivity of the Ornstein-Uhlenbeck semigroup).

The proof of Proposition \ref{Prop:NL-1-nlpart} will rely on a certain (deterministic) cancelation in one region of frequency space.  This cancelation is one of the more delicate points of this paper; we proceed to discuss its details before we begin the proof of Proposition \ref{Prop:NL-1-nlpart}.

As mentioned in Remark \ref{Rem:cancel}, it is in fact necessary that we explain these details before starting the proof of Proposition \ref{Prop:NL-1-nlpart}.  Indeed, in the statement of Proposition \ref{Prop:NL-1-nlpart} and during the proof of Theorem \ref{Thm:LWP}, our use of the notation $\mathcal{N}(u_1,u_2,u_3,u_4)$ and  $\mathcal{N}_{1}(\mathcal{D}(u_5,u_6,u_7,u_8),u_2,u_3,u_4)$ with different input functions is somewhat misleading.  When we write
$\mathcal{N}(u_1,u_2,u_3,u_4)$ and  $\mathcal{N}_{1}(\mathcal{D}(u_5,u_6,u_7,u_8),u_2,u_3,u_4)$ in this paper, if any of the input functions differ, we do not mean for these expressions to be interpreted literally.  Instead, these expressions are defined through a decomposition in frequency space, with a certain cancelation enforced in a problematic region.  There are two important questions that arise: 1. What is the nature of this cancelation?  2. Why are we allowed to enforce this cancelation?

Before we answer question 1 in detail, let us forecast the answer to question 2.  The crucial point here is that this cancelation only applies when the functions $u_2,\ldots,u_8$ placed into the nonlinearity $\mathcal{N}_{1}(\mathcal{D}(u_5,u_6,u_7,u_8),u_2,u_3,u_4)$ are all the same ($u=u_2=\cdots =u_8$).  Luckily, during the proof of Theorem \ref{Thm:LWP}, when we establish \eqref{Eqn:LWP-1a}-\eqref{Eqn:LWP-17},  we will only need to consider the nonlinearity with different input functions through the addition and subtraction such terms, in order to produce factors on the right-hand side of the nonlinear estimates with the difference of two solutions (e.g. $u^N-u^M$) inserted (see \eqref{Eqn:LWP-5} and \eqref{Eqn:LWP-6a}).  Because we are adding and subtracting these factors, we can define $\mathcal{N}(u_1,u_2,u_3,u_4)$ and  $\mathcal{N}_{1}(\mathcal{D}(u_5,u_6,u_7,u_8),u_2,u_3,u_4)$, with different input functions, to be any multilinear expressions that are suitable to our needs.  We will only modify the definitions (from literal interpretation) in a single region of frequency space (a subset of $A_{1}$), to ensure that the cancelation which holds when $u=u_2=\cdots =u_8$ is preserved for different input functions.


We proceed to identify the cancelation, and to properly define $\mathcal{N}(u_1,u_2,u_3,u_4)$ and  $\mathcal{N}_{1}(\mathcal{D}(u_5,u_6,u_7,u_8),u_2,u_3,u_4)$ with different input functions.  Following these definitions, we will present the proof of Proposition \ref{Prop:NL-1-nlpart}.  First suppose all factors are the same, and consider
\begin{align}
\mathcal{N}_{1}(\mathcal{D}&(u,u,u,u),
u,u,u)^{\wedge}(n,\tau) \notag \\ &=
\sum_{\substack{(n_1,n_2,n_3,n_4)\in \zeta(n)\\(n_5,n_6,n_7,n_8)\in \zeta(n_1)}}\int_{\tau=\tau_{2}+\cdots +\tau_{8}}\chi_{A_{1}}
\frac{-n_{1}n_{5}}{\sigma_{1}} \prod_{j=2}^{8}\hat{u}(n_{j},\tau_{j}).
\label{Eqn:NL-1-20}
\end{align}
We will induce cancelation in the contribution to \eqref{Eqn:NL-1-20} from when $n_5=n$ and the remaining frequencies satisfy certain smallness conditions.  Consider
\begin{align}
A_{1,c}=\bigg\{&(n,n_2,n_3,n_4,n_5,n_6,n_7,n_8,\tau,\tau_2,\tau_3,\tau_4,
\tau_5,\tau_6,\tau_7,\tau_8)\in (\mathbb{Z}\setminus \{0\})^8\times \mathbb{R}^8 : \notag \\
&\ \ \ \  n=n_5,\ \ \tau=\tau_2+\cdots +\tau_8, \ \  n_{2}+n_3+n_4+n_6+n_7+n_8=0, \notag \\
&\ \ \ \  n_{2}+n_3+n_4\neq 0,\ \ |\sigma|<|n|^{\sqrt{2\delta}},\  |\sigma_{k}|<|n|^{\sqrt{2\delta}}, \ |n_{k}|<|n|^{\sqrt{2\delta}} \notag \\ &\quad \quad \quad \ \text{for}\ k=2,3,4,6,7,8\bigg\}.
\label{Eqn:cancel}
\end{align}
Notice that, if $(n,n_2,\ldots,n_8,\tau, \tau_2,\ldots$
$\tau_8)\in A_{1,c}$, then
$(n_1,n_2,n_3,n_4)\in\zeta(n)$, $(n,n_6,n_7,n_8)\in\zeta(n_1)$ and $(n,n_1,n_2,n_3,n_4,\tau_1,\tau_2,\tau_3,\tau_4)\in A_{1}$.
Indeed, the restrictions $|n_k|<|n|^{\sqrt{2\delta}}$
for $k=2,3,4$ and $n_{2}+n_3+n_4\neq 0$ guarantee that $n\neq n_k$ for $k=1,2,3,4$ and $n_1\neq -n_k$ for $k=2,3,4$, thus $(n_1,n_2,n_3,n_4)\in\zeta(n)$.
Similarly $|n_k|<|n|^{\sqrt{2\delta}}$
for $k=6,7,8$ and $n_{6}+n_7+n_8\neq 0$ guarantees $n_1 \neq n,n_k$ and $n\neq -n_k$ for $k=6,7,8$, and thus $(n,n_6,n_7,n_8)\in\zeta(n_1)$.  Lastly using the restrictions $|\sigma|,|\sigma_{k}|,|n_k|<|n|^{\sqrt{2\delta}}$ for $k=2,3,4,6,7,8$ and $n=n_5$, we can easily show that $|\sigma_1|\gtrsim |n_{\text{max}}|^2$, and therefore $(n,n_1,n_2,$ $n_3,n_4,\tau_1,\tau_2,\tau_3,\tau_4)\in A_{1}$.

Because of this, we can consider the following contribution to \eqref{Eqn:NL-1-20}
\begin{align}
\mathcal{N}_{1}(\mathcal{D}(u,&
u,u,u),u,u,u)^{\wedge}(n,\tau)|_{A_{1,c}}
\notag \\
&:= -n\sum_{n_2+n_3+n_4+n_6+n_7+n_8=0}
\int_{\tau=\tau_2+\tau_3+\tau_4+\tau_5+\tau_6+\tau_7 +\tau_{8}}
\chi_{A_{1,c}}\prod_{j=2}^{8}\hat{u}(n_{j},\tau_{j})
\notag \\
&\quad \quad \quad \quad \quad \quad
\quad
\cdot\Big(\frac{n-n_{2}-n_{3}-n_{4}}{\tau-\tau_{2}-\tau_{3}-\tau_{4}
-(n-n_{2}-n_{3}-n_{4})^{3}}\Big)
\notag \\
&=
-n\sum_{n_2+n_3+n_4+n_6+n_7+n_8=0}
\int_{\tau=\tau_2+\tau_3+\tau_4+\tau_5+\tau_6+\tau_7 +\tau_{8}}
\chi_{A_{1,c}}\prod_{j=2}^{8}\hat{u}(n_{j},\tau_{j})
\notag \\
& \quad \quad \quad \quad\quad \quad \quad
\cdot\Big(\frac{n}{\tau-\tau_{2}-\tau_{3}-\tau_{4}-(n-n_{2}-n_{3}-n_{4})^{3}}
\notag \\
&  \quad \quad \quad \quad \quad \quad \quad \quad \quad -\frac{n_{2}+n_{3}+n_{4}}{\tau-\tau_{2}-\tau_{3}-\tau_{4}
-(n-n_{2}-n_{3}-n_{4})^{3}}\Big)
\notag \\
&=:
K_{1}(n,\tau) + K_{2}(n,\tau),
\label{Eqn:NLdef3}
\end{align}
where we have defined $K_{1}(n,\tau)$, $K_{2}(n,\tau)$ by expanding the parentheses in the second last line.  We will only need cancelation to control $K_{1}(n,\tau)$ ($K_{2}(n,\tau)$ will be estimated directly).  Let us now describe this cancelation.  We swap the variable names
$(n_{2},n_{3},n_{4},\tau_{2},\tau_{3},\tau_{4})$
with $(n_{6},n_{7},n_{8},\tau_{6},\tau_{7},\tau_{8})$ and use the invariance of $A_{1,c}$ under this modification to obtain
\begin{align}
K_{1}(n,\tau) &=
-\frac{n^2}{2}\sum_{n_2+n_3+n_4+n_6+n_7+n_8=0}
\int_{\tau=\tau_2+\tau_3+\tau_4+\tau_5+\tau_6+\tau_7 +\tau_{8}}\chi_{A_{1,c}}\prod_{j=2}^{8}
\hat{u}(n_{j},\tau_{j})
\notag
\\
&\quad \quad \quad
\cdot\Big(\frac{1}{\tau-\tau_{6}-\tau_{7}-\tau_{8}
-(n-n_{6}-n_{7}-n_{8})^{3}}
\notag \\
&\ \ \ \ \ \ \
\ \ \ \ \ \ \
\ \ \ \ \ \ \
+ \frac{1}{\tau-\tau_{2}-\tau_{3}-\tau_{4}
-(n-n_{2}-n_{3}-n_{4})^{3}}\Big).
\label{Eqn:2A-18}
\end{align}
Using $n_{2}+n_{3}+n_{4}+n_{6}+n_{7}+n_{8}=0$, we find
\begin{align}
&\frac{1}{\tau-\tau_{6}-\tau_{7}-\tau_{8}
-(n-n_{6}-n_{7}-n_{8})^{3}}
 \notag \\
& \ \
+ \frac{1}{\tau-\tau_{2}-\tau_{3}-\tau_{4}
-(n-n_{2}-n_{3}-n_{4})^{3}}
 \notag \\
&= \frac{1}{3n^2(n_{2}+n_{3}+n_{4})-(3n-(n_2+n_3+n_4))(n_{2}+n_{3}+n_{4})^2
+ \tau-n^{3}-\tau_{6}-\tau_{7}-\tau_{8}}
\notag
\\
& \ \
-\frac{1}{3n^2(n_{2}+n_{3}+n_{4})+(3n+(n_2+n_3+n_4))(n_{2}+n_{3}+n_{4})^2
- \sigma+\tau_{2}+\tau_{3}+\tau_{4}}
\notag
\\
&\ \ \ =\frac{-6n(n-n_{1})^{2}+\tau-\tau_5-2\sigma}
{(3nn_{1}(n-n_{1})
+ \sigma-\tau_{6}-\tau_{7}-\tau_{8})
(3nn_{1}(n-n_{1})
- \sigma+\tau_{2}+\tau_{3}+\tau_{4})},
\label{Eqn:NL-1-50}
\end{align}
where we have used $\tau=\tau_2+\cdots+\tau_8$ to obtain the last line.  This gives
\begin{align}
K_{1}&(n,\tau)= -\frac{n^2}{2}\sum_{n_2+n_3+n_4+n_6+n_7+n_8=0}
\int_{\tau=\tau_2+\tau_3+\tau_4+\tau_5+\tau_6+\tau_7 +\tau_{8}}
\chi_{A_{1,c}}
\prod_{j=2}^{8}\hat{u}(n_{j},\tau_{j}) \notag \\
&\cdot\frac{-6n(n-n_{1})^{2}+\tau-\tau_5-2\sigma}
{(3nn_{1}(n-n_{1})
+ \sigma-\tau_{6}-\tau_{7}-\tau_{8})
(3nn_{1}(n-n_{1})
- \sigma+\tau_{2}+\tau_{3}+\tau_{4})}.
\label{Eqn:NLdef}
\end{align}

Now, as anticipated, we will \textit{define}  $\mathcal{N}(u_1,u_2,u_3,u_4)$ and $\mathcal{N}_{1}(\mathcal{D}(u_{5},
u_{6},u_{7},u_{8}),u_{2}
,u_{3},u_{4})$ with (potentially non-equivalent) input functions by extending
the definition of $K_{1}(n,\tau)$ according to \eqref{Eqn:NLdef}.
That is, $\mathcal{N}(u_1,u_2,u_3,u_4)$ is defined piecewise through a decomposition in frequency space.  The region of integration $A$ is divided into $A_{-1}$,$A_{0}$,$A_{1}$,$A_{2}$,$A_{3}$,$A_{4}$.  In the regions $A_{k}$ for $k=-1,0,2,3,4$, we interpret $\mathcal{N}_{k}(u_1,u_2,u_3,u_4)$ directly (without modifying the definition for non-equivalent inputs).  In the region $A_{1}$, we will exploit cancelation after the second iteration.  It is here where we must emphasize that our use, during the proof of Theorem \ref{Thm:LWP}, of the notation $\mathcal{N}(u_1,u_2,u_3,u_4)$ (and more specifically of the notation $\mathcal{N}_{1}(u_1,u_2,u_3,u_4)$) with potentially non-equivalent inputs, is misleading.  Indeed, the definition of $\mathcal{N}_1(u_1,u_2,u_3,u_4)$ depends upon inserting an equation satisfied by $u_1$ (the second iteration), and our definition of $\mathcal{N}_1(u_1,u_2,u_3,u_4)$ will vary with this equation.  However, the algorithm for determining this definition is straightforward, and we describe it here.  During the proof of Theorem \ref{Thm:LWP}, the factor $u_1$ will satisfy an equation of the form \eqref{Eqn:LWP-2b} or one of its variants.  The important point is that the equation satisfied by $u_1$ will always be decomposed into contributions of type (I) (linear part, rough but random) and type (II) (nonlinear part, smooth and deterministic).  The contributions from the type (I) part of $u_1$ are always interpreted directly.  For the contributions from the type (II) factor, we either (i) bound this factor using the higher temporal regularity $b=\frac{1}{2}+\delta$, via the estimate \eqref{Eqn:NL-k-1}, in which case the nonlinearity is interpreted directly, or (ii) we expand the type (II) contribution into a heptilinear expression.  In case (ii), in the complement of $A_{1,c}$, we interpret the nonlinearity directly, as in \eqref{Eqn:NL-1-20}.  For the contribution from the region $A_{1,c}$, we will force the cancelation \eqref{Eqn:NL-1-50}.  That is, for each $n>0$ and $\tau\in\mathbb{R}$, with $|\tau-n^3|<|n|^{\sqrt{2\delta}}$, we define
\begin{align}
\mathcal{N}_{1}(\mathcal{D}(u_5,
u_6,u_7,u_8),u_2,u_3,u_4)^{\wedge}(n,\tau)|_{A_{1,c}}
:= K_{1}(n,\tau) + K_{2}(n,\tau)
\label{Eqn:NLdef2b}
\end{align}
where $K_{2}(n,\tau)$ is given as in \eqref{Eqn:NLdef3} (but with potentially non-equivalent factors $u_{j}$), and
\begin{align}
&K_{1}(n,\tau):= -\frac{n^2}{2}\sum_{n_2+n_3+n_4+n_6+n_7+n_8=0}
\int_{\tau=\tau_2+\tau_3+\tau_4+\tau_5+\tau_6+\tau_7 +\tau_{8}}
\chi_{A_{1,c}}
\prod_{j=2}^{8}\hat{u}_{j}(n_{j},\tau_{j}) \notag \\
&\ \ \ \ \ \ \ \cdot\frac{-6n(n-n_{1})^{2}+\tau-\tau_5-2\sigma}
{(3nn_{1}(n-n_{1})
+ \sigma-\tau_{6}-\tau_{7}-\tau_{8})
(3nn_{1}(n-n_{1})
- \sigma+\tau_{2}+\tau_{3}+\tau_{4})}.
\label{Eqn:NLdef2}
\end{align}
This way we can still take advantage of the cancelation of \eqref{Eqn:NL-1-50} with nonequal factors.


Having defined the nonlinearity in the statement of Proposition \ref{Prop:NL-1-nlpart}, we proceed with the proof.

\begin{proof}[Proof of Proposition \ref{Prop:NL-1-nlpart}:]

We split into cases depending on the relative sizes of the spatial frequencies $n,n_2,\ldots,n_{8}$.  Here is a list of the cases we will consider.

\vspace{0.1in}

\noindent $\bullet$ \textbf{CASE 1.}
$|\sigma|\gtrsim |n|^{\sqrt{2\delta}}$, $|n_{k}|\gtrsim |n|^{\sqrt{2\delta}}$ or $|\sigma_k|\gtrsim |n|^{\sqrt{2\delta}}$ for some $k\in\{2,3,4,6,7,8\}$.

\noindent $\bullet$ \textbf{CASE 2.}
$|\sigma|\ll |n|^{\sqrt{2\delta}}$, $|n_{k}|\ll |n|^{\sqrt{2\delta}}$ and $|\sigma_k|\ll |n|^{\sqrt{2\delta}}$ for each $k\in\{2,3,4,6,7,8\}$.

\vspace{0.1in}

\hspace{0.4in} $\bullet$ \textbf{CASE 2.a.}
$u_{5}$ type (II).

\hspace{0.4in} $\bullet$ \textbf{CASE 2.b.}
$u_{5}$ type (I).

\vspace{0.1in}

\hspace{0.8in} $\bullet$ \textbf{CASE 2.b.i:} $n\neq n_{5}$.

\hspace{0.8in} $\bullet$ \textbf{CASE 2.b.ii:} $n= n_{5}$.

\vspace{0.1in}

We proceed with the analysis of each case.

\vspace{0.1in}

\noindent $\bullet$ \textbf{CASE 1.}
$|\sigma|\gtrsim |n|^{\sqrt{2\delta}}$, $|n_{k}|\gtrsim |n|^{\sqrt{2\delta}}$ or $|\sigma_k|\gtrsim |n|^{\sqrt{2\delta}}$ for some $k\in\{2,3,4,6,7,8\}$.

\noindent In this subcase we show that
\begin{align}
\|\mathcal{N}_{1}(\mathcal{D}(u_{5},u_{6},
u_{7},u_{8}),u_{2},u_{3},u_{4})|_{\text{Case 1}}
\|_{\frac{1}{2}+\delta,-\frac{1}{2}+\delta}
\lesssim T^{\theta}\prod_{j=2}^{8}
\|u_{j}\|_{\frac{1}{2}-\delta,
\frac{1}{2}-\delta,T}.
\label{Eqn:2A-4}
\end{align}
First suppose $|n_{k}|\gtrsim |n|^{\sqrt{2\delta}}$ for some $k\in\{2,3,4,6,7,8\}$.  We estimate
\begin{align}
\frac{|n|^{\frac{1}{2}+\delta}|n_{1}||n_{5}|}
{|\sigma_{1}||n_{5}|^{\frac{1}{2}-\delta}
|n_{k}|^{\sqrt{2\delta}}}  \lesssim  1.
\label{Eqn:2A-5}
\end{align}
Using \eqref{Eqn:2A-5} and \eqref{Eqn:NL-1-2}, \eqref{Eqn:2A-4} follows from
\begin{align}
\Big\|\prod_{j=2}^{8}f_{j}\Big\|_{L^{2}_{x,t}}
\lesssim T^{\theta}\|f_{5}\|_{0,\frac{1}{2}-\delta}
\|f_{k}\|_{\frac{1}{2}-\delta-\sqrt{2\delta},\frac{1}{2}-\delta}
\prod_{\substack{j=2\\j\neq 5,k}}^{8}
\|f_{j}\|_{\frac{1}{2}-\delta,\frac{1}{2}-\delta}.
\label{Eqn:2A-6}
\end{align}
Then \eqref{Eqn:2A-6} is obtained with H\"{o}lder,
\eqref{Eqn:X-onehalfminus}, \eqref{Eqn:X-Str-1} and Lemma \ref{Lemma:gainpowerofT},
\begin{align*}
\Big\|\prod_{j=2}^{8}f_{j}\Big\|_{L^{2}_{x,t}}
&\lesssim
\|f_{5}\|_{L^{4}_{x,t}} \prod_{\substack{j=2\\j\neq 5}}^{8} \|f_{j}\|_{L^{24}_{x,t}}  \notag \\
&\lesssim  \|f_{5}\|_{0,\frac{1}{3}}
\prod_{\substack{j=2\\j\neq 5}}^{8} \|f_{j}\|_{\frac{1}{2}-\delta-\sqrt{2\delta},\frac{1}{2}
-\delta-\sqrt{2\delta}}
\notag \\
&\lesssim  T^{\theta} \|f_{5}\|_{0,\frac{1}{2}-\delta} \|f_{k}\|_{\frac{1}{2}-\delta-\sqrt{2\delta},
\frac{1}{2}-\delta} \prod_{\substack{j=2\\j\neq 5,k}}^{8} \|f_{j}\|_{\frac{1}{2}-\delta,\frac{1}{2}-\delta}.
\end{align*}

Next suppose $|\sigma_k|\gtrsim |n|^{\sqrt{2\delta}}$ for some $k\in\{2,3,4,6,7,8\}$.  We estimate
\begin{align}
\frac{|n|^{\frac{1}{2}+\delta}|n_{1}||n_{5}|}
{|\sigma_{1}||n_{5}|^{\frac{1}{2}-\delta}
|\sigma_k|^{\sqrt{2\delta}}}  \lesssim  1.
\label{Eqn:2A-7}
\end{align}
Using \eqref{Eqn:2A-7} and \eqref{Eqn:NL-1-2}, \eqref{Eqn:2A-4} follows from
\begin{align}
\Big\|\prod_{j=2}^{8}f_{j}\Big\|_{L^{2}_{x,t}}
\lesssim T^{\theta}\|f_{5}\|_{0,\frac{1}{2}-\delta}
\|f_{k}\|_{\frac{1}{2}-\delta,\frac{1}{2}-\delta-\sqrt{2\delta}}
\prod_{\substack{j=2\\j\neq 5,k}}^{8}
\|f_{j}\|_{\frac{1}{2}-\delta,\frac{1}{2}-\delta}.
\label{Eqn:2A-8}
\end{align}
Then \eqref{Eqn:2A-8} is obtained with H\"{o}lder,
\eqref{Eqn:X-onehalfminus}, \eqref{Eqn:X-Str-1} and Lemma \ref{Lemma:gainpowerofT},
\begin{align*}
\Big\|\prod_{j=2}^{8}f_{j}\Big\|_{L^{2}_{x,t}}
&\lesssim
\|f_{5}\|_{L^{4}_{x,t}} \prod_{\substack{j=2\\j\neq 5}}^{8} \|f_{j}\|_{L^{24}_{x,t}}  \notag \\
&\lesssim  \|f_{5}\|_{0,\frac{1}{3}}
\prod_{\substack{j=2\\j\neq 5}}^{8} \|f_{j}\|_{\frac{1}{2}-\delta-\sqrt{2\delta},\frac{1}{2}-\delta-\sqrt{2\delta}}
\notag \\
&\lesssim  T^{\theta} \|f_{5}\|_{0,\frac{1}{2}-\delta} \|f_{k}\|_{\frac{1}{2}-\delta,
\frac{1}{2}-\delta-\sqrt{2\delta}} \prod_{\substack{j=2\\j\neq 5,k}}^{8} \|f_{j}\|_{\frac{1}{2}-\delta,\frac{1}{2}-\delta}.
\end{align*}
If $|\sigma|\gtrsim |n|^{\sqrt{2\delta}}$, the justification of \eqref{Eqn:2A-8} follows the same method.

\vspace{0.1in}

\noindent $\bullet$ \textbf{CASE 2.}
$|\sigma|\ll |n|^{\sqrt{2\delta}}$, $|n_{k}|\ll |n|^{\sqrt{2\delta}}$ and $|\sigma_k|\ll |n|^{\sqrt{2\delta}}$ for each $k\in\{2,3,4,6,7,8\}$.

\vspace{0.1in}

Notice that the region $A_{1,c}$ must be treated as a subset of this case.

\noindent $\bullet$ \textbf{CASE 2.a.}
$u_{5}$ type (II).

In this subcase we establish
\begin{align}
\|\mathcal{N}_{1}(\mathcal{D}(u_{5},u_{6},
u_{7},u_{8}),&u_{2},u_{3},u_{4})|_{\text{Case 2.a.}}
\|_{\frac{1}{2}+\delta,-\frac{1}{2}+\delta} \notag \\
&\lesssim T^{\theta}\|u_{5}\|_{\frac{1}{2}+\delta,
\frac{1}{2}-\delta,T}\prod_{
\substack{j=2\\j\neq 5}}^{8}
\|u_{j}\|_{\frac{1}{2}-\delta,
\frac{1}{2}-\delta,T}.
\label{Eqn:2A-0}
\end{align}
We begin by estimating
\begin{align}
\frac{|n|^{\frac{1}{2}+\delta}|n_{1}||n_{5}|}
{|\sigma_{1}||n_{5}|^{\frac{1}{2}+\delta}}  \lesssim  1.
\label{Eqn:2A-1}
\end{align}
Using \eqref{Eqn:2A-1}, \eqref{Eqn:2A-0} follows from
\begin{align}
\Big\|\prod_{j=2}^{8}f_{j}\Big\|_{L^{2}_{x,t}}
\lesssim T^{\theta}\|f_{5}\|_{0,\frac{1}{2}-\delta}
\prod_{\substack{j=2\\j\neq 5}}^{8}
\|f_{j}\|_{\frac{1}{2}-\delta,\frac{1}{2}-\delta}.
\label{Eqn:2A-3}
\end{align}
Then \eqref{Eqn:2A-3} is obtained with H\"{o}lder,
\eqref{Eqn:X-onehalfminus}, \eqref{Eqn:X-Str-1} and Lemma \ref{Lemma:gainpowerofT},
\begin{align*}
\Big\|\prod_{j=2}^{8}f_{j}\Big\|_{L^{2}_{x,t}}
&\lesssim
\|f_{5}\|_{L^{4}_{x,t}} \prod_{\substack{j=2\\j\neq 5}}^{8} \|f_{j}\|_{L^{24}_{x,t}}  \notag \\
&\lesssim  \|f_{5}\|_{0,\frac{1}{3}} \prod_{\substack{j=2\\j\neq 5}}^{8} \|f_{j}\|_{\frac{1}{2}-\delta,\frac{1}{2}-\delta}
\notag \\
&\lesssim  T^{\theta} \|f_{5}\|_{0,\frac{1}{2}-\delta} \prod_{\substack{j=2\\j\neq 5}}^{8} \|f_{j}\|_{\frac{1}{2}-\delta,\frac{1}{2}-\delta}.
\end{align*}

\vspace{0.1in}

\noindent $\bullet$ \textbf{CASE 2.b.}
$u_{5}$ type (I).

\vspace{0.1in}

In this case we show there exists $\beta>0$ and $\Omega_{T}\subset \Omega$, with $P(\Omega_{T}^{c})<e^{-\frac{1}{T^{\beta}}}$, such that if $\omega\in\Omega_{T}$, then we have
\begin{align}
\|\mathcal{N}_{1}(\mathcal{D}(S(t)\Lambda_5 u_{0,\omega},&u_{6},
u_{7},u_{8}),u_{2},u_{3},u_{4})|_{\text{Case 2.b.}}
\|_{\frac{1}{2}+\delta,-\frac{1}{2}+\delta,T} \notag \\
&\lesssim T^{\theta}M_{5}^{-\varepsilon}\prod_{\substack{j=2\\j\neq 5}}^{8}
\|u_{j}\|_{\frac{1}{2}-\delta,
\frac{1}{2}-\delta,T}.
\label{Eqn:2A-11}
\end{align}
Recall that $\Lambda_5$ is the Fourier multiplier corresponding to the characteristic function of the interval $[M_5,K_5]$ in frequency space, for some dyadic integers $M_5, K_5>0$.  We will establish \eqref{Eqn:2A-11} using a dyadic decomposition in all factors.  That is, we assume that $|n|\sim N, |n_{i}|\sim N_{i}$, and as in the proof of \eqref{Eqn:NL-neg1}, we will order the frequencies (and corresponding dyadic shells) from largest to smallest using superscripts.
To simplify notation, for the remainder of this case, we will let $u_5=S(t)\Lambda_5 u_{0,\omega}$, and we will drop the explicit notation
$\mathcal{N}_{1}(\mathcal{D}(u_{5},u_{6},
u_{7},u_{8}),u_{2},u_{3},u_{4})|_{\text{Case 2.b.}}$, but maintain the restrictions of this case in our analysis.

Starting on the left-hand side of \eqref{Eqn:2A-11}, we use the restriction $|\tau-n^{3}|<(N^0)^{\sqrt{2\delta}}$ to consider
\begin{align}
\|\mathcal{N}_{1}&(\mathcal{D}(u_{5},u_{6},
u_{7},u_{8}),u_{2},u_{3},u_{4})
\|_{\frac{1}{2}+\delta,-\frac{1}{2}+\delta,T}
\notag
\\
&=  \Big(\sum_{n}\int_{|\tau-n^{3}|<(N^0)^{\sqrt{2\delta}}}
\frac{\langle n\rangle^{1+2\delta}}{\langle \tau-n^{3}\rangle^{1-2\delta}}
\notag
\\
&\ \ \ \ \ \ \ \ \ \ \ \
\cdot|\mathcal{N}_{1}(\mathcal{D}(u_{5},u_{6},
u_{7},u_{8}),u_{2},u_{3},u_{4})(n,\tau)|
^{2}\Big)^{\frac{1}{2}}
\notag
\\
&= \Big(\int_{|\lambda|<(N^0)^{\sqrt{2\delta}}}
\langle \lambda\rangle^{-1+2\delta}\sum_{n}
\langle n\rangle^{1+2\delta}
\notag
\\
&\ \ \ \ \ \ \ \ \ \ \ \
\cdot|\mathcal{N}_{1}(\mathcal{D}(u_{5},u_{6},
u_{7},u_{8}),u_{2},u_{3},u_{4})(n,\lambda + n^{3})|^{2}\Big)^{\frac{1}{2}}
\notag
\\
&\leq
\Big(\int_{|\lambda|<(N^0)^{\sqrt{2\delta}}}
\langle \lambda\rangle^{1-2\delta}\Big)^{\frac{1}{2}}
\notag
\\
& \ \ \ \ \ \ \ \ \ \ \ \
\cdot\sup_{|\lambda|<(N^0)^{\sqrt{2\delta}}}
\|(\mathcal{N}_{1}(\mathcal{D}(u_{5},u_{6},
u_{7},u_{8}),u_{2},u_{3},u_{4}))(n,\lambda + n^{3})\|
_{H^{\frac{1}{2}+\delta}_{x}}
\notag\\
&\leq
(N^0)^{\sqrt{2}\delta^{3/2}}
\notag
\\
& \ \ \ \ \ \ \ \ \ \ \ \
\cdot
\sup_{|\lambda|<(N^0)^{\sqrt{2\delta}}}
\|(\mathcal{N}_{1}(\mathcal{D}(u_{5},u_{6},
u_{7},u_{8}),u_{2},u_{3},u_{4}))(n,\lambda + n^{3})\|
_{H^{\frac{1}{2}+\delta}_{x}}.
\label{Eqn:NL-1-4}
\end{align}
To exploit the restriction $|\tau_{j}-n_{j}^{3}|<(N^0)^{\sqrt{2\delta}}$ for each $j=2,3,4,6,7,8$, we will use the representation \eqref{Eqn:NL-neg1-24}.
That is, for each $j=2,3,4,6,7,8$, we consider
\begin{align*}
\hat{u}_{j}(n_{j},\tau_{j}) = \int_{|\lambda_{j}|<(N^0)^{\sqrt{2\delta}}}\langle \lambda_{j}\rangle^{-\frac{1}{2}+\delta}
c_{j}(\lambda_{j})a_{\lambda_{j}}(n_{j})\delta(\tau_{j}-n_{j}^{3}-\lambda_{j})d\lambda_{j},
\end{align*}
with $\sum_{n}\langle n\rangle^{1-2\delta}|a_{\lambda_j}(n)|^{2}=1$ and $c_{j}(\lambda_{j})=\Big(\sum_{n}\langle n\rangle^{1-2\delta}\langle \lambda_{j}\rangle^{1-2\delta}
|\hat{v}(n,n^{3}+\lambda)|^{2}\Big)^{\frac{1}{2}}$.
Also, by \eqref{Eqn:NL-neg1-25}, we have
\begin{align}
\int_{|\lambda_{j}|<(N^0)^{\sqrt{2\delta}}}
\langle\lambda_{j}\rangle^{-\frac{1}{2}+\delta}
c_{j}(\lambda_{j})
d\lambda_{j}
\lesssim (N^0)^{\sqrt{2}\delta^{3/2}}\|u_{j}\|
_{\frac{1}{2}-\delta,\frac{1}{2}-\delta,T}.
\label{Eqn:2A-13}
\end{align}

\vspace{0.1in}

\noindent $\bullet$ \textbf{CASE 2.b.i:} $n\neq n_{5}$.

\vspace{0.1in}

\noindent For fixed $n\in\mathbb{Z}$, $\tau\in\mathbb{R}$, we consider
\begin{align}
|\mathcal{N}_{1}&(\mathcal{D}(u_{5},
u_{6},u_{7},u_{8}),u_{2}
,u_{3},u_{4})(n,\tau)|
\notag
\\
&= \Big|\sum_{\substack{(n_1,n_2,n_3,n_4)\in \zeta(n)\\(n_5,n_6,n_7,n_8)\in \zeta(n_1)\\ n_5\neq n}}\int_{\tau=\tau_{2}+\cdots +\tau_{8}}
\frac{-n_{1}n_{5}g_{n_{5}}(\omega)}{|\sigma_{1}||n_{5}|}\delta(\tau_{5}-n_{5}^{3}) \prod_{\substack{j=2\\ j\neq 5}}^{8}\hat{u}_{j}(n_{j},\tau_{j})d\tau_{j}\Big|
\notag
\\
&= \Big|\sum_{\substack{(n_1,n_2,n_3,n_4)\in \zeta(n)\\(n_5,n_6,n_7,n_8)\in \zeta(n_1)\\ n_5\neq n}}\int_{\tau=\tau_{2}+\cdots +\tau_{8}}
\frac{-n_{1}n_{5}g_{n_{5}}(\omega)}{|\sigma_{1}||n_{5}|}\delta(\tau_{5}-n_{5}^{3}) \notag \\
&\ \ \ \ \ \ \ \ \ \ \ \ \ \
\ \ \ \ \ \ \ \ \ \ \ \prod_{\substack{j=2\\ j\neq 5}}^{8}
\int_{|\lambda_{j}|<(N^0)^{\sqrt{2\delta}}}\langle \lambda_{j}\rangle^{-\frac{1}{2}+\delta}
c_{j}(\lambda_{j})a_{\lambda_{j}}(n_{j})
\delta(\tau_{j}-n_{j}^{3}-\lambda_{j})d\lambda_{j}\Big|
\notag
\\
&= \Big|\sum_{\substack{(n_1,n_2,n_3,n_4)\in \zeta(n)\\(n_5,n_6,n_7,n_8)\in \zeta(n_1)\\ n_5\neq n}}
\int_{\substack{|\lambda_{j}|<(N^0)^{\sqrt{2\delta}}
\\ j=2,3,4,6,7,8}}
\frac{-n_{1}n_{5}
g_{n_{5}}(\omega)}{|\sigma_{1}||n_{5}|}
\delta(\tau-n_{5}^{3}-
\sum_{\substack{j=2\\ j\neq 5}}^{8}
(n_{j}^{3}+\lambda_{j}))
\notag \\
&\ \ \ \ \ \ \ \ \ \ \ \ \ \
\ \ \ \ \ \ \ \ \ \ \
\prod_{\substack{j=2\\ j\neq 5}}^{8}
\langle \lambda_{j}\rangle^{-\frac{1}{2}+\delta}
c_{j}(\lambda_{j})a_{\lambda_{j}}(n_{j})
d\lambda_{j}\Big|,
\label{Eqn:NL-1-3}
\end{align}
by Fubini, and integration in $\tau_{2},\ldots,\tau_{8}$, since
\begin{align*}
\int_{\tau=\tau_{2}+\cdots+\tau_{8}}
\delta(\tau_{5}-n_{5}^{3})
\prod_{\substack{j=2 \\ j\neq 5}}^{8}
\delta(\tau_{j}-n_{j}^{3}-\lambda_{j})d\tau_{j}
= \delta(\tau-n_{5}^{3}-
\sum_{\substack{j=2\\ j\neq 5}}^{8}
(n_{j}^{3}+\lambda_{j})).
\end{align*}
Then, for each fixed $\lambda\in\mathbb{R}$, we use \eqref{Eqn:NL-1-3} and apply the Minkowski inequality to bring the integral(s) in $\lambda_{j}$ outside of the $l^{2}_{n}$ norm.
\begin{align}
\|&(\mathcal{N}_{1}(\mathcal{D}(u_{5},u_{6},
u_{7},u_{8}),u_{2},u_{3},u_{4}))(n,\lambda + n^{3})\|
_{H^{\frac{1}{2}-\delta}_{x}}
 \notag \\
&=
\Big\||n|^{\frac{1}{2}+\delta}\sum_{\substack{(n_1,n_2,n_3,n_4)\in \zeta(n)\\(n_5,n_6,n_7,n_8)\in \zeta(n_1)\\ n_5\neq n}}
\int_{\substack{|\lambda_{j}|<(N^0)^{\sqrt{2\delta}}
\\ j=2,3,4,6,7,8}}
\frac{-n_{1}n_{5}
g_{n_{5}}(\omega)}{|\sigma_{1}||n_{5}|}
\delta(\lambda + n^{3}-n_{5}^{3}-
\sum_{\substack{j=2\\ j\neq 5}}^{8}
(n_{j}^{3}+\lambda_{j}))
\notag \\
&\ \ \ \ \ \ \ \ \ \ \ \ \ \
\ \ \ \ \ \ \ \ \ \ \ \ \ \ \ \
\ \ \ \ \ \ \ \ \ \ \
\cdot\prod_{\substack{j=2\\ j\neq 5}}^{8}
\langle \lambda_{j}\rangle^{-\frac{1}{2}+\delta}
c_{j}(\lambda_{j})a_{\lambda_{j}}(n_{j})
d\lambda_{j}\Big\|_{l^{2}_{n}}
\notag \\
&\lesssim
\int_{\substack{|\lambda_{j}|<(N^0)^{\sqrt{2\delta}}
\\ j=2,3,4,6,7,8}}\prod_{\substack{j=2\\ j\neq 5}}^{8}
\langle \lambda_{j}\rangle^{-\frac{1}{2}+\delta}
c_{j}(\lambda_{j})d\lambda_{j}
\notag \\
&\ \
\sup_{\substack{|\lambda_{j}|<(N^0)^{\sqrt{2\delta}}
\\ j=2,3,4,6,7,8}}\Big\||n|^{\frac{1}{2}+\delta}
\sum_{\substack{(n_1,n_2,n_3,n_4)\in \zeta(n)\\(n_5,n_6,n_7,n_8)\in \zeta(n_1)\\ n_5\neq n}}
\frac{-n_{1}n_{5}
g_{n_{5}}(\omega)}{|\sigma_{1}||n_{5}|}
\delta(\lambda + n^{3}-n_{5}^{3}-
\sum_{\substack{j=2\\ j\neq 5}}^{8}
(n_{j}^{3}+\lambda_{j}))
\notag \\
&\ \ \ \ \ \ \ \ \ \ \ \ \ \
\ \ \ \ \ \ \ \ \ \ \ \ \ \ \ \
\ \ \ \ \ \ \ \ \ \ \ \ \ \ \
\cdot\prod_{\substack{j=2\\ j\neq 5}}^{8}
a_{\lambda_{j}}(n_{j})
\Big\|_{l^{2}_{n}}
\notag \\
&\lesssim
(N^0)^{6\sqrt{2}\delta^{3/2}}\prod_{\substack{j=2\\ j\neq 5}}^{8}\|u_{j}\|_{\frac{1}{2}-\delta,
\frac{1}{2}-\delta,T}
\notag \\
&\ \
\sup_{\substack{|\lambda_{j}|<(N^0)^{\sqrt{2\delta}}
\\ j=2,3,4,6,7,8}}\Big\||n|^{\frac{1}{2}+\delta}
\sum_{\substack{(n_1,n_2,n_3,n_4)\in \zeta(n)\\(n_5,n_6,n_7,n_8)\in \zeta(n_1)\\ n_5\neq n}}
\frac{-n_{1}n_{5}
g_{n_{5}}(\omega)}{|\sigma_{1}||n_{5}|}
\delta(\lambda + n^{3}-n_{5}^{3}-
\sum_{\substack{j=2\\ j\neq 5}}^{8}
(n_{j}^{3}+\lambda_{j}))
\notag \\
&\ \ \ \ \ \ \ \ \ \ \ \ \ \
\ \ \ \ \ \ \ \ \ \ \ \ \ \ \ \
\ \ \ \ \ \ \ \ \ \ \ \ \ \ \
\cdot\prod_{\substack{j=2\\ j\neq 5}}^{8}
a_{\lambda_{j}}(n_{j})
\Big\|_{l^{2}_{n}}.
\label{Eqn:2A-14}
\end{align}
Note that we have applied \eqref{Eqn:2A-13} to obtain the last line.  
Now letting $\mu=\sum_{j=2,j\neq 5}^{8}\lambda_{j}-\lambda$, we find
\begin{align}
\|\mathcal{N}_{1}(\mathcal{D}&(u_{5},u_{6},
u_{7},u_{8}),u_{2},u_{3},u_{4})
\|_{\frac{1}{2}+\delta,-\frac{1}{2}+\delta,T}
\notag \\
&\lesssim
T^{\delta-}(N^0)^{8\sqrt{2}\delta^{3/2}}\prod_{\substack{j=2\\ j\neq 5}}^{8}\|u_{j}\|_{\frac{1}{2}-\delta,
\frac{1}{2}-\delta,T}
\notag \\
&\ \ \ \ \ \ \ \ \ \ \ \ \ \ \ \ \ \
\cdot \sup_{\mu<C(N^0)^{\sqrt{2\delta}}}
\Big\||n|^{\frac{1}{2}+\delta}
\sum_{\substack{*(n,\mu)\\ n\neq n_5}}
\frac{-n_{1}n_{5}
g_{n_{5}}(\omega)}{|\sigma_{1}||n_{5}|}
\prod_{\substack{j=2\\ j\neq 5}}^{8}
a_{n_{j}}
\Big\|_{l^{2}_{n}}
\label{Eqn:2A-15}
\end{align}
where $a_{n_{j}}:=a_{\lambda_{j}}(n_{j})$ (we have removed the dependance on $\lambda_{j}$ because our estimates will hold uniformly with respect to these parameters), and
\begin{align*}
*(n,\mu) = \Big\{(n_{2},n_{3},n_{4},&n_{5},n_{6},n_{7},n_{8})\in \mathbb{Z}^{7}:
(n_1,n_2,n_3,n_4)\in\zeta(n), \\
&\ \ \ \ \ \ (n_5,n_6,n_7,n_8)\in\zeta(n_1) \ \ \text{and}\ \
n^{3}-\sum_{j=2}^{8} n_{j}^{3} = \mu\Big\}.
\end{align*}
Let us highlight a crucial property of the set $*(n,\mu)$. Suppose the six variables $n_{2},n_{3},n_{4},n_{6},n_{7},n_{8}$ are fixed, and $(n_{2},\ldots,n_{8})\in *(n,\mu)$.  Then the relation $n=n_{2} + \cdots + n_{8}$ determines $n_{5}$ as a function of $n$, and $n$ satisfies
$n^{3}-\sum_{j=2}^{8} n_{j}^{3} = \mu$, which is a (nondegenerate, since $n\neq n_{5}$) quadratic equation in $n$ with at most 2 roots.  That is, in this subcase (with $n\neq n_5$), we can sum over the set of integers $\displaystyle \{(n,n_{2},n_{3},n_{4},n_{5},n_{6},n_{7},n_{8})\in \mathbb{Z}^{8}:
(n_{2},\ldots,n_{8})\in *(n,\mu)\}$ by summing over the six variables $n_{2},n_{3},n_{4},n_{6},n_{7},n_{8}$.  We will use this observation in the estimates that follow.

By bringing the absolute value inside and applying Lemma \ref{Lemma:prob1}, there exists $\Omega_{T}$ satisfying $P(\Omega_{T}^{c})<e^{-\frac{1}{T^{\beta}}}$ such that for each $\omega\in\Omega_{T}$ we have
\begin{align}
\Big\||n|^{\frac{1}{2}+\delta}
\sum_{\substack{*(n,\mu)\\n_5\neq n}}
&\frac{-n_{1}n_{5}
g_{n_{5}}(\omega)}{|\sigma_{1}||n_{5}|}
\prod_{\substack{j=2\\ j\neq 5}}^{8}
a_{n_{j}}
\Big\|_{l^{2}_{n}}
\\
&\lesssim  \frac{T^{-\beta/2}}{(N^0)^{\frac{1}{2}-\delta-\beta}}\sup_{|\mu|<C(N^{0})^{2}}
\Bigg(\sum_{|n|\sim N}
\Big|\sum_{\substack{*(n,\mu)\\n_5\neq n}}\prod_{\substack{j=2\\ j\neq 5}}^{8} |a_{n_{j}}|\Big|^{2} \Bigg)^{\frac{1}{2}},
\label{Eqn:2A-16}
\end{align}
where we have used the condition $|\sigma_{1}|\gtrsim (N^0)^{2}$ of the region $A_1$.  With repeated applications of Cauchy-Schwarz, we find
\begin{align}
\eqref{Eqn:2A-16} &=  \frac{T^{-\beta/2}}{(N^0)^{\frac{1}{2}-\delta-\beta}}\sup_{|\mu|<C(N^{0})^{2}}
\Bigg(\sum_{|n|\sim N}
\Big|\sum_{|n_{2}|\sim N_{2}}|a_{n_{2}}|
\sum_{\{(n_3,\ldots,n_8):(n_2,\ldots,n_8)\in *(n,\mu), n_5\neq n\}}
\prod_{\substack{j=3\\ j\neq 5}}^{8} |a_{n_{j}}|\Big|^{2} \Bigg)^{\frac{1}{2}}
\notag \\
&\leq  \frac{T^{-\beta/2}}{(N^0)^{\frac{1}{2}-\delta-\beta}}\sup_{|\mu|<C(N^{0})^{2}}
\Bigg(\sum_{|n|\sim N,|n_{2}|\sim N_{2}}
\frac{1}{|n_{2}|^{1-2\delta}}\Big|
\sum_{\{(n_3,\ldots,n_8):(n_2,\ldots,n_8)\in *(n,\mu), n_5\neq n\}}\prod_{\substack{j=3\\ j\neq 5}}^{7} |a_{n_{j}}|\Big|^{2} \Bigg)^{\frac{1}{2}} \notag \\
&\leq  \frac{T^{-\beta/2}}{(N^0)^{\frac{1}{2}-\delta-\beta}}\sup_{|\mu|<C(N^{0})^{2}}
\Bigg(\sum_{\substack{|n|\sim N,|n_{k}|\sim N_{k},\\ 2\leq k \leq 7, k\neq 5 \\ \{(n_5,n_{8}):(n_{2}
,\ldots,n_{8})\in *(n,\mu), n_5\neq n\}}} \prod_{\substack{j=2\\ j\neq 5}}^{8}\frac{1}{|n_{j}|^{1-2\delta}} \Bigg)^{\frac{1}{2}}  \notag \\
&=  \frac{T^{-\beta/2}}{(N^0)^{\frac{1}{2}-7\delta-\beta-3\gamma}}\sup_{|\mu|<C(N^{0})^{2}}
\Bigg(\sum_{\substack{|n_{k}|\sim N_{k}, 2\leq k \leq 8, k\neq 5 \\ \{(n,n_{5}):(n_{2}
,\ldots,n_{8})\in *(n,\mu),n\neq n_5\}}} \prod_{\substack{j=2\\ j\neq 5}}^{8}\frac{1}{|n_{j}|^{1+\gamma}} \Bigg)^{\frac{1}{2}} \notag \\
&\lesssim
\frac{T^{-\beta/2}}{(N^0)^{\frac{1}{2}-7\delta-\beta-3\gamma}}\sup_{|\mu|<C(N^{0})^{2}}
\Bigg(\sum_{|n_{k}|\sim N_{k}, k=2,3,4,6,7,8
} \prod_{\substack{j=2\\ j\neq 5}}^{8}\frac{1}{|n_{j}|^{1+\gamma}} \Bigg)^{\frac{1}{2}}
\notag \\
&\leq
\frac{T^{-\beta/2}}{(N^0)^{\frac{1}{2}-7\delta-\beta-3\gamma}}.
\label{Eqn:2A-17}
\end{align}
Notice that, in the 3rd line above, we have used the condition $n_5\neq -n_8$.  Indeed, as discussed above, for fixed $n,$ $n_2, n_3, n_4, n_6, n_7$ and $\mu$, $n_5$ is determined by $n_8$, which satisfies a non-degenerate (if $n_5\neq -n_8$) quadratic equation with at most two roots.  Also notice that we have used the same argument (with $n\neq n_5$) to avoid summation with respect to $n$ in the second last line.  The condition $n_5\neq -n_8$ holds because $n=n_{2}+\cdots + n_8$ and $|n_{k}|\ll |n|^{\sqrt{2\delta}}$ for $k=2,3,4,6,7,8$ gives $|n|\sim |n_5|$, and thus $n_5\neq -n_k$ for all $k=2,3,4,6,7,8$.

Combining \eqref{Eqn:2A-16}-\eqref{Eqn:2A-17}, the inequality \eqref{Eqn:2A-11} follows, and case 2.b.i. is complete.

\vspace{0.1in}

\noindent $\bullet$ \textbf{CASE 2.b.ii:} $n= n_{5}$.

Notice that, with the restrictions of this case, we are considering contributions to $\mathcal{N}_{1}(\mathcal{D}(u_5,u_6,u_7,u_8),u_2,u_3,u_4)$ from frequencies $(n,n_2,\ldots,n_8,\tau,\tau_2,\ldots,\tau_8)\in A_{1,c}$.  Therefore the definition of our nonlinearity in this case is given by \eqref{Eqn:NLdef2b}-\eqref{Eqn:NLdef2}.  Notice that, using $|\sigma|, |\sigma_{k}|, |n_k| \ll |n|^{\sqrt{2\delta}}$, we have
\begin{align}
&\frac{|n||-6n(n-n_{1})^{2}+\tau-\tau_5-2\sigma|}
{|3nn_{1}(n-n_{1})
+ \sigma-\tau_{6}-\tau_{7}-\tau_{8}|
|3nn_{1}(n-n_{1})
- \sigma+\tau_{2}+\tau_{3}+\tau_{4}|} \notag \\ & \quad\quad\lesssim \frac{1}{(N^0)^{2-\gamma}}
\label{Eqn:Ac1}
\end{align}
and
\begin{align}
\frac{|n_{2}+n_{3}+n_{4}|}{|\sigma_1|} \lesssim \frac{1}{(N^0)^{2-\gamma}}.
\label{Eqn:Ac2}
\end{align}
Substituting that $u_5$ is type (I), with \eqref{Eqn:NLdef2b}-\eqref{Eqn:NLdef2} this gives
\begin{align*}
|\mathcal{N}_{1}(\mathcal{D}(&u_{5},
u_{6},u_{7},u_{8}),u_{2}
,u_{3},u_{4})(n,\tau)|_{\text{Case 2.b.ii.}}|
\notag
\\
&\lesssim  \frac{1}{(N^0)^{2-2\gamma}}|g_{n}(\omega)|
\notag \\ &
\ \ \ \ \ \cdot\Bigg|\sum_{\substack{n_{2}+n_{3}+n_{4}+n_{6}+n_{7}+n_{8}=0 \\ |n_{j}|<(N^0)^{\alpha}}}
\int_{\substack{\tau-n^{3}=\tau_{2}
+\tau_{3}+\tau_{4}+\tau_{6}+\tau_{7}+\tau_{8}
\\ |\sigma_{j}|<(N^0)^{\alpha}}}
\prod_{\substack{j=2\\ j\neq 5}}^{8}
\hat{u}_{j}(n_{j},\tau_{j}) \Bigg|.
\end{align*}
Following the approach used in the previous case (see \eqref{Eqn:NL-1-3}-\eqref{Eqn:2A-15}), we find
\begin{align}
\|\mathcal{N}_{1}(\mathcal{D}&(u_{5},u_{6},
u_{7},u_{8}),u_{2},u_{3},u_{4})
\|_{\frac{1}{2}+\delta,-\frac{1}{2}+\delta}
\notag \\
&\lesssim
\frac{1}{(N^0)^{2-2\gamma-6\sqrt{2}\delta^{3/2}}}\prod_{\substack{j=2\\ j\neq 5}}^{8}\|u_{j}\|_{\frac{1}{2}-\delta,
\frac{1}{2}-\delta,T}
\notag \\
&\ \ \sup_{|\mu|<C(N^0)^{\sqrt{2\delta}}}\Big\||n|^{\frac{1}{2}+\delta}
|g_{n}(\omega)|
\sum_{\substack{*(n,\mu)\cap \{n=n_5\}\\  |n_{j}|<(N^0)^{\alpha}}}
\prod_{\substack{j=2\\ j\neq 5}}^{8}
a_{n_{j}}
\Big\|_{l^{2}_{n}}
\label{Eqn:2A-15}
\end{align}
where $a_{n_{j}}:=a_{\lambda_{j}}(n_{j})$ (we have removed the dependance on $\lambda_{j}$ because our estimates will hold uniformly with respect to these parameters).  Then by Lemma \ref{Lemma:prob1} and repeated applications of Cauchy-Schwarz, we have for $\omega\in \tilde{\Omega}_T$,
\begin{align}
\|&\mathcal{N}_{1}(\mathcal{D}(u_{5},u_{6},
u_{7},u_{8}),u_{2},u_{3},u_{4})
\|_{\frac{1}{2}+\delta,-\frac{1}{2}+\delta}
\notag \\
&\lesssim
\frac{T^{-\frac{\beta}{2}}}{(N^0)^{\frac{3}{2}-2\gamma-6\sqrt{2}\delta^{3/2}-2\beta}}\prod_{\substack{j=2\\ j\neq 5}}^{8}\|u_{j}\|_{\frac{1}{2}-\delta,
\frac{1}{2}-\delta,T}
\sup_{|\mu|<C(N^0)^{\sqrt{2\delta}}}\Big\|
\sum_{\substack{*(n,\mu)\cap \{n=n_5\}\\  |n_{j}|<(N^0)^{\alpha}}}
\prod_{\substack{j=2\\ j\neq 5}}^{8}
a_{n_{j}}
\Big\|_{l^{2}_{n}}
\notag \\
&\lesssim
\frac{T^{-\frac{\beta}{2}}}{(N^0)^{1-2\gamma-6\sqrt{2}\delta^{3/2}-2\beta}}\prod_{\substack{j=2\\ j\neq 5}}^{8}\|u_{j}\|_{\frac{1}{2}-\delta,
\frac{1}{2}-\delta,T}
\sup_{|n|\sim N,|\mu|<C(N^0)^{\sqrt{2\delta}}}\Big|
\sum_{\substack{*(n,\mu)\cap \{n=n_5\}\\  |n_{j}|<(N^0)^{\alpha}}}
\prod_{\substack{j=2\\ j\neq 5}}^{8}
a_{n_{j}}
\Big|
\notag \\
&\lesssim
\frac{T^{-\frac{\beta}{2}}}{(N^0)^{1-2\gamma-6\sqrt{2}\delta^{3/2}-2\beta}}\prod_{\substack{j=2\\ j\neq 5}}^{8}\|u_{j}\|_{\frac{1}{2}-\delta,
\frac{1}{2}-\delta,T}
\sup_{|n|\sim N}\Big|
\sum_{\substack{n_2+n_3+n_4+n_6+n_7+n_8=0\\  |n_{j}|<(N^0)^{\alpha}}}
\prod_{\substack{j=2\\ j\neq 5}}^{8}
|a_{n_{j}}|
\Big|
\notag \\
&\lesssim
\frac{T^{-\frac{\beta}{2}}}{(N^0)^{1-2\gamma-6\sqrt{2}\delta^{3/2}-2\beta}}\prod_{\substack{j=2\\ j\neq 5}}^{8}\|u_{j}\|_{\frac{1}{2}-\delta,
\frac{1}{2}-\delta,T}
\Big(\sum_{|n_k|\sim N_k,\ k=2,3,4,6,7,8}
\prod_{\substack{j=2\\ j\neq 5}}^{8}
\frac{1}{\langle n_{j}\rangle^{1-2\delta}}
\Big)^{\frac{1}{2}}
\notag \\
&\lesssim
\frac{T^{-\frac{\beta}{2}}}{(N^0)^{1-3\gamma-6\sqrt{2}\delta^{3/2}-\delta-2\beta}}
\prod_{\substack{j=2\\ j\neq 5}}^{8}\|u_{j}\|_{\frac{1}{2}-\delta,
\frac{1}{2}-\delta,T}
\Big(\sum_{|n_k|\sim N_k,\ k=2,3,4,6,7,8}
\prod_{\substack{j=2\\ j\neq 5}}^{8}
\frac{1}{\langle n_{j}\rangle^{1+2\gamma}}
\Big)^{\frac{1}{2}}
\notag \\
&\lesssim
\frac{T^{-\frac{\beta}{2}}}{(N^0)^{1-3\gamma-6\sqrt{2}\delta^{3/2}-\delta-2\beta}}
\prod_{\substack{j=2\\ j\neq 5}}^{8}\|u_{j}\|_{\frac{1}{2}-\delta,
\frac{1}{2}-\delta,T},
\label{Eqn:NL-1-7}
\end{align}
and \eqref{Eqn:2A-11} follows by dyadic summation.  This completes the analysis of case 2.b.ii.
We have therefore established \eqref{Eqn:NL-1-nlpart}, and \eqref{Eqn:NL-1b-nlpart} is justified with the exact same arguments.  The proof of Proposition \ref{Prop:NL-1-nlpart} is complete.
\end{proof}

\subsection{Deterministic nonlinear estimates}
\label{Sec:NLproof2}

In this subsection we present the proof of Proposition \ref{Prop:nonlin2}.  That is, we establish the deterministic estimates \eqref{Eqn:NL-0-a}-\eqref{Eqn:NL-2b}.  In this section we require the following calculus inequality:

\begin{lemma}
Let $0<\delta_{1}\leq\delta_{2}$ satisfy $\delta_{1}+\delta_{2}>1$, and let $a\in\mathbb{R}$, then
\[ \int_{-\infty}^{\infty}\frac{d\theta}{\langle\theta\rangle^{\delta_{1}}\langle a-\theta\rangle^{\delta_{2}}}
 \lesssim  \frac{1}{\langle a\rangle^{\alpha}},\]
where $\alpha=\delta_{1}-(1-\delta_{2})_{+}$.  Recall that $(\lambda)_{+}:= \lambda$ if $\lambda>0$, $= \varepsilon>0$ if $\lambda=0$, and $=0$ if $\lambda<0$.
\label{Lemma:integralineq}
\end{lemma}
The proof of Lemma \ref{Lemma:integralineq} can be found in \cite{GTV}.

\begin{proof}[Proof of Proposition \ref{Prop:nonlin2}:]

\vspace{0.1in}

We establish Proposition \ref{Prop:nonlin2} with $\delta_0=0$.  It is straight-forward to adapt the proof to $0<\delta_0<\delta$ (see Remark \ref{Rem:deltazero}). The ordering of inequalities \eqref{Eqn:NL-0-a} and \eqref{Eqn:NL-0-b} in the statement of Proposition \ref{Prop:nonlin2} is a little misleading; we will establish \eqref{Eqn:NL-0-b} first, then use it in the proof of \eqref{Eqn:NL-0-a}.  Our choice to order these inequalities as written was based on the instinct of discussing (the estimates needed for) existence before continuity.
We proceed with the proof of \eqref{Eqn:NL-0-b}.
In fact, we establish
\begin{align}
\|\mathcal{N}_{0}&(u_{1},u_{2},u_{3},u_{4})\|_
{Y^{\frac{1}{2}+\delta,-1}_{T}}
\lesssim
\prod_{j=1}^{4} \|u_{j}\|_{\frac{1}{2}-\delta,\frac{1}{2}-\frac{3\delta}{2},T}.
\label{Eqn:NL-0-e1}
\end{align}
Then \eqref{Eqn:NL-0-b} follows easily from \eqref{Eqn:NL-0-e1} and Lemma \ref{Lemma:gainpowerofT} with $\theta=2\delta-$.

Using $ \|\cdot\|_{Y^{\frac{1}{2}+\delta,-1}_{T}}\leq \|\chi_{[0,T]}(t)\cdot\|
_{Y^{\frac{1}{2}+\delta,-1}}$ (from the definition of the norm) and
$$\chi_{[0,T]}(t)\mathcal{N}_{0}(u_{1},u_{2},u_{3},u_{4})= \mathcal{N}_{0}(\chi_{[0,T]}(t)u_{1},\chi_{[0,T]}(t)u_{2},
\chi_{[0,T]}(t)u_{3},\chi_{[0,T]}(t)u_{4}),$$
it suffices to establish
\begin{align}
\|\mathcal{N}_{0}&(u_{1},u_{2},u_{3},u_{4})\|_
{Y^{\frac{1}{2}+\delta,-1}}
\lesssim T^{\theta}
\prod_{j=1}^{4} \|u_{j}\|_{\frac{1}{2}-\delta,\frac{1}{2}-\frac{3\delta}{2}},
\label{Eqn:NL-0-eq2}
\end{align}
where each $u_{j}$ satisfies $u_{j}=\chi_{[0,T]}(t)u_{j}$.  We proceed to establish \eqref{Eqn:NL-0-eq2} as written, introducing factors of $\chi_{[0,T]}(t)$ (in front of the $u_{j}$) when needed.

Let
\begin{align*}
f_{j}(n_{j},\tau_{j}):=\langle n_{j}
\rangle^{\frac{1}{2}-\delta}
\langle \tau_{j}-n_{j}^{3}
\rangle^{\frac{1}{2}-\frac{3\delta}{2}}
\widehat{u_{j}}(n_{j},\tau_{j})
\end{align*}
for each $j=1,2,3,4$.  To prove \eqref{Eqn:NL-0-b} it is sufficient to establish
\begin{align}
\Big\|\frac{\langle n\rangle^{\frac{1}{2}+\delta}}{\langle \sigma \rangle}
\sum_{\substack{n_{1},n_{2},n_{3}\\
n=n_{1}+\cdots+n_{4}}}
\int_{\substack{\tau_{1},\tau_{2},\tau_{3}\\
\tau=\tau_{1}+\cdots+\tau_{4}}}
&\chi_{A_{0}}\cdot|n_{1}|\prod_{j=1}^{4}\frac{|f_{j}(n_{j},\tau_{j})|}
{\langle n_{j}\rangle^{\frac{1}{2}-\delta}
\langle \sigma_{j}\rangle^{\frac{1}{2}-\frac{3\delta}{2}}}
\Big\|_{l^{2}_{n}L^{1}_{\tau}} \notag \\
&\lesssim
\prod_{j=1}^{4}\|f_{j}\|_{L^{2}_{n,\tau}}.
\label{Eqn:NL-2mod}
\end{align}

Using the condition $|\sigma|\gtrsim |n_{\text{max}}|^{2}$, we have
\begin{align}
\frac{|n|^{\frac{1}{2}+\delta}|n_{1}|}
{\langle\sigma\rangle^{1-6\delta-\gamma}
|n_{1}|^{\frac{1}{2}-\delta}}
\lesssim
\frac{1}{|n|^{1-17\delta-5\gamma}
\prod_{j=2}^{4}\langle n_{j}\rangle^{\delta+\gamma}}.
\label{Eqn:NL-0-3}
\end{align}
Applying \eqref{Eqn:NL-0-3}, and subsequently removing all restrictions in frequency space (which is allowed because we have brought absolute values inside), we have
\begin{align}
&\text{LHS of \eqref{Eqn:NL-2mod}} \notag \\
&\ \ \ \ \ \ \ \ \lesssim \Big\| \frac{1}{\langle\sigma \rangle^{6\delta+\gamma} |n|^{1-17\delta-5\gamma}}
\sum_{\substack{n_{1},n_{2},n_{3}\\
n=n_{1}+\cdots+n_{4}}}
\int_{\substack{\tau_{1},\tau_{2},\tau_{3}\\
\tau=\tau_{1}+\cdots+\tau_{4}}}\frac{|f_{1}(n_{1},\tau_{1})|}
{\langle \sigma_{1}\rangle^{\frac{1}{2}-\frac{3\delta}{2}}}
\prod_{j=2}^{4} \frac{|f_{j}(n_{j},\tau_{j})|}
{\langle n_{j}\rangle^{\frac{1}{2}+\gamma}\langle \sigma_{j}\rangle^{\frac{1}{2}-\frac{3\delta}{2}}}\Big\|_{l^{2}_{n}L^{1}_{\tau}}
\notag \\ &\ \ \ \ \ \ \ \ \leq \Big\|\frac{1}{\langle\sigma \rangle^{6\delta+\gamma}|n|^{1-17\delta-5\gamma}}
\Big(\sum_{
\substack{ n_{1},n_{2},n_{3}
\\ n=n_{1}+\cdots+n_{4}}}\int
_{
\substack{\tau_{1},\tau_{2},\tau_{3} \\
\tau=\tau_{1}+\cdots+\tau_{4}
 }}
 \prod_{j=1}^{4}|f(n_{j},\tau_{j})|^{2}
 \Big)^{\frac{1}{2}}
\notag \\
&\ \ \ \ \ \ \ \ \ \ \ \ \ \ \ \ \ \
\Big(\sum_{
\substack{ n_{1},n_{2},n_{3}
\\ n=n_{1}+\cdots+n_{4} }}\int
_{
\substack{\tau_{1},\tau_{2},\tau_{3} \\
\tau=\tau_{1}+\cdots+\tau_{4}
 }}
\frac{1}{\langle \sigma_{1}\rangle^{1-3\delta}}
 \prod_{j=2}^{4}\frac{1}{\langle n_{j}\rangle^{1+2\gamma}
\langle \sigma_{j}\rangle^{1-3\delta}}
 \Big)^{\frac{1}{2}}
\Big\|_{l^{2}_{n}L^{1}_{\tau}}
\label{Eqn:NL-0-1}
\end{align}
by Cauchy-Schwarz in
$n_{1},n_{2},n_{3},\tau_{1},\tau_{2},\tau_{3}$, for fixed $n,\tau$.  Next we fix $\tau,n,n_{1},n_{2},n_{3},$ and  repeatedly apply Lemma \ref{Lemma:integralineq} to obtain
\begin{align}
\int_{\tau_{1},\tau_{2},\tau_{3}}&
\frac{d\tau_{1}d\tau_{2}d\tau_{3}}{\langle \tau_{1}-n_{1}^{3}\rangle^{1-3\delta}
\langle \tau_{2}-n_{2}^{3}\rangle^{1-3\delta}
\langle \tau_{3}-n_{3}^{3}\rangle^{1-3\delta}
\langle \tau-\tau_{1}-\tau_{2}-\tau_{3}
-n_{4}^{3}\rangle^{1-3\delta}}
\notag \\
&\lesssim
\int_{\tau_{1},\tau_{2}}
\frac{d\tau_{1}d\tau_{2}}{\langle \tau_{1}-n_{1}^{3}\rangle^{1-3\delta}
\langle \tau_{2}-n_{2}^{3}\rangle^{1-3\delta}
\langle \tau-\tau_{1}-\tau_{2}-n_{3}^{3}-n_{4}^{3}\rangle^{1-6\delta}}
\notag \\
&\lesssim
\int_{\tau_{1}}
\frac{d\tau_{1}}{\langle \tau_{1}-n_{1}^{3}\rangle^{1-3\delta}
\langle \tau-\tau_{1}-n_{2}^{3}-n_{3}^{3}-n_{4}^{3}
\rangle^{1-9\delta}}
\notag \\
&\lesssim
\frac{1}
{\langle \tau-n_{1}^{3}-n_{2}^{3}-n_{3}^{3}-n_{4}^{3}
\rangle^{1-12\delta}}.
\label{Eqn:NL-0-2}
\end{align}
Using \eqref{Eqn:NL-0-2}, we have
\begin{align}
\eqref{Eqn:NL-0-1} &\lesssim
\Big\|\frac{1}{\langle\sigma \rangle^{6\delta+\gamma}|n|^{1-17\delta-5\gamma}} \Big(\sum_{
\substack{ n_{1},n_{2},n_{3}
\\ n=n_{1}+\cdots+n_{4}}}\int
_{
\substack{\tau_{1},\tau_{2},\tau_{3} \\
\tau=\tau_{1}+\cdots+\tau_{4}
 }}
 \prod_{j=1}^{4}|f(n_{j},\tau_{j})|^{2}
 \Big)^{\frac{1}{2}}
\notag \\
&\ \ \ \ \ \ \ \ \
\Big(\sum_{
\substack{ n_{1},n_{2},n_{3}
\\ n=n_{1}+\cdots+n_{4}}}
 \frac{1}{\langle \tau-n_{1}^{3}-\cdots -n_{4}^{3}
\rangle^{1-12\delta}}\prod_{j=2}^{4}\frac{1}{\langle n_{j}\rangle^{1+2\gamma}}
 \Big)^{\frac{1}{2}}
\Big\|_{l^{2}_{n}L^{1}_{\tau}}
\notag \\
&\lesssim
\Big\|
\Big(\sum_{
\substack{ n_{1},n_{2},n_{3}
\\ n=n_{1}+\cdots+n_{4}}}\int
_{
\substack{\tau_{1},\tau_{2},\tau_{3} \\
\tau=\tau_{1}+\cdots+\tau_{4}
 }}
 \prod_{j=1}^{4}|f(n_{j},\tau_{j})|^{2}
 \Big)^{\frac{1}{2}}\Big\|_{L^{2}_{n,\tau}}
\notag \\
&\ \
\sup_{n\neq 0}\frac{1}{|n|^{1-17\delta-5\gamma}}
\notag \\
&\
\cdot\Big(\int_{\tau}
\sum_{\substack{ n_{1},n_{2},n_{3}
\\ n=n_{1}+\cdots+n_{4}}}
 \frac{1}{
 \langle\sigma\rangle^{12\delta+2\gamma} \langle\tau-n_{1}^{3}-\cdots
 -n_{4}^{3}\rangle^{1-12\delta}}
 \prod_{j=2}^{4}\frac{1}{\langle n_{j}\rangle^{1+2\gamma}}
 \Big)^{\frac{1}{2}}.
\label{Eqn:NL-0-4}
\end{align}
In the last line we applied Cauchy-Schwarz in $\tau$, and took out the supremum in $n$ afterward.  Applying Fubini we compute that
\begin{align}
\Big\|
\Big(\sum_{
\substack{ n_{1},n_{2},n_{3}
\\ n=n_{1}+\cdots+n_{4}}}\int
_{
\substack{\tau_{1},\tau_{2},\tau_{3} \\
\tau=\tau_{1}+\cdots+\tau_{4}
 }}
 \prod_{j=1}^{4}|f(n_{j},\tau_{j})|^{2}
 \Big)^{\frac{1}{2}}\Big\|_{L^{2}_{n,\tau}}
= \prod_{j=1}^{4}\|f_{j}\|_{L^{2}_{n,\tau}}.
\label{Eqn:NL-0-5}
\end{align}
It remains to estimate the second factor in
\eqref{Eqn:NL-0-4}.
We change the order of integration and summation inside the supremum, and integrate in $\tau$, for fixed $n, n_{1},n_{2},n_{3}$. Since we have
\begin{align}
\int_{\tau}
 \frac{1}{\langle \tau-n^{3}\rangle^{12\delta+2\gamma} \langle\tau-n_{1}^{3}-\cdots
 -n_{4}^{3}\rangle^{1-12\delta}}
 \lesssim \frac{1}{\langle n^{3}-n_{1}^{3}-\cdots -n_{4}^{3}\rangle^{2\gamma}} \leq 1
\label{Eqn:NL-0-6}
\end{align}
by Lemma \ref{Lemma:integralineq}, this gives
\begin{align}
\sup_{n\neq 0}\frac{1}{|n|^{1-17\delta-5\gamma}}\Big(\int_{\tau}&
\sum_{\substack{ n_{1},n_{2},n_{3}
\\ n=n_{1}+\cdots+n_{4}}}
 \frac{1}{
 \langle\sigma\rangle^{12\delta+2\gamma} \langle\tau-n_{1}^{3}-\cdots
 -n_{4}^{3}\rangle^{1-12\delta}}
 \prod_{j=2}^{4}\frac{1}{\langle n_{j}\rangle^{1+2\gamma}}
 \Big)^{\frac{1}{2}}
\notag \\
&\lesssim
\sup_{n\neq 0}\frac{1}{|n|^{1-17\delta-5\gamma}}\Big(
\sum_{\substack{ n_{1},n_{2},n_{3}
\\ n=n_{1}+\cdots+n_{4}}}
 \prod_{j=2}^{4}\frac{1}{\langle n_{j}\rangle^{1+2\gamma}}
 \Big)^{\frac{1}{2}}
\notag \\
&\leq 1.
\label{Eqn:NL-0-9}
\end{align}
Combining \eqref{Eqn:NL-0-1}, \eqref{Eqn:NL-0-4}, \eqref{Eqn:NL-0-5}
and \eqref{Eqn:NL-0-9}, we obtain the estimate \eqref{Eqn:NL-2mod}.  The proof of \eqref{Eqn:NL-0-b} is complete.

Next we establish \eqref{Eqn:NL-0-a}.  
We will prove \eqref{Eqn:NL-0-a} using linear estimates, \eqref{Eqn:NL-0-b}, and the following estimate on the nonlinearity,
\begin{align}
\|\mathcal{N}_{0}(u_{1},u_{2},u_{3},u_{4})\|_
{\frac{1}{2}+\delta,-\frac{1}{2}-\delta,T}
\lesssim
\prod_{j=1}^{4} \|u_{j}\|_{\frac{1}{2}-\delta,\frac{1}{2}-
\frac{3\delta}{2},T}.
\label{Eqn:NL-0-a2}
\end{align}
That is, \eqref{Eqn:NL-0-a} follows from Lemma \ref{Lemma:lin2}, \eqref{Eqn:NL-0-a2}, \eqref{Eqn:NL-0-b} and Lemma \ref{Lemma:gainpowerofT}:
\begin{align*}
\|\mathcal{D}_{0}(u_{1},u_{2},u_{3},u_{4})\|_
{\frac{1}{2}+\delta,\frac{1}{2}-\delta,T}
&\lesssim \|\mathcal{N}_{0}(u_{1},u_{2},u_{3},u_{4})\|_
{\frac{1}{2}+\delta,-\frac{1}{2}-\delta,T}
+ \|\mathcal{N}_{0}(u_{1},u_{2},u_{3},u_{4})\|_
{Y^{\frac{1}{2}+\delta,-1-\delta}_{T}}
\notag \\
&\lesssim \prod_{j=1}^{4}
\|u_{j}\|_{\frac{1}{2}-\delta,
\frac{1}{2}-\frac{3\delta}{2},T}
+ \|\mathcal{N}_{0}(u_{1},u_{2},u_{3},u_{4})\|_
{Y^{\frac{1}{2}+\delta,-1}_{T}}
\notag \\
&\lesssim T^{2\delta-}\prod_{j=1}^{4}
\|u_{j}\|_{\frac{1}{2}-\delta,
\frac{1}{2}-\delta,T}.
\end{align*}
We proceed to justify
\eqref{Eqn:NL-0-a2}, which is equivalent to:
\begin{align*}
\|\chi_{[0,T]}(t)\mathcal{N}_{0}&(u_{1},u_{2},u_{3},u_{4})
\|_{\frac{1}{2}+\delta,-\frac{1}{2}-\delta}
\lesssim
\prod_{j=1}^{4} \|\chi_{[0,T]}(t)u_{j}\|_{\frac{1}{2}-\delta,\frac{1}{2}-\frac{3\delta}{2}}.
\end{align*}
Then since $\chi_{[0,T]}(t)\mathcal{N}_{0}(u_{1},u_{2},u_{3},u_{4})= \mathcal{N}_{0}(\chi_{[0,T]}(t)u_{1},\chi_{[0,T]}(t)u_{2},
\chi_{[0,T]}(t)u_{3},\chi_{[0,T]}(t)u_{4})$,
it suffices to establish
\begin{align}
\|\mathcal{N}_{0}&(u_{1},u_{2},u_{3},u_{4})\|_
{\frac{1}{2}+\delta,-\frac{1}{2}-\delta}
\lesssim T^{\theta}
\prod_{j=1}^{4} \|u_{j}\|_{\frac{1}{2}-\delta,\frac{1}{2}-\frac{3\delta}{2}},
\label{Eqn:NL-0-eq1}
\end{align}
where each $u_{j}$ satisfies $u_{j}=\chi_{[0,T]}(t)u_{j}$.  We proceed to establish \eqref{Eqn:NL-0-eq1} as written, introducing factors of $\chi_{[0,T]}(t)$ (in front of the $u_{j}$) when needed.

Using the condition $|\sigma|\gtrsim |n_{\text{max}}|^{2}$, we have
\begin{align}
\frac{\langle n\rangle^{\frac{1}{2}+\delta}|n_{1}|}
{\langle \sigma\rangle^{\frac{1}{2}+\delta}
\langle n_{1}\rangle^{\frac{1}{2}-\delta}} \lesssim 1.
\label{Eqn:D0bound}
\end{align}
Define the function $w_{1}:=
\big(\langle n_{1}\rangle^{\frac{1}{2}-\delta}
\widehat{u_{1}}(n_{1},\tau_{1})\big)^{\vee}$.  By using \eqref{Eqn:D0bound}, subsequently removing all restrictions in frequency space, and applying Plancherel, we find
\begin{align}
\|\mathcal{N}_{0}(u_{1},u_{2},u_{3},u_{4})\|
_{\frac{1}{2}+\delta,-\frac{1}{2}-\delta}
&= \Big\| \frac{\langle n\rangle^{\frac{1}{2}+\delta}}{\langle \sigma \rangle^{\frac{1}{2}+\delta}}
\sum_{\zeta(n)}
\int_{\tau=\tau_{1}+\cdots +\tau_{4}}
\chi_{A_{0}}(in_{1})\prod_{j=1}^{4}\widehat{u_{j}}(n_{j},\tau_{j})
\Big\|_{L^{2}_{n,\tau}}
\notag \\
&\lesssim
\Big\|
\sum_{\zeta(n)}
\int_{\tau=\tau_{1}+\cdots +\tau_{4}}
\chi_{A_{0}}\widehat{w_{1}}(n_{1},\tau_{1})
\prod_{j=2}^{4}\widehat{u_{j}}(n_{j},\tau_{j})
\Big\|_{L^{2}_{n,\tau}}
\notag \\
&\leq
\Big\|
\sum_{n=n_{1}+\cdots +n_{4}}
\int_{\tau=\tau_{1}+\cdots +\tau_{4}}
\widehat{w_{1}}(n_{1},\tau_{1})
\prod_{j=2}^{4}\widehat{u_{j}}(n_{j},\tau_{j})
\Big\|_{L^{2}_{n,\tau}}
\notag \\
&=
\|w_{1}u_{2}u_{3}u_{4}\|
_{L^{2}_{x,t}}.
\label{Eqn:NL-0-11}
\end{align}
To prove \eqref{Eqn:NL-0-a2} it now suffices to show that
\begin{align}
\|w_{1}u_{2}u_{3}u_{4}\|
_{L^{2}_{x,t}}
\lesssim
\|w_{1}\|_{0,\frac{1}{2}-\delta}
\prod_{j=2}^{4}
\|u_{j}\|_{\frac{1}{2}-\delta,\frac{1}{2}-\delta},
\label{Eqn:NL-0-10}
\end{align}
where $w_{1}=\chi_{[0,T]}(t)w_{1}$ and
$u_{j}=\chi_{[0,T]}(t)u_{j}$ for each $j$.
\noindent The inequality \eqref{Eqn:NL-0-10} is obtained with H\"{o}lder's inequality,
\eqref{Eqn:X-Str-1}, \eqref{Eqn:X-onehalfminus} and Lemma \ref{Lemma:gainpowerofT}:
\begin{align*}
\|w_{1}u_{2}u_{3}u_{4}\|_{L^{2}_{x,t}} &\leq \|w_{1}\|_{L^{4}_{x,t}}\prod_{j=2}^{4}
\|u_{j}\|_{L^{12}_{x,t}} \\
&\leq  \|w_{1}\|_{0,\frac{1}{3}}\prod_{j=2}^{4}
\|u_{j}\|_{\frac{1}{2}-\delta,\frac{1}{2}-\delta},
\end{align*}
with $\delta<\frac{1}{12}$.  The proof of \eqref{Eqn:NL-0-a2} (and thus of \eqref{Eqn:NL-0-a}) is complete.

\vspace{0.1in}

Next we justify \eqref{Eqn:NL-k-1} and \eqref{Eqn:NL-k-2}.  Using the condition $|\sigma_1|\gtrsim |n_{\text{max}}|^2$, we have
\begin{align}
\frac{|n|^{\frac{1}{2}+\delta}|n_1|}{|n_1|^{\frac{1}{2}-\delta}|\sigma_1|^{\frac{1}{2}+\delta}}\lesssim 1.
\label{Eqn:lasthope2}
\end{align}
Using \eqref{Eqn:lasthope2} and duality, \eqref{Eqn:NL-k-1} follows from
\begin{align}
\int v\cdot f_1 u_2 u_3 u_4 dxdt \lesssim \|v\|_{0,\frac{1}{2}-\delta,T}\|f_1\|_{L^2_{x,t\in[0,T]}}
\prod_{j=2}^{4}
\|u_j\|_{\frac{1}{2}-\delta,\frac{1}{2}-\delta,T}.
\label{Eqn:lasthope2}
\end{align}
Then \eqref{Eqn:lasthope2} is established using H\"{o}lder, \eqref{Eqn:X-onehalfminus}, \eqref{Eqn:X-Str-1} and \eqref{Eqn:X-Str-interp} (as in various cases above).

For \eqref{Eqn:NL-k-2}, we use the condition $|\sigma_2|\gtrsim |n_{\text{max}}|^2$ and find
\begin{align}
\frac{|n|^{\frac{1}{2}}|n_1|}{|n_1|^{\frac{1}{2}-\delta}|\sigma_2|^{\frac{1}{2}+\delta}}\lesssim \frac{1}{|n|^{\delta}}.
\label{Eqn:lasthope3}
\end{align}
Using \eqref{Eqn:lasthope3} and duality, \eqref{Eqn:NL-k-2} follows from
\begin{align}
\int v\cdot f_1 u_2 u_3 u_4 dxdt \lesssim \|v\|_{\delta,\frac{1}{2}-\delta,T}\|f_1\|_{0,\frac{1}{2}-\delta,T}
\|u_2\|_{L^2_{x,t\in[0,T]}}\prod_{j=3}^{4}
\|u_j\|_{\frac{1}{2}-\delta,\frac{1}{2}-\delta,T}.
\label{Eqn:lasthope2}
\end{align}
Then \eqref{Eqn:lasthope2} is established using H\"{o}lder, \eqref{Eqn:X-onehalfminus}, \eqref{Eqn:X-Str-1} and \eqref{Eqn:X-Str-interp} (as in various cases above).

We turn to the justification of \eqref{Eqn:NL-2}.  We consider
\begin{align}
\|\mathcal{N}_{2}(u_{1},\mathcal{D}&(u_{5},u_{6},u_{7},u_{8}),
u_{3},u_{4})\|_{\frac{1}{2}+\delta,-\frac{1}{2}+\delta} \notag \\ &=\Big\|\frac{|n|^{\frac{1}{2}+\delta}}
{\langle \sigma\rangle^{\frac{1}{2}-\delta}}\sum_{n=n_{2}+\cdots +n_{8}}\int_{\tau=\tau_{2}+\cdots +\tau_{4}}\chi_{A_{2}}
\frac{-n_{1}n_{5}}{\sigma_{2}} \prod_{j=1,j\neq 2}^{8}\hat{u}_{j}(n_{j},\tau_{j})d\tau_{j}
\Big\|_{L^{2}_{n,\tau}}.
\label{Eqn:NL-1-2}
\end{align}
Using $|\sigma_{2}|\gtrsim |n_{\text{max}}|^{2}$ we have
\begin{align}
\frac{|n|^{\frac{1}{2}+\delta}|n_{1}||n_{5}|}
{|\sigma_{2}||n_{1}|^{\frac{1}{2}-\delta}|n_{5}|^{2\delta}}  \lesssim  1.
\label{Eqn:NL-1-11}
\end{align}
Using \eqref{Eqn:NL-1-2} and \eqref{Eqn:NL-1-11}, \eqref{Eqn:NL-2} follows from
\begin{align}
\Big\|\prod_{j=2}^{8}f_{j}\Big\|_{L^{2}_{x,t}}
\lesssim T^{2\delta-}\|f_{1}\|_{0,\frac{1}{2}-\delta}
\|f_{5}\|_{\frac{1}{2}-3\delta,\frac{1}{2}-\delta}
\prod_{\substack{j=3\\j\neq 5,k}}^{8}
\|f_{j}\|_{\frac{1}{2}-\delta,\frac{1}{2}-\delta}.
\label{Eqn:NL-1-10}
\end{align}
Then \eqref{Eqn:NL-1-10} is obtained with H\"{o}lder,
\eqref{Eqn:X-onehalfminus}, \eqref{Eqn:X-Str-1} and Lemma \ref{Lemma:gainpowerofT},
\begin{align*}
\Big\|\prod_{j=2}^{8}f_{j}\Big\|_{L^{2}_{x,t}}
&\lesssim
\|f_{1}\|_{L^{4}_{x,t}} \prod_{\substack{j=3\\j\neq 5}}^{8} \|f_{j}\|_{L^{24}_{x,t}}  \notag \\
&\lesssim  \|f_{1}\|_{0,\frac{1}{3}} \prod_{\substack{j=3\\j\neq 5}}^{8} \|f_{j}\|_{\frac{1}{2}-3\delta,\frac{1}{2}-3\delta}
\notag \\
&\lesssim  T^{\theta} \|f_{1}\|_{0,\frac{1}{2}-\delta}
\|f_{5}\|_{\frac{1}{2}-3\delta,\frac{1}{2}-\delta}
 \prod_{\substack{j=3\\j\neq 5}}^{8} \|f_{j}\|_{\frac{1}{2}-\delta,\frac{1}{2}-\delta}.
\end{align*}
We have established \eqref{Eqn:NL-2}, and \eqref{Eqn:NL-2b} can be obtained with the same argument.  The proof of Proposition \ref{Prop:nonlin2} is complete.
\end{proof}

\end{document}